\documentclass{article}
\usepackage{amsmath,amssymb,bbm,theorem}

\newcommand{\assign}{:=}
\newcommand{\cdummy}{\cdot}
\newcommand{\longhookrightarrow}{{\lhook\joinrel\relbar\joinrel\rightarrow}}
\numberwithin{equation}{section} 
\newcommand{\nequiv}{\not\equiv}
\newcommand{\nocomma}{}
\newcommand{\noplus}{}
\newcommand{\nosymbol}{}
\newcommand{\tmem}[1]{{\em #1\/}}
\newcommand{\tmname}[1]{\textsc{#1}}
\newcommand{\tmop}[1]{\ensuremath{\operatorname{#1}}}
\newcommand{\tmsamp}[1]{\textsf{#1}}
\newcommand{\tmstrong}[1]{\textbf{#1}}
\newcommand{\tmtextbf}[1]{{\bfseries{#1}}}
\newcommand{\tmtexttt}[1]{{\ttfamily{#1}}}
\newenvironment{proof}{\noindent\textbf{Proof\ }}{\hspace*{\fill}$\Box$\medskip}
\newtheorem{corollary}{Corollary}
\newtheorem{definition}{Definition}
\newtheorem{lemma}{Lemma}
\newtheorem{proposition}{Proposition}
{\theorembodyfont{\rmfamily}\newtheorem{remark}{Remark}}
\newtheorem{theorem}{Theorem}
\newtheorem{Conjecture}{Conjecture}

\newcommand{\XXint}[3]{{\setbox}0=\text{\ensuremath{#1 #2 #3 \int}}
{\vcenter{\text{\ensuremath{#2 #3}}}}{\kern}-.5{\tmwd}0}

\newcommand{\opn}[2]{\newcommand{\1}{\}} {\opn}{\Rm{Rm}} {\opn}{\Ric{Ric}}
{\opn}{\Rc{Rc}} {\opn}{\Scal{Sc}} {\opn}{\Tr{Tr}} {\opn}{\Trac{Tr}}
{\opn}detdet {\opn}{\diam{diam}} {\opn}{\dist{dist}} {\opn}{\Im}Im
{\opn}{\div}div {\opn}{\Ker{Ker}} {\opn}expexp {\opn}{\Vol{Vol}}
{\opn}{\exph{exph}} {\opn}{\Herm{Herm}} {\opn}{\End{End}} {\opn}{\Hess{Hess}}
{\opn}{\Vol{Vol}}}

\newcommand{\R}{\mathbb{R}}

\newcommand{\contract}{{\kern}-1.5pt{\vrule} width6.0pt height0.4pt depth0pt
{\vrule} width0.4pt height4.0pt depth0pt}
\newcommand{\retract}{{\kern}-1.5pt{\vrule} width0.4pt height4.0pt depth0pt
{\vrule} width6.0pt height0.4pt depth0pt}
\newcommand{\Openbox}{{\leavevmode} {\hfil}{\vrule} width{\boxrulethickness}
{\vbox} to{\Openboxwidth{{\advance}{\Openboxwidth} -2{\boxrulethickness}
{\hrule} height {\boxrulethickness} width{\Openboxwidth}{\vfil} {\hrule}
height{\boxrulethickness}}}{\vrule} width{\boxrulethickness}{\hfil} }

\begin{document}

\title{The Soliton-Ricci Flow with variable volume forms}\author{\\
{\tmstrong{NEFTON PALI}}}\date{}\maketitle

\begin{abstract}
  We introduce a flow of Riemannian metrics and positive volume forms over
  compact oriented manifolds whose formal limit is a shrinking Ricci soliton.
  The case of a fixed volume form has been considered in our previous work. We
  still call this new flow, the Soliton-Ricci flow. It corresponds to a forward
  Ricci type flow up to a gauge transformation. This gauge is generated by the gradient of
  the density of the volumes. The new Soliton-Ricci flow exist for all times.
  It represents the gradient flow of Perelman's $\mathcal{W}$ functional with
  respect to a pseudo-Riemannian structure over the space of metrics and
  normalized positive volume forms. We obtain an expression of the Hessian of
  the $\mathcal{W}$ functional with respect to such structure. Our expression
  shows the elliptic nature of this operator in the orthogonal directions to the
  orbits obtained by the action of the group of diffeomorphism. In the case
  that initial data is K\"ahler, the Soliton-Ricci flow over a Fano manifold preserves the
  K\"ahler condition and the symplectic form. Over a Fano manifold, the space of tamed complex
  structures embeds naturally, via the Chern-Ricci map, into the space of metrics and normalized positive
  volume forms. Over such space the pseudo-Riemannian
  structure restricts to a Riemannian one. We perform a study of the sign of
  the restriction of the Hessian of the $\mathcal{W}$ functional over such
  space. This allows us to obtain a finite dimensional reduction of the stability problem for
  K\"ahler-Ricci solitons. This reduction represents the solution of this well known problem.
\end{abstract}
\section{Introduction and statement of the main result}

This is the first of a series of papers whose purpose is the study the following
problem.

Let $(X, J_0)$ be a Fano manifold. We remind that the first Chern class $c_1
(X, \left[ J_0 \right]) \in H^2_d (X, \mathbbm{R})$ depends only on $X$ and
the coboundary class $\left[ J_0 \right]$ of the complex structure $J$.

Let also $\omega \in 2 \pi c_1 (X, \left[ J_0 \right])$ be an arbitrary
$J_0$-invariant K\"ahler form over $X$. We want to find under which conditions
on $J_0$ and $\omega$ there exists a smooth complex structure $J \in \left[
J_0 \right]$ and a smooth volume form $\Omega > 0$ over $X$ such that
\[ \left\{ \begin{array}{l}
     \omega  = \tmop{Ric}_J (\Omega) \hspace{0.25em},\\
     \\
     \overline{\partial}_{T_{X, J}}  \left( \omega^{- 1} d \log
     \frac{\;\omega^n}{\Omega} \right)  =  0
     \hspace{0.25em},
   \end{array} \right. \]
i.e. the Riemannian metric $g \assign - \omega J$, is a $J$-invariant
K\"ahler-Ricci soliton.

This set up represents a particular case of the Hamilton-Tian conjecture with
a stronger conclusion. Namely we avoid the singularities in the solution of
the K\"ahler-Ricci soliton equation.

Proofs of the Hamilton-Tian conjecture have been posted on
the arXiv server in (2013) by Tian-Zhang \cite{Ti-Zha} in complex dimension 3 and
quite recently by Chen-Wang \cite{Ch-Wa} in arbitrary dimensions.

Our starting point of view is Perelman's twice contracted second Bianchi type
identity introduced in \cite{Per}.

We remind first what this is about. Let $\Omega > 0$ be a smooth volume form
over an oriented compact and connected Riemannian manifold $(X, g)$. We remind
that the $\Omega$-Bakry-Emery-Ricci tensor of $g$ is defined by the formula
\[ \tmop{Ric}_g (\Omega) \hspace{0.75em} : = \hspace{0.75em} \tmop{Ric} (g)
   \hspace{0.75em} + \hspace{0.75em} \nabla_g d \log \frac{dV_g}{\Omega} . \]
A Riemannian metric $g$ is called a $\Omega$-Shrinking Ricci soliton if $g = \tmop{Ric}_g (\Omega)$. 
We equip the set of smooth Riemannian metrics $\mathcal{M}$ with the scalar
product
\begin{equation}
  \label{Glb-Rm-m}  (u, v) \longmapsto \int_X \left\langle \hspace{0.25em} u,
  v \right\rangle_g \Omega,
\end{equation}
for all $u, v \in \mathcal{H} : = L^2 (X, S^2_{\mathbbm{R}} T_X^{\ast})$. Let
$P_g^{\ast}$ be the formal adjoint of some operator $P$ with respect to a
metric $g$. We observe that the operator
\begin{eqnarray*}
  P^{\ast_{\Omega}}_g & : = & e^f P^{\ast}_g  \left( e^{- f} \bullet \right),
\end{eqnarray*}
with $f \assign \log \frac{d V_g}{\Omega}$ , is the formal adjoint of $P$ with
respect to the scalar product (\ref{Glb-Rm-m}). We define also the
$\Omega$-Laplacian operator
\begin{eqnarray*}
  \Delta^{\Omega}_g & \assign & \nabla_g^{\ast_{\Omega}} \nabla_g = \Delta_g +
  \nabla_g f \neg \nabla_g .
\end{eqnarray*}
It is also useful to introduce the $\Omega$-divergence operator acting on
vector fields as follows:
\begin{eqnarray*}
  \tmop{div}^{\Omega} \xi & \assign & \frac{d (\xi \neg \Omega)}{\Omega} = e^f
  \tmop{div}_g \left( e^{- f} \xi \right) = \tmop{div}_g \xi - g \left( \xi,
  \nabla_g f \right) .
\end{eqnarray*}
(We denote by $\neg$ the contraction operator). We infer in particular the identity $\tmop{div}^{\Omega} \nabla_g u = -
\Delta^{\Omega}_g u$, for all functions $u$. We observe also the integration
by parts formula
\begin{eqnarray*}
  - \int_X u \tmop{div}^{\Omega} \xi \,\Omega & = & \int_X g (\nabla_g u, \xi)\,
  \Omega .
\end{eqnarray*}
We define now the following fundamental
objects
\begin{eqnarray*}
  h & \equiv & h_{g, \Omega} : = \tmop{Ric}_g (\Omega) - g \nocomma,\\
  &  & \\
  2 H & \equiv & 2 H_{g, \Omega} \assign - \Delta^{\Omega}_g f \noplus \noplus
  + \tmop{Tr}_g h \noplus + 2 f,\\
  &  & \\
  f & \assign & \log \frac{d V_g}{\Omega} .
\end{eqnarray*}
An elementary computation made by Perelman \cite{Per} (see also \cite{Pal2}) shows that the maps $h$ and $H$ 
satisfy {\tmstrong{Perelman's
twice contracted second Bianchi type identity}}
\begin{equation}
  \label{II-contr-Bianchi} \nabla_g^{\ast_{\Omega}} h_{g, \Omega}^{\ast}
  \noplus \noplus + \nabla_g H_{g, \Omega} = 0,
\end{equation}
where $h_{g, \Omega}^{\ast} \assign g^{- 1} h_{g, \Omega}$ is the endomorphism
associated to $h_{g, \Omega}$. We remind now that for any symmetric $2$-tensor
$u$ the tensor $\mathcal{R}_g \ast u$, defined by the formula
\begin{eqnarray*}
  ( \mathcal{R}_g \ast u) (\xi, \eta) & \assign & - \tmop{Tr}_g \left[ u
  \left( \mathcal{R}_g \left( \xi, \cdot \left) \eta, \cdot \right) \right],
  \right. \right.
\end{eqnarray*}
is also symmetric (see section \ref{main-res}). For any smooth symmetric $2$-tensor $u$ we
define the $\Omega$-Lichnerowicz Laplacian $\Delta^{\Omega}_{L, g}$ as
\begin{eqnarray*}
  \Delta^{\Omega}_{L, g} u & \assign & \Delta^{\Omega}_g u - 2\mathcal{R}_g
  \ast u + u \tmop{Ric}^{\ast}_g (\Omega) + \tmop{Ric}_g (\Omega) u_g^{\ast} .
\end{eqnarray*}
This operator is self-adjoint with respect to the scalar product
(\ref{Glb-Rm-m}) thanks to the identity
\begin{equation}
  \left\langle \mathcal{R}_g \ast u, v \right\rangle_g = \left\langle u,
  \mathcal{R}_g \ast v \right\rangle_g, \label{SymRmOp}
\end{equation}
for all symmetric $2$-tensors $u$ and $v$ (see section \ref{main-res}). We define also the set of normalized volume forms
$\mathcal{V}_1 \assign \left\{ \Omega > 0 \mid \int_X \Omega = 1 \right\}$.
>From now on we consider the maps $h$ and $H$ over
$\mathcal{M} \times \mathcal{V}_1$. Notice that the tangent space of
$\mathcal{M} \times \mathcal{V}_1$ is $T_{\mathcal{M} \times \mathcal{V}_1} =
C^{\infty} (X, S^2 T^{\ast}_X) \oplus C^{\infty} (X, \Lambda^m T^{\ast}_X)_0$,
where $m=\dim_{\R} X$ and
\begin{eqnarray*}
  C^{\infty} (X, \Lambda^m T^{\ast}_X)_0 & \assign & \left\{ V \in C^{\infty}
  (X, \Lambda^m T^{\ast}_X) \mid \int_X V = 0 \right\} .
\end{eqnarray*}
We denote by $\tmop{End}_g \left( T_X \right)$ the bundle of $g$-symmetric
endomorphisms of $T_X$ and by $C_{\Omega}^{\infty} (X, \mathbbm{R})_0$ the
space of smooth functions with zero integral with respect to $\Omega$. 
We will systematically use the fact that for any $\left( g, \Omega \right)
\in \mathcal{M} \times \mathcal{V}_1$ the tangent space $T_{\mathcal{M} \times
\mathcal{V}_1, \left( g, \Omega \right)}$ identifies with $C^{\infty} (X,
\tmop{End}_g \left( T_X \right)) \oplus C_{\Omega}^{\infty} (X,
\mathbbm{R})_0$ via the isomorphism
\begin{eqnarray*}
  \left( v, V \right) & \longmapsto & \left( v^{\ast}_g, V_{\Omega}^{\ast}
  \right) \assign \left( g^{- 1} v, V / \Omega \right) .
\end{eqnarray*}
With these notations holds the fundamental variation formulas
\begin{equation}
  \label{var-h} 2 D_{g, \Omega} h \left( v, V \right) = \Delta^{\Omega}_{L, g}
  v - L_{\nabla_g^{\ast_{\Omega}} v_g^{\ast} + \nabla_g V^{\ast}_{\Omega}} g -
  2 v ,
\end{equation}
and
\begin{equation}
  \label{var-H} 2 D_{g, \Omega} H \left( v, V \right) = \Delta^{\Omega}_g
  V^{\ast}_{\Omega} - \left( L_{\nabla_g^{\ast_{\Omega}} v_g^{\ast} + \nabla_g
  V^{\ast}_{\Omega}} \Omega \right)_{\Omega}^{\ast} - 2 V^{\ast}_{\Omega} -
  \left\langle v, h_{g, \Omega} \right\rangle_g ,
\end{equation}
where $L_\xi$ denotes the Lie derivative in the direction $\xi$.
(We will give a detailed proof in section \ref{main-res}). We infer that the variations of
the non-linear operators $h$ and $H$ are strictly elliptic in restriction to the space
\begin{eqnarray*}
  \mathbbm{F}_{g, \Omega} & \assign & \left\{ (v, V) \in T_{\mathcal{M} \times
  \mathcal{V}_1} \mid \nabla_g^{\ast_{\Omega}} v_g^{\ast} + \nabla_g
  V^{\ast}_{\Omega} = 0 \right\} .
\end{eqnarray*}
This fact strongly suggests that the following flow represents a strictly
parabolic system.

\begin{definition}
  The Soliton-Ricci flow is the smooth curve $(g_t, \Omega_t)_{t \geqslant 0}
  \subset \mathcal{M} \times \mathcal{V}_1$ solution of the evolution system
 \[ \left\{ \begin{array}{l}
     \dot{g}_t = - h_{g_t, \Omega_t},\\
     \\
     \dot{\Omega}_t = - \underline{H}_{g_t, \Omega_t} \Omega_t,
   \end{array} \right. \] 
  with
  \begin{eqnarray*}
    \underline{H}_{g, \Omega} & \assign & H_{g, \Omega} - \int_X H_{g, \Omega}
    \Omega .
  \end{eqnarray*}
\end{definition}
Indeed this is the case. We show the strict parabolic statement in the proof of the following basic fact.

\begin{lemma}
  \label{Exist-SRF}For every $( \check{g}_0, \Omega_0) \in \mathcal{M} \times
  \mathcal{V}_1$ there exists a unique smooth solution $(g_t, \Omega_t)_{t
  \geqslant 0} \subset \mathcal{M} \times \mathcal{V}_1$ of the Soliton-Ricci
  flow equation with initial data $( \check{g}_0 / \lambda, \Omega_0)$, for
  some $\lambda > 0$. In the case $(X, J_0)$ is a Fano variety and 
  $
  \check{g}_0$ is $J_0$ invariant K\"ahler such that
  $
  \check{g}_0 J_0 \in 2 \pi c_1 (X)$, we can choose $\lambda = 1$. In this case
  the Soliton-Ricci-flow represents a smooth family of K\"ahler structures and
  normalized positive volumes $(J_t, g_t, \Omega_t)_{t \geqslant 0}$ uniquely
  determined by the evolution system
  \[ \left\{ \begin{array}{l}
       \dot{g}_t = - h_{g_t, \Omega_t},\\
       \\
       \dot{\Omega}_t = - \underline{H}_{g_t, \Omega_t} \Omega_t,\\
       \\
       2 \dot{J}_t = \left[ J_t, \dot{g}_t^{\ast} \right].
     \end{array} \right. \]
We call the latter the {\tmstrong{Soliton-K\"ahler-Ricci flow}}.
\end{lemma}
Let $\tmop{Ric}_{J} (\Omega)$ be the Chern-Ricci form associated to the volume form $\Omega$ with respect to the complex structure $J$.
We will show in section 3 that, if the initial data $(J_0, g_0, \Omega_0)$
satisfies
\begin{eqnarray*}
  \omega  \assign  g_0 J_0 = \tmop{Ric}_{J_0} (\Omega_0), \int_X \Omega_0 =
  1,
\end{eqnarray*}
then the Soliton-K\"ahler-Ricci flow equation is equivalent with the evolution
system

\begin{equation}
  \label{SSCool-SKRF}  \left\{ \begin{array}{l}
    \omega  = \tmop{Ric}_{J_t} (\Omega_t) \hspace{0.25em},
    \int_X \Omega_t = 1,\\
    \\
    \dot{J}_t =  \overline{\partial}_{T_{X, J_t}}  \left(
    \omega^{- 1} d \log \frac{\;\omega^n}{\Omega_t} \right) .
  \end{array} \right.
\end{equation}
Thus the Soliton-K\"ahler-Ricci flow preserves the initial symplectic
structure $\omega$.

Over a $m$-dimensional compact Riemannian manifold $\left( X, g \right)$ we
consider Perelman's $\mathcal{W}$-functional \cite{Per}
\begin{eqnarray*}
  \mathcal{W} (g, f) & : = & \int_X \left[ | \nabla_g f|^2_g  +
   \tmop{Scal} (g)  +  2 f
   - m \right] e^{- f} dV_g\\
  &  & \\
  & = & \int_X \left[ - \Delta_g f  + 
  \tmop{Scal} (g) +  2 f  - m
  \right] e^{- f} dV_g  .
\end{eqnarray*}
(We can use here the identity $\Delta_g e^{- f} = - (| \nabla_g f|^2_g +
\Delta_g f) e^{- f}$). If we use the identifications $f \longleftrightarrow
\Omega \assign e^{- f} d V_g$ and $\mathcal{W} (g, f) \equiv \mathcal{W} (g,
\Omega)$, then
\[ \mathcal{W} (g, \Omega)  =  \int_X \left[
   \tmop{Tr}_g h_{g, \Omega}  + 2 \log
   \frac{dV_g}{\Omega} \right] \Omega \hspace{0.25em} = 2 \int_X H_{g, \Omega}
   \Omega . \]
With these notations Perelman's first variation formula for the functional\\
$\mathcal{W}: \mathcal{M} \times \mathcal{V}_1 \longrightarrow \mathbbm{R}$ in
\cite{Per} writes as
\begin{eqnarray*}
  D_{g, \Omega} \mathcal{W} (v, V)  =  - \int_X \left[ \left\langle v, h_{g,
  \Omega} \right\rangle_g - 2 V^{\ast}_{\Omega}  \underline{H}_{g, \Omega}
  \right] \Omega .
\end{eqnarray*}
We consider the pseudo-Riemannian structure over the space
$\mathcal{M} \times \mathcal{V}_1$ given by the formula $(g, \Omega) \in
\text{$\mathcal{M} \times \mathcal{V}_1$} \longmapsto G_{g, \Omega}$, with
\begin{eqnarray*}
  G_{g, \Omega}  (u, U ; v, V)  =  \int_X \left[ \left\langle u, v
  \right\rangle_g - 2\, U^{\ast}_{\Omega} V^{\ast}_{\Omega} \right] \Omega,
\end{eqnarray*}
for all $(u, U), (v, V) \in T_{\mathcal{M} \times \mathcal{V}_1}$. We
infer the identity
\begin{eqnarray*}
  \nabla_G \mathcal{W} (g, \Omega)  =  - \left( h_{g, \Omega},
  \underline{H}_{g, \Omega} \,\Omega\right) .
\end{eqnarray*}
This shows that the Soliton-Ricci flow is the gradient flow of the
$\mathcal{W}$ functional with respect to the pseudo-Riemmanian structure $G$.
Perelman's twice contracted second Bianchi identity (\ref{II-contr-Bianchi})
implies the equality
\begin{eqnarray*}
  \left\{ (g, \Omega) \in \mathcal{M} \times \mathcal{V}_1 \mid D_{g, \Omega}
  \mathcal{W}= 0 \right\}  =  \left\{ (g, \Omega) \in \mathcal{M} \times
  \mathcal{V}_1 \mid h_{g, \Omega} = 0 \right\},
\end{eqnarray*}
i.e the critical points of $\mathcal{W}$ are precisely the shrinking Ricci
solitons. We provide at this point a geometric interpretation of the space
$\mathbbm{F}_{g, \Omega}$. Let 
\begin{eqnarray*}
  \left[ g, \Omega \right]  =  \tmop{Diff}_0 (X) \cdot (g, \Omega),
\end{eqnarray*}
be the orbit of the point $(g, \Omega)$ under the action of the identity component $\tmop{Diff}_0 (X)$ of the
group of smooth diffeomorphisms of $X$. 
Then $\mathbbm{F}_{g, \Omega}$ represents the orthogonal space, with respect to
$G$, to the tangent space $T_{\left[ g, \Omega \right], (g, \Omega)}$ at the point $(g, \Omega) \in \mathcal{M} \times
\mathcal{V}_1$ of the orbit $\left[ g, \Omega \right] $. In formal terms
holds the equality
\begin{equation}
  T^{\bot_G}_{\left[ g, \Omega \right], (g, \Omega)} =\mathbbm{F}_{g, \Omega}
  . \label{F-OrtoOrb}
\end{equation}
We define the anomaly space of the pseudo-Riemannian structure $G$ at an
arbitrary point $(g, \Omega)$ as the vector space
\begin{eqnarray*}
  \mathbbm{A}^{\Omega}_g  \assign  \mathbbm{F}_{g, \Omega} \cap T_{[g,
  \Omega], g, \Omega} .
\end{eqnarray*}
In the case $(g, \Omega)$ is a shrinking Ricci-Soliton then the map
\begin{eqnarray*}
  \tmop{Ker} (\Delta^{\Omega}_g - 2\mathbbm{I}) & \longrightarrow &
  \mathbbm{A}^{\Omega}_g\\
  &  & \\
  u & \longmapsto & 2 \left( \nabla_g d u, - u \Omega \right),
\end{eqnarray*}
is an isomorphism (see section \ref{Anomaly}). In the case $(J, g, \Omega)$ is a
K\"ahler-Ricci soliton then $\mathbbm{A}^{\Omega}_g$ is canonically isomorphic
with the space of Killing vector fields of $g$. This is a consequence of a non
trivial result (see corollary \ref{zero-ortog}).

We denote by $\nabla^2_G \mathcal{W} (g, \Omega)$ the Hessian endomorphism of
the $\mathcal{W}$ functional with respect to the pseudo-Riemannian structure
$G$ at the point $(g, \Omega) \in \mathcal{M} \times \mathcal{V}_1$. We show
in lemma \ref{Sec-Var-W} that its restriction to the space $\mathbbm{F}_{g, \Omega}$ is a
strictly elliptic operator for any point $\left( g, \Omega \right)$. A simple
consequence of Perelman's twice contracted second Bianchi type identity
(\ref{II-contr-Bianchi}) is that the map
\begin{equation}
\nabla^2_G \mathcal{W} (g, \Omega) : \mathbbm{F}_{g, \Omega}
   \longrightarrow \mathbbm{F}_{g, \Omega}, \label{WellHesWF}
\end{equation}
is well defined in the case $(g, \Omega)$ is a shrinking Ricci-Soliton (see section \ref{inv-F}). In
this case holds also the inclusion
\[ \mathbbm{A}^{\Omega}_g \subseteq \mathbbm{F}_{g, \Omega} \cap \tmop{Ker}
   \nabla^2_G \mathcal{W} (g, \Omega) . \]
(See lemma \ref{Ker-D2W}). In general (see section \ref{inv-F}) for any point $(g, \Omega)$ holds the fundamental and
deep property
\begin{equation}
  \nabla^2_G \mathcal{W} (g, \Omega)  (h_{g, \Omega}, \underline{H}_{g,
  \Omega} \Omega) \in \mathbbm{F}_{g, \Omega} . \label{fundam-property-SRF}
\end{equation}
This is quite crucial for the stability of the Soliton-K\"ahler-Ricci flow
(see \cite{Pal7}). The following basic fact is a meaningful geometric reformulation of
the monotony statement for Perelman's $\mathcal{W}$ functional discovered by
the author in 2006 \cite{Pal1} and published in 2008.

\begin{lemma}
  \label{Monot-SKRF}Let $(X, J)$ be a Fano manifold, let $g$ be a
  $J$-invariant K\"ahler metric with symplectic form $\omega \assign g J \in 2
  \pi c_1 (X, \left[ J \right])$ and let $\Omega > 0$ be the unique smooth
  volume form with $\int_X \Omega = 1$ such that $\omega = \tmop{Ric}_J
  (\Omega)$. Then Perelman's $\mathcal{W}$ functional is monotone increasing
  along the Soliton-K\"ahler-Ricci flow with initial data $(J_0, g_0, \Omega_0) = (J, g, \Omega)$.
  The monotony is strict unless $(J, g)$ is a K\"ahler-Ricci soliton.
\end{lemma}
>From now on we will refer to the Soliton-K\"ahler-Ricci flow only if the
initial data are as in the previous lemma. Let $\mathcal{J}_{\tmop{int}}$ be the space of smooth integrable complex
structures over $X$. We consider the space of $\omega$-compatible
complex structures
\begin{eqnarray*}
  \mathcal{J}_{\omega} & \assign & \left\{ J \in \mathcal{J}_{\tmop{int}} \mid
  \omega = J^{\ast} \omega J, \omega J < 0 \right\}.
\end{eqnarray*}
Over a Fano manifold, the Chern-Ricci map provides a natural embedding of $\mathcal{J}_{\omega}$ inside 
$\mathcal{M} \times\mathcal{V}_1$. The image $\mathcal{S}_{\omega} \subset \mathcal{M} \times
\mathcal{V}_1$ of this embedding is
\begin{eqnarray*}
  \mathcal{S}_{\omega} & \assign & \left\{ (g, \Omega) \in
  \mathcal{M}_{\omega} \times \mathcal{V}_1 \mid \omega = \tmop{Ric}_J
  (\Omega), J = g^{- 1} \omega \right\} ,
\end{eqnarray*}
with $\mathcal{M}_{\omega} \assign - \omega \cdot
\mathcal{J}_{\omega} \subset \mathcal{M}$. 
The fact that the
Soliton-K\"ahler-Ricci flow preserves the symplectic form $\omega$ strongly
suggests the study of the restriction of Perelman's $\mathcal{W}$ functional
over $\mathcal{S}_{\omega}$.

The space $\mathcal{J}_{\omega}$ may be singular in general. This
implies that also the space $\mathcal{S}_{\omega}$ may be singular. We denote
by $\tmop{TC}_{\mathcal{S}_{\omega}, (g, \Omega)}$ the tangent cone of
$\mathcal{S}_{\omega}$ at an arbitrary point $(g, \Omega) \in
\mathcal{S}_{\omega}$. This is by definition the union of all tangent vectors
of $\mathcal{S}_{\omega}$ at the point $\left( g, \Omega \right)$. We notice
that (see for example \cite{Pal3}) the tangent cone
$\tmop{TC}_{\mathcal{M}_{\omega}, g}$ of $\mathcal{M}_{\omega}$ at an
arbitrary point $g \in \mathcal{M}_{\omega}$ satisfies the inclusion
\begin{equation}
  \tmop{TC}_{\mathcal{M}_{\omega}, g} \subseteq \mathbbm{D}^J_{g, \left[ 0
  \right]}, \label{TConeM}
\end{equation}
with
\begin{eqnarray*}
  \mathbbm{D}^J_{g, [0]}  \assign  \left\{ v \in C^{\infty} \left( X,
  S_{\mathbbm{R}}^2 T^{\ast}_X \right)_{_{_{_{_{_{}}}}}} \mid \hspace{0.25em}
 v = -J^{\ast} v J \;
  ,\; \overline{\partial}_{T_{X, J}}^{} v_g^{\ast} = 0 \right\}.
\end{eqnarray*}
The first variation of the Chern-Ricci form (see lemma \ref{Lm-var-O-Rc-fm}) shows that for any $(g, \Omega)
\in \mathcal{S}_{\omega}$ hold the inclusion
\begin{equation}
  \tmop{TC}_{\mathcal{S}_{\omega}, (g, \Omega)} \subseteq \mathbbm{T}^J_{g,
  \Omega}, \label{TConeS}
\end{equation}
with
\begin{eqnarray*}
  \mathbbm{T}^J_{g, \Omega}  \assign  \left\{ (v, V) \in \mathbbm{D}^J_{g,
  \left[ 0 \right]} \times T_{\mathcal{V}_1} \mid L_{ \nabla_g^{\ast_{\Omega}} v_g^{\ast} + \nabla_g V^{\ast}_{\Omega}}\,\omega  = 0 \right\} .
\end{eqnarray*}
We consider also its sub-space
\begin{eqnarray*}
  \mathbbm{F}^J_{g, \Omega} [0]  \assign  \left\{ (v, V) \in
  \mathbbm{F}_{g, \Omega} \mid v \in \mathbbm{D}^J_{g, \left[ 0 \right]}
  \right\} .
\end{eqnarray*}
In the case $\left( X, J, g \right)$ is a compact K\"ahler-Ricci soliton then
the map
\begin{eqnarray}
 \nabla^2_G \mathcal{W} (g, \Omega) : \mathbbm{F}^J_{g, \Omega} [0]
   \longrightarrow \mathbbm{F}^J_{g, \Omega} [0], \label{welldefF0}
 \end{eqnarray}  
is well defined. Furthermore for any $(g, \Omega) \in \mathcal{S}_{\omega}$
the fundamental property (\ref{fundam-property-SRF}) implies
\begin{eqnarray}
\nabla^2_G \mathcal{W} (g, \Omega)  (h_{g, \Omega}, \underline{H}_{g,
   \Omega} \Omega) \in \mathbbm{F}^J_{g, \Omega} [0] . \label{fundPF0}
 \end{eqnarray}    
This is precisely the key statement needed for the study of the stability of
the Soliton-K\"ahler-Ricci flow in \cite{Pal7}. For any point $(g, \Omega) \in
\mathcal{S}_{\omega}$ we denote by
\begin{eqnarray*}
  \left[ g, \Omega \right]_{\omega}  \assign  \tmop{Symp}^0 (X, \omega)
  \cdot (g, \Omega) \subset \mathcal{S}_{\omega},
\end{eqnarray*}
the orbit of $(g, \Omega)$ under the action of the identity
component $\tmop{Symp}^0 (X,\omega)$ of the group of smooth symplectomorphisms of $X$. With these notations hold the property
\begin{equation}
  T^{\bot_G}_{\left[ g, \Omega \right]_{\omega}, (g, \Omega)} \cap
  \mathbbm{T}^J_{g, \Omega} =\mathbbm{F}^J_{g, \Omega} \left[ 0 \right] .
  \label{geom-F}
\end{equation}
This combined with (\ref{TConeS}) implies directly the geometric identity
\begin{equation}
  T^{\bot_G}_{\left[ g, \Omega \right]_{\omega}, (g, \Omega)} \cap
  \tmop{TC}_{\mathcal{S}_{\omega}, (g, \Omega)} =\mathbbm{F}^J_{g, \Omega}
  \left[ 0 \right] \cap \tmop{TC}_{\mathcal{S}_{\omega}, (g, \Omega)},
  \label{TConeS1}
\end{equation}
for any $(g, \Omega) \in \mathcal{S}_{\omega}$. An other remarkable fact is
that for any $(g, \Omega) \in \mathcal{S}_{\omega}$ the restriction of the
symmetric form $G_{g, \Omega}$ over the vector space $\mathbbm{T}^J_{g,
\Omega}$, with $J : = g^{- 1} \omega$, is positive definite. This implies the
$G$-orthogonal decomposition (see corollary \ref{Ortg-dec-TS} and the sub-section \ref{Trip-Split})
\begin{eqnarray}
   \mathbbm{T}^J_{g, \Omega} & = &  T_{\left[ g, \Omega
  \right]_{\omega}, (g, \Omega)} \oplus_G  \mathbbm{F}_{g, \Omega}^J [0].\label{doub-spl-T}
\end{eqnarray}
The vector
space of $\Omega$-harmonic $T_{X, J}$-valued $(0,1)$-forms 
$\mathcal{H}_{g, \Omega}^{0, 1} \left( T_{X, J} \right)$ 
embeds naturally inside $\mathbbm{F}^J_{g, \Omega} [0]$ via the map $A\in \mathcal{H}_{g, \Omega}^{0, 1} \left( T_{X, J} \right)
\longmapsto \left( g A, 0 \right)$. By abuse of notations we still denote by
$\mathcal{H}_{g, \Omega}^{0, 1} \left( T_{X, J} \right) \subset
\mathbbm{F}^J_{g, \Omega} [0]$ the image of this embedding. There exists an
infinite dimensional vector space $\mathbbm{E}^J_{g, \Omega} [0] \subseteq
\mathbbm{F}^J_{g, \Omega} [0]$, (see the sub-section \ref{Trip-Split} for its definition) such that
the $G$-orthogonal decomposition holds true
\begin{eqnarray*}
  \mathbbm{F}^J_{g, \Omega} [0]  =  \mathbbm{E}^J_{g, \Omega} [0] \oplus_G
  \mathcal{H}_{g, \Omega}^{0, 1} \left( T_{X, J} \right) .
\end{eqnarray*}
We can explain now a more precise property of the tangent cone
$\tmop{TC}_{\mathcal{S}_{\omega}, \left( g, \Omega \right)}$. For this
purpose we consider the Kuranishi space $\mathcal{K}_{J, g} \subset
\mathcal{H}_{g, \Omega}^{0, 1} \left( T_{X, J} \right)$, $0 \in \mathcal{K}_{J, g}$ of $X$. (See theorem
\ref{Kuranishi} in the sub-section \ref{Kuran-Sect} of appendix B for its definition and properties.) In the sub-section \ref{Pol-Kuran-Sect} we define also the Kuranishi
space of $\omega$-polarized complex deformations $\mathcal{K}^{\omega}_J
\subseteq \mathcal{K}_{J, g}$ of the Fano manifold $\left( X, J, \omega
\right)$. (See the definition \ref{Pol-Kuranishi}). Then holds the
inclusions
\begin{eqnarray}
 & & T_{\left[ g, \Omega \right]_{\omega}, (g, \Omega)} \oplus_G 
  \mathbbm{E}^J_{g, \Omega} [0] \oplus_G \tmop{TC}_{\mathcal{K}^{\omega}_J, 0}\nonumber
  \\\nonumber
  &  & \\
  & \subseteq &  \tmop{TC}_{\mathcal{S}_{\omega}, \left( g, \Omega
  \right)}\label{First-incTC}
  \\\nonumber
  &  & \\
  & \subseteq &  T_{\left[ g, \Omega \right]_{\omega}, (g, \Omega)}
  \oplus_G  \mathbbm{E}^J_{g, \Omega} [0] \oplus_G
  \tmop{TC}_{\mathcal{K}_{J, g}, 0} .\label{Second-incTC}
\end{eqnarray}
Let $F:=f-\int_X f \,\Omega$. We define the non-negative cone of $\Omega$-harmonic variations
\begin{eqnarray*}
  \mathcal{H}_{g, \Omega}^{0, 1} \left( T_{X, J} \right)_{\geqslant 0} &
  \assign & \left\{ A \in \mathcal{H}_{g, \Omega}^{0, 1} \left( T_{X, J}
  \right) \mid \int_X \left| A \right|^2_g F \,\Omega \geqslant 0 \right\},
\end{eqnarray*}
and the sub-cone
\begin{eqnarray*}
  \mathcal{H}_{g, \Omega}^{0, 1} \left( T_{X, J} \right)_0 & \assign & \left\{
  A \in \mathcal{H}_{g, \Omega}^{0, 1} \left( T_{X, J} \right) \mid \int_X
  \left| A \right|^2_g F \,\Omega = 0 \right\}.
\end{eqnarray*}
In the K\"ahler-Einstein case holds the obvious identities 
$$\mathcal{H}_{g, \Omega}^{0, 1} \left(
T_{X, J} \right)_{\geqslant 0} =\mathcal{H}_{g, \Omega}^{0, 1} \left( T_{X, J}
\right)_0 =\mathcal{H}_{g, \Omega}^{0, 1} \left( T_{X, J} \right).
$$
In the
Dancer-Wang K\"ahler-Ricci soliton case $\mathcal{H}_{g, \Omega}^{0, 1} \left(
T_{X, J} \right)_0 \neq \{0\}$, thanks to a result in Hall-Murphy \cite{Ha-Mu2}. 
Let $H_{T_{X, J}}$ be the $L^2_{\Omega}$-projector over the space $\mathcal{H}_{g, \Omega}^{0, 1} \left( T_{X, J}\right)$.
We define
also the non-negative cone
\begin{eqnarray*}
  \mathbbm{T}^{J, \geqslant 0}_{g, \Omega} & \assign & \left\{ \left( v, V
  \right) \in \mathbbm{T}^J_{g, \Omega} \mid H_{T_{X, J}} v^{\ast}_g \in
  \mathcal{H}_{g, \Omega}^{0, 1} \left( T_{X, J} \right)_{\geqslant 0}
  \right\},
\end{eqnarray*}
and in a similar way $\mathbbm{T}^{J, 0}_{g, \Omega}$. An interesting
non-negative cone from the geometric point of view is also
\begin{eqnarray*}
  \tmop{TC}^{\geqslant 0}_{\mathcal{S}_{\omega}, \left( g, \Omega \right)} &
  \assign & \tmop{TC}_{\mathcal{S}_{\omega}, \left( g, \Omega \right)} \cap
  \mathbbm{T}^{J, \geqslant 0}_{g, \Omega} .
\end{eqnarray*}
Let now $\tmop{KRS}_{\omega}$ be the set of all K\"ahler-Ricci solitons inside
$\mathcal{S}_{\omega}$.
We observe that Perelman's twice contracted second Bianchi type
identity implies 
\begin{eqnarray*}
  \tmop{KRS}_{\omega} & = & \left\{ \left( g, \Omega \right)_{_{_{_{}}}}
  \in \mathcal{S}_{\omega} \mid \underline{H}_{g, \Omega} = 0 \right\} .
\end{eqnarray*}
Notice also that for any $\left( g, \Omega \right)_{_{_{_{}}}} \in
\tmop{KRS}_{\omega}$ holds the inclusions $[g, \Omega]_{\omega}
\subseteq \tmop{KRS}_{\omega}$ and
\[ T_{\left[ g, \Omega \right]_{\omega}, (g, \Omega)} \subseteq
   \tmop{TC}_{\tmop{KRS}_{\omega}, \left( g, \Omega \right)} \subseteq
   \tmop{Ker} D_{g, \Omega}  \underline{H}^{}_{} \cap
   \tmop{TC}_{\mathcal{S}_{\omega}, \left( g, \Omega \right)} . \]
The following statement provides a finite dimensional reduction of the stability problem for
  K\"ahler-Ricci solitons. This reduction represents the solution of this well known problem. 
\begin{theorem}
  \label{Main-teorem}{\tmstrong{$($Main result. The stability of
  K\"ahler-Ricci solitons$)$}}
  
  Let $\left( X, J, g \right)$ be a compact K\"ahler-Ricci soliton and let
  $\Omega > 0$ be the unique smooth volume form with $\int_X \Omega = 1$ such
  that $\omega :=gJ= \tmop{Ric}_J (\Omega)$. Then for all $\left( v, V \right) \in
  \mathbbm{T}^{J, \geqslant 0}_{g, \Omega}$ the Hessian form of Perelman's
  $\mathcal{W}$ functional with respect to the pseudo-Riemannian structure $G$
  at the point $(g, \Omega)$, in the direction $\left( v, V \right)$ satisfies
  the inequality
  \begin{eqnarray}
    \nabla_G D\mathcal{W} \left( g, \Omega \right) \left( v, V ; v, V \right)
    & \leqslant & 0 \,,
  \end{eqnarray}
  with equality if and only if
  \begin{eqnarray}
    \left( v, V \right) & \in & \tmop{Ker} D_{g, \Omega}  \underline{H}^{}_{}
    \cap \mathbbm{T}^{J, 0}_{g, \Omega}\\\nonumber
    &  & \\
    & = & T_{\left[ g, \Omega \right]_{\omega}, (g, \Omega)} \oplus_G
    \mathcal{H}_{g, \Omega}^{0, 1} \left( T_{X, J} \right)_0\label{equal-MainThm}
    \\\nonumber
    &  & \\
    & \supseteq & \tmop{TC}_{\tmop{KRS}_{\omega}, \left( g, \Omega \right)}\label{incl-MainThm} .
  \end{eqnarray}
\end{theorem}

In more explicit/classic terms the previous statement shows \ that for any
smooth curve $(g_t, \Omega_t)_{t \in \mathbbm{R}} \subset \mathcal{M} \times
\mathcal{V}_1$ (not necessarily in $\mathcal{S}_{\omega}$!) with $(g_0,
\Omega_0) = (g, \Omega)$ a K\"ahler-Ricci soliton and with $( \dot{g}_0, \dot{\Omega}_0) = (v, V)\in
  \mathbbm{T}^{J, \geqslant 0}_{g, \Omega}$ holds
the inequality
\begin{eqnarray*}
  \frac{d^2}{d t^2} _{\mid_{t = 0}} \mathcal{W} (g_t, \Omega_t) & \leqslant &
  0 \;,
\end{eqnarray*}
with equality if and only if $\left( v, V \right) \in \tmop{Ker} D_{g, \Omega}
\underline{H}^{}_{} \cap \mathbbm{T}^{J, 0}_{g, \Omega}$. The identity (\ref{equal-MainThm}) and the inclusion (\ref{incl-MainThm}) are part of the
statement in the main theorem \ref{Main-teorem}.

In section \ref{Sign-Sect} we obtain also a quite sharp second variation
formula for Perelman's $\mathcal{W}$ functional with respect to more general variations
$\left( v, V \right) \in \mathbbm{F}_{g, \Omega}$ over a K\"ahler-Ricci
soliton point $(g,\Omega)$. These variations arise from variations of K\"ahler structures preserving
the first Chern class of $X$.

This formula provide a precise control of the sign of the second variation
of Perelman's $\mathcal{W}$ functional over a K\"ahler-Ricci soliton point.
This can be of independent interest for experts. (In particular we will see
below some general consequences for the classical stability of
K\"ahler-Einstein metrics.) For our geometric applications the most striking
particular case is the one corresponding to the main theorem
\ref{Main-teorem}.

The highly geometric nature of the Soliton-K\"ahler-Ricci flow combined with
the main theorem \ref{Main-teorem}, suggest to the author the following
version of the Hamilton-Tian conjecture (compare with the statements made in \cite{Ti-Zha} and \cite{Ch-Wa}).

\begin{Conjecture}
  Let $(X, J_0)$ be a Fano manifold and let $\omega \in 2 \pi c_1 (X, \left[
  J_0 \right])$ be an arbitrary $J_0$-invariant K\"ahler form. Then there
  exists a complex analytic subset $\Sigma$ of complex codimension greater or
  equal to $2$, which is empty for generic choices of $J_0$ inside $\left[J_0 \right]$ and $\omega$ inside $2 \pi c_1 (X, \left[
  J_0 \right])$, there exists a smooth
  complex structure $J \in \left[ J_0 \right]$ outside $\Sigma$ and a smooth
  volume form $\Omega > 0$ outside $\Sigma$ such that;
  \[ \left\{ \begin{array}{l}
       \omega \hspace{0.75em} = \hspace{0.75em}\tmop{Ric}_J (\Omega) \hspace{0.25em},\\
       \\
       \overline{\partial}_{T_{X, J}}  \left( \omega^{- 1} d \log
       \frac{\;\omega^n}{\Omega} \right) \hspace{0.75em} = \hspace{0.75em} 0
       \hspace{0.25em},
     \end{array} \right. \]
  outside $\Sigma$, i.e. the Riemannian metric $g \assign - \omega J$, is a
  smooth $J$-invariant K\"ahler-Ricci soliton outside $\Sigma$. The triple
  $\left( J, g, \Omega \right)$ is obtained as the limit in the smooth
  topology of $X \smallsetminus \Sigma$, as $t \rightarrow + \infty$, of the
  Soliton-K\"ahler-Ricci flow with initial data $\left( J_0, g_0, \Omega_0
  \right)$ where $g_0 \assign - \omega J_0$ and $\omega = \tmop{Ric}_{J_0}
  (\Omega_0)$, with $\int_X \Omega_0 = 1$.
\end{Conjecture}

We explain now a very particular consequence of our study of the
stability problem. This consequence provides a result on the stability in the classical sense of
K\"ahler-Einstein manifolds. We introduce first a few basic notations.

Let $\left( X, J \right)$ be a compact K\"ahler manifold and let $c_1 \equiv
c_1 (X, \left[ J \right]) \in H^2_d (X, \mathbbm{R})$. We denote by
$\mathcal{K}\mathcal{S}$ the space of K\"ahler structures over $X$ and we set
\begin{eqnarray*}
  \mathcal{K}\mathcal{S}_{2 \pi c_1} & \assign & \left\{ (J, g) \in
  \mathcal{K}\mathcal{S} \mid g J \in 2 \pi c_1 \right\} .
\end{eqnarray*}
We define also the set $\mathbbm{K}\mathbbm{V}^J_g (2 \pi c_1)$ of symmetric variations of K\"ahler structures
preserving the first Chern class of $X$. The latter is defined as the set of elements $ v \in C^{\infty}
  \left( X, S_{\mathbbm{R}}^2 T^{\ast}_X \right)$ such that there exists a smooth curve $(J_t^{}, g_t)_t
  \subset \mathcal{K}\mathcal{S}_{2 \pi c_1}$ with $(J_0, g_0) = (J,g)$, $\dot{g}_0 = v$ and $\dot{J}_0^{} = (
  \dot{J}_0^{})_g^T$.
In section \ref{Sect-SmVKS} we show the inclusion
\begin{equation}
  \mathbbm{K}\mathbbm{V}^J_g (2 \pi c_1) \subseteq \mathbbm{D}^J_{g, 0}\,,
  \label{Kahl-VarC1}
\end{equation}
with
\begin{eqnarray*}
  \mathbbm{D}^J_{g, 0}  \assign  \left\{ v \in C^{\infty} \left( X,
  S_{\mathbbm{R}}^2 T^{\ast}_X \right) \mid \hspace{0.25em} \partial^g_{T_{X,
  J}} (v_J')_g^{\ast} = 0, \overline{\partial}_{T_{X, J}} (v_J'')_g^{\ast} =
  0, \left\{ v_J' J \right\}_d = 0 \right\} ,
\end{eqnarray*}
where $v_J'$ and $v_J''$ denote respectively the $J$-invariant and
$J$-anti-invariant parts of $v$ and $\{\alpha\}_d$ denotes the De Rham cohomology class of any $d$-closed
form $\alpha$. We introduce also the classical stability operator (see \cite{Bes})
\begin{eqnarray*}
  \mathcal{L}_g & \assign & \Delta_g - 2\mathcal{R}_g \ast\,,
\end{eqnarray*}
acting on smooth symmetric $2$-tensors. With these notations we can state the
following stability (in the classical sense) result.

\begin{theorem}
  \label{KEstability}Let $\left( X, J, g \right)$ be a Fano K\"ahler-Einstein
  manifold. Then for any $v \in \tmop{Ker} \nabla^{\ast}_g \cap
  \mathbbm{D}^J_{g, 0}$, holds the inequality
  \begin{eqnarray*}
    \int_X \left\langle \mathcal{L}^{}_g v, v \right\rangle_g d V_g &
    \geqslant & 0\,,
  \end{eqnarray*}
  with equality if and only if $v^{\ast}_g \in \mathcal{H}_g^{0, 1} \left(
  T_{X, J} \right)$.
\end{theorem}
(See sub-section
\ref{KEcase} for the proof).
A similar result in the case of negative or vanishing first Chern class has
been proved in the remarkable paper \cite{D-W-W2} (see also \cite{D-W-W1}). The statement about the
equality case holds also under more general assumptions (see lemma
\ref{harmonicKER} in the appendix B).

In the next section we enlighten the results obtained by other authors in the
long standing problem of the stability of K\"ahler-Ricci solitons and on the
Hamilton-Tian conjecture.

\section{Other works on the subject}

A question of central importance in complex differential geometry is the
Hamilton-Tian conjecture.

Solutions of this conjecture have been posted on the arXiv
server in (2013) by Tian-Zhang \cite{Ti-Zha} in complex dimension 3 and quite recently by Chen-Wang \cite{Ch-Wa} in general.

Since we have learned about this conjecture in 2004 we immediately asked ourself which one is the precise notion of gauge needed for the convergence. (The
K\"ahler-Ricci flow $\left( J_0, \hat{g}_t \right)_{t \geqslant 0}$ needs to
be modified since its {\tmstrong{formal}} limit $\left( J_0, \hat{g}_{\infty}
\right)$ as $t \rightarrow + \infty$ is a is a K\"ahler-Einstein metric, but
Fano manifolds do not always admit such ones!)

It turns out that the Soliton-K\"ahler-Ricci flow introduced in this paper
corresponds to a modification of the K\"ahler-Ricci flow via the gauge
provided by the gradient of the Ricci potentials.

To the very best of our knowledge the Soliton-K\"ahler-Ricci flow with
variable volume forms introduced in this paper does not appear nowhere in the
literature.

In our previous works \cite{Pal4} and \cite{Pal5}, we introduced also the notion of
Soliton-K\"ahler-Ricci flow with fixed volume form. This leads to a complete
different approach which conducts naturally to the study of the existence of ancient
solutions of the K\"ahler-Ricci flow and their modified (according to \cite{Pal4} and \cite{Pal5})
convergence as $t \rightarrow - \infty$. This approach requires some particular
geometric conditions (which imply some strong regularity) on the initial data.
The key point in \cite{Pal4} and \cite{Pal5} is that these conditions represent a conservative law along
the Soliton-K\"ahler-Ricci flow with fixed volume form. These conditions imply
good convexity properties for the convergence of this flow.

We review now the modifications of the K\"ahler-Ricci flow made by other
authors. We can find two frequent approaches in the literature. One is based on the gauge transformation generated by a
holomorphic vector field with imaginary part generating an $S^1$-action on
the manifold (see \cite{Ti-Zhu1} and \cite{P-S-S-W2} for a very elegant construction). 
A K\"ahler-Ricci-soliton vector field provides such
example.

The second approach, which has been used quite intensively in the last years is
based on the gauge modification constructed via the minimizers of Perelman's
$\mathcal{W}$ functional (see \cite{Ti-Zhu3} and \cite{Su-Wa}).  As far as known the minimizers are unique only in a small
neighborhood of the K\"ahler-Ricci soliton. Therefore the ''modified
K\"ahler-Ricci flow'' in \cite{Ti-Zhu3} and \cite{Su-Wa} exists only in such small neighborhood.

For historical reasons it is important to remind that Hamilton \cite{Ham} pointed out
first that to any flow of K\"ahler structures with fixed complex structure
corresponds an other flow of K\"ahler structures which preserves the
symplectic form (see also Donaldson \cite{Don} for the same remark). He suggested this
approach for the study of the K\"ahler-Ricci flow. As far as we know he did not
pursuit on this idea.

As explained in the introduction our definition of the Soliton-Ricci flow with
variable volume forms was inspired to us from Perelman's twice contracted
second Bianchi type identity and from the strict ellipticity of the first variation
of the maps $h$ and $H$ in the directions $\mathbbm{F}$.

It was surprising for us to discover that the corresponding
Soliton-K\"ahler-Ricci flow with variable volume forms (from now on we will
refer only to this flow) preserves the symplectic structure.

We realized quickly the power of this fact. It allows indeed the application of
Futaki's weighted complex Bochner identity and the uniform lower bound on the
first eigenvalue of the complex weighted Laplacian \cite{Fu1}. The main feature of
the Soliton-K\"ahler-Ricci flow in this paper is the jumping
of the complex structure at the limit when $t \rightarrow + \infty$. This
phenomenon is necessary for the existence of K\"ahler-Ricci solitons in
general. We learned for the first time about this key phenomenon in the
Pioneer work of \cite{P-S1}. In this fundamental work the authors introduce a 
condition on stability (is the condition (B) in \cite{P-S1}) witch is the 
key phenomenon occurring in the convergence of the K\"ahler-Ricci flow. We refer also to \cite{P-S-S-W3} for further developments.

We remind now that by definition, the stability of a critical point of a
functional corresponds to determine a sign of its second variation in
determinate directions.

The stability of critical metrics for natural geometric functionals was
naturally born with differential geometry (see \cite{Bes}). The main classic example is the
Einstein metric. In the case of this metric the corresponding functional is
the integral of the scalar curvature.

In 2003 Grigory Perelman astonished the mathematical community with his
spectacular proof of the Poincar\'e conjecture. In this celebrated paper \cite{Per}
he introduced various entropy functionals for Ricci-solitons. Shrinking
Ricci-solitons correspond to critical points of his $\mathcal{W}$ functional
or his entropy functional $\nu$.

Since then, the second variation of Perelman's functionals $\mathcal{W}$ and $\nu$ has been
studied quite intensively. It started in 2004 with the works of Cao-Hamilton-Imlanen \cite{C-H-I}, \cite{Ca-Zhu} and
Tian-Zhu \cite{Ti-Zhu2} independently. It continued with \cite{Ca-He} and \cite{Ha-Mu1}, \cite{Ha-Mu2}.

We wish to point out that the results in this paper and in \cite{Pal3} are of
completely different nature with respect to the previous works. The reason is
that in our work we compute the second variation of Perelman's $\mathcal{W}$
functional with respect to the pseudo-Riemannian structure $G$. (The work \cite{Pal3} is a particular case.)

An important fact about K\"ahler-Ricci solitons is that once they exist, one
can obtain the Einstein condition by proving the vanishing of the Futaki
invariant \cite{Fut}. From our point of view they provide a natural and necessary
generalization in order to control the Einstein condition.

The stability of K\"ahler-Ricci solitons is important in order to understand
the convergence of the K\"ahler-Ricci flow. The first work on the subject is
due to Tian-Zhu, see \cite{Ti-Zhu2}.

In 2009 Sun-Wang \cite{Su-Wa} posted on the arXiv server a stability result for the K\"ahler-Ricci flow
basing on the Lojasiewicz inequality (see \cite{Co-Mi}). In this paper the authors use the
modified flow in \cite{Ti-Zhu3}. The
same method was used in Ache \cite{Ach}, where a uniform bound assumption on the curvature is made.
We report finally a quite recent work on the same subject by Kr\"oncke \cite{Kro}. This author combines the
technical details in \cite{Su-Wa}, \cite{Ach} and \cite{Co-Mi} in the Riemannian set up.

The statements made in this section are based on the very best of our knowledge
and understanding of the subject. We sincerely apologize to other authors in
case of inaccuracies or omissions in the claims of this section.

\section{Proof of the first variation formulas for the maps $h$ and $H$
}\label{main-res}

\subsection{The first variation of the Bakry-Emery-Ricci tensor}

We remind (see \cite{Pal3}) that the first variation of the Bakry-Emery-Ricci
tensor with fixed volume form $\Omega > 0$ is given by the formula
\begin{equation}
  \label{var-Om-Ric} 2 \frac{d}{d t} \tmop{Ric}_{g_t} (\Omega) = -
  \nabla_{g_t}^{\ast_{\Omega}} \mathcal{D}_{g_t}  \dot{g}_t,
\end{equation}
where $\mathcal{D}_g \assign \hat{\nabla}_g - 2 \nabla_g$, with
$\hat{\nabla}_g$ being the symmetrization of
{\tmname{{\tmsamp{\tmtexttt{$\nabla_g$}}}}} acting on symmetric 2-tensors.
Explicitly
\begin{eqnarray*}
  \hat{\nabla}_g \alpha (\xi_0, ..., \xi_p) & : = & \sum_{j = 0}^p \nabla
  \alpha (\xi_j, \xi_0, ..., \hat{\xi}_j, ..., \xi_p),
\end{eqnarray*}
for all $p$-tensors $\alpha$. Fixing an arbitrary time $\tau$ and time
deriving at $t = \tau$ the decomposition
\begin{eqnarray*}
  \tmop{Ric}_{g_t} (\Omega_t) & = & \tmop{Ric}_{g_t} (\Omega_{\tau}) -
  \nabla_{g_t} d \log \frac{\Omega_t}{\Omega_{\tau}},
\end{eqnarray*}
we deduce, thanks to (\ref{var-Om-Ric}), the general variation formula
\begin{equation}
  \label{var-OM-Ric} 2 \frac{d}{d t} \tmop{Ric}_{g_t} (\Omega_t) = -
  \nabla_{g_t}^{\ast_{\Omega_t}} \mathcal{D}_{g_t}  \dot{g}_t - 2 \nabla_{g_t}
  d \frac{\dot{\Omega}_t}{\Omega_t} .
\end{equation}
This formula implies directly Perelman's general first variation formula for the
$\mathcal{W}$ functional (see appendix A). We define the Hodge Laplacian
(resp. the $\Omega$-Hodge Laplacian) operators acting on $q$-forms as
\begin{eqnarray*}
  \Delta_{T_{X, g}} & : = & \nabla_{T_{X, g}} \nabla^{\ast}_g +
  \nabla_g^{\ast} \nabla_{T_{X, g}},\\
  &  & \\
  \Delta^{\Omega}_{T_{X, g}} & \assign & \nabla_{T_{X, g}}
  \nabla^{\ast_{\Omega}}_g + \nabla^{\ast_{\Omega}}_g \nabla_{T_{X, g}} .
\end{eqnarray*}
We remind also the following Weitzenb\"ock type formula proved in \cite{Pal4}

\begin{lemma}
  \label{OmTX-Lap-RmLap}Let $(X, g)$ be a orientable Riemannian manifold, let
  $\Omega > 0$ be a smooth volume form and let $A \in C^{\infty} (X,
  \tmop{End} (T_X))$. Then
  \begin{eqnarray*}
    \Delta^{\Omega}_{T_{X, g}} A & = & \Delta^{\Omega}_g A -\mathcal{R}_g \ast
    A + A \tmop{Ric}^{\ast}_g (\Omega),
  \end{eqnarray*}
  Where $\left( \mathcal{R}_g \ast A \right) \xi \assign \tmop{Tr}_g \left[
  \left( \xi \neg \mathcal{R}_g \right) A \right]$ for all $\xi \in T_X$.
\end{lemma}

In analogy to the $\Omega$-Hodge Laplacian we can define the Laplace type
operator
\[ \hat{\Delta}^{\Omega}_g \assign \nabla^{\ast_{\Omega}}_g  \hat{\nabla}_g -
   \hat{\nabla}_g \nabla^{\ast_{\Omega}}_g : C^{\infty} (X, S^p T^{\ast}_X)
   \longrightarrow C^{\infty} (X, S^2 T^{\ast}_X) . \]
Using this notation we observe that for any $u \in C^{\infty} (X, S^2
T^{\ast}_X)$ hold the identities
\begin{eqnarray*}
  - \nabla_g^{\ast_{\Omega}} \mathcal{D}_g u & = & \left( 2 \Delta^{\Omega}_g
  - \hat{\Delta}^{\Omega}_g \right) u - \hat{\nabla}_g
  \nabla^{\ast_{\Omega}}_g u\\
  &  & \\
  & = & \left( 2 \Delta^{\Omega}_g - \hat{\Delta}^{\Omega}_g \right) u -
  L_{\nabla_g^{\ast_{\Omega}} u_g^{\ast}} g,
\end{eqnarray*}
The last one follows from the equalities $\nabla^{\ast_{\Omega}}_g u = g
\nabla^{\ast_{\Omega}}_g u^{\ast}_g$ and $\hat{\nabla}_g  (g \xi) = L_{\xi}
g$, $\xi \in C^{\infty} (X, T_X)$. We observe now that for any symmetric
$2$-tensor $u$ the tensor $\mathcal{R}_g \ast u$ is also symmetric. In fact
let $(e_k)_k$ be a $g (x)$-orthonormal base of $T_{X, x}$. Then
\begin{eqnarray*}
  - ( \mathcal{R}_g \ast u) (\xi, \eta) & = & R_g (\xi, e_k, u_g^{\ast} e_k,
  \eta) = R_g (\eta, u_g^{\ast} e_k, e_k, \xi) .
\end{eqnarray*}
Furthermore if we choose the $g (x)$-orthonormal base $(e_k)_k$ such that $u$
is diagonal with respect to this one, then
\begin{eqnarray*}
  R_g (\eta, u_g^{\ast} e_k, e_k, \xi) & = & R_g (\eta, e_k, u_g^{\ast} e_k,
  \xi) = - ( \mathcal{R}_g \ast u) (\eta, \xi) .
\end{eqnarray*}
We observe also that the previous computation shows the identity
\begin{eqnarray*}
  ( \mathcal{R}_g \ast u) (\xi, \eta) & = & R_g (\xi, e_k, \eta, u_g^{\ast}
  e_k)\\
  &  & \\
  & = & g (\mathcal{R}_g (\xi, e_k) u_g^{\ast} e_k, \eta)\\
  &  & \\
  & = & g \left( \left( \mathcal{R}_g \ast u^{\ast}_g \right) \xi, \eta
  \right),
\end{eqnarray*}
i.e
\begin{equation}
  \label{ast-curv-Id}  ( \mathcal{R}_g \ast u)_g^{\ast} =\mathcal{R}_g \ast
  u^{\ast}_g .
\end{equation}
We deduce in particular the equality
\begin{equation}
  \label{Trns-curv-id} \mathcal{R}_g \ast u^{\ast}_g = \left( \mathcal{R}_g
  \ast u^{\ast}_g \right)_g^T .
\end{equation}
We remind that the $\Omega$-Lichnerowicz Laplacian $\Delta^{\Omega}_{L, g}$ is
self-adjoint with respect to the scalar product (\ref{Glb-Rm-m}) thanks to the
identity (\ref{SymRmOp}) that we show now.

We pick a $g (x)$-orthonormal base $(e_k)_k \subset T_{X, x}$ such that $v$ is
diagonal with respect to this one at the point $x$. Using (\ref{ast-curv-Id})
we infer
\begin{eqnarray*}
  \left\langle \mathcal{R}_g \ast u, v \right\rangle_g & = &
  \tmop{Tr}_{\mathbbm{R}} \left[ \left( \mathcal{R}_g \ast u^{\ast}_g \right)
  v^{\ast}_g \right]\\
  &  & \\
  & = & R_g (v^{\ast}_g e_l, e_k, e_l, u^{\ast}_g e_k)\\
  &  & \\
  & = & R_g (e_l, e_k, v^{\ast}_g e_l, u^{\ast}_g e_k)\\
  &  & \\
  & = & R_g (e_k, e_l, u^{\ast}_g e_k, v^{\ast}_g e_l)\\
  &  & \\
  & = & \left\langle \mathcal{R}_g \ast v, u \right\rangle_g,
\end{eqnarray*}
since these identities are independent of the choice of the $g (x)$-orthonormal
base $(e_k)_k \subset T_{X, x}$.

\begin{lemma}
  \label{Lich-Weitz}For any $g \in \mathcal{M}$ and $u \in C^{\infty} (X,
  S_{}^2 T^{\ast}_X)$ holds the Weitzenb\"ock type formula
  \begin{eqnarray*}
    - \nabla_g^{\ast_{\Omega}} \mathcal{D}_g u & = & \Delta^{\Omega}_{L, g} u
    - L_{\nabla_g^{\ast_{\Omega}} u_g^{\ast}} g .
  \end{eqnarray*}
\end{lemma}

\begin{proof}
  The required formula follows from the identity
  \begin{equation}
    \label{Part-Weitz} \Delta^{\Omega}_{L, g} u = \left( 2 \Delta^{\Omega}_g -
    \hat{\Delta}^{\Omega}_g \right) u .
  \end{equation}
  In order to show this identity we expand $\hat{\Delta}^{\Omega}_g u =
  \nabla^{\ast_{\Omega}}_g  \hat{\nabla}_g u - \hat{\nabla}_g
  \nabla^{\ast_{\Omega}}_g u$. We observe first
  \begin{eqnarray*}
    \nabla^{\ast_{\Omega}}_g  \hat{\nabla}_g u (\xi, \eta) & = &
    \nabla^{\ast_{}}_g  \hat{\nabla}_g u (\xi, \eta) + \hat{\nabla}_g u
    (\nabla_g f, \xi, \eta) .
  \end{eqnarray*}
  We fix an arbitrary point $x_0\in X$ and we choose the vector fields $\xi$ and
  $\eta$ such that $0 = \nabla_g \xi (x_0) = \nabla_g \eta (x_0)$. Let
  $(e_k)_k$ be a $g$-orthonormal local frame such that $\nabla_g e_k  (x_0) =
  0$. Then at the point $x_0$ hold the identities
  \begin{eqnarray*}
    \nabla^{\ast_{}}_g  \hat{\nabla}_g u (\xi, \eta) & = & - \nabla_{g, e_k}
    \hat{\nabla}_g u (e_k, \xi, \eta)\\
    &  & \\
    & = & - \nabla_{g, e_k} \left[ \hat{\nabla}_g u (e_k, \xi, \eta)
    \right]\\
    &  & \\
    & = & - \nabla_{g, e_k} \left[ \nabla_g u (e_k, \xi, \eta) + \nabla_g u
    (\xi, e_k, \eta) + \nabla_g u (\eta, e_k, \xi) \right]\\
    &  & \\
    & = & - \nabla_{g, e_k} \nabla_{g, e_k} u (\xi, \eta) - \nabla_{g, e_k}
    \nabla_{g, \xi} u (e_k, \eta) - \nabla_{g, e_k} \nabla_{g, \eta} u (e_k,
    \xi),
  \end{eqnarray*}
  and
  \begin{eqnarray*}
    \hat{\nabla}_g u (\nabla_g f, \xi, \eta) & = & \nabla_g u (\nabla_g f,
    \xi, \eta) + \nabla_g u (\xi, \nabla_g f, \eta) + \nabla_g u (\eta,
    \nabla_g f, \xi) .
  \end{eqnarray*}
  Moreover
  \begin{eqnarray*}
    \hat{\nabla}_g \nabla^{\ast_{\Omega}}_g u (\xi, \eta) & = & \hat{\nabla}_g
    \nabla^{\ast_{}}_g u (\xi, \eta) + \hat{\nabla}_g  \left( \nabla_g f \neg
    u \right) (\xi, \eta),
  \end{eqnarray*}
  and at the point $x_0$ hold the identities
  \begin{eqnarray*}
    \hat{\nabla}_g \nabla^{\ast_{}}_g u (\xi, \eta) & = & \nabla_{g, \xi}
    \nabla^{\ast_{}}_g u \cdot \eta + \nabla_{g, \eta} \nabla^{\ast_{}}_g u
    \cdot \xi\\
    &  & \\
    & = & \nabla_{g, \xi}  \left[ \nabla^{\ast_{}}_g u \cdot \eta \right] +
    \nabla_{g, \eta}  \left[ \nabla^{\ast_{}}_g u \cdot \xi \right]\\
    &  & \\
    & = & - \nabla_{g, \xi}  \left[ \nabla^{}_{g, e_k} u (e_k, \eta) \right]
    - \nabla_{g, \eta}  \left[ \nabla^{}_{g, e_k} u (e_k, \xi) \right]\\
    &  & \\
    & = & - \nabla_{g, \xi} \nabla^{}_{g, e_k} u (e_k, \eta) - \nabla_{g,
    \eta} \nabla^{}_{g, e_k} u (e_k, \xi),
  \end{eqnarray*}
  and
  \begin{eqnarray*}
    \hat{\nabla}_g  \left( \nabla_g f \neg u \right) (\xi, \eta) & = &
    \nabla_{g, \xi}  \left( \nabla_g f \neg u \right) \cdot \eta + \nabla_{g,
    \eta}  \left( \nabla_g f \neg u \right) \cdot \xi\\
    &  & \\
    & = & \nabla_{g, \xi}  \left[ u (\nabla_g f, \eta) \right] + \nabla_{g,
    \eta}  \left[ u (\nabla_g f, \xi) \right]\\
    &  & \\
    & = & \nabla_g u (\xi, \nabla_g f, \eta) + \left( u \nabla^2_g f \right) 
    (\xi, \eta)\\
    &  & \\
    & + & \nabla_g u (\eta, \nabla_g f, \xi) + \left( \nabla_g d f u^{\ast}_g
    \right)  (\xi, \eta) .
  \end{eqnarray*}
  Let now $A \in C^{\infty} (X, \tmop{End} (T_X))$. We denote by $A \neg u$
  the $2$-tensor defined by the formula
  \begin{eqnarray*}
    (A \neg u)  (\xi, \eta) & \assign & u (A \xi, \eta) + u (\xi, A \eta) .
  \end{eqnarray*}
  We observe that if $\mu, \zeta$ are two germs of vector fields near $x_0$
  such that $\left[ \mu, \zeta \right] (x_0) = 0$ then holds the identity at
  the point $x_0$
  \begin{eqnarray*}
    \nabla_{g, \mu} \nabla_{g, \zeta} u - \nabla_{g, \zeta} \nabla_{g, \mu} u
    & = & -\mathcal{R}_g (\mu, \zeta) \neg u .
  \end{eqnarray*}
  Using this identity we infer the equalities at the point $x_0$
  \begin{eqnarray*}
    \left( \nabla_{g, \xi} \nabla_{g, e_k} u - \nabla_{g, e_k} \nabla_{g, \xi}
    u \right)  (e_k, \eta) & = & - u \left( \mathcal{R}_g (\xi, e_k) e_k, \eta
    \right) - u \left( e_k, \mathcal{R}_g (\xi, e_k) \eta \right)\\
    &  & \\
    & = & - \left( u \tmop{Ric}^{\ast} (g) \right)  (\xi, \eta) +
    (\mathcal{R}_g \ast u)  (\xi, \eta),\\
    &  & \\
    \left( \nabla_{g, \eta} \nabla_{g, e_k} u - \nabla_{g, e_k} \nabla_{g,
    \eta} u \right)  (e_k, \xi) & = & - \left( \tmop{Ric} (g) u^{\ast}_g
    \right)  (\xi, \eta) + (\mathcal{R}_g \ast u)  (\xi, \eta),
  \end{eqnarray*}
  by obvious symmetries. Combining the identities obtained so far and
  simplifying we obtain the identity
  \begin{eqnarray*}
    \hat{\Delta}^{\Omega}_g u & = & \Delta^{\Omega}_g u + 2\mathcal{R}_g \ast
    u - u \tmop{Ric}^{\ast}_g (\Omega) - \tmop{Ric}_g (\Omega) u_g^{\ast},
  \end{eqnarray*}
  which in its turn implies the required identity (\ref{Part-Weitz}).
\end{proof}

The Weitzenb\"ock type identity in lemma \ref{Lich-Weitz} combined with the
variation formula (\ref{var-OM-Ric}) implies directly the variation formula
(\ref{var-h}).

\subsection{Proof of the first variation formula for Perelman's $H$-function}

We show now the variation formula (\ref{var-H}). For this purpose let $0 <
(g_t, \Omega_t)_t\subset \mathcal{M} \times \mathcal{V}_1$ be a smooth family and set as usual $f_t \assign \log
\frac{d V_{g_t}}{\Omega_t}$. We start time deriving the identity
\begin{eqnarray*}
  - \Delta^{\Omega_t}_{g_t} f_t & = & ^{} \tmop{div}^{\Omega_t} \nabla_{g_t}
  f_t .
\end{eqnarray*}
We compute first the variation of the $\Omega$-divergence operator. Set $u_t
\assign \dot{\Omega}^{\ast}_t$ and time derive the definition identity
\begin{eqnarray*}
  d \left( \xi \neg \Omega_t \right) & = & (\tmop{div}^{\Omega_t} \xi)
  \Omega_t.
\end{eqnarray*}
We infer
\begin{eqnarray*}
  d \left( \xi \neg u_t \Omega_t \right) & = & \left( \frac{d}{d t}
  \tmop{div}^{\Omega_t} \xi \right) \Omega_t + u_t  (\tmop{div}^{\Omega_t}
  \xi) \Omega_t .
\end{eqnarray*}
Moreover expanding the left hand side we obtain
\begin{eqnarray*}
  d \left( \xi \neg u_t \Omega_t \right) & = & (\xi . u_t) \Omega_t + u_t d
  \left( \xi \neg \Omega_t \right),
\end{eqnarray*}
which implies the formula
\begin{eqnarray*}
  \left( \frac{d}{d t} \tmop{div}^{\Omega_t} \right) \xi & = & g \left(
  \nabla_g  \dot{\Omega}^{\ast}_t, \xi \right) .
\end{eqnarray*}
We observe also the variation formulas
\begin{equation}
  \label{var-grad}  \frac{d}{d t}  \left( \nabla_{g_t} f_t \right) =
  \nabla_{g_t}  \dot{f}_t - \dot{g}^{\ast}_t \nabla_{g_t} f_t,
\end{equation}
and
\begin{equation}
  \label{var-f}  \dot{f}_t = \frac{1}{2} \tmop{Tr}_{g_t}  \dot{g}_t -
  \dot{\Omega}^{\ast}_t .
\end{equation}
Combining all these formulas we obtain
\begin{eqnarray*}
  - \frac{d}{d t} \Delta^{\Omega_t}_{g_t} f_t & = & \left( \frac{d}{d t}
  \tmop{div}^{\Omega_t} \right) \nabla_{g_t} f_t + \tmop{div}^{\Omega_t}
  \frac{d}{d t}  \left( \nabla_{g_t} f_t \right)\\
  &  & \\
  & = & g_t \left( \nabla_{g_t}  \dot{\Omega}^{\ast}_t, \nabla_{g_t} f_t
  \right) + \Delta^{\Omega_t}_{g_t} \left( \dot{\Omega}^{\ast}_t - \frac{1}{2}
  \tmop{Tr}_{g_t}  \dot{g}_t \right)\\
  &  & \\
  & - & \tmop{div}^{\Omega_t} \left( \dot{g}^{\ast}_t \nabla_{g_t} f_t
  \right) .
\end{eqnarray*}
We expand last term using the identity
\begin{eqnarray*}
  \tmop{div}^{\Omega} \xi & = & \tmop{Tr}_{\mathbbm{R}} \left( \nabla_g \xi
  \right) - g \left( \xi, \nabla_g f \right) .
\end{eqnarray*}
We obtain with respect to a $g_t (x)$-orthonormal basis $(e_k)_k \subset T_{X,
x}$ at an arbitrary space-time point $(x, t)$
\begin{eqnarray*}
 && \tmop{div}^{\Omega_t} \left( \dot{g}^{\ast}_t \nabla_{g_t} f_t \right) 
 \\
 \\
 & = &
  g_t (\nabla_{g, e_k} \left( \dot{g}^{\ast}_t \nabla_{g_t} f_t \right), e_k)
  - g_t \left( \dot{g}^{\ast}_t \nabla_{g_t} f_t, \nabla_{g_t} f_t \right)\\
  &  & \\
  & = & g_t \left( \nabla_{g, e_k}  \dot{g}^{\ast}_t \cdot \nabla_{g_t} f_t +
  \dot{g}^{\ast}_t \nabla_{g, e_k} \nabla_{g_t} f_t, e_k \right) - g_t \left(
  \dot{g}^{\ast}_t \nabla_{g_t} f_t, \nabla_{g_t} f_t \right)\\
  &  & \\
  & = & g_t \left( \nabla_{g_t} f_t, \nabla_{g, e_k}  \dot{g}^{\ast}_t \cdot
  e_k \right) + g_t \left( \nabla^2_{g, e_k} f_t, \dot{g}^{\ast}_t e_k \right)
  - g_t \left( \nabla_{g_t} f_t, \dot{g}^{\ast}_t \nabla_{g_t} f_t \right)\\
  &  & \\
  & = & - g_t \left( \nabla_{g_t}^{\ast_{\Omega_t}} \dot{g}^{\ast}_t,
  \nabla_{g_t} f_t \right) + \left\langle \nabla_{g_t} d f_t, \dot{g}_t
  \right\rangle_{g_t} .
\end{eqnarray*}
We infer the variation formula
\begin{eqnarray}
  - \frac{d}{d t} \Delta^{\Omega_t}_{g_t} f_t 
 &=&
  \Delta^{\Omega_t}_{g_t} \left( \dot{\Omega}^{\ast}_t - \frac{1}{2}
  \tmop{Tr}_{g_t}  \dot{g}_t \right) + g_t \left(
  \nabla_{g_t}^{\ast_{\Omega_t}} \dot{g}^{\ast}_t + \nabla_{g_t} 
  \dot{\Omega}^{\ast}_t, \nabla_{g_t} f_t \right) \nonumber
  \\\nonumber
  \\
  &-& \left\langle \dot{g}_t,
  \nabla_{g_t} d f_t \right\rangle_{g_t} .\label{var-Lapf}
\end{eqnarray}
We observe next the identity
\begin{eqnarray*}
  2 \frac{d}{d t} h^{\ast}_t & = & 2 \dot{h}^{\ast}_t - 2 \dot{g}^{\ast}_t
  h^{\ast}_t\\
  &  & \\
  & = & \Delta^{\Omega_t}_{g_t}  \dot{g}^{\ast}_t - 2 (\mathcal{R}_{g_t} \ast
  \dot{g}_t)^{\ast}_t + \dot{g}^{\ast}_t \tmop{Ric}^{\ast}_{g_t} (\Omega_t) +
  \tmop{Ric}^{\ast}_{g_t} (\Omega_t)  \dot{g}^{\ast}_t \\
  &  & \\
  & - & \left( L_{\nabla_{g_t}^{\ast_{\Omega_t}} \dot{g}_t^{\ast} +
  \nabla_{g_t}  \dot{\Omega}^{\ast}_t} g_t \right)_t^{\ast} - 2
  \dot{g}^{\ast}_t - 2 \dot{g}^{\ast}_t h^{\ast}_t,
\end{eqnarray*}
thanks to the variation formula (\ref{var-h}). We deduce the formula
\begin{equation}
  \label{var-Trh}  \frac{d}{d t} \tmop{Tr}_{g_t} h_t = \frac{1}{2}
  \Delta^{\Omega_t}_{g_t} \tmop{Tr}_{g_t}  \dot{g}_t - \left\langle \dot{g}_t,
  \tmop{Ric} (g_t) \right\rangle_{g_t} - \tmop{div}_{g_t} \left(
  \nabla_{g_t}^{\ast_{\Omega_t}} \dot{g}^{\ast}_t + \nabla_{g_t} 
  \dot{\Omega}^{\ast}_t \right),
\end{equation}
thanks to the identities
\begin{eqnarray}
   \tmop{Tr}_g \left( \mathcal{R}_g \ast v \right) &=&
  \left\langle v, \tmop{Ric} (g) \right\rangle_g,\label{Trace-Curv}
\\\nonumber
\\
 \tmop{Tr}_g  (L_{\xi} g) &=& 2 \tmop{Tr}_{\mathbbm{R}}
  (\nabla_g \xi) \;=\; 2 \tmop{div}_g \xi .\label{Trace-Lie}
\end{eqnarray}
In order to show the identity (\ref{Trace-Curv}) we expand with respect to a
$g (x)$-orthonormal basis $(e_k)_k \subset T_{X, x}$ the term
\begin{eqnarray*}
  \tmop{Tr}_g \left( \mathcal{R}_g \ast v \right) & = & \left( \mathcal{R}_g
  \ast v \right) \left( e_k, e_k \right)\\
  &  & \\
  & = & - v (\mathcal{R}_g (e_k, e_l) e_k, e_l)\\
  &  & \\
  & = & g \left( v_g^{\ast} e_l, \tmop{Ric}^{\ast} (g) e_l \right)\\
  &  & \\
  & = & \left\langle v, \tmop{Ric} (g) \right\rangle_g .
\end{eqnarray*}
The first equality in (\ref{Trace-Lie}) follows from the elementary identity
\begin{eqnarray*}
  (L_{\xi} g)_g^{\ast} & = & \nabla_g \xi + (\nabla_g \xi)_g^T,
\end{eqnarray*}
where $A_g^T$ denotes the transpose of an endomorphism $A$ of $T_X$ with
respect to $g$.

In conclusion combining the variation formulas (\ref{var-Lapf}),
(\ref{var-Trh}) and (\ref{var-f}) we infer the variation identity
\[ 2 D_{g, \Omega} H \left( v, V \right) = \Delta^{\Omega}_g V^{\ast}_{\Omega}
   - \tmop{div}^{\Omega} \left( \nabla_g^{\ast_{\Omega}} v_g^{\ast} + \nabla_g
   V^{\ast}_{\Omega} \right) - 2 V^{\ast}_{\Omega} - \left\langle v, h_{g,
   \Omega} \right\rangle_g, \]
and thus the required variation formula (\ref{var-H}).

\section{The Soliton-K\"ahler-Ricci Flow with variable volume forms}

\subsection{Existence of the Soliton-K\"ahler-Ricci flow}

We prove in this sub-section lemma \ref{Exist-SRF}.

\begin{proof}
  From now on we will set for notation simplicity $h_t \equiv h_{g_t,
  \Omega_t}$, $H_t \equiv H_{g_t, \Omega_t}$ and $\underline{H}_t \equiv
  \underline{H}_{g_t, \Omega_t}$. We observe that for any smooth curve $(g_t,
  \Omega_t)_{t \geqslant 0} \subset \mathcal{M} \times \mathcal{V}_1$ the
  identity
  \begin{eqnarray*}
    f_t & = & \log \frac{d V_{g_t}}{\Omega_t},
  \end{eqnarray*}
  is equivalent to the evolution equation
  \begin{equation}
    \label{dif-car-f} 2 \dot{f}_t = \tmop{Tr}_{g_t}  \dot{g}_t - 2
    \dot{\Omega}^{\ast}_t .
  \end{equation}
  with initial data $f_0 : = \log \frac{d V_{g_0}}{\Omega_0}$. Along the Soliton-Ricci flow,
  the equation (\ref{dif-car-f}) rewrites as
  \begin{eqnarray*}
    2 \dot{f}_t & = & - \tmop{Tr}_{g_t} h_t + 2 \underline{H}_t\\
    &  & \\
    & = & - \Delta^{\Omega_t}_{g_t} f_t + 2 f_t - 2 \int_X H_t \Omega_t .
  \end{eqnarray*}
  We infer that the Soliton-Ricci flow equation is equivalent to the evolution system
  \begin{equation}
    \label{SRF-f}  \left\{ \begin{array}{l}
      \dot{g}_t = g_t - \tmop{Ric} (g_t) \hspace{0.75em} - \nabla_{g_t} d f_t
      \hspace{0.25em},\\
      \\
      2 \hspace{0.25em} \dot{f}_t \hspace{0.75em} = - \Delta_{g_t} f_t - |
      \nabla_{g_t} f_t |^2_{g_t} + 2 f_t -\mathcal{W}(g_t, f_t),
    \end{array} \right.
  \end{equation}
  with $f_0 : = \log \frac{d V_{g_0}}{\Omega_0}$. We consider now the flow of
  diffeomorphisms $(\varphi_t)_{t \geqslant 0}$ solution of the equation
  \begin{eqnarray*}
    2 \dot{\varphi}_t & = & \left( \nabla_{g_t} f_t \right) \circ \varphi_t,
  \end{eqnarray*}
  with $\varphi_0 = \tmop{Id}_X$ and we define $( \hat{g}_t, \hat{f}_t) : =
  \varphi_t^{\ast} (g_t, f_t)$. We observe the evolution formulas
  \begin{eqnarray*}
    \frac{d}{d t}  \hat{g}_t & = & \varphi_t^{\ast} \left( \dot{g}_t +
    \frac{1}{2} L_{\nabla_{g_t} f_t} g_t \right)\\
    &  & \\
    & = & \varphi_t^{\ast} \left[ g_t - \tmop{Ric} (g_t) \right]\\
    &  & \\
    & = & \hat{g}_t - \tmop{Ric} ( \hat{g}_t),
  \end{eqnarray*}
  and
 \begin{eqnarray*} 
    2 \frac{d}{d t}  \hat{f}_t & = & 2 \dot{f}_t \circ \varphi_t + 2
    d_{\varphi_t} f_t \cdot \dot{\varphi}_t \\
    &  & \\
    & = & 2 \dot{f}_t \circ \varphi_t + d_{\varphi_t} f_t \cdot \left[ \left(
    \nabla_{g_t} f_t \right) \circ \varphi_t \right]\\
    &  & \\
    & = & \left( 2 \dot{f}_t + d f_t \cdot \nabla_{g_t} f_t \right) \circ
    \varphi_t\\
    &  & \\
    & = & \left( 2 \dot{f}_t + | \nabla_{g_t} f_t |^2_{g_t} \right) \circ
    \varphi_t,
  \end{eqnarray*}
  We deduce thanks to the diffeomorphism invariance of the $\mathcal{W}$
  functional that the evolution system (\ref{SRF-f}) is equivalent to
  \begin{equation}
    \label{SRF-F}  \left\{ \begin{array}{l}
      \frac{d}{d t}  \hat{g}_t = \hat{g}_t - \tmop{Ric} ( \hat{g}_t),\\
      \\
      2 \frac{d}{d t}  \hat{f}_t = - \Delta_{\hat{g}_t}  \hat{f}_t + 2
      \hat{f}_t -\mathcal{W}( \hat{g}_t, \hat{f}_t),
    \end{array} \right.
  \end{equation}
  with initial data $( \hat{g}_0, \hat{f}_0) : = (g_0, f_0)$. Notice indeed
  that we can obtain (\ref{SRF-f}) from (\ref{SRF-F}) by performing the
  inverse gauge transformation $(g_t, f_t) : = \psi_t^{\ast} ( \hat{g}_t,
  \hat{f}_t)$ with $\psi_t = \varphi^{- 1}_t$ being characterized by the
  evolution equation
  \begin{eqnarray*}
    2 \dot{\psi}_t & = & - \left( \nabla_{\hat{g}_t}  \hat{f}_t \right) \circ
    \psi_t,
  \end{eqnarray*}
  $\psi_0 = \tmop{Id}_X$. In order to show all time existence and uniqueness
  of the solutions of the evolution system (\ref{SRF-F}) we consider a
  solution of the Ricci flow $( \check{g}_t)_{t \in [0, T)}$,
  \begin{eqnarray*}
    \frac{d}{d t}  \check{g}_t & = & - 2 \tmop{Ric} ( \check{g}_t),
  \end{eqnarray*}
  with initial data $\check{g}_0$ and $0 < T < + \infty$. Then $(
  \hat{g}_t)_{t \geqslant 0}$ defined by
  \begin{eqnarray*}
    \hat{g}_t & \assign & \frac{e^t}{2 T}  \check{g}_{T (1 - e^{- t})},
  \end{eqnarray*}
  satisfies the evolution equation relative to the metrics in (\ref{SRF-F}).
  Then we set $\lambda \assign 2 T$. In the case $(X, J_0)$ is a Fano variety
  and $ \check{g}_0 J_0 \in 2 \pi c_1 (X)$ we can choose $\lambda = 1$ since
  the the evolution equation of $\hat{g}_t$ in (\ref{SRF-F}) represents a
  solution of the K\"ahler-Ricci flow equation.
  
  The existence and uniqueness of the solutions of the evolution equation for
  $\hat{f}_t$ in (\ref{SRF-F}) follows directly from standard parabolic theory
  with respect to H\"older spaces. Notice indeed that the presence of the
  integral term $\mathcal{W} ( \hat{g}_t, \hat{f}_t)$ (we consider the
  expression involving the $H^1 (X)$ norm of $f$) does not produce any issue
  in this theory.
  
  In the Fano set up we define the complex structures $J_t \assign
  \psi_t^{\ast} J_0$. Then the family $(J_t, g_t)_{t \geqslant 0}$ represents
  a flow of K\"ahler structures since $(J_0, \hat{g}_t)_{t \geqslant 0}$ is
  also a flow of K\"ahler structures. The identity $\varphi_t^{\ast} J_t
  \equiv J_0$ is equivalent to the equality
  \begin{eqnarray*}
    0 & = & \frac{d}{dt} \hspace{0.25em} (\varphi_t^{\ast} J_t)\\
    &  & \\
    & = & \varphi_t^{\ast} \left( \dot{J}_t \hspace{0.75em} + \hspace{0.75em}
    \frac{1}{2} \hspace{0.25em} L_{\nabla_{g_t} f_t} J_t \right)\\
    &  & \\
    & = & \varphi_t^{\ast} \left( \dot{J}_t \hspace{0.75em} + \hspace{0.75em}
    J_t \hspace{0.25em} \overline{\partial}_{T_{X, J_t}} \nabla_{g_t} f_t
    \right),
  \end{eqnarray*}
  i.e to the equation
  \begin{eqnarray*}
    \dot{J}_t & = & - J_t \hspace{0.25em} \overline{\partial}_{T_{X, J_t}}
    \nabla_{g_t} f_t .
  \end{eqnarray*}
  This combined with the $J_t$-linearity of the first two terms in the right
  hand side of the complex decomposition
  \begin{eqnarray*}
    \tmop{Ric}^{\ast}_{g_t} (\Omega_t) & = & \tmop{Ric}^{\ast} (g_t) +
    \partial^{g_t}_{T_{X, J_t}} \nabla_{g_t} f_t + \overline{\partial}_{T_{X,
    J_t}} \nabla_{g_t} f_t,
  \end{eqnarray*}
  implies the required characterization $2 \dot{J}_t = \left[ J_t,
  \dot{g}^{\ast}_t \right]$ of the evolution of the complex structures $J_t$.

\end{proof}

\subsection{Monotony of Perelman's $\mathcal{W}$-functional along the
Soliton-K\"ahler-Ricci flow}

We observe first the following elementary fact.

\begin{lemma}
  \label{Carac-KRS}Let $(X, J)$ be a Fano manifold and let $g$ be a
  $J$-invariant K\"ahler metric with symplectic form $\omega \assign g J \in 2
  \pi c_1 (X, \left[ J \right])$. Then $(J, g)$ is a K\"ahler-Ricci soliton if and only
  if there exists a smooth volume form $\Omega > 0$ with $\int_X \Omega = 1$
  such that
  \[ (S)  \left\{ \begin{array}{l}
       \omega = \tmop{Ric}_J (\Omega),\\
       \\
       \Delta^{\Omega}_g f - 2 f + 2 \int_X f \Omega = 0, f \assign \log
       \frac{\omega^n}{n! \Omega} .
     \end{array} \right. \]
\end{lemma}

\begin{proof}
  We assume first that $(J, g)$ is a K\"ahler-Ricci soliton. Then Perelman's twice
  contracted Bianchi type identity (\ref{II-contr-Bianchi}) implies
  $\underline{H}_{g, \Omega} = 0$. The latter is equivalent to the second
  equation of the system (S) thanks to the identity $\tmop{Tr}_g h_{g, \Omega}
  = 0$. We show now that the solution of the system (S) implies that $(J, g)$
  is a K\"ahler-Ricci soliton. Indeed multiplying by $\nabla_g f$ both sides of the identity
  (\ref{II-contr-Bianchi}) and integrating by parts we obtain the general
  formula
  \begin{equation}
    \label{int-cntr-Bianchi}  \int_X \left\langle h^{\ast}_{g, \Omega},
    \nabla^2_g f \right\rangle_g \Omega = - \int_X \underline{H}^{}_{g,
    \Omega} \Delta^{\Omega}_g f \Omega .
  \end{equation}
  In our case this rewrites as
  \begin{equation}
    \label{f-Boch} 2 \int_X \left| \overline{\partial}_{T_{X, J}} \nabla_g f
    \right|^2_g \Omega = \int_X  (\Delta^{\Omega}_g - 2\mathbbm{I}) f
    \Delta^{\Omega}_g f \Omega,
  \end{equation}
  thanks to the condition $\omega = \tmop{Ric}_J (\Omega)$ and the complex
  decomposition of the Bakry-Emery-Ricci tensor. We infer the required conclusion.
\end{proof}

We provide now a proof of the monotony statement in lemma \ref{Monot-SKRF}.

\begin{proof}
  {\tmstrong{STEP I}}. Let $(J, \hat{g}_t)_{t \geqslant 0}$ be a solution of
  the K\"ahler-Ricci flow and observe that this equation rewrites in the equivalent form
  \begin{equation}
    \label{KRF-eqv}  \left\{ \begin{array}{l}
      \frac{d}{d t}  \hat{\omega}_t = i \hspace{0.25em} \partial_J
      \overline{\partial}_J \log
      \frac{\hat{\omega}_t^n}{\hat{\Omega}_t},\\
      \\
      \hat{\omega}_t = \tmop{Ric}_J ( \hat{\Omega}_t), \int_X
      \hat{\Omega}_t = 1,
    \end{array} \right.
  \end{equation}
  with $\hat{\omega}_t \assign \hat{g}_t J$, and $\hat{\omega}_0 : =
  \omega$. We define the function
  \begin{eqnarray*}
    \hat{f}_t & \assign & \log \frac{\hat{\omega}_t^n}{ \hat{\Omega}_t
    n!},
  \end{eqnarray*}
  and we observe the analogue of (\ref{dif-car-f})
  \begin{eqnarray*}
    2 \frac{d}{d t}  \hat{f}_t & = & \tmop{Tr}_{\hat{\omega}_t} 
    \frac{d}{d t} \hat{\omega}_t - 2 \left( \frac{d}{d t}  \hat{\Omega}_t
    \right)_t^{\ast} .
  \end{eqnarray*}
  This combined with the first equation in (\ref{KRF-eqv}) implies
  \begin{equation}
    \label{evol-F1} 2 \frac{d}{d t}  \hat{f}_t = - \Delta_{\hat{g}_t} 
    \hat{f}_t - 2 \left( \frac{d}{d t}  \hat{\Omega}_t \right)_t^{\ast} .
  \end{equation}
  On the other hand time differentiating the identity $\hat{\omega}_t =
  \tmop{Ric}_J ( \hat{\Omega}_t)$ in (\ref{KRF-eqv}) we obtain
  \begin{eqnarray*}
    2 \frac{d}{d t} \hat{\omega}_t & = & 2 \frac{d}{d t} \tmop{Ric}_J (
    \hat{\Omega}_t) = - 2 i \hspace{0.25em} \partial_J \overline{\partial}_J 
    \left( \frac{d}{d t}  \hat{\Omega}_t \right)_t^{\ast},
  \end{eqnarray*}
  which combined with (\ref{evol-F1}) implies
  \begin{eqnarray*}
    2 i \hspace{0.25em} \partial_J \overline{\partial}_J  \hat{f}_t & = & i
    \hspace{0.25em} \partial_J \overline{\partial}_J  \left( 2 \frac{d}{d t} 
    \hat{f}_t + \Delta_{\hat{g}_t}  \hat{f}_t \right),
  \end{eqnarray*}
  i.e.
  \begin{eqnarray*}
    2 \frac{d}{d t}  \hat{f}_t & = & - \Delta_{\hat{g}_t}  \hat{f}_t + 2
    \hat{f}_t + C_t,
  \end{eqnarray*}
  for some time dependent constant $C_t$ which can be determined time deriving
  the integral condition $\int_X \hat{\Omega}_t = 1$. Indeed using
  (\ref{evol-F1}) we obtain
  \begin{eqnarray*}
    0 & = & - 2 \int_X \left( \frac{d}{d t}  \hat{\Omega}_t \right)_t^{\ast}
    \hat{\Omega}_t\\
    &  & \\
    & = & \int_X \left[ 2 \frac{d}{d t}  \hat{f}_t + \Delta_{\hat{g}_t} 
    \hat{f}_t \right]  \hat{\Omega}_t\\
    &  & \\
    & = & C_t + 2 \int_X \hat{f}_t  \hat{\Omega}_t .
  \end{eqnarray*}
  We infer the evolution formula
  \begin{equation}
    \label{evol-Fcool} 2 \frac{d}{d t}  \hat{f}_t = - \Delta_{\hat{g}_t} 
    \hat{f}_t + 2 \hat{f}_t - 2 \int_X \hat{f}_t e^{- \hat{f}_t} d
    V_{\hat{g}_t},
  \end{equation}
  with initial data
  \begin{eqnarray*}
    \hat{f}_0 & \assign & \log \frac{\omega^n}{\Omega n!} .
  \end{eqnarray*}
  We observe now that the identity $\hat{\omega}_t = \tmop{Ric}_J (
  \hat{\Omega}_t)$ in (\ref{KRF-eqv}) implies
  \begin{equation}
    \label{Ric-ident}  \hat{g}_t = - \tmop{Ric}_J ( \hat{\Omega}_t) J =
    \tmop{Ric}_{\hat{g}_t} ( \hat{\Omega}_t) - \hat{g}_t 
    \overline{\partial}_{T_{X, J}} \nabla_{\hat{g}_t}  \hat{f}_t,
  \end{equation}
  and thus $\tmop{Tr}_{\hat{g}_t} h_{\hat{g}_t, \hat{\Omega}_t} = 0$. We
  deduce the equality
  \begin{equation}
    \label{Cool-W} \mathcal{W} ( \hat{g}_t, \hat{f}_t) = 2 \int_X \hat{f}_t
    e^{- \hat{f}_t} d V_{\hat{g}_t} .
  \end{equation}
  We infer by Cauchy's uniqueness that the evolution equation
  (\ref{evol-Fcool}) is equivalent with the second evolution equation in
  (\ref{SRF-F}). We obtain, as in the proof of lemma \ref{Exist-SRF}, a Soliton-K\"ahler-Ricci flow
  $(J_t, \omega_t, \Omega_t)_{t \geqslant 0}$ with initial data $(J_0,
  \omega_0, \Omega_0) = (J, \omega, \Omega)$. We observe that thanks to
  (\ref{Ric-ident}) and (\ref{Cool-W}) the Soliton-Ricci flow evolution system (\ref{SRF-f})
  writes in our case as
  \begin{equation}
    \label{SRF-Cool}  \left\{ \begin{array}{l}
      \dot{g}_t = - g_t \overline{\partial}_{T_{X, J_t}} \nabla_{g_t} f_t,\\
      \\
      2 \dot{f}_t = - \Delta^{\Omega_t}_{g_t} f_t + 2 f_t - 2 \int_X f_t e^{-
      f_t} d V_{g_t} .
    \end{array} \right.
  \end{equation}
  Time deriving the identity $\omega_t = g_t J_t$ and using the evolution
  formula for the complex structure $2 \dot{J}_t = \left[ J_t,
  \dot{g}_t^{\ast} \right]$ in the Soliton-K\"ahler-Ricci flow equation we infer
  \begin{eqnarray*}
    \dot{\omega}_t & = & \dot{g}_t J_t + g_t \dot{J}_t\\
    &  & \\
    & = & \frac{1}{2} g_t \left( \dot{g}_t^{\ast} J_t + J_t  \dot{g}_t^{\ast}
    \right)\\
    &  & \\
    & = & \frac{1}{2} \omega_t \left( \dot{g}_t^{\ast} - J_t 
    \dot{g}_t^{\ast} J_t \right)\\
    &  & \\
    & = & \omega_t  ( \dot{g}_t^{\ast})_J^{1, 0} \\
    &  & \\
    & = & 0,
  \end{eqnarray*}
  thanks to the first equation in (\ref{SRF-Cool}). We deduce $\omega_t \equiv
  \omega$ and thus the identity in time
  \begin{equation}
    \label{Symplectic-inv} \omega = \tmop{Ric}_{J_t} (\Omega_t) .
  \end{equation}
  {\tmstrong{STEP IIa}}. We provide now a first proof of the monotony
  statement for the Soliton-K\"ahler-Ricci flow. The equality (\ref{Cool-W}) rewrites as
  \begin{equation}
    \label{SCool-W} \mathcal{W} (g_t, \Omega_t) = 2 \int_X f_t e^{- f_t} 
    \frac{\omega^n}{n!} \equiv 2 \int_X f_t \Omega_t,
  \end{equation}
  thanks to the invariance by diffeomorphisms of $\mathcal{W}$. Let
  \begin{eqnarray*}
    F_t & \assign & f_t - \int_X f_t \Omega_t,
  \end{eqnarray*}
  and observe that the second evolution equation in (\ref{SRF-Cool}) rewrites
  as
  \begin{equation}
    \label{Cool-evol-f} 2 \dot{f}_t = - \Delta^{\Omega_t}_{g_t} F_t + 2 F_t .
  \end{equation}
  Time deriving the expression (\ref{SCool-W}) and using the evolution
  equation (\ref{Cool-evol-f}) we infer
  \begin{eqnarray*}
    \frac{d}{d t} \mathcal{W} (g_t, \Omega_t) & = & 2 \int_X \left( \dot{f}_t
    - f_t  \dot{f}_t \right) \Omega_t\\
    &  & \\
    & = & - 2 \int_X f_t  \dot{f}_t \,\Omega_t\\
    &  & \\
    & = & \int_X \left( \Delta^{\Omega_t}_{g_t} F_t - 2 F_t \right) F_t\,
    \Omega_t \geqslant 0,
  \end{eqnarray*}
  thanks to the estimate $\lambda_1 (\Delta^{\Omega}_g) \geqslant 2$ for the
  first eigenvalue $\lambda_1 (\Delta^{\Omega}_g)$ of the weighted Laplacian
  $\Delta^{\Omega}_g$ in the case $g J = \tmop{Ric}_J (\Omega)$. (See the
  estimate (\ref{Estim-FirstEigen}) in the section \ref{CxBochSec}.) Indeed by
  the variational characterization of the first eigenvalue holds the estimate
  \begin{equation}
    \label{Eigen-estm} 2 \leqslant \lambda_1 (\Delta^{\Omega}_g) = \inf
    \left\{ \int_X \Delta^{\Omega}_g u\, u \,\Omega \mid u \in C_{\Omega}^{\infty}
    (X, \mathbbm{R})_0 : \int_X u^2 \Omega = 1 \right\},
  \end{equation}
  which implies
  \begin{equation}
    \label{GPoinc-Ineq} 0 \leqslant \int_X \left( \Delta^{\Omega}_g F - 2 F
    \right) F \Omega,
  \end{equation}
  with
  \begin{eqnarray*}
    F & \assign & f - \int_X f \Omega, \quad f \;\assign\; \log \frac{d V_g}{\Omega} .
  \end{eqnarray*}
  We assume now equality in (\ref{GPoinc-Ineq}). We assume also $F \neq 0$
  otherwise $g$ will be a $J$-invariant K\"ahler-Einstein metric. Equality in
  (\ref{GPoinc-Ineq}) implies $2 = \lambda_1 (\Delta^{\Omega}_g)$ and
  \begin{eqnarray*}
    u_0 & \assign & F \left[ \int_X F^2 \Omega \right]^{- 1 / 2},
  \end{eqnarray*}
  attains the infinitum in (\ref{Eigen-estm}). Thus we can apply the method of
  Lagrange multipliers to the functionals
  \begin{eqnarray*}
    \Phi (u) & \assign & \int_X \Delta^{\Omega}_g u u \Omega,\\
    &  & \\
    \Psi (u) & \assign & \int_X u^2 \Omega,
  \end{eqnarray*}
  over the space $C_{\Omega}^{\infty} (X, \mathbbm{R})_0$. We have the
  equalities
  \begin{eqnarray*}
    2 & = & \min_{\Psi = 1} \Phi = \Phi (u_0),
  \end{eqnarray*}
  which imply $D_{u_0} \Phi = \mu D_{u_0} \Psi$, i.e. $\Delta^{\Omega}_g u_0 =
  \mu u_0$, with $\mu = 2$. The latter is equivalent to the equation
  $\Delta^{\Omega}_g F = 2 F$. Then the required conclusion follows from lemma
  \ref{Carac-KRS}.

  {\tmstrong{STEP IIb}}. We give here a different proof of the monotony
  statement. We remind first that Perelman's first variation formula for the
  $\mathcal{W}$ functional \cite{Per} writes as
  \begin{eqnarray*}
    D_{g, \Omega} \mathcal{W} (v, V) & = & - \int_X \left[ \left\langle v,
    h_{g, \Omega} \right\rangle_g - 2 V^{\ast}_{\Omega}  \underline{H}_{g,
    \Omega} \right] \Omega .
  \end{eqnarray*}
  Thus along the Soliton-Ricci flow holds the identity
  \begin{eqnarray*}
    \frac{d}{d t} \mathcal{W} (g_t, \Omega_t) & = & \int_X \left[ \left|
    h_{g_t, \Omega_t} \right|^2_{g_t} - 2 \underline{H}^2_{g_t, \Omega_t}
    \right] \Omega_t .
  \end{eqnarray*}
  Then the conclusion follows from the identity (\ref{Symplectic-inv})
  combined with the elementary lemma below.
\end{proof}

\begin{lemma}
  Let $(X, J)$ be a Fano manifold, let $g$ be a $J$-invariant K\"ahler metric
  with symplectic form $\omega \assign g J \in 2 \pi c_1 (X, \left[ J
  \right])$ and let $\Omega > 0$ be a smooth volume form with $\int_X \Omega =
  1$ such that $\omega = \tmop{Ric}_J (\Omega)$. Then
  \begin{equation}
    \label{Poinc-Ineq}  \int_X \left| h_{g, \Omega} \right|^2_g \Omega
    \geqslant 2 \int_X \underline{H}^2_{g, \Omega} \Omega,
  \end{equation}
  with equality if and only if $(J, g)$ is a K\"ahler-Ricci soliton.
\end{lemma}

\begin{proof}
  The condition $\omega = \tmop{Ric}_J (\Omega)$ and the complex decomposition
  of the Bakry-Emery-Ricci tensor in \cite{Pal5} imply
  \begin{eqnarray*}
    h_{g, \Omega} & = & g \overline{\partial}_{T_{X, J}} \nabla_g f,
  \end{eqnarray*}
  and thus $\tmop{Tr}_g h_{g, \Omega} = 0$. We deduce
  \begin{eqnarray}
    2 \underline{H}^{}_{g, \Omega} & = & - (\Delta^{\Omega}_g - 2\mathbbm{I})
    F . \label{H-KRS} 
  \end{eqnarray}
  Then
  \begin{eqnarray*}
    \int_X \left[ \left| h_{g, \Omega} \right|^2_g - 2 \underline{H}^2_{g,
    \Omega} \right] \Omega & = & \int_X  [ \left| \overline{\partial}_{T_{X,
    J}} \nabla_g f \right|^2_g - \frac{1}{2}  \left| (\Delta^{\Omega}_g -
    2\mathbbm{I}) F \right|^2] \Omega\\
    &  & \\
    & = & \int_X (\Delta^{\Omega}_g - 2\mathbbm{I}) F \cdot F \Omega,
  \end{eqnarray*}
  thanks to the integral identity (\ref{f-Boch}). The conclusion follows from
  the variational argument at the end of step IIa.
\end{proof}

\begin{remark}
  We observe that the elementary identities 
  $$
  \nabla_g f = J \omega^{- 1} d f =
  2 \omega^{- 1} d_J^c f,
  $$ with $2 d_J^c f \assign - d f \cdot J$, allow to
  rewrite the Soliton-Ricci flow evolution system (\ref{SRF-Cool}) as
  \begin{equation}
    \label{SRF-SpCool}  \left\{ \begin{array}{l}
      \dot{J}_t = \overline{\partial}_{T_{X, J_t}} \left( \omega^{- 1} d f_t
      \right),\\
      \\
      2 \dot{f}_t = \tmop{Tr}_{\omega} \left( d d_{J_t}^c f_t - d f_t \wedge
      d_{J_t}^c f_t \right) + 2 f_t - 2 \int_X f_t e^{- f_t} 
      \frac{\omega^n}{n!} .
    \end{array} \right.
  \end{equation}
  We notice also that the Soliton-K\"ahler-Ricci flow evolution system with initial data $(J_0, g_0,
  \Omega_0) = (J, g, \Omega)$ such that $\omega \assign g J = \tmop{Ric}_J
  (\Omega)$ is equivalent to the system (\ref{SSCool-SKRF}). Indeed the
  argument in step I of the proof of lemma \ref{Monot-SKRF} shows that our
  Soliton-K\"ahler-Ricci flow is equivalent to the K\"ahler-Ricci flow equation (\ref{KRF-eqv}) via the gauge
  transformation given by the diffeomorphisms $\varphi_t$. But
  (\ref{SSCool-SKRF}) is also equivalent to (\ref{KRF-eqv}) via the same gauge
  transformation. Notice in fact the identities
  \begin{eqnarray*}
    \frac{d}{d t}  \hat{\omega}_t & = & \frac{1}{2} \varphi^{\ast}_t 
    \left( L_{\nabla_{g_t} f_t} \omega \right) = \varphi^{\ast}_t  \left( i
    \hspace{0.25em} \partial_{J_t} \overline{\partial}_{J_t} f_t \right) = i
    \hspace{0.25em} \partial_J \overline{\partial}_J  \hat{f}_t .
  \end{eqnarray*}
  The corresponding identities for the transformation of the complex structure
  have been considered at the end of the proof of lemma \ref{Exist-SRF}. We
  infer the equivalence of our Soliton-K\"ahler-Ricci flow with (\ref{SSCool-SKRF}).
\end{remark}

\begin{remark}
  Let $(g_t, \Omega_t)_{t \geqslant 0}$ be the Soliton-Ricci flow and set for notation
  simplicity $\mathcal{W}_t \assign \mathcal{W} (g_t, \Omega_t)$, $h_t \assign
  h_{g_t, \Omega_t}$, $\underline{H}_t \assign \underline{H}_{g_t, \Omega_t}$.
  Perelman's twice contracted differential Bianchi identity
  (\ref{II-contr-Bianchi}) implies
  \begin{equation}
    \nabla_{g_t}^{\ast_{\Omega_t}} \dot{g}^{\ast}_t + \nabla_{g_t} 
    \dot{\Omega}^{\ast}_t = 0 \label{evolv-Bianch} .
  \end{equation}
  Then the fundamental variation formula (\ref{var-H}) implies the evolution
  equation along the Soliton-Ricci flow
  \begin{equation}
    2 \frac{d}{d t}  \underline{H}_t = - (\Delta^{\Omega_t}_{g_t} -
    2\mathbbm{I})  \underline{H}_t + \left| h_t \right|^2_{g_t} -
    \dot{\mathcal{W}}_t . \label{evolv-H}
  \end{equation}
  This combined with the monotony statement in lemma \ref{Monot-SKRF} or in \cite{Pal1} implies the inequality
  \begin{equation}
    2 \frac{d}{d t}  \underline{H}_t \leqslant - (\Delta^{\Omega_t}_{g_t} -
    2\mathbbm{I})  \underline{H}_t + \left| h_t \right|^2_{g_t},
    \label{estm-evolv-H}
  \end{equation}
  along the Soliton-K\"ahler-Ricci flow.
\end{remark}

\section{The second variation of the $\mathcal{W}$ functional along the
Soliton-K\"ahler-Ricci flow}

Let $(J_t, g_t, \Omega_t)_{t \geqslant 0}$ be the Soliton-K\"ahler-Ricci flow. In the proof of step I
of lemma \ref{Monot-SKRF} we obtained the identity
\begin{eqnarray*}
  \dot{\mathcal{W}}_t & = & - 2 \int_X f_t  \dot{f}_t e^{- f_t} 
  \frac{\omega^n}{n!} .
\end{eqnarray*}
Time deriving this we obtain
\begin{eqnarray*}
  \ddot{\mathcal{W}}_t & = & - 2 \int_X \dot{f}^2_t \Omega_t - 2 \int_X f_t 
  \left( \ddot{f}_t - \dot{f}^2_t \right) \Omega_t .
\end{eqnarray*}
Time deriving the identity
\begin{eqnarray*}
  0 & = & \int_X \dot{f}_t \Omega_t \equiv \int_X \dot{f}_t e^{- f_t} 
  \frac{\omega^n}{n!},
\end{eqnarray*}
we deduce
\begin{eqnarray*}
  0 & = & \int_X \left( \ddot{f}_t - \dot{f}^2_t \right) \Omega_t,
\end{eqnarray*}
and thus the evolution formula
\begin{equation}
  \label{II-varW-SKRF}  \ddot{\mathcal{W}}_t = - 2 \int_X \dot{f}^2_t \Omega_t
  - 2 \int_X F_t  \left( \ddot{f}_t - \dot{f}^2_t \right) \Omega_t .
\end{equation}
We observe now that the second evolution equation in the system
(\ref{SRF-Cool}) rewrites as
\begin{eqnarray*}
  2 \dot{f}_t & = & - \Delta^{\Omega_t}_{g_t} f_t + 2 f_t -\mathcal{W}_t,
\end{eqnarray*}
thanks to (\ref{SCool-W}). Time deriving this we infer
\begin{equation}
  \label{II-var-f} - 2 \ddot{f}_t = \frac{d}{d t} \Delta^{\Omega_t}_{g_t} f_t
  - 2 \dot{f}_t + \dot{\mathcal{W}}_t .
\end{equation}
Plugging the identity (\ref{evolv-Bianch}) in the variation formula
(\ref{var-Lapf}) and using the first equation in the system (\ref{SRF-Cool})
we obtain
\begin{eqnarray*}
  \frac{d}{d t} \Delta^{\Omega_t}_{g_t} f_t & = & \Delta^{\Omega_t}_{g_t} 
  \underline{H}_t - \left| \overline{\partial}_{T_{X, J_t}} \nabla_{g_t} f_t 
  \right|^2_{g_t} \\
  &  & \\
  & = & \Delta^{\Omega_t}_{g_t}  \underline{H}_t - \left| h_t
  \right|^2_{g_t},
\end{eqnarray*}
with $2 \underline{H}_t = - (\Delta^{\Omega_t}_{g_t} - 2\mathbbm{I}) F_t$.
Thus $\dot{f}_t = \underline{H}_t$ thanks to (\ref{Cool-evol-f}). Using
(\ref{II-var-f}) we infer
\begin{eqnarray*}
  - 2 \ddot{f}_t & = & (\Delta^{\Omega_t}_{g_t} - 2\mathbbm{I}) 
  \underline{H}_t - \left| h_t \right|^2_{g_t} + \dot{\mathcal{W}}_t .
\end{eqnarray*}
(This last follows also from the general evolution formula (\ref{evolv-H}).)
Integrating by parts we obtain the identity
\begin{eqnarray*}
  - 2 \int_X F_t  \ddot{f}_t \Omega_t & = & - \int_X \left[ 2
  \underline{H}^2_t + F_t  \left| h_t \right|^2_{g_t} \right]
  \Omega_t,
\end{eqnarray*}
(since $\int_X F_t \Omega_t = 0$). Plunging this identity in the evolution
formula (\ref{II-varW-SKRF}) we deduce the simple second variation formula
\begin{eqnarray*}
  \ddot{\mathcal{W}}_t & = & - \int_X \left[ 4\, \underline{H}^2_t + \left(
  \left| h_t \right|^2_{g_t} - 2\,  \underline{H}^2_t
  \right) F_t \right] \Omega_t\\
  &  & \\
  &=&- \int_X \left[ \left| \overline{\partial}_{T_{X, J_t}} \nabla_{g_t} f_t \right|^2_{g_t} F_t + \frac{1}{2} 
  \left| (\Delta^{\Omega_t}_{g_t} - 2\mathbbm{I})
  \right|^2 (2 - F_t)\right]\Omega_t .
\end{eqnarray*}

\section{The Levi-Civita connection of the \\
pseudo-Riemannian structure $G$}

In this section we compute the Levi-Civita connection of the pseudo-Riemannian
structure $G$. This is needed for the computation of the second variation of
the $\mathcal{W}$ functional with respect to such structure. We set for notations simplicity $\mathcal{T}
\assign T_{\mathcal{M} \times \mathcal{V}_1}$ and we compute the first
variation of $G$ at an arbitrary point $(g, \Omega)$,
\[ D_{g, \Omega} G : \mathcal{T} \times \mathcal{T} \longrightarrow
   \mathcal{T}^{\ast} . \]
In a direction $(\theta, \Theta) \in \mathcal{T}$ this is given by the
identity
\begin{eqnarray*}
  D_{g, \Omega} G (\theta, \Theta ; u, V)  (v, V) & = & \frac{d}{d t}
  _{\mid_{t = 0}} G_{g_t, \Omega_t} (u, U ; v, V),
\end{eqnarray*}
where $(g_t, \Omega_t)_{t \in (- \varepsilon, \varepsilon)} \subset
\mathcal{M} \times \mathcal{V}_1$ is a smooth curve with $(g_0, \Omega_0) =
(g, \Omega)$ and $( \dot{g}_0, \dot{\Omega}_0) = (\theta, \Theta)$. For
notation simplicity let denote $u^{\ast}_t : = g^{- 1}_t u$ and $U^{\ast}_t
\assign U / \Omega_t$. Then holds the equality
\begin{eqnarray*}
  D_{g, \Omega} G (\theta, \Theta ; u, V)  (v, V) & = & \frac{d}{d t}
  _{\mid_{t = 0}}  \left[ \int_X \tmop{Tr}_{\mathbbm{R}} (u_t^{\ast}
  \hspace{0.25em} v^{\ast}_t) \Omega_t - 2 \int_X U^{\ast}_t V \right] \\
  &  & \\
  & = & \int_X \left[ \frac{d}{d t} _{\mid_{t = 0}} \tmop{Tr}_{\mathbbm{R}}
  (u_t^{\ast} \hspace{0.25em} v^{\ast}_t) \right] \Omega_t + \int_X
  \tmop{Tr}_{\mathbbm{R}} (u_t^{\ast} \hspace{0.25em} v^{\ast}_t) \Theta\\
  &  & \\
  & - & 2 \int_X \frac{d}{d t} _{\mid_{t = 0}} U^{\ast}_t V .
\end{eqnarray*}
Using the identity $\frac{d}{dt} \hspace{0.25em} u_t^{\ast} = -
\dot{g}_t^{\ast} \hspace{0.25em} u_t^{\ast}$, which follows from the formula
$$
\frac{d}{dt} \hspace{0.25em} g_t^{- 1} = - g_t^{- 1} \hspace{0.25em}
\dot{g}_t \hspace{0.25em} g_t^{- 1},
$$
we obtain
\begin{eqnarray*}
  \frac{d}{d t} \tmop{Tr}_{\mathbbm{R}} (u_t^{\ast} \hspace{0.25em}
  v^{\ast}_t) & = & \tmop{Tr}_{\mathbbm{R}} \left( \frac{d}{dt} u_t^{\ast}
  \hspace{0.25em} v^{\ast}_t \hspace{0.75em} + \hspace{0.75em} u_t^{\ast}
  \hspace{0.25em}  \frac{d}{dt} \hspace{0.25em} v^{\ast}_t \right)\\
  &  & \\
  & = & - \tmop{Tr}_{\mathbbm{R}} ( \dot{g}_t^{\ast} \hspace{0.25em}
  u_t^{\ast} \hspace{0.25em} v^{\ast}_t \hspace{0.75em} + \hspace{0.75em}
  u_t^{\ast} \hspace{0.25em}  \dot{g}_t^{\ast} \hspace{0.25em} v_t^{\ast})\\
  &  & \\
  & = & - \hspace{0.75em} 2 \tmop{Tr}_{\mathbbm{R}} ( \dot{g}_t^{\ast}
  \hspace{0.25em} u_t^{\ast} \hspace{0.25em} v^{\ast}_t) \hspace{0.25em},
\end{eqnarray*}
since $\dot{g}_t$ is also symmetric. Indeed we observe the elementary
identities
\begin{eqnarray*}
  \tmop{Tr}_{\mathbbm{R}} \left[ (u_t^{\ast} \hspace{0.25em} 
  \dot{g}_t^{\ast}) v_t^{\ast} \right] & = & \tmop{Tr}_{\mathbbm{R}} \left[
  v_t^{\ast} (u_t^{\ast} \hspace{0.25em}  \dot{g}_t^{\ast}) \right]\\
  &  & \\
  & = & \tmop{Tr}_{\mathbbm{R}} \left[ v_t^{\ast} (u_t^{\ast} \hspace{0.25em}
  \dot{g}_t^{\ast}) \right]_t^T\\
  &  & \\
  & = & \tmop{Tr}_{\mathbbm{R}} \left[ (u_t^{\ast} \hspace{0.25em} 
  \dot{g}_t^{\ast})_t^T v^{\ast}_t \right]\\
  &  & \\
  & = & \tmop{Tr}_{\mathbbm{R}} ( \dot{g}_t^{\ast} \hspace{0.25em} u_t^{\ast}
  \hspace{0.25em} v^{\ast}_t),
\end{eqnarray*}
where $A_t^T$ denotes the transpose of $A$ with respect to $g_t$. Time
deriving the identity $U = U_t^{\ast} \Omega_t$ we infer
\begin{eqnarray*}
  0 & = & \frac{d U^{\ast}_t}{d t} \Omega_t + U_t^{\ast}  \dot{\Omega}_t,
\end{eqnarray*}
and thus
\begin{eqnarray*}
  \frac{d U^{\ast}_t}{d t} & = & - U_t^{\ast}  \dot{\Omega}^{\ast}_t .
\end{eqnarray*}
Summing up we infer the expression of the variation of $G$ at the point $(g,
\Omega)$ in the direction $(\theta, \Theta)$
\[ D_{g, \Omega} G (\theta, \Theta ; u, U)  (v, V) \hspace{0.75em} = \int_X \{
   \tmop{Tr}_{\mathbbm{R}} [ (\Theta^{\ast}_{\Omega} - 2 \theta_g^{\ast})
   u_g^{\ast} \hspace{0.25em} v^{\ast}_g] + 2 \Theta^{\ast}_{\Omega}
   U^{\ast}_{\Omega} V^{\ast}_{\Omega} \} \hspace{0.25em} \Omega . \]
We can compute now the Levi-Civita connection $\nabla_G = D + \Gamma_G$ of the
pseudo-Riemannian structure $G$. At a point $(g, \Omega)$ the symmetric
bilinear form
\[ \Gamma_G (g, \Omega) : \mathcal{T} \times \mathcal{T} \longrightarrow
   \mathcal{T}, \]
is identified by the expression
\begin{eqnarray*}
 && 2 G_{g, \Omega} \left( \Gamma_G (g, \Omega) (u, U ; v, V) ; \theta, \Theta
  \right) 
  \\
  \\
  & = & [ D_{g, \Omega} G (u, U ; v, V) + D_{g, \Omega} G (v, V ; u,
  U)] \left( \theta, \Theta \right)\\
  &  & \\
  & - & D_{g, \Omega} G (\theta, \Theta ; u, U)  (v, V) .
\end{eqnarray*}
Expanding and arranging the terms of the right hand side we obtain
\begin{eqnarray*}
  &  & 2 G_{g, \Omega} \left( \Gamma_G (g, \Omega) (u, U ; v, V) ; \theta,
  \Theta \right)\\
  &  & \\
  & = & \int_X \left\{ \tmop{Tr}_{\mathbbm{R}} \left[ (U^{\ast}_{\Omega} - 2
  u_g^{\ast}) v_g^{\ast} \hspace{0.25em} \theta^{\ast}_g\right] + 2
  U^{\ast}_{\Omega} V^{\ast}_{\Omega} \Theta^{\ast}_{\Omega} \right\}
  \hspace{0.25em} \Omega\\
  &  & \\
  & + & \int_X \left\{ \tmop{Tr}_{\mathbbm{R}} \left[ (V^{\ast}_{\Omega} - 2
  v_g^{\ast}) u_g^{\ast} \hspace{0.25em} \theta^{\ast}_g\right] + 2
  V^{\ast}_{\Omega} U^{\ast}_{\Omega} \Theta^{\ast}_{\Omega} \right\}
  \hspace{0.25em} \Omega\\
  &  & \\
  & - & \int_X \left\{ \tmop{Tr}_{\mathbbm{R}} \left[ (\Theta^{\ast}_{\Omega} - 2
  \theta_g^{\ast}) u_g^{\ast} \hspace{0.25em} v^{\ast}_g\right] + 2
  \Theta^{\ast}_{\Omega} U^{\ast}_{\Omega} V^{\ast}_{\Omega} \right\}
  \hspace{0.25em} \Omega\\
  &  & \\
  & = & \int_X \left\{ \tmop{Tr}_{\mathbbm{R}} \left[ (U^{\ast}_{\Omega} - 2
  u_g^{\ast}) v_g^{\ast} \hspace{0.25em} \theta^{\ast}_g + V^{\ast}_{\Omega}
  u_g^{\ast} \hspace{0.25em} \theta^{\ast}_g - \Theta^{\ast}_{\Omega}
  u_g^{\ast} \hspace{0.25em} v^{\ast}_g\right] + 2 U^{\ast}_{\Omega}
  V^{\ast}_{\Omega} \Theta^{\ast}_{\Omega} \right\} \Omega\\
  &  & \\
  & = & \int_X \tmop{Tr}_{\mathbbm{R}} \left[ ( u_g^{\ast} (V^{\ast}_{\Omega} -
  v_g^{\ast}) + v_g^{\ast} (U^{\ast}_{\Omega} - u_g^{\ast})) \hspace{0.25em}
  \theta^{\ast}_g\right]  \hspace{0.25em} \Omega\\
  &  & \\
  & - & \int_X \left[ \tmop{Tr}_{\mathbbm{R}} (u_g^{\ast} \hspace{0.25em}
  v^{\ast}_g) - 2 U^{\ast}_{\Omega} V^{\ast}_{\Omega}\right] \Theta^{\ast}_{\Omega} 
  \hspace{0.25em} \Omega\\
  &  & \\
  & = & \int_X \left\langle u (V^{\ast}_{\Omega} - v_g^{\ast}) + v
  (U^{\ast}_{\Omega} - u_g^{\ast}), \theta \right\rangle_g  \hspace{0.25em}
  \Omega \\
  &  & \\
  & - & 2 \int_X \left[ \frac{1}{2}  \left\langle u, v \right\rangle_g -
  U^{\ast}_{\Omega} V^{\ast}_{\Omega} - \frac{1}{2} G_{g, \Omega} (u, U ; v,
  V)\right] \Theta^{\ast}_{\Omega}  \hspace{0.25em} \Omega,
\end{eqnarray*}
since $\int_X \Theta = 0$. We infer the expression
\begin{eqnarray*}
  (\psi, \Psi) & \equiv & \Gamma_G (g, \Omega) (u, U ; v, V),\\
  &  & \\
  \psi & = & \frac{1}{2}  \left[ u (V^{\ast}_{\Omega} - v_g^{\ast}) + v
  (U^{\ast}_{\Omega} - u_g^{\ast})\right],\\
  &  & \\
  \Psi & = & \frac{1}{4}  \left[ \left\langle u, v \right\rangle_g - 2
  U^{\ast}_{\Omega} V^{\ast}_{\Omega} - G_{g, \Omega} (u, U ; v, V)\right] \Omega .
\end{eqnarray*}
This concludes the computation of the Levi-Civita connection $\nabla_G$.

\section{The second variation of the $\mathcal{W}$ functional with respect to
the pseudo-Riemannian structure $G$}

We justify first the geometric interpretation of $\mathbbm{F}_{g, \Omega}$
provided by the identity (\ref{F-OrtoOrb}). We observe indeed that $(v, V) \in
T^{\bot_G}_{\left[ g, \Omega \right], (g, \Omega)}$ if and only if 
$$G_{g,
\Omega} (L_{\xi} g, L_{\xi} \Omega ; v, V) = 0,
$$ for all $\xi \in C^{\infty}
(X, T_X)$, i.e
\begin{eqnarray*}
  0 & = & \int_X \left[ \left\langle L_{\xi} g, v \right\rangle_g - 2 (L_{\xi}
  \Omega)^{\ast}_{\Omega} V^{\ast}_{\Omega} \right] \Omega\\
  &  & \\
  & = & 2 \int_X \left[ \left\langle \nabla_g \xi, v_g^{\ast} \right\rangle_g
  - (\tmop{div}^{\Omega} \xi) V^{\ast}_{\Omega} \right] \Omega\\
  &  & \\
  & = & 2 \int_X \left\langle \xi, \nabla_g^{\ast_{\Omega}} v_g^{\ast} +
  \nabla_g V^{\ast}_{\Omega} \right\rangle_g \Omega,
\end{eqnarray*}
which shows the required conclusion. We introduce now the operator
\[ \mathcal{L}^{\Omega}_g : C^{\infty} (X, \tmop{End} (T_X)) \longrightarrow
   C^{\infty} (X, \tmop{End} (T_X)), \]
defined by the formula
\begin{eqnarray*}
  \mathcal{L}^{\Omega}_g A & \assign & \Delta^{\Omega}_g A - 2\mathcal{R}_g
  \ast A .
\end{eqnarray*}
By abuse of notations we define also
\[ \mathcal{L}^{\Omega}_g : C^{\infty} (X, S^2 T^{\ast}_X) \longrightarrow
   C^{\infty} (X, S^2 T^{\ast}_X), \]
defined by the same formula
\begin{eqnarray*}
  \mathcal{L}^{\Omega}_g v & \assign & \Delta^{\Omega}_g v - 2\mathcal{R}_g
  \ast v .
\end{eqnarray*}
We observe that (\ref{ast-curv-Id}) implies the identity
$(\mathcal{L}^{\Omega}_g v)_g^{\ast} =\mathcal{L}^{\Omega}_g v_g^{\ast}$. We
show now the second variation formula for the $\mathcal{W}$ functional.

\begin{lemma}
  \label{Sec-Var-W}The Hessian endomorphism $\nabla^2_G \mathcal{W} (g,
  \Omega)$ of the $\mathcal{W}$ functional with respect to the
  pseudo-Riemannian structure $G$ at the point $(g, \Omega) \in \mathcal{M}
  \times \mathcal{V}_1$ in the directions $(v, V) \in \mathbbm{F}_{g, \Omega}$
  is given by the expressions
  \begin{eqnarray*}
    (u, U) & \equiv & \nabla^2_G \mathcal{W} (g, \Omega) (v, V),\\
    &  & \\
    u & \assign & - \frac{1}{2}  \left( \mathcal{L}^{\Omega}_g +
    \underline{H}_{g, \Omega} \right) v - \frac{1}{2} V_{\Omega}^{\ast} h_{g,
    \Omega},\\
    &  & \\
    U^{\ast}_{\Omega} & \assign & - \frac{1}{2}  \left( \Delta^{\Omega}_g +
    \underline{H}_{g, \Omega} - 2\mathbbm{I} \right) V_{\Omega}^{\ast} +
    \frac{1}{4}  \left\langle h_{g, \Omega}, v \right\rangle_g + \frac{1}{4}
    D_{g, \Omega} \mathcal{W} (v, V) .
  \end{eqnarray*}
  In particular if $h_{g, \Omega} = 0$ then
  \begin{eqnarray*}
    u & = & - \frac{1}{2} \mathcal{L}^{\Omega}_g v,\\
    &  & \\
    U^{\ast}_{\Omega} & = & - \frac{1}{2}  (\Delta^{\Omega}_g - 2\mathbbm{I})
    V_{\Omega}^{\ast} .
  \end{eqnarray*}
\end{lemma}

\begin{proof}
  We consider a smooth curve $(g_t, \Omega_t)_{t \in \mathbbm{R}} \subset
  \mathcal{M} \times \mathcal{V}_1$ with $(g_0, \Omega_0) = (g, \Omega)$ and
  with arbitrary speed $( \dot{g}_0, \dot{\Omega}_0) = (v, V)$. We observe
  that the $G$-covariant derivative of its speed is given by the expressions
  \begin{eqnarray*}
    (\theta_t, \Theta_t) & \equiv & \nabla_G ( \dot{g}_t, \dot{\Omega}_t)  (
    \dot{g}_t, \dot{\Omega}_t) = ( \ddot{g}_t, \ddot{\Omega}_t) + \Gamma (g_t,
    \Omega_t)  ( \dot{g}_t, \dot{\Omega}_t ; \dot{g}_t, \dot{\Omega}_t),\\
    &  & \\
    \theta_t & \assign & \ddot{g}_t + \dot{g}_t \left( \dot{\Omega}^{\ast}_t -
    \dot{g}_t^{\ast} \right),\\
    &  & \\
    \Theta_t & \assign & \ddot{\Omega}_t + \frac{1}{4}  \left[ \left|
    \dot{g}_t \right|^2_t - 2 ( \dot{\Omega}^{\ast}_t)^2 - G_{g_t, \Omega_t} (
    \dot{g}_t, \dot{\Omega}_t ; \dot{g}_t, \dot{\Omega}_t) \right] \Omega_t .
  \end{eqnarray*}
  We infer
  \begin{eqnarray*}
    \theta^{\ast}_t & = & \frac{d}{d t}  \dot{g}_t^{\ast} +
    \dot{\Omega}^{\ast}_t  \dot{g}_t^{\ast},\\
    &  & \\
    \Theta^{\ast}_t & = & \frac{d}{d t}  \dot{\Omega}_t^{\ast} + \frac{1}{2} 
    ( \dot{\Omega}^{\ast}_t)^2 + \frac{1}{4}  \left| \dot{g}_t \right|^2_{g_t}
    - \frac{1}{4} G_{g_t, \Omega_t} ( \dot{g}_t, \dot{\Omega}_t ; \dot{g}_t,
    \dot{\Omega}_t) .
  \end{eqnarray*}
  Using this expressions and Perelman's first variation formula we expand the
  Hessian form
  \begin{eqnarray*}
    &  & \nabla_G D\mathcal{W} (g_t, \Omega_t)  ( \dot{g}_t, \dot{\Omega}_t ;
    \dot{g}_t, \dot{\Omega}_t)\\
    &  & \\
    & = & \frac{d^2}{d t^2} \mathcal{W} (g_t, \Omega_t) - D_{g_t, \Omega_t}
    \mathcal{W} (\theta_t, \Theta_t)\\
    &  & \\
    & = & - \frac{d}{d t} \int_X \left[ \tmop{Tr}_{\mathbbm{R}} \left(
    \dot{g}_t^{\ast} h^{\ast}_t \right) - 2 \dot{\Omega}_t^{\ast} H_t \right]
    \Omega_t\\
    &  & \\
    & + & \int_X \left[ \tmop{Tr}_{\mathbbm{R}} \left( \theta^{\ast}_t
    h^{\ast}_t \right) - 2 \Theta^{\ast}_t H_t \right] \Omega_t\\
    &  & \\
    & = & - \int_X \left[ \tmop{Tr}_{\mathbbm{R}} \left( \frac{d}{d t} 
    \dot{g}_t^{\ast} h^{\ast}_t + \dot{g}_t^{\ast}  \frac{d}{d t} h^{\ast}_t
    \right) - 2 \frac{d}{d t}  \dot{\Omega}_t^{\ast} H_t - 2
    \dot{\Omega}_t^{\ast}  \dot{H}_t \right] \Omega_t\\
    &  & \\
    & - & \int_X \left[ \tmop{Tr}_{\mathbbm{R}} \left( \dot{g}_t^{\ast}
    h^{\ast}_t \right) - 2 \dot{\Omega}_t^{\ast} H_t \right]  \dot{\Omega}_t\\
    &  & \\
    & + & \int_X \left[ \tmop{Tr}_{\mathbbm{R}} \left( \theta^{\ast}_t
    h^{\ast}_t \right) - 2 \Theta^{\ast}_t H_t \right] \Omega_t\\
    &  & \\
    & = & - \int_X \left\{ \tmop{Tr}_{\mathbbm{R}} \left[ \dot{g}_t^{\ast} 
    \left( \dot{h}^{\ast}_t - \dot{g}_t^{\ast} h^{\ast}_t \right) \right] - 2
    \dot{\Omega}_t^{\ast}  \dot{H}_t \right\} \Omega_t\\
    &  & \\
    & - & \frac{1}{2} \int_X \left[ \left| \dot{g}_t \right|^2_{g_t} - 2 (
    \dot{\Omega}_t^{\ast})^2 - G_{g_t, \Omega_t} ( \dot{g}_t, \dot{\Omega}_t ;
    \dot{g}_t, \dot{\Omega}_t) \right] H_t \Omega_t .
  \end{eqnarray*}
  Using the variation formulas (\ref{var-h}) and (\ref{var-H}) and evaluating
  the previous identity at time $t = 0$ we obtain the expression
  \begin{eqnarray*}
    &  & \nabla_G D\mathcal{W} (g, \Omega) (v, V ; v, V)\\
    &  & \\
    & = & - \frac{1}{2}  \int_X \left[ \left\langle \mathcal{L}^{\Omega}_g v
    - L_{\nabla_g^{\ast_{\Omega}} v_g^{\ast} + \nabla_g V^{\ast}_{\Omega}} g,
    v \right\rangle_g \right] \Omega\\
    &  & \\
    & - & \frac{1}{2} \int_X \{ - 2 V_{\Omega}^{\ast} [ (\Delta^{\Omega}_g -
    2\mathbbm{I}) V_{\Omega}^{\ast} - \tmop{div}^{\Omega} \left(
    \nabla_g^{\ast_{\Omega}} v_g^{\ast} + \nabla_g V^{\ast}_{\Omega} \right) -
    \left\langle v, h_{g, \Omega} \right\rangle_g] \} \Omega\\
    &  & \\
    & - & \frac{1}{2}  \int_X [ \left| v \right|^2_g - 2
    (V_{\Omega}^{\ast})^2]  \underline{H}_{g, \Omega} \Omega,
  \end{eqnarray*}
  since $\int_X \underline{H}_{g, \Omega} \Omega = 0$. Arranging symmetrically
  the integrand terms via the identity
  \begin{eqnarray*}
    \nabla_G D\mathcal{W} (g, \Omega) (v, V ; v, V) & = & G_{g, \Omega} (u, U
    ; v, V),\\
    &  & \\
    (u, U) & \equiv & \nabla^2_G \mathcal{W} (g, \Omega) (v, V),
  \end{eqnarray*}
  we infer the general expressions
  \begin{eqnarray*}
    u & = & - \frac{1}{2}  \left( \mathcal{L}^{\Omega}_g + \underline{H}_{g,
    \Omega} \right) v + \frac{1}{2} L_{\nabla_g^{\ast_{\Omega}} v_g^{\ast} +
    \nabla_g V^{\ast}_{\Omega}} g - \frac{1}{2} V_{\Omega}^{\ast} h_{g,
    \Omega},\\
    &  & \\
    &  & \\
    U^{\ast}_{\Omega} & = & - \frac{1}{2}  \left( \Delta^{\Omega}_g +
    \underline{H}_{g, \Omega} - 2\mathbbm{I} \right) V_{\Omega}^{\ast} \\
    &  & \\
    & + & \frac{1}{2}  \left( L_{\nabla_g^{\ast_{\Omega}} v_g^{\ast} +
    \nabla_g V^{\ast}_{\Omega}} \Omega \right)_{\Omega}^{\ast} + \frac{1}{4} 
    \left\langle h_{g, \Omega}, v \right\rangle_g + \frac{1}{4} D_{g, \Omega}
    \mathcal{W} (v, V) .
  \end{eqnarray*}
  Then the required expression of the Hessian of $\mathcal{W}$ follows from
  the assumption $(v, V) \in \mathbbm{F}_{g, \Omega}$. If $h_{g, \Omega} = 0$
  then the required conclusion follows from Perelman's twice contracted second
  Bianchi identity (\ref{II-contr-Bianchi}) which implies $\underline{H}_{g,
  \Omega} = 0$.
\end{proof}

\section{The anomaly space of the pseudo-Riemannian structure $G$}\label{Anomaly}

Let $\tmop{Isom}^0_{g, \Omega}$ be the identity component of the group
\begin{eqnarray*}
  \tmop{Isom}_{g, \Omega} & : = & \left\{ \varphi \in \tmop{Diff} (X) \mid
  \varphi^{\ast} g = g, \varphi^{\ast} \Omega = \Omega \right\},
\end{eqnarray*}
and let
\begin{eqnarray*}
  \tmop{Kill}_{g, \Omega} & \assign & \tmop{Lie} \left( \tmop{Isom}^0_{g,
  \Omega} \right) \equiv \left\{ \xi \in C^{\infty} (X, T_X) \mid L_{\xi} g =
  0, L_{\xi} \Omega = 0 \right\} .
\end{eqnarray*}
We define the anomaly space of the pseudo-Riemannian structure $G$ at an
arbitrary point $(g, \Omega)$ as the vector space
\begin{eqnarray*}
  \mathbbm{A}^{\Omega}_g & \assign & \mathbbm{F}_{g, \Omega} \cap T_{[g,
  \Omega], g, \Omega} .
\end{eqnarray*}
We will study some properties of this space. It is clear by definition that
this space is generated by the vector fields $\xi \in C^{\infty} (X, T_X)$
such that
\begin{eqnarray*}
  0 & = & \left[ \nabla_g^{\ast_{\Omega}}  \hat{\nabla}_g + d
  \tmop{div}_g^{\Omega} \right] (g \xi) = \hat{\Delta}^{\Omega}_g  (g \xi) .
\end{eqnarray*}
More precisely there exists the exact sequence of finite dimensional vector
spaces
\begin{eqnarray*}
  0 \longrightarrow \tmop{Kill}_{g, \Omega} \longrightarrow & \tmop{Ker}
  \hat{\Delta}^{\Omega}_g & \longrightarrow \mathbbm{A}^{\Omega}_g
  \longrightarrow 0\\
  &  & \\
  \xi \mapsto & g \xi = \alpha & \mapsto \left( \hat{\nabla}_g \alpha,
  (\tmop{div}_g^{\Omega} \alpha) \Omega \right) .
\end{eqnarray*}
We observe that if $\alpha = d u \in \tmop{Ker} \hat{\Delta}^{\Omega}_g$ then
the function $u$ satisfies the equation
\begin{eqnarray*}
  2 \Delta^{\Omega}_g \nabla_g u - \nabla_g \Delta^{\Omega}_g u & = & 0,
\end{eqnarray*}
which is equivalent to the equation
\begin{equation}
  \label{grad-eq-anom}  \left[ \Delta^{\Omega}_g - \tmop{Ric}_g^{\ast}
  (\Omega) \right] \nabla_g u = 0,
\end{equation}
thanks to the general identity
\begin{equation}
  \label{com-grad-Lap} \nabla_g \Delta^{\Omega}_g u = \Delta^{\Omega}_g
  \nabla_g u + \tmop{Ric}_g^{\ast} (\Omega) \nabla_g u .
\end{equation}
We set
\begin{eqnarray*}
  \mathbbm{V}_{g, \Omega} &: =& \left\{ \alpha \in \tmop{Ker}
  \hat{\Delta}^{\Omega}_g \mid \alpha = d u \right\}
  \\
  \\
  & \cong & \left\{ u \in
  C_{\Omega}^{\infty} (X, \R)_0 \mid \left[ \Delta^{\Omega}_g -
  \tmop{Ric}_g^{\ast} (\Omega) \right] \nabla_g u = 0 \right\} .
\end{eqnarray*}
We observe that in the soliton case $h_{g, \Omega} = 0$ we have
\begin{equation}
  \label{V-iso-Eig} \mathbbm{V}_{g, \Omega} \cong \tmop{Ker}
  (\Delta^{\Omega}_g - 2\mathbbm{I}) \subset C_{\Omega}^{\infty} (X,
  \mathbbm{R})_0,
\end{equation}
thanks to the identity (\ref{com-grad-Lap}). By duality we can consider
$\tmop{Kill}_{g, \Omega} \subset \tmop{Ker} \hat{\Delta}^{\Omega}_g$ and we
observe the inclusion
\begin{equation}
  \label{V-cont-OrtLIE} \mathbbm{V}_{g, \Omega} \subseteq \tmop{Kill}_{g,
  \Omega}^{\bot_{g, \Omega}},
\end{equation}
where the symbol $\bot_{g, \Omega}$ indicates the orthogonal space inside
$\tmop{Ker} \hat{\Delta}^{\Omega}_g$ with respect to the scalar product
(\ref{Glb-Rm-m}) at the level of 1-forms. The previous inclusion holds true for
any $(g, \Omega)$ since
\begin{eqnarray*}
  \int_X \left\langle d u, \beta \right\rangle_g \Omega & = & - \int_X
  \left\langle u, \tmop{div}_g^{\Omega} \beta \right\rangle_g \Omega = 0,
\end{eqnarray*}
for any $\beta \in \tmop{Kill}_{g, \Omega}$. We infer that in the soliton case
the previous exact sequence can be reduced to the sequence
\begin{eqnarray*}
  0 \longrightarrow \tmop{Ker} (\Delta^{\Omega}_g - 2\mathbbm{I}) &
  \longrightarrow & \mathbbm{A}^{\Omega}_g\\
  &  & \\
  u & \mapsto & 2 \left( \nabla_g d u, - u \Omega \right) .
\end{eqnarray*}
In order to show that the previous map is also surjective we need to show a
few differential identities. We show first the Weitzenb\"ock type formula
\begin{equation}
  \label{Sm-Wei-1-form}  \hat{\Delta}^{\Omega}_g \alpha = \Delta^{\Omega}_g
  \alpha - \alpha \tmop{Ric}_g^{\ast} (\Omega) .
\end{equation}
(This implies in particular the identification of $\mathbbm{V}_{g, \Omega}$ in
terms of functions). We decompose the expression
\begin{equation}
  \label{def-OmSmLap1}  \hat{\Delta}^{\Omega}_g \alpha = \left[
  \nabla^{\ast_{\Omega}}_g  \hat{\nabla}_g - d \nabla^{\ast_{\Omega}}_g
  \right] \alpha .
\end{equation}
We decompose first the term
\begin{eqnarray*}
  \nabla^{\ast_{\Omega}}_g  \hat{\nabla}_g \alpha \cdot \xi & = &
  \nabla^{\ast_{}}_g  \hat{\nabla}_g \alpha \cdot \xi + \hat{\nabla}_g \alpha
  (\nabla_g f, \xi) .
\end{eqnarray*}
We fix an arbitrary point $p$ and we choose the vector fields $\xi$ and $\eta$
such that $0 = \nabla_g \xi (p) = \nabla_g \eta (p)$. Let $(e_k)_k$ be a
$g$-orthonormal local frame such that $\nabla_g e_k  (p) = 0$. Then at the
point $p$ hold the identities
\begin{eqnarray*}
  \nabla^{\ast_{}}_g  \hat{\nabla}_g \alpha \cdot \xi & = & - \nabla_{g, e_k}
  \hat{\nabla}_g \alpha (e_k, \xi)\\
  &  & \\
  & = & - \nabla_{g, e_k} \left[ \hat{\nabla}_g \alpha (e_k, \xi) \right]\\
  &  & \\
  & = & - \nabla_{g, e_k} \left[ \nabla_{g, e_k} \alpha \cdot \xi +
  \nabla_{g, \xi} \alpha \cdot e_k \right]\\
  &  & \\
  & = & - \nabla_{g, e_k} \nabla_{g, e_k} \alpha \cdot \xi - \nabla_{g, e_k}
  \nabla_{g, \xi} \alpha \cdot e_k .
\end{eqnarray*}
We infer the expression
\begin{eqnarray*}
  \nabla^{\ast_{\Omega}}_g  \hat{\nabla}_g \alpha \cdot \xi & = &
  \Delta^{\Omega}_g \alpha - \nabla_{g, e_k} \nabla_{g, \xi} \alpha \cdot e_k
  + \nabla_g \alpha (\xi, \nabla_g f) .
\end{eqnarray*}
Moreover
\begin{eqnarray*}
  d \nabla^{\ast_{\Omega}}_g \alpha (\xi) & = & - \nabla_{g, \xi}
  \nabla^{}_{g, e_k} \alpha \cdot e_k + \nabla^{}_{g, \xi} \alpha \cdot
  \nabla_g f + \alpha \cdot \nabla^2_{g, \xi} f .
\end{eqnarray*}
Summing up we deduce
\begin{eqnarray*}
  \hat{\Delta}^{\Omega}_g \alpha \cdot \xi & = & \Delta^{\Omega}_g \alpha
  \cdot \xi + \left( \nabla_{g, \xi} \nabla^{}_{g, e_k} \alpha - \nabla_{g,
  e_k} \nabla_{g, \xi} \alpha \right) \cdot e_k - \alpha \cdot \nabla^2_{g,
  \xi} f\\
  &  & \\
  & = & \Delta^{\Omega}_g \alpha \cdot \xi - \alpha \cdot \mathcal{R}_g (\xi,
  e_k) e_k - \alpha \cdot \nabla^2_{g, \xi} f,
\end{eqnarray*}
thanks to the dual identity
\begin{equation}
  \label{dual-com} \nabla_{g, \xi} \nabla^{}_{g, \eta} \alpha - \nabla_{g,
  \eta} \nabla_{g, \xi} \alpha = \nabla_{g, \left[ \xi, \eta \right]} \alpha -
  \alpha \cdot \mathcal{R}_g (\xi, \eta),
\end{equation}
and to the fact that $\left[ \xi, e_k \right] (p) = 0$. We infer the required
formula (\ref{Sm-Wei-1-form}). We deduce that in the soliton case $h_{g,
\Omega} = 0$ holds the equality
\begin{equation}
  \label{KerSMLap} \tmop{Ker} \hat{\Delta}^{\Omega}_g = \tmop{Ker}
  (\Delta^{\Omega}_g -\mathbbm{I}) \subset C^{\infty} (X, T^{\ast}_X) .
\end{equation}
We define now the $\Omega$-Hodge Laplacian acting on scalar valued
differential forms as the operator
\begin{eqnarray*}
  \Delta^{\Omega}_{d, g} & \assign & d \nabla^{\ast_{\Omega}}_g +
  \nabla^{\ast_{\Omega}}_g d .
\end{eqnarray*}
At the level of scalar valued 1-forms we observe the identities
\begin{eqnarray*}
  \left( \Delta^{\Omega}_{d, g} + \hat{\Delta}^{\Omega}_g \right) \alpha & = &
  \nabla^{\ast_{\Omega}}_g  \left( d + \hat{\nabla}_g \right) \alpha 
  \\
  \\
  &=& 2
  \nabla^{\ast_{\Omega}}_g \nabla_g \alpha 
  \\
  \\
  &=& 2 \Delta^{\Omega}_g \alpha .
\end{eqnarray*}
We infer thanks to the identity (\ref{Sm-Wei-1-form}) that for any scalar
valued 1-form $\alpha$ holds the Weitzenb\"ock type formula
\begin{equation}
  \label{Hg-Wei-1-form} \Delta^{\Omega}_g \alpha = \Delta^{\Omega}_{d, g}
  \alpha - \alpha \tmop{Ric}_g^{\ast} (\Omega) .
\end{equation}
Applying the $\nabla^{\ast_{\Omega}}_g$-operator to both sides of this
identity and using the fact that $(\nabla^{\ast_{\Omega}}_g)^2 = 0$ at the
level of scalar valued differential forms we obtain
\begin{eqnarray*}
  \nabla^{\ast_{\Omega}}_g \Delta^{\Omega}_g \alpha & = & \Delta^{\Omega}_g
  \nabla^{\ast_{\Omega}}_g \alpha - \nabla^{\ast_{\Omega}}_g \left[ \alpha
  \tmop{Ric}_g^{\ast} (\Omega) \right] .
\end{eqnarray*}
In the soliton case $h_{g, \Omega} = 0$ this implies the formula
\begin{equation}
  \label{com-div-Lap} \nabla^{\ast_{\Omega}}_g \Delta^{\Omega}_g \alpha =
  \Delta^{\Omega}_g \nabla^{\ast_{\Omega}}_g \alpha - \nabla^{\ast_{\Omega}}_g
  \alpha .
\end{equation}
Then the identity (\ref{KerSMLap}) implies that the map
\begin{eqnarray*}
  \tmop{Ker} \hat{\Delta}^{\Omega}_g & \longrightarrow & \tmop{Ker}
  (\Delta^{\Omega}_g - 2\mathbbm{I}) \subset C_{\Omega}^{\infty} (X,
  \mathbbm{R})_0 \\
  &  & \\
  \alpha & \mapsto & \tmop{div}_g^{\Omega} \alpha,
\end{eqnarray*}
is well defined. More precisely there exists the exact sequence of finite
dimensional vector spaces
\begin{eqnarray*}
  0 \longrightarrow \tmop{Kill}_{g, \Omega} \longrightarrow & \tmop{Ker}
  \hat{\Delta}^{\Omega}_g & \longrightarrow \tmop{Ker} (\Delta^{\Omega}_g -
  2\mathbbm{I}) \longrightarrow 0\\
  &  & \\
  \xi \mapsto & g \xi = \alpha & \mapsto \tmop{div}_g^{\Omega} \alpha .
\end{eqnarray*}
Indeed the surjectivity follows from the isomorphism (\ref{V-iso-Eig}). The
injectivity follows from the fact that
\begin{eqnarray*}
  \tmop{Kill}_{g, \Omega} & \cong & \left\{ \alpha \in \tmop{Ker}
  \hat{\Delta}^{\Omega}_g \mid \tmop{div}_g^{\Omega} \alpha = 0 \right\} .
\end{eqnarray*}
This hold true thanks to the identity
\begin{eqnarray*}
  \int_X \left| \tmop{div}_g^{\Omega} \alpha \right|^2_g \Omega & = &
  \frac{1}{2}  \int_X \left| \hat{\nabla}_g \alpha \right|^2_g \Omega,
\end{eqnarray*}
which follows from the expression
\begin{eqnarray*}
  \hat{\Delta}^{\Omega}_g \alpha & = & \left[ \frac{1}{2} 
  \hat{\nabla}^{\ast_{\Omega}}_g  \hat{\nabla}_g - d \nabla^{\ast_{\Omega}}_g
  \right] \alpha .
\end{eqnarray*}
For dimensional reasons we conclude the existence of the required exact
sequence
\begin{eqnarray*}
  0 \longrightarrow \tmop{Ker} (\Delta^{\Omega}_g - 2\mathbbm{I}) &
  \longrightarrow & \mathbbm{A}^{\Omega}_g \longrightarrow 0\\
  &  & \\
  u & \mapsto & 2 \left( \nabla_g d u, - u \Omega \right) .
\end{eqnarray*}
(We observe also that for dimensional reasons (\ref{V-cont-OrtLIE}) is an
equality.)

\section{Properties of the kernel of the Hessian of $\mathcal{W}$}

\begin{lemma}
  \label{Ker-D2W}In the soliton case $h_{g, \Omega} = 0$ holds the inclusion
  \[ \mathbbm{A}^{\Omega}_g \subseteq \mathbbm{F}_{g, \Omega} \cap \tmop{Ker}
     \nabla^2_G \mathcal{W} (g, \Omega) . \]
\end{lemma}

We start with a few notations. For any tensor $A \in C^{\infty} (X,
(T_X^{\ast})^{\otimes p + 1} \otimes T_X)$ we define the divergence type
operations
\begin{eqnarray*}
  \underline{\tmop{div}}_g A (u_1, \ldots, u_p) & \assign & \tmop{Tr}_g \left[
  \nabla_g A \right( \cdot, u_1, \ldots, u_p, \cdot \left) \right],\\
  &  & \\
  \underline{\tmop{div}}^{\Omega}_g A (u_1, \ldots, u_p) & : = &
  \underline{\tmop{div}}_g A (u_1, \ldots, u_p) - A (u_1, \ldots, u_p,
  \nabla_g f) .
\end{eqnarray*}
The once contracted differential Bianchi identity writes often as
$\underline{\tmop{div}}_g \mathcal{R}_g = - \nabla_{T_X, g}
\tmop{Ric}^{\ast}_g$. This combined with the identity $\nabla_{T_X, g}
\nabla^2_g f =\mathcal{R}_g \cdot \nabla_g f$ implies
\begin{equation}
  \label{Om-cntr-Bianc}  \underline{\tmop{div}}^{\Omega}_g \mathcal{R}_g = -
  \nabla_{T_X, g} \tmop{Ric}^{\ast}_g (\Omega) .
\end{equation}
We define the $\Omega$-Lichnerowicz Laplacian $\Delta^{\Omega}_{L, g}$ acting
on $g$-symmetric endomorphisms $A$ as
\begin{eqnarray*}
  \Delta^{\Omega}_{L, g} A & \assign & \mathcal{L}^{\Omega}_g A +
  \tmop{Ric}_g^{\ast} (\Omega) A + A \tmop{Ric}_g^{\ast} (\Omega) .
\end{eqnarray*}
We fix a point $p\in
X$ and we take an arbitrary vector field $\xi$ such that $\nabla_g \xi (p) = 0$. 
Let also $(e_k)_k$ be a $g$-orthonormal local frame such that $\nabla_g e_k  (p) =0$. We expand the identity at
the point $p$
\begin{eqnarray*}
  (\Delta^{\Omega}_g \nabla^2_g u) \xi & = & - \nabla_{g, e_k} \nabla^3_g u
  (e_k, \xi) + \nabla^3_g u (\nabla_g f, \xi) .
\end{eqnarray*}
Commuting derivatives at the point $p$ we obtain
\begin{eqnarray*}
  \nabla_{g, e_k} \nabla^3_g u (e_k, \xi) & = & \nabla_{g, e_k} \left[
  \nabla^3_g u (e_k, \xi) \right]\\
  &  & \\
  & = & \nabla_{g, e_k} \left[ \nabla_{g, e_k} \nabla_{g, \xi} \nabla_g u -
  \nabla^2_g u \cdot \nabla_{g, e_k} \xi \right]\\
  &  & \\
  & = & \nabla_{g, e_k} \left[ \nabla_{g, \xi} \nabla_{g, e_k} \nabla_g u
  +\mathcal{R}_g (e_k, \xi) \nabla_g u - \nabla^2_g u \cdot \nabla_{g, \xi}
  e_k \right]\\
  &  & \\
  & = & \nabla_{g, \xi} \nabla_{g, e_k} \nabla_{g, e_k} \nabla_g u +
  2\mathcal{R}_g (e_k, \xi) \nabla_{g, e_k} \nabla_g u\\
  &  & \\
  & + & \nabla_{g, e_k} \mathcal{R}_g (e_k, \xi) \nabla_g u - \nabla^2_g u
  \cdot \nabla_{g, e_k} \nabla_{g, \xi} e_k,
\end{eqnarray*}
since $\left[ e_k, \xi \right](p)=0$. Taking a covariant derivative of the identity
\begin{eqnarray*}
  \Delta_g \nabla_g u & = & - \nabla_{g, e_k} \nabla_{g, e_k} \nabla_g u +
  \nabla^2_g u \cdot \nabla_{g, e_k} e_k ,
\end{eqnarray*}
 we infer
\begin{eqnarray*}
  \nabla_{g, \xi} \Delta_g \nabla_g u & = & - \nabla_{g, \xi} \nabla_{g, e_k}
  \nabla_{g, e_k} \nabla_g u + \nabla^2_g u \cdot \nabla_{g, \xi} \nabla_{g,
  e_k} e_k,
\end{eqnarray*}
at the point $p$. Combining with the previous expression we obtain
\begin{eqnarray*}
  \nabla_{g, e_k} \nabla^3_g u (e_k, \xi) & = & - 2 \left( \mathcal{R}_g \ast
  \nabla^2_g u \right) \xi - \nabla_{g, \xi} \Delta_g \nabla_g u\\
  &  & \\
  & + & \nabla^2_g u\cdot \tmop{Ric}^{\ast} (g) \xi +
  \nabla_{g, e_k} \mathcal{R}_g (e_k, \xi) \nabla_g u .
\end{eqnarray*}
On the other hand deriving the identity
\begin{eqnarray*}
  \Delta^{\Omega}_g \nabla_g u & = & \Delta_g \nabla_g u + \nabla^2_g u \cdot
  \nabla_g f,
\end{eqnarray*}
we infer
\begin{eqnarray*}
  \nabla_{g, \xi} \Delta^{\Omega}_g \nabla_g u & = & \nabla_{g, \xi} \Delta_g
  \nabla_g u + \nabla_{g, \xi} \nabla^2_g u \cdot \nabla_g f + \nabla^2_g u
  \cdot \nabla^2_{g, \xi} f,
\end{eqnarray*}
and thus
\begin{eqnarray*}
  \nabla_{g, e_k} \nabla^3_g u (e_k, \xi) & = & - 2 \left( \mathcal{R}_g \ast
  \nabla^2_g u \right) \xi - \nabla_{g, \xi} \Delta^{\Omega}_g \nabla_g u +
  \nabla_{g, \xi} \nabla^2_g u \cdot \nabla_g f\\
  &  & \\
  & + & \nabla^2_g u \cdot \tmop{Ric}_g^{\ast} (\Omega) \xi -
  \underline{\tmop{div}}_g \mathcal{R}_g  (\xi, \nabla_g u) - \left( \nabla_g
  u \neg \nabla^{\ast}_g \mathcal{R}_g \right) \xi,
\end{eqnarray*}
thanks to the algebraic Bianchi identity. We obtain
\begin{eqnarray*}
  (\Delta^{\Omega}_g \nabla^2_g u) \xi & = & 2 \left( \mathcal{R}_g \ast
  \nabla^2_g u \right) \xi - \nabla_{T_X, g} \nabla^2_g u (\xi, \nabla_g f)\\
  &  & \\
  & + & \underline{\tmop{div}}_g \mathcal{R}_g  (\xi, \nabla_g u) + \left(
  \nabla_g u \neg \nabla^{\ast}_g \mathcal{R}_g \right) \xi\\
  &  & \\
  & + & \nabla_{g, \xi} \Delta^{\Omega}_g \nabla_g u - \nabla^2_g u \cdot
  \tmop{Ric}_g^{\ast} (\Omega) \xi .
\end{eqnarray*}
The identity $\nabla_{T_X, g} \nabla^2_g u =\mathcal{R}_g \cdot \nabla_g u$
implies
\begin{eqnarray*}
  - \nabla_{T_X, g} \nabla^2_g u (\xi, \nabla_g f) & = & \mathcal{R}_g
  (\nabla_g f, \xi) \nabla_g u\\
  &  & \\
  & = & -\mathcal{R}_g (\xi, \nabla_g u) \nabla_g f +\mathcal{R}_g (\nabla_g
  f, \nabla_g u) \xi,
\end{eqnarray*}
thanks again to the algebraic Bianchi identity. We infer
\begin{eqnarray*}
  (\mathcal{L}^{\Omega}_g \nabla^2_g u) \xi & = & \left[ \nabla_g u \neg
  \left( \nabla^{\ast_{\Omega}}_g \mathcal{R}_g -
  \underline{\tmop{div}}^{\Omega}_g \mathcal{R}_g \right) \right] \xi\\
  &  & \\
  & + & \nabla_{g, \xi} \Delta^{\Omega}_g \nabla_g u - \nabla^2_g u \cdot
  \tmop{Ric}_g^{\ast} (\Omega) \xi\\
  &  & \\
  & = & \left[ \nabla_g u \neg \left( \nabla^{\ast_{\Omega}}_g \mathcal{R}_g
  + \nabla_{T_X, g} \tmop{Ric}^{\ast}_g (\Omega) \right) \right] \xi\\
  &  & \\
  & + & (\nabla^2_g \Delta^{\Omega}_g u) \xi - \nabla_{g, \xi} [
  \tmop{Ric}^{\ast}_g (\Omega) \nabla_g u] - \nabla^2_g u \cdot
  \tmop{Ric}_g^{\ast} (\Omega) \xi,
\end{eqnarray*}
thanks to (\ref{Om-cntr-Bianc}) and (\ref{com-grad-Lap}). Thus
\begin{equation}
  \label{Lich-Hess} \Delta^{\Omega}_{L, g} \nabla^2_g u = \nabla^2_g
  \Delta^{\Omega}_g u + \nabla_g u \neg \left[ \nabla^{\ast_{\Omega}}_g
  \mathcal{R}_g + \nabla_g \tmop{Ric}^{\ast}_g (\Omega) \right] - 2 \nabla_g
  \tmop{Ric}^{\ast}_g (\Omega) \nabla_g u .
\end{equation}
We observe now that the endomorphism section $\nabla_g u \neg
\nabla^{\ast_{\Omega}}_g \mathcal{R}_g$ is $g$-anti-symmetric thanks to the
identity \
\begin{eqnarray*}
  \mathcal{R}_g (\xi, \eta) & = & - \left( \mathcal{R}_g (\xi, \eta)
  \right)_g^T,
\end{eqnarray*}
which is a consequence of the alternating property of the $(4, 0)$-Riemann
curvature operator. Notice indeed that the previous identity implies
\begin{eqnarray*}
  \nabla_{g, \mu} \mathcal{R}_g (\xi, \eta) & = & - \left( \nabla_{g, \mu}
  \mathcal{R}_g (\xi, \eta) \right)_g^T,
\end{eqnarray*}
for all vector fields $\xi, \eta, \mu$. Combining the $g$-symmetric and
$g$-anti-symmetric parts in the identity (\ref{Lich-Hess}) we infer the
formulas
\begin{eqnarray*}
  \Delta^{\Omega}_{L, g} \nabla^2_g u & = & \left. \nabla^2_g
  \Delta^{\Omega}_g u + \nabla_g u \neg \nabla_{T_X, g} \tmop{Ric}^{\ast}_g
  (\Omega) - [ \nabla_g \tmop{Ric}^{\ast}_g (\Omega) \nabla_g u \right]_g^T,\\
  &  & \\
  \xi \neg \nabla^{\ast_{\Omega}}_g \mathcal{R}_g & = & \left. \nabla_g
  \tmop{Ric}^{\ast}_g (\Omega) \xi - [ \nabla_g \tmop{Ric}^{\ast}_g (\Omega)
  \xi \right]_g^T,
\end{eqnarray*}
for all $\xi \in T_X$ since the function $u$ is arbitrary. In the case
$\nabla_g \tmop{Ric}^{\ast}_g (\Omega) = 0$ we deduce the identities
$\Delta^{\Omega}_{L, g} \nabla^2_g u = \nabla^2_g \Delta^{\Omega}_g u$ and
$\nabla^{\ast_{\Omega}}_g \mathcal{R}_g = 0$. More in particular in the
soliton case $h_{g, \Omega} = 0$ the first formula reduces to the differential
identity
\begin{equation}
  \label{com-Lich-Hess} \mathcal{L}^{\Omega}_g \nabla^2_g u = \nabla^2_g 
  (\Delta^{\Omega}_g - 2\mathbbm{I}) u .
\end{equation}
We infer the conclusion of lemma \ref{Ker-D2W}. This formula will be also
quite crucial for the study of the sign of the second variation of the
$\mathcal{W}$ functional at a K\"ahler-Ricci soliton point.

\section{Invariance of $\mathbbm{F}$ under the action of the \\
Hessian endomorphism of $\mathcal{W}$}\label{inv-F}

We observe that Perelman's twice contracted second Bianchi type identity
(\ref{II-contr-Bianchi}) rewrites as;
\[ \nabla_g^{\ast_{\Omega}} h_{g, \Omega} \noplus + d H_{g, \Omega} = 0 . \]
If we differentiate this over the space $\mathcal{M} \times \mathcal{V}_1$ we
obtain
\begin{eqnarray*}
  \left[ (D_{g, \Omega} \nabla_{\bullet}^{\ast_{\bullet}}) \left( v, V \right)
  \right] h_{g, \Omega} + \nabla_g^{\ast_{\Omega}} \left[ D_{g, \Omega} h
  \left( v, V \right) \right] \noplus + d \left[ D_{g, \Omega} H \left( v, V
  \right) \right] = 0 . &  & 
\end{eqnarray*}
We deduce using the fundamental variation formulas (\ref{var-h}) and
(\ref{var-H})
\begin{eqnarray*}
  &  & \nabla_g^{\ast_{\Omega}} \left[ \mathcal{L}^{\Omega}_g v + v
  h^{\ast}_{g, \Omega} + h_{g, \Omega} v_g^{\ast} \right]\\
  &  & \\
  & + & d \left[ (\Delta^{\Omega}_g - 2\mathbbm{I}) V^{\ast}_{\Omega} -
  \left\langle v, h_{g, \Omega} \right\rangle_g \right]\\
  &  & \\
  & = & - 2 \left[ D_{g, \Omega} \nabla_{\bullet}^{\ast_{\bullet}}  \left( v,
  V \right) \right] h_{g, \Omega} 
\end{eqnarray*}
in the directions $(v, V) \in \mathbbm{F}_{g, \Omega}$. We infer that in the
soliton case $h_{g, \Omega} = 0$ the map
\[ \nabla^2_G \mathcal{W} (g, \Omega) : \mathbbm{F}_{g, \Omega}
   \longrightarrow \mathbbm{F}_{g, \Omega}, \]
is well defined. In order to investigate the general case we use a different
method which has the advantage to involve less computations. Let $(e_k)_k$ be
a $g$-orthonormal local frame of $T_X$. For any $u, v \in C^{\infty} \left( X,
S^2 T^{\ast}_X \right)$ we define the real valued $1$-form
\begin{eqnarray*}
  M_g (u, v) (\xi) & \assign & 2 \nabla_g v (e_k, u^{\ast}_g e_k, \xi) +
  \nabla_g u (\xi, v^{\ast}_g e_k, e_k) \hspace{1.2em},
\end{eqnarray*}
for all $\xi \in T_X$. One can show that the operator
\begin{eqnarray*}
  T_g (u, v) & : = & M_g (u, v) - M_g (v, u),
\end{eqnarray*}
is related with the torsion of the distribution $\mathbbm{F}$. We observe now
that by lemma \ref{OmTX-Lap-RmLap} holds the identity
\begin{eqnarray*}
  \Delta^{\Omega}_g v^{\ast}_g - \nabla_g \nabla^{\ast_{\Omega}}_g v^{\ast}_g
  & = & \nabla^{\ast_{\Omega}}_g \nabla_{T_X, g} v^{\ast}_g +\mathcal{R}_g
  \ast v^{\ast}_g - v^{\ast}_g \tmop{Ric}_g^{\ast} (\Omega) .
\end{eqnarray*}
Applying the $\nabla^{\ast_{\Omega}}_g$-operator to both sides of this
identity we deduce the commutation formula
\begin{eqnarray*}
  \left[ \nabla^{\ast_{\Omega}}_g, \Delta^{\Omega}_g \right] v^{\ast}_g & = &
  \nabla^{\ast_{\Omega}}_g \left[ \nabla^{\ast_{\Omega}}_g \nabla_{T_X, g}
  v^{\ast}_g +\mathcal{R}_g \ast v^{\ast}_g - v^{\ast}_g h_{g, \Omega}^{\ast}
  - v^{\ast}_g \right] .
\end{eqnarray*}
We observe now that for any $\psi \in C^{\infty} (X, \Lambda^2 T_X
\otimes_{\mathbbm{R}} T_X)$ and $\xi \in C^{\infty} (X, T_X)$ hold the
equalities
\begin{eqnarray*}
  \int_X \left\langle (\nabla^{\ast_{\Omega}}_g)^2 \psi, \xi \right\rangle_g
  \Omega & = & \int_X \left\langle \nabla^{\ast_{\Omega}}_g \psi, \nabla_g \xi
  \right\rangle_g \Omega\\
  &  & \\
  & = & \frac{1}{2} \int_X \left\langle \nabla^{\ast_{\Omega}}_{T_X, g} \psi,
  \nabla_g \xi \right\rangle_g \Omega\\
  &  & \\
  & = & \frac{1}{2} \int_X \left\langle \psi, \nabla^2_{T_X, g} \xi
  \right\rangle_g \Omega\\
  &  & \\
  & = & \frac{1}{2} \int_X \left\langle \psi, \mathcal{R}_g \cdot \xi
  \right\rangle_g \Omega,
\end{eqnarray*}
and
\begin{eqnarray*}
  \left\langle \psi, \mathcal{R}_g \cdot \xi \right\rangle_g & = &
  \left\langle \psi (e_k, e_l), \mathcal{R}_g (e_k, e_l) \xi \right\rangle_g =
  - \left\langle \mathcal{R}_g (e_k, e_l) \psi (e_k, e_l), \xi \right\rangle_g
  .
\end{eqnarray*}
We infer
\begin{eqnarray*}
  (\nabla^{\ast_{\Omega}}_g)^2 \nabla_{T_X, g} v^{\ast}_g & = & - \frac{1}{2}
  \mathcal{R}_g (e_k, e_l)  \left[ \nabla_g v^{\ast}_g (e_k, e_l) - \nabla_g
  v^{\ast}_g (e_l, e_k) \right]\\
  &  & \\
  & = & \mathcal{R}_g (e_l, e_k) \nabla_g v^{\ast}_g  (e_k, e_l) .
\end{eqnarray*}
This combined with the expression
\begin{eqnarray*}
  \nabla^{\ast_{\Omega}}_g (\mathcal{R}_g \ast v^{\ast}_g) & = &
  \nabla^{\ast_{\Omega}}_g \mathcal{R}_g (e_k) v^{\ast}_g e_k +\mathcal{R}_g
  (e_l, e_k) \nabla_g v^{\ast}_g (e_k, e_l),
\end{eqnarray*}
implies the identity
\begin{eqnarray*}
  \nabla^{\ast_{\Omega}}_g \mathcal{L}^{\Omega}_g v^{\ast}_g  & = &
  (\Delta^{\Omega}_g -\mathbbm{I}) \nabla^{\ast_{\Omega}}_g v^{\ast}_g +
  \nabla_g v^{\ast}_g  (e_k, h_{g, \Omega}^{\ast} e_k) \\
  &  & \\
  & - & v^{\ast}_g \nabla^{\ast_{\Omega}}_g h_{g, \Omega}^{\ast} -
  \nabla^{\ast_{\Omega}}_g \mathcal{R}_g (e_k) v^{\ast}_g e_k,
\end{eqnarray*}
which rewrites also under the form
\begin{eqnarray*}
  \nabla^{\ast_{\Omega}}_g \mathcal{L}^{\Omega}_g v & = & (\Delta^{\Omega}_g
  -\mathbbm{I}) \nabla^{\ast_{\Omega}}_g v + \nabla_g v (e_k, h_{g,
  \Omega}^{\ast} e_k, \bullet) \\
  &  & \\
  & - & v \nabla^{\ast_{\Omega}}_g h_{g, \Omega}^{\ast} + v \left( e_k,
  \nabla^{\ast_{\Omega}}_g \mathcal{R}_g (e_k) \bullet \right),
\end{eqnarray*}
thanks to (\ref{ast-curv-Id}) and the anti-symmetry property
\begin{eqnarray*}
  e_k \neg \nabla^{\ast_{\Omega}}_g \mathcal{R}_g & = & - \left( e_k \neg
  \nabla^{\ast_{\Omega}}_g \mathcal{R}_g \right)_g^T .
\end{eqnarray*}
On the other hand the once contracted differential Bianchi type identity
(\ref{Om-cntr-Bianc}) rewrites as
\begin{eqnarray*}
  - \nabla_{T_X, g} \tmop{Ric}^{\ast}_g (\Omega) & = & \tmop{Alt} \left(
  \nabla^{\ast_{\Omega}}_g \mathcal{R}_g \right),
\end{eqnarray*}
thanks to the algebraic Bianchi identity. Therefore for any $\xi \in
C^{\infty} (X, T_X)$ hold the identities
\begin{eqnarray*}
  v \left( e_k, \nabla^{\ast_{\Omega}}_g \mathcal{R}_g (e_k) \xi \right) & = &
  v \left( \nabla^{\ast_{\Omega}}_g \mathcal{R}_g (e_k) \xi, e_k \right)\\
  &  & \\
  & = & v \left( \nabla^{\ast_{\Omega}}_g \mathcal{R}_g (\xi) e_k, e_k
  \right) + v \left( \left[ \xi \neg \nabla_{T_X, g} \tmop{Ric}^{\ast}_g
  (\Omega) \right] e_k, e_k \right)\\
  &  & \\
  & = & \tmop{Tr}_{\mathbbm{R}} \left[ v^{\ast}_g \nabla^{\ast_{\Omega}}_g
  \mathcal{R}_g (\xi) \right] + \tmop{Tr}_g \left[ v \left( \xi \neg
  \nabla_{T_X, g} \tmop{Ric}^{\ast}_g (\Omega) \right) \right]\\
  &  & \\
  & = & \tmop{Tr}_g \left[ v \left( \xi \neg \nabla_{T_X, g} h_{g,
  \Omega}^{\ast} \right) \right],
\end{eqnarray*}
since the endomorphism section $\nabla^{\ast_{\Omega}}_g \mathcal{R}_g (\xi)$
is $g$-anti-symmetric. Notice indeed that if $A, B \in C^{\infty} \left( X,
\tmop{End} (T_X) \right)$ satisfy $A = A_g^T$ and $B = - B_g^T$ then
\begin{eqnarray*}
  \tmop{Tr}_{\mathbbm{R}} (A B) & = & \tmop{Tr}_{\mathbbm{R}} (B A)\\
  &  & \\
  & = & \tmop{Tr}_{\mathbbm{R}} (B A)_g^T\\
  &  & \\
  & = & \tmop{Tr}_{\mathbbm{R}} (A_g^T B_g^T)\\
  &  & \\
  & = & - \tmop{Tr}_{\mathbbm{R}} (A B),
\end{eqnarray*}
i.e $\tmop{Tr}_{\mathbbm{R}} (A B) = 0$. We deduce in conclusion the formula
\begin{eqnarray*}
  \nabla^{\ast_{\Omega}}_g \mathcal{L}^{\Omega}_g v & = & (\Delta^{\Omega}_g
  -\mathbbm{I}) \nabla^{\ast_{\Omega}}_g v - v \nabla^{\ast_{\Omega}}_g h_{g,
  \Omega}^{\ast}\\
  &  & \\
  & + & \nabla_g v (e_k, h_{g, \Omega}^{\ast} e_k, \bullet) + \tmop{Tr}_g
  \left[ v \left( \bullet \neg \nabla_{T_X, g} h_{g, \Omega}^{\ast} \right)
  \right] .
\end{eqnarray*}
Using the general formula
\[ \nabla_g^{\ast_{\Omega}}  \left( \varphi v \right) = - v \nabla_g \varphi +
   \varphi \nabla_g^{\ast_{\Omega}} v, \]
with $\varphi \in C^{\infty} (X, \mathbbm{R})$ we infer
\begin{eqnarray*}
  \nabla_g^{\ast_{\Omega}} \left[  \underline{H}_{g, \Omega} v +
  V^{\ast}_{\Omega} h_{g, \Omega} \right] & = & - v \nabla_g 
  \underline{H}_{g, \Omega} + \underline{H}_{g, \Omega}
  \nabla_g^{\ast_{\Omega}} v - h_{g, \Omega} \nabla_g V^{\ast}_{\Omega} +
  V^{\ast}_{\Omega} \nabla_g^{\ast_{\Omega}} h_{g, \Omega}\\
  &  & \\
  & = & v \nabla_g^{\ast_{\Omega}} h^{\ast}_{g, \Omega} + \underline{H}_{g,
  \Omega} \nabla_g^{\ast_{\Omega}} v - d V^{\ast}_{\Omega} \cdot h^{\ast}_{g,
  \Omega} - V^{\ast}_{\Omega} d \underline{H}_{g, \Omega},
\end{eqnarray*}
thanks to Perelman's twice contracted differential Bianchi type identity
(\ref{II-contr-Bianchi}). Using the identity (\ref{com-d-Lap}) we expand the
term
\begin{eqnarray*}
  &  & d \left[ \left( \Delta^{\Omega}_g + \underline{H}_{g, \Omega} -
  2\mathbbm{I} \right) V^{\ast}_{\Omega} - \frac{1}{2}  \left\langle v, h_{g,
  \Omega} \right\rangle_g \right]\\
  &  & \\
  & = & \Delta^{\Omega}_g d V^{\ast}_{\Omega} + d V^{\ast}_{\Omega} \cdot
  \tmop{Ric}^{\ast}_g (\Omega) + V^{\ast}_{\Omega} d \underline{H}_{g, \Omega}
  + \underline{H}_{g, \Omega} d V^{\ast}_{\Omega} - 2 d V^{\ast}_{\Omega} -
  \frac{1}{2} d \left\langle v, h_{g, \Omega} \right\rangle_g\\
  &  & \\
  & = & (\Delta^{\Omega}_g + \underline{H}_{g, \Omega} -\mathbbm{I}) d
  V^{\ast}_{\Omega} + d V^{\ast}_{\Omega} \cdot h^{\ast}_{g, \Omega} +
  V^{\ast}_{\Omega} d \underline{H}_{g, \Omega} - \frac{1}{2} d \left\langle
  v, h_{g, \Omega} \right\rangle_g .
\end{eqnarray*}
Summing up we infer
\begin{eqnarray*}
  &  & \nabla_g^{\ast_{\Omega}} \left[ \left( \mathcal{L}^{\Omega}_g +
  \underline{H}_{g, \Omega} \right) v + V^{\ast}_{\Omega} h_{g, \Omega}
  \right]\\
  &  & \\
  & + & d \left[ \left( \Delta^{\Omega}_g + \underline{H}_{g, \Omega} -
  2\mathbbm{I} \right) V^{\ast}_{\Omega} - \frac{1}{2}  \left\langle v, h_{g,
  \Omega} \right\rangle_g \right]\\
  &  & \\
  & = & (\Delta^{\Omega}_g + \underline{H}_{g, \Omega} -\mathbbm{I})  \left(
  \nabla_g^{\ast_{\Omega}} v + d V^{\ast}_{\Omega} \right) \\
  &  & \\
  & + & \nabla_g v (e_k, h_{g, \Omega}^{\ast} e_k, \bullet) + \tmop{Tr}_g
  \left[ v \left( \bullet \neg \nabla_{T_X, g} h_{g, \Omega}^{\ast} \right)
  \right] - \frac{1}{2} d \left\langle v, h_{g, \Omega} \right\rangle_g .
\end{eqnarray*}
We observe now the identity
\begin{eqnarray*}
  \nabla_g v (e_k, h_{g, \Omega}^{\ast} e_k, \bullet) + \tmop{Tr}_g \left[ v
  \left( \bullet \neg \nabla_{T_X, g} h_{g, \Omega}^{\ast} \right) \right] -
  \frac{1}{2} d \left\langle v, h_{g, \Omega} \right\rangle_g & = &
  \frac{1}{2} T_g (h_{g, \Omega}, v) .
\end{eqnarray*}
We deduce the formula
\begin{eqnarray*}
  &  & \nabla_g^{\ast_{\Omega}} \left[ \left( \mathcal{L}^{\Omega}_g +
  \underline{H}_{g, \Omega} \right) v + V^{\ast}_{\Omega} h_{g, \Omega}
  \right]\\
  &  & \\
  & + & d \left[ \left( \Delta^{\Omega}_g + \underline{H}_{g, \Omega} -
  2\mathbbm{I} \right) V^{\ast}_{\Omega} - \frac{1}{2}  \left\langle v, h_{g,
  \Omega} \right\rangle_g \right]\\
  &  & \\
  & = & (\Delta^{\Omega}_g + \underline{H}_{g, \Omega} -\mathbbm{I})  \left(
  \nabla_g^{\ast_{\Omega}} v + d V^{\ast}_{\Omega} \right) + \frac{1}{2} T_g
  (h_{g, \Omega}, v) .
\end{eqnarray*}
Setting $(v, V) = (h_{g, \Omega}, \underline{H}_{g, \Omega} \Omega) \in
\mathbbm{F}_{g, \Omega}$ in the previous identity we infer
\begin{eqnarray*}
  0 & = & \nabla_g^{\ast_{\Omega}} \left[ \left( \mathcal{L}^{\Omega}_g + 2
  \underline{H}_{g, \Omega} \right) h_{g, \Omega} \right]\\
  &  & \\
  & + & d \left[ \left( \Delta^{\Omega}_g + \underline{H}_{g, \Omega} -
  2\mathbbm{I} \right)  \underline{H}_{g, \Omega} - \frac{1}{2}  \left| h_{g,
  \Omega} \right|^2_g \right] .
\end{eqnarray*}
This shows the fundamental property (\ref{fundam-property-SRF}) of the
Soliton-Ricci flow.

\section{The K\"ahler set up}

In this section we introduce a few basic notations needed in sequel. Let $(X,
J, g)$ be a compact connected K\"ahler manifold with symplectic form $\omega
\assign g J$. Let $h \assign g - i g J = 2 g \pi^{1, 0}_J$ be the hermitian
metric over $T_{X, J}$ induced by $g$. We remind that in the K\"ahler case the
Chern connection
\begin{eqnarray*}
  D^g_{T_{X, J}} = & \partial^g_{T_{X, J}} + \overline{\partial}_{T_{X, J}} :
  & C^{\infty} (T_{X, J}) \longrightarrow C^{\infty} (T^{\ast}_X
  \otimes_{\mathbbm{R}} T_{X, J}),
\end{eqnarray*}
of the hermitian vector bundle $(T_{X, J}, h)$ coincides with the Levi-Civita
connection $\nabla_g$. We set $\mathbbm{C}T_X \assign T_X
\otimes_{\mathbbm{R}} \mathbbm{C}$ and $\mathbbm{C}T^{\ast}_X \assign
T^{\ast}_X \otimes_{\mathbbm{R}} \mathbbm{C}$. We observe further that the
sesquiliner extension of $g$
\[ g_{\mathbbm{C}} \in C^{\infty} (X, \mathbbm{C}T^{\ast}_X
   \otimes_{\mathbbm{C}} \overline{\mathbbm{C}} T^{\ast}_X), g_{\mathbbm{C}}
   (\xi, \eta) \assign g (\xi, \bar{\eta}), \forall \xi, \eta \in
   \mathbbm{C}T_X, \]
is a hermitian metric over $\mathbbm{C}T_X$ and the $\mathbbm{C}$-linear
extension of the Levi-Civita connection $\nabla_g$
\begin{eqnarray*}
  \nabla^{\mathbbm{C}}_g : C^{\infty} (\mathbbm{C}T_X) & \longrightarrow &
  C^{\infty} (\mathbbm{C}T^{\ast}_X \otimes_{\mathbbm{C}} \mathbbm{C}T_X),
\end{eqnarray*}
is a $g_{\mathbbm{C}}$-hermitian connection over the vector bundle $T_X
\otimes_{\mathbbm{R}} \mathbbm{C}$. We will focus our interest on the sections
of the hermitian vector bundle
\[ \left( (\mathbbm{C}T^{\ast}_X)^{\otimes p} \otimes_{\mathbbm{C}} T_{X, J},
   g_{\mathbbm{C}} \otimes h \right), \]
and we will denote by abuse of notations $\nabla_g \equiv
\nabla^{\mathbbm{C}}_g \otimes D^g_{T_{X, J}}$ the $g_{\mathbbm{C}} \otimes
h$-hermitian connection over this vector bundle. Still by abuse of notations
we will use the identification $\left\langle \cdot, \cdot
\right\rangle_{\omega} \assign g_{\mathbbm{C}} \otimes h$. With these notations
we define the operators
\begin{eqnarray*}
  \nabla_{g, J}^{1, 0}  :  C^{\infty} \left(
  (\mathbbm{C}T^{\ast}_X)^{\otimes p} \otimes_{\mathbbm{C}} T_{X, J} \right)
  &\longrightarrow& C^{\infty} \left( \Lambda^{1, 0}_J T^{\ast}_X
  \otimes_{\mathbbm{C}} (\mathbbm{C}T^{\ast}_X)^{\otimes p}
  \otimes_{\mathbbm{C}} T_{X, J} \right),\\
  &  & \\
  \nabla_{g, J}^{0, 1}  :  C^{\infty} \left(
  (\mathbbm{C}T^{\ast}_X)^{\otimes p} \otimes_{\mathbbm{C}} T_{X, J} \right)
  &\longrightarrow& C^{\infty} \left( \Lambda^{0, 1}_J T^{\ast}_X
  \otimes_{\mathbbm{C}} (\mathbbm{C}T^{\ast}_X)^{\otimes p}
  \otimes_{\mathbbm{C}} T_{X, J} \right),
\end{eqnarray*}
by the formulas
\begin{eqnarray*}
  2 \nabla_{g, J}^{1, 0} & \assign & \nabla_g - J \nabla_{g, J \bullet},\\
  &  & \\
  2 \nabla_{g, J}^{0, 1} & \assign & \nabla_g + J \nabla_{g, J \bullet_{}} .
\end{eqnarray*}
Then the formal adjoints of the operators $\partial^g_{T_{X, J}}$ and
$\overline{\partial}_{T_{X, J}}$ acting on $T_{X, J}$-valued differential
forms satisfy the identities (see \cite{Pal6})
\begin{eqnarray*}
  \partial^{\ast_g}_{T_{X, J}} \alpha & = & - q \tmop{Tr}_g \nabla_{g, J}^{0,
  1} \alpha,\\
  &  & \\
  \overline{\partial}^{\ast_g}_{T_{X, J}} \alpha & = & - q \tmop{Tr}_g
  \nabla_{g, J}^{1, 0} \alpha,
\end{eqnarray*}
for any $\alpha \in C^{\infty} (X, \Lambda_{}^q T^{\ast}_X
\otimes_{\mathbbm{C}} T_{X, J})$. We remind now that with our conventions (see \cite{Pal3}) 
the Hodge Laplacian operator acting on $T_X$-valued $q$-forms
satisfies the identity
\begin{eqnarray*}
  \Delta_{T_{X, g}} & = &  \frac{1}{q} \nabla_{T_{X, g}} \nabla^{\ast}_{T_{X,
  g}} + \frac{1}{q + 1} \nabla^{\ast}_{T_{X, g}} \nabla_{T_{X, g}} .
\end{eqnarray*}
We define also the holomorphic and antiholomorphic Hodge Laplacian operators
acting on $T_X$-valued $q$-forms as
\begin{eqnarray*}
  \Delta^J_{T_{X, g}} & \assign & \frac{1}{q} \partial^g_{T_{X, J}}
  \partial^{\ast_g}_{T_{X, J}} + \frac{1}{q + 1} \partial^{\ast_g}_{T_{X, J}}
  \partial^g_{T_{X, J}},\\
  &  & \\
  \Delta^{- J}_{T_{X, g}} & \assign & \frac{1}{q}  \overline{\partial}_{T_{X,
  J}} \overline{\partial}^{\ast_g}_{T_{X, J}} + \frac{1}{q + 1} 
  \overline{\partial}^{\ast_g}_{T_{X, J}} \overline{\partial}_{T_{X, J}},
\end{eqnarray*}
with the usual convention $\infty \cdot 0 = 0$. This Hodge Laplacian operators
coincide with the standard ones used in the literature. We remind that in the
K\"ahler case holds the decomposition identity
\begin{eqnarray*}
  \Delta_{T_{X, g}} & = & \Delta^J_{T_{X, g}} +\; \Delta^{- J}_{T_{X, g}} .
\end{eqnarray*}
We observe now that the formal adjoint of the $\partial^g_{T_{X, J}}$ operator
with respect to the hermitian product
\begin{equation}
  \label{L2omOm-prod}  \left\langle \cdot, \cdot \right\rangle_{\omega,
  \Omega} \;\assign \; \int_X \left\langle \cdot, \cdot \right\rangle_{\omega}
  \Omega,
\end{equation}
is the operator
\begin{eqnarray*}
  \partial^{\ast_{g, \Omega}}_{T_{X, J}} & : = & e^f \partial^{\ast_g}_{T_{X,
  J}} \left( e^{- f} \bullet \right) .
\end{eqnarray*}
In a similar way the formal adjoint of the $\overline{\partial}_{T_{X, J}}$
operator with respect to the hermitian product (\ref{L2omOm-prod}) is the
operator
\begin{eqnarray*}
  \overline{\partial}^{\ast_{g, \Omega}}_{T_{X, J}} & : = & e^f
  \overline{\partial}^{\ast_g}_{T_{X, J}} \left( e^{- f} \bullet \right) .
\end{eqnarray*}
With these notations we define the holomorphic and anti-holomorphic
$\Omega$-Hodge Laplacian operators acting on $T_X$-valued $q$-forms as
\begin{eqnarray*}
  \Delta^{\Omega, J}_{T_{X, g}} & \assign & \frac{1}{q} \partial^g_{T_{X, J}}
  \partial^{\ast_{g, \Omega}}_{T_{X, J}} + \frac{1}{q + 1} \partial^{\ast_{g,
  \Omega}}_{T_{X, J}} \partial^g_{T_{X, J}},\\
  &  & \\
  \Delta^{\Omega, - J}_{T_{X, g}} & \assign & \frac{1}{q} 
  \overline{\partial}_{T_{X, J}} \overline{\partial}^{\ast_{g, \Omega}}_{T_{X,
  J}} + \frac{1}{q + 1}  \overline{\partial}^{\ast_{g, \Omega}}_{T_{X, J}}
  \overline{\partial}_{T_{X, J}} .
\end{eqnarray*}

\section{The decomposition of the operator $\mathcal{L}^{\Omega}_g$ in the
K\"ahler case}

For any $A \in \tmop{End} (T_X)$ we denote by $A'_J$ and by $A_J''$ the
$J$-linear, respectively the $J$-anti-linear parts of $A$. We observe that the
operator
\[ \mathcal{L}^{\Omega}_g : C^{\infty} (X, \tmop{End} (T_X)) \longrightarrow
   C^{\infty} (X, \tmop{End} (T_X)), \]
defined by the formula
\begin{eqnarray*}
  \mathcal{L}^{\Omega}_g A & \assign & \Delta^{\Omega}_g A - 2\mathcal{R}_g
  \ast A,
\end{eqnarray*}
restricts as;
\begin{equation}
  \label{lin-Lich} \mathcal{L}^{\Omega}_g : C^{\infty} \left( X, T^{\ast}_{X,
  J} \otimes T_{X, J} \right) \longrightarrow C^{\infty} (X, T^{\ast}_{X, J}
  \otimes T_{X, J}),
\end{equation}

\begin{equation}
  \label{ant-lin-Lich} \mathcal{L}^{\Omega}_g : C^{\infty} \left( X,
  T^{\ast}_{X, - J} \otimes T_{X, J} \right) \longrightarrow C^{\infty} (X,
  T^{\ast}_{X, - J} \otimes T_{X, J}),
\end{equation}
Indeed these properties follow from the identities
\begin{equation}
  \label{J-curv}  (\mathcal{R}_g \ast A)'_J =\mathcal{R}_g \ast A'_J,
\end{equation}
\begin{equation}
  \label{Ant-J-curv}  (\mathcal{R}_g \ast A)_J'' =\mathcal{R}_g \ast A_J'',
\end{equation}
for any $A \in \tmop{End} (T_X)$. In their turn they are direct consequence of
the identities
\begin{equation}
  \label{Jcurv} J (\mathcal{R}_g \ast A) =\mathcal{R}_g \ast (J A),
\end{equation}
\begin{equation}
  \label{curvJ}  (\mathcal{R}_g \ast A) J =\mathcal{R}_g \ast (A J),
\end{equation}

In order to see (\ref{Jcurv}) and (\ref{curvJ}) let $(e_k)_k$ be a
$g$-orthonormal real basis. Using the $J$-invariant properties of the
curvature operator we infer
\begin{eqnarray*}
  J (\mathcal{R}_g \ast A) \xi & = & J\mathcal{R}_g (\xi, e_k) A e_k
  =\mathcal{R}_g (\xi, e_k) J A e_k = \left[ \mathcal{R}_g \ast (J A) \right]
  \xi,\\
  &  & \\
  (\mathcal{R}_g \ast A) J \xi & = & \mathcal{R}_g (J \xi, e_k) A e_k =
  -\mathcal{R}_g (\xi, J e_k) A e_k =\mathcal{R}_g (\xi, \eta_k) A J \eta_k,
\end{eqnarray*}
where $\eta_k \assign J e_k$. The fact that $(\eta_k)_k$ is also a
$g$-orthonormal real frame implies (\ref{curvJ}). By (\ref{lin-Lich}) and
(\ref{ant-lin-Lich}) we conclude the decomposition formula
\begin{equation}
  \label{lin-ant-lin-Integ-dec}  \int_X \left\langle \mathcal{L}^{\Omega}_g A,
  A \right\rangle_g \Omega = \int_X \left\langle \mathcal{L}^{\Omega}_g A_J',
  A_J' \right\rangle_g \Omega + \int_X \left\langle \mathcal{L}^{\Omega}_g
  A_J'', A_J'' \right\rangle_g \Omega .
\end{equation}
We observe that the properties (\ref{lin-Lich}) and (\ref{ant-lin-Lich}) imply
also that $A \in \tmop{Ker} \mathcal{L}^{\Omega}_g$ if and only if $A_J' \in
\tmop{Ker} \mathcal{L}^{\Omega}_g$ and $A_J'' \in \tmop{Ker}
\mathcal{L}^{\Omega}_g$. We observe further that the identity
(\ref{com-Lich-Hess}) combined with the properties (\ref{lin-Lich}) and
(\ref{ant-lin-Lich}) implies the formulas
\begin{equation}
  \label{com-Lich-HessIO} \mathcal{L}^{\Omega}_g \partial^g_{T_{X, J}}
  \nabla_g u = \partial^g_{T_{X, J}} \nabla_g  (\Delta^{\Omega}_g -
  2\mathbbm{I}) u,
\end{equation}

\begin{equation}
  \label{com-Lich-HessOI} \mathcal{L}^{\Omega}_g \overline{\partial}_{T_{X,
  J}} \nabla_g u = \overline{\partial}_{T_{X, J}} \nabla_g  (\Delta^{\Omega}_g
  - 2\mathbbm{I}) u,
\end{equation}
in the K\"ahler-Ricci soliton case. The properties (\ref{WellHesWF}) and (\ref{fundam-property-SRF}) combined with
(\ref{ant-lin-Lich}) imply (\ref{welldefF0}) and (\ref{fundPF0}).

\section{Basic complex Bochner type formulas}\label{CxBochSec}

We need to review in detail now some fact from \cite{Fut}, (see also \cite{Pal1}). 
Most of the formulas in this section will be intensively used in the rest of the paper. Let $(X, J, g)$ be a compact 
connected K\"ahler manifold with symplectic form
$\omega \assign g J$. We remind that the hermitian product induced by $\omega$
over the bundle $\Lambda_J^{1, 0} T^{\ast}_X$ satisfies the identity
\begin{eqnarray*}
  2 \left\langle \alpha, \beta \right\rangle_{\omega} & = & \tmop{Tr}_{\omega}
  \left( i \alpha \wedge \bar{\beta} \right) .
\end{eqnarray*}
Let $\Omega > 0$ be a smooth volume form and set as usual $f \assign \log
\frac{d V_g}{\Omega}$. We define the $\Omega$-weighted complex Laplace type
operator acting on functions $u \in C^{\infty} (X, \mathbbm{C})$ as
\begin{eqnarray*}
  \Delta^{\Omega}_{g, J} u & \assign & e^f \tmop{Tr}_{\omega} \left[ i
  \overline{\partial}_J \left( e^{- f} \partial_J u \right) \right]\\
  &  & \\
  & = & \Delta_g u + 2 \left\langle \partial_J u, \partial_J f
  \right\rangle_{\omega}\\
  &  & \\
  & = & \Delta_g u + 2 \partial_J u \cdot \nabla_g f .
\end{eqnarray*}
We notice the identities $\Delta^{\Omega}_{g, J} f = \Delta^{\Omega}_g f$ and
$2 \Delta^{\Omega}_g = \tmop{Re} (\Delta^{\Omega}_{g, J})$. The complex
operator $\Delta^{\Omega}_{g, J}$ is self-adjoint with respect to the the
$L_{\Omega}^2$-hermitian product
\begin{equation}
  \label{L2Om-prod}  \left\langle u, v \right\rangle_{\Omega} \assign \int_X u
  \overline{v} \Omega .
\end{equation}
Indeed integrating by parts we obtain
\begin{eqnarray*}
  \int_X i \overline{\partial}_J \left( e^{- f} \partial_J u \right)
  \overline{v} \wedge \omega^{n - 1} & = & \int_X \partial_J u \wedge i e^{-
  f}  \overline{\partial}_J \overline{v} \wedge \omega^{n - 1}\\
  &  & \\
  & = & - \int_X u i \partial_J \left( e^{- f}  \overline{\partial}_J
  \overline{v} \right) \wedge \omega^{n - 1} .
\end{eqnarray*}
(Notice the equality $\Omega = e^{- f} \omega^n / n!$.) We observe in
particular the identity
\begin{eqnarray*}
  \int_X \Delta^{\Omega}_{g, J} u \cdot \overline{v} \Omega & = & \int_X 2
  \left\langle \partial_J u, \partial_J v \right\rangle_{\omega} \Omega,
\end{eqnarray*}
which implies that all the eigenvalues satisfy $\lambda_j (\Delta^{\Omega}_{g,
J}) \geqslant 0$. For any function $u \in C^{\infty} (X, \mathbbm{C})$ we
define the $J$-complex $g$-gradient as the real vector field;
\begin{eqnarray*}
  \nabla_{g, J} u & \assign & \nabla_g \tmop{Re} u + J \nabla_g \tmop{Im} u
  \in C^{\infty} (X, T_X) .
\end{eqnarray*}
With these notations hold the complex decomposition formula
\begin{equation}
  \label{dec-Cx-grad} \nabla_{g, J} u \neg g = \partial_J \overline{u} +
  \overline{\partial}_J u .
\end{equation}
We consider now the linear operator
\begin{eqnarray*}
  B^{\Omega}_{g, J} : C^{\infty} (X, \mathbbm{R}) & \longrightarrow &
  C_{\Omega}^{\infty} (X, \mathbbm{R})_0,\\
  &  & \\
  B^{\Omega}_{g, J} u & \assign & \tmop{div}^{\Omega} (J \nabla_g u) .
\end{eqnarray*}
This is a first order differential operator. Indeed
\begin{eqnarray*}
  B^{\Omega}_{g, J} u & = & \tmop{Tr}_{\mathbbm{R}} \left( J \nabla^2_g u
  \right) - d f \cdot J \nabla_g u\\
  &  & \\
  & = & g (\nabla_g u, J \nabla_g f),
\end{eqnarray*}
since $J$ is $g$-anti-symmetric. We extend $B^{\Omega}_{g, J}$ over
$C^{\infty} (X, \mathbbm{C})$ by complex linearity. Let also
\begin{eqnarray*}
  2 d_J^c u & \assign & i ( \overline{\partial}_J - \partial_J) u = - d u
  \cdot J .
\end{eqnarray*}
Then the identity $2 \partial_J = d + 2 i d_J^c$ implies the decomposition
\[ \Delta^{\Omega}_{g, J} = \Delta^{\Omega}_g + 2 i \nabla_g f \neg d_J^c . \]
In other terms
\begin{eqnarray*}
  \Delta^{\Omega}_{g, J} & = & \Delta^{\Omega}_g - i B^{\Omega}_{g,J} .
\end{eqnarray*}
The following lemma is needed for the study of the operator
$\Delta^{\Omega}_{g, J}$. (Compare also with \cite{Fut}.)

\begin{lemma}
  Let $(X, J, g)$ be a K\"ahler manifold with symplectic form $\omega \assign
  g J$ and let $\Omega > 0$ be a smooth volume form. Then for all $u \in
  C^{\infty} (X, \mathbbm{R})$ and $v \in C^{\infty} (X, \mathbbm{C})$ hold
  the complex Bochner type formulas
  \begin{equation}
    \label{del-OmBoch-fnct} 2 \partial^{\ast_{g, \Omega}}_{T_{X, J}}
    \partial^g_{T_{X, J}} \nabla_g u = \nabla_{g, J} \Delta^{\Omega}_{g, J} u
    - 2 \overline{\partial}_{T_{X, J}} \nabla_g f \nabla_g u,
  \end{equation}
\end{lemma}
\begin{equation}
  \label{dbar-OmBoch-fnct} 2 \overline{\partial}^{\ast_{g, \Omega}}_{T_{X, J}}
  \overline{\partial}_{T_{X, J}} \nabla_{g, J}  \overline{v} = \nabla_{g, J}
  \overline{\Delta^{\Omega}_{g, J} v} - 2 \tmop{Ric}_J^{\ast}
  (\Omega)_{\omega} \nabla_{g, J}  \overline{v} .
\end{equation}
\begin{proof}
  Let $\xi \in C^{\infty} (X, T_X)$ and observe that for bi-degree reasons hold
  the identity
  \begin{eqnarray*}
    2 \partial^{\ast_{g, \Omega}}_{T_{X, J}} \partial^g_{T_{X, J}} \xi & = & 2
    \nabla_g^{\ast_{\Omega}} \partial^g_{T_{X, J}} \xi\\
    &  & \\
    & = & \Delta^{\Omega}_g \xi - \nabla_g^{\ast_{\Omega}} \left( J
    \nabla_{g, J \bullet} \xi \right)\\
    &  & \\
    & = & \Delta^{\Omega}_g \xi - \nabla_g^{\ast} \left( J \nabla_{g, J
    \bullet} \xi \right) - J \nabla_{g, J \nabla_g f} \xi .
  \end{eqnarray*}
  Let $(e_k)^{2 n}_{k = 1}$ be a local $g$-orthonormal frame over a
  neighborhood of an arbitrary point $p$ such that $\nabla_g e_k  (p) = 0$.
  Then at the point $p$ hold the equalities
  \begin{eqnarray*}
    - \nabla_g^{\ast} \left( J \nabla_{g, J \bullet} \xi \right) & = & J
    \nabla_{g, e_k} \nabla_{g, J e_k} \xi\\
    &  & \\
    & = & \frac{1}{2}  \left( J \nabla_{g, e_k} \nabla_{g, J e_k} \xi - J
    \nabla_{g, J e_k} \nabla_{g, e_k} \xi \right),
  \end{eqnarray*}
  since $(J e_k)^{2 n}_{k = 1}$ is also a local $g$-orthonormal frame. Then
  the fact that 
  \\
  $\left[ e_k, J e_k \right] (p) = 0$ implies
  \begin{eqnarray*}
    - \nabla_g^{\ast} \left( J \nabla_{g, J \bullet} \xi \right) & = &
    \frac{1}{2} J\mathcal{R}_g (e_k, J e_k) \xi = \tmop{Ric}^{\ast} (g) \xi .
  \end{eqnarray*}
  We infer the complex Bochner type formula
  \begin{equation}
    \label{del-OmBoch} 2 \partial^{\ast_{g, \Omega}}_{T_{X, J}}
    \partial^g_{T_{X, J}} \xi = \Delta^{\Omega}_g \xi + \tmop{Ric}^{\ast} (g)
    \xi - J \nabla_{g, J \nabla_g f} \xi .
  \end{equation}
  In a similar way we obtain
  \begin{equation}
    \label{dbar-OmBoch} 2 \overline{\partial}^{\ast_{g, \Omega}}_{T_{X, J}}
    \overline{\partial}_{T_{X, J}} \xi = \Delta^{\Omega}_g \xi -
    \tmop{Ric}^{\ast} (g) \xi + J \nabla_{g, J \nabla_g f} \xi .
  \end{equation}
  Using formulas (\ref{del-OmBoch}) and (\ref{com-grad-Lap}) we deduce the
  expressions
  \begin{eqnarray*}
    2 \partial^{\ast_{g, \Omega}}_{T_{X, J}} \partial^g_{T_{X, J}} \nabla_g u
    & = & \nabla_g \Delta^{\Omega}_g u - \nabla^2_g f \nabla_g u - J
    \nabla^2_g u J \nabla_g f\\
    &  & \\
    & = & \nabla_g \Delta^{\Omega}_g u - \left( \nabla^2_g f + J \nabla^2_g f
    J \right) \nabla_g u\\
    &  & \\
    & - & J \left( \nabla^2_g u J \nabla_g f - \nabla^2_g f J \nabla_g u
    \right)\\
    &  & \\
    & = & \nabla_g \Delta^{\Omega}_g u - 2 \overline{\partial}_{T_{X, J}}
    \nabla_g f \nabla_g u - J \nabla_g  [ g (\nabla_g u, J \nabla_g f)] .
  \end{eqnarray*}
  Using the first order expression of $B^{\Omega}_{J, g}$ we obtain
  \begin{eqnarray*}
    2 \partial^{\ast_{g, \Omega}}_{T_{X, J}} \partial^g_{T_{X, J}} \nabla_g u
    & = & \left[ \nabla_g \Delta^{\Omega}_g - J \nabla_g B^{\Omega}_{g, J}
    \right] u - 2 \overline{\partial}_{T_{X, J}} \nabla_g f \nabla_g u .
  \end{eqnarray*}
  We infer the complex differential Bochner type formula
  (\ref{del-OmBoch-fnct}). In a similar way using formulas (\ref{dbar-OmBoch})
  and (\ref{com-grad-Lap}) we deduce
  \begin{eqnarray*}
    2 \overline{\partial}^{\ast_{g, \Omega}}_{T_{X, J}}
    \overline{\partial}_{T_{X, J}} \nabla_g u & = & \nabla_g \Delta^{\Omega}_g
    u - 2 \tmop{Ric}_g^{\ast} (\Omega) \nabla_g u + \nabla^2_g f \nabla_g u +
    J \nabla^2_g u J \nabla_g f\\
    &  & \\
    & = & \nabla_g \Delta^{\Omega}_g u - 2 \tmop{Ric}_g^{\ast} (\Omega)
    \nabla_g u + \left( \nabla^2_g f + J \nabla^2_g f J \right) \nabla_g u\\
    &  & \\
    & + & J \left( \nabla^2_g u J \nabla_g f - \nabla^2_g f J \nabla_g u
    \right)\\
    &  & \\
    & = & \nabla_g \Delta^{\Omega}_g u - 2 \tmop{Ric}_g^{\ast} (\Omega)
    \nabla_g u + 2 \overline{\partial}_{T_{X, J}} \nabla_g f \nabla_g u\\
    &  & \\
    & + & J \nabla_g  [ g (\nabla_g u, J \nabla_g f)] .
  \end{eqnarray*}
  Using the first order expression of $B^{\Omega}_{J, g}$ we obtain
  \begin{eqnarray*}
    2 \overline{\partial}^{\ast_{g, \Omega}}_{T_{X, J}}
    \overline{\partial}_{T_{X, J}} \nabla_g u & = & \left[ \nabla_g
    \Delta^{\Omega}_g + J \nabla_g B^{\Omega}_{g, J} \right] u - 2
    \tmop{Ric}_J^{\ast} (\Omega)_{\omega} \nabla_g u .
  \end{eqnarray*}
  We infer the complex differential Bochner type formula
  \begin{equation}
    \label{dbar-OmBoch-Relfnct} 2 \overline{\partial}^{\ast_{g,
    \Omega}}_{T_{X, J}} \overline{\partial}_{T_{X, J}} \nabla_g u = \nabla_{g,
    J} \overline{\Delta^{\Omega}_{g, J} u} - 2 \tmop{Ric}_J^{\ast}
    (\Omega)_{\omega} \nabla_g u .
  \end{equation}
  More in general for all $v \in C^{\infty} (X, \mathbbm{C})$ this writes as
  (\ref{dbar-OmBoch-fnct}).
\end{proof}

Notice that for bi-degree reasons the identity (\ref{com-grad-Lap}) decomposes
as
\begin{eqnarray*}
  2 \partial^{\ast_{g, \Omega}}_{T_{X, J}} \partial^g_{T_{X, J}} \nabla_g u +
  2 \overline{\partial}^{\ast_{g, \Omega}}_{T_{X, J}}
  \overline{\partial}_{T_{X, J}} \nabla_g u & = & \nabla_{g, J}
  \Delta^{\Omega}_{g, J} u + \nabla_{g, J} \overline{\Delta^{\Omega}_{g, J} u}
  \\
  &  & \\
  & - & 2 \overline{\partial}_{T_{X, J}} \nabla_g f \nabla_g u - 2
  \tmop{Ric}_J^{\ast} (\Omega)_{\omega} \nabla_g u .
\end{eqnarray*}
Then we can obtain (\ref{dbar-OmBoch-Relfnct}) from (\ref{del-OmBoch-fnct})
and vice versa. We observe also that the complex Bochner identities
(\ref{del-OmBoch-fnct}), (\ref{dbar-OmBoch-fnct}) write in the K\"ahler-Ricci-Soliton case as
\begin{equation}
  \label{del-OmBoch-fnctSOL} 2 \partial^{\ast_{g, \Omega}}_{T_{X, J}}
  \partial^g_{T_{X, J}} \nabla_g u = \nabla_{g, J} \Delta^{\Omega}_{g, J} u,
\end{equation}

\begin{equation}
  \label{dbar-OmBoch-fnctSOL} 2 \overline{\partial}^{\ast_{g, \Omega}}_{T_{X,
  J}} \overline{\partial}_{T_{X, J}} \nabla_{g, J}  \overline{v} = \nabla_{g,
  J} \overline{ (\Delta^{\Omega}_{g, J} - 2\mathbbm{I}) v},
\end{equation}
for all $u \in C^{\infty} (X, \mathbbm{R})$ and $v \in C^{\infty} (X,
\mathbbm{C})$. Obviously the identity (\ref{dbar-OmBoch-fnctSOL}) still hold
in the more general case $\tmop{Ric}_J (\Omega) = \omega$. We observe now an
other integration by parts formula.

Let $\xi \in C^{\infty} \left( X, T_X \right)$, $A \in C^{\infty} \left( X,
T^{\ast}_{X, - J} \otimes T_{X, J} \right)$ and observe that the comparison
between Riemannian and hermitian norms of $T_X$-valued 1-forms (see the appendix in \cite{Pal2}) implies
\begin{eqnarray*}
  \int_X \left\langle \overline{\partial}_{T_{X, J}} \xi, A \right\rangle_g
  \Omega & = & \frac{1}{2} \int_X \left[ \left\langle
  \overline{\partial}_{T_{X, J}} \xi, A \right\rangle_{\omega} + \left\langle
  A, \overline{\partial}_{T_{X, J}} \xi \right\rangle_{\omega} \right]
  \Omega\\
  &  & \\
  & = & \frac{1}{2} \int_X \left[ \left\langle \xi,
  \overline{\partial}^{\ast_{g, \Omega}}_{T_{X, J}} A \right\rangle_{\omega} +
  \left\langle \overline{\partial}^{\ast_{g, \Omega}}_{T_{X, J}} A, \xi
  \right\rangle_{\omega} \right] \Omega\\
  &  & \\
  & = & \int_X \left\langle \xi, \overline{\partial}^{\ast_{g,
  \Omega}}_{T_{X, J}} A \right\rangle_g \Omega .
\end{eqnarray*}
Using this and multiplying both sides of (\ref{dbar-OmBoch-fnctSOL}) by
$\nabla_{g, J}  \overline{v}$ we obtain the identity
\begin{equation}
  \label{CxBoch-frm} 2 \int_X \left| \overline{\partial}_{T_{X, J}} \nabla_{g,
  J}  \overline{v}  \right|^2_g \Omega = \int_X \left\langle \nabla_{g, J}
  \overline{(\Delta^{\Omega}_{g, J} - 2\mathbbm{I}) v}, \nabla_{g, J} 
  \overline{v} \right\rangle_g \Omega,
\end{equation}
in the case $\tmop{Ric}_J (\Omega) = \omega$. We consider now the
$J$-anti-linear component of the complex Hessian map;
\begin{eqnarray*}
  \mathcal{H}^{0, 1}_{g, J} : C_{\Omega}^{\infty} (X, \mathbbm{C})_0 &
  \longrightarrow & C^{\infty} (X, \Lambda_J^{0, 1} T^{\ast}_X
  \otimes_{\mathbbm{C}} T_{X, J})\\
  &  & \\
  u & \longmapsto & \overline{\partial}_{T_{X, J}} \nabla_{g, J} u .
\end{eqnarray*}
We observe that $H_d^1 (X, \mathbbm{C}) = 0$ in the case of Fano manifolds and
we remind the following well known fact

\begin{lemma}
  \label{iso-CxGrad}Let $(X, J, g)$ be a compact connected K\"ahler manifold
  such that $H_d^1 (X, \mathbbm{C}) = 0$. Then the map
  \begin{eqnarray*}
    \tmop{Ker} \mathcal{H}^{0, 1}_{g, J} & \longrightarrow & H^0 (X, T_{X,
    J})\\
    &  & \\
    u & \longmapsto & \nabla_{g, J} u,
  \end{eqnarray*}
  is an isomorphism of complex vector spaces.
\end{lemma}

\begin{proof}
  We observe first the injectivity. Using the complex decomposition
  (\ref{dec-Cx-grad}) we infer the formula
  \begin{eqnarray*}
    d (\nabla_{g, J} u \neg g) & = & \partial_J \overline{\partial}_J  (u -
    \overline{u}),
  \end{eqnarray*}
  which in the case $\nabla_{g, J} u = 0$ implies $\tmop{Im} u = 0$ and thus
  $\tmop{Re} u = 0$. In order to show the surjectivity we consider an
  arbitrary $\xi \in H^0 (X, T_{X, J})$. Then the identity
  (\ref{clos-hol-vct}) below implies
  \begin{eqnarray*}
    \overline{\partial}_J  (\xi_J^{1, 0} \neg \omega) & = & 0 .
  \end{eqnarray*}
  By Hodge decomposition hold the identity $H_{\overline{\partial}}^{0, 1} (X,
  \mathbbm{C}) = 0$. We deduce the existence of a unique function $u \in
  C_{\Omega}^{\infty} (X, \mathbbm{C})_0$ such that
  \begin{eqnarray*}
    i \overline{\partial}_J u & = & \xi_J^{1, 0} \neg \omega = i \xi_J^{1, 0}
    \neg g .
  \end{eqnarray*}
  Thus $\xi = \nabla_{g, J} u$ thanks to the complex decomposition
  (\ref{dec-Cx-grad}).
\end{proof}

\begin{lemma}
  Let $(X, J)$ be a complex manifold and let $\omega \in C^{\infty} (X,
  \Lambda^{1, 1}_J T^{\ast}_X)$, $\xi \in C^{\infty} (X, T^{1, 0}_X)$. Then
  hold the identity \ \ \
  \begin{equation}
    \label{clos-hol-vct}  \overline{\partial}_J  (\xi \neg \omega) =
    \overline{\partial}_{T^{1, 0}_{X, J}} \xi \neg \omega - \xi \neg
    \overline{\partial}_J \omega .
  \end{equation}
\end{lemma}

\begin{proof}
  Let $\eta, \mu \in C^{\infty} (X, T^{0, 1}_X)$ and observe the identities
  (see \cite{Pal})
  \begin{eqnarray*}
    \overline{\partial}_J  (\xi \neg \omega)  (\eta, \mu) & = & \eta . \omega
    (\xi, \mu) - \mu . \omega (\xi, \eta) - \omega (\xi, \left[ \eta, \mu
    \right]),\\
    &  & \\
    \overline{\partial}_J \omega (\eta, \xi, \mu) & = & \eta . \omega (\xi,
    \mu) + \mu . \omega (\eta, \xi)\\
    &  & \\
    & - & \omega ( \left[ \eta, \xi, \right]^{1, 0}, \mu) + \omega ( \left[
    \eta, \mu \right], \xi) - \omega ( \left[ \xi, \mu \right]^{1, 0}, \eta)\\
    &  & \\
    & = & \overline{\partial}_J  (\xi \neg \omega)  (\eta, \mu) - \omega
    \left( \overline{\partial}_{T^{1, 0}_{X, J}} \xi (\eta), \mu \right) +
    \omega \left( \overline{\partial}_{T^{1, 0}_{X, J}} \xi (\mu), \eta
    \right)\\
    &  & \\
    & = & \overline{\partial}_J  (\xi \neg \omega)  (\eta, \mu) - \omega
    \left( \overline{\partial}_{T^{1, 0}_{X, J}} \xi (\eta), \mu \right) -
    \omega \left( \eta, \overline{\partial}_{T^{1, 0}_{X, J}} \xi (\mu)
    \right)\\
    &  & \\
    & = & \left[ \overline{\partial}_J (\xi \neg \omega) -
    \overline{\partial}_{T^{1, 0}_{X, J}} \xi \neg \omega \right] (\eta, \mu),
  \end{eqnarray*}
  which implies the required identity.
\end{proof}

On the other hand the identities (\ref{dbar-OmBoch-fnctSOL}) and
(\ref{CxBoch-frm}) show that in the case $\tmop{Ric}_J (\Omega) = \omega$ hold
the identity
\begin{equation}
  \label{kern-HessOI}  \overline{\tmop{Ker} (\Delta^{\Omega}_{g, J} -
  2\mathbbm{I})} = \tmop{Ker} \mathcal{H}^{0, 1}_{g, J} .
\end{equation}
We infer the following well known result due to Futaki \cite{Fut}. (See also \cite{Gau} and 
the sub-section \ref{Poiss-Strc} in appendix B for a more
more complete statement.)

\begin{corollary}
  \label{Eigenf-HolVct}Let $(X, J)$ be a Fano manifold and let $g$ be a
  $J$-invariant K\"ahler metric such that $\omega \assign g J \in 2 \pi c_1
  (X, \left[ J \right])$. Let also $\Omega > 0$ be the unique smooth volume
  form with$\int_X \Omega = 1$ such that $\tmop{Ric}_J (\Omega) = \omega$.
  Then the map
  \begin{eqnarray*}
    \overline{\tmop{Ker} (\Delta^{\Omega}_{g, J} - 2\mathbbm{I})} &
    \longrightarrow & H^0 (X, T_{X, J})\\
    &  & \\
    u & \longmapsto & \nabla_{g, J} u,
  \end{eqnarray*}
  is well defined and it represents an isomorphism of complex vector spaces.
  The first eigenvalue $\lambda_1 (\Delta^{\Omega}_{g, J})$ of the operator
  $\Delta^{\Omega}_{g, J}$ satisfies the estimate $\lambda_1
  (\Delta^{\Omega}_{g, J}) \geqslant 2$, with equality in the case $H^0 (X,
  T_{X, J})  \neq 0$. Moreover if we set $\tmop{Kill}_g \assign \tmop{Lie}
  (\tmop{Isom}^0_g)$ and
  \begin{eqnarray*}
    \tmop{Ker}_{\mathbbm{R}} (\Delta^{\Omega}_{g, J} - 2\mathbbm{I}) & \assign
    & \tmop{Ker} (\Delta^{\Omega}_{g, J} - 2\mathbbm{I}) \cap
    C_{\Omega}^{\infty} (X, \mathbbm{R})_0,
  \end{eqnarray*}
  then the map
  \begin{equation}
    \label{RealEig-Killmap} J \nabla_g : \tmop{Ker}_{\mathbbm{R}}
    (\Delta^{\Omega}_{g, J} - 2\mathbbm{I}) \longrightarrow \tmop{Kill}_g,
  \end{equation}
  is well defined and it represents an isomorphism of real vector spaces.
\end{corollary}

\begin{proof}
  We only need to show the statement concerning the map
  (\ref{RealEig-Killmap}). Let $\xi \in \tmop{Kill}_g$ and let $(\varphi_t)_{t
  \in \mathbbm{R}} \subset \tmop{Isom}^0_g$ be the corresponding 1-parameter
  sub-group. The K\"ahler condition $\nabla_g J = 0$ implies $\Delta_{d, g}
  \omega = 0$ and thus $\Delta_{d, g}  (\varphi^{\ast}_t \omega) = 0$. Time
  deriving the latter at $t = 0$ we infer
  \begin{equation}
    \label{HLap-Lie} \Delta_{d, g} L_{\xi} \omega = 0 .
  \end{equation}
  But $L_{\xi} \omega = d (\xi \neg \omega)$ and (\ref{HLap-Lie}) rewrites as
  $d^{\ast_g} d (\xi \neg \omega) = 0$. We infer
  \begin{eqnarray*}
    0 & = & L_{\xi} \omega = g L_{\xi} J = 2 \omega \overline{\partial}_{T_{X,
    J}} \xi,
  \end{eqnarray*}
  and thus
  \begin{eqnarray*}
    \tmop{Kill}_g & = & \left\{ \xi \in H^0 (X, T_{X, J}) \mid L_{\xi} \omega
    = 0 \right\}\\
    &  & \\
    & = & \left\{ \xi \in H^0 (X, T_{X, J}) \mid d (\xi \neg \omega) = 0
    \right\} \\
    &  & \\
    & = & \left\{ \xi \in H^0 (X, T_{X, J}) \mid \exists u \in
    C_{\Omega}^{\infty} (X, \mathbbm{R})_0 : \xi \neg \omega = d u \right\},
  \end{eqnarray*}
  thanks to the fact that $H_d^1 (X, \mathbbm{R}) = 0$. But the latter identity
  rewrites as
  \begin{eqnarray*}
    \tmop{Kill}_g & = & \left\{ \xi \in H^0 (X, T_{X, J}) \mid \exists u \in
    C_{\Omega}^{\infty} (X, \mathbbm{R})_0 : \xi = J \nabla_g u \right\},
  \end{eqnarray*}
  which shows that the map (\ref{RealEig-Killmap}) is well defined thanks to
  the first statement of corollary \ref{Eigenf-HolVct}. The surjectivity of
  the map (\ref{RealEig-Killmap}) follows from the identity
  (\ref{dbar-OmBoch-fnctSOL}) applied to the function $v \assign - i u$, with
  $u \in C_{\Omega}^{\infty} (X, \mathbbm{R})_0$ such that $J \nabla_g u \in
  \tmop{Kill}_g$. The injectivity of the map (\ref{RealEig-Killmap}) is
  obvious.
\end{proof}

Using the variational characterization of the first eigenvalue we observe;
\begin{eqnarray*}
  \lambda_1 (\Delta^{\Omega}_{g, J}) & = & \inf \left\{ \frac{\int_X
  \Delta^{\Omega}_{g, J} u \overline{u} \Omega}{\int_X \left| u \right|^2
  \Omega} \mid u \in C_{\Omega}^{\infty} (X, \mathbbm{C})_0 \smallsetminus
  \left\{ 0 \right\} \right\}\\
  &  & \\
  & \leqslant & \inf \left\{ \frac{\int_X 2 \left| \partial_J u
  \right|^2_{\omega} \Omega}{\int_X u^2 \Omega} \mid u \in C_{\Omega}^{\infty}
  (X, \mathbbm{R})_0 \smallsetminus \left\{ 0 \right\} \right\}\\
  &  & \\
  & = & \inf \left\{ \frac{\int_X \Delta^{\Omega}_g u u \Omega}{\int_X u^2
  \Omega} \mid u \in C_{\Omega}^{\infty} (X, \mathbbm{R})_0 \smallsetminus
  \left\{ 0 \right\} \right\} \\
  &  & \\
  & = & \lambda_1 (\Delta^{\Omega}_g),
\end{eqnarray*}
thanks to the identity $2 \left| \partial_J u \right|^2_{\omega} = \left|
\nabla_g u \right|^2_g$. We deduce that in the set-up of corollary
\ref{Eigenf-HolVct} hold the estimate
\begin{equation}
  \label{Estim-FirstEigen} \lambda_1 (\Delta^{\Omega}_g) \geqslant 2 .
\end{equation}

\section{Symmetric variations of K\"ahler structures}\label{Sect-SmVKS}

We show a few fundamental facts about the space of symmetric variations of
K\"ahler structures $\mathbbm{K}\mathbbm{V}^J_g$ given by the elements $v \in C^{\infty} \left( X,
  S_{\mathbbm{R}}^2 T^{\ast}_X \right)$ such that there exists a smooth family $(J_t^{}, g_t)_t \subset
  \mathcal{K}\mathcal{S}$ with $(J_0, g_0) = (J, g)$, $\dot{g}_0 = v$ and $\dot{J}_0^{} = (
  \dot{J}_0^{})_g^T$. One can observe (see \cite{Pal3}) that $\mathbbm{K}\mathbbm{V}^J_g \subseteq
\mathbbm{D}^J_g$ with
\begin{equation}
  \label{Kah-D} \mathbbm{D}^J_g : = \left\{ v \in C^{\infty} \left( X,
  S_{\mathbbm{R}}^2 T^{\ast}_X \right) \mid \hspace{0.25em} \partial^g_{T_{X,
  J}} (v_g^{\ast})_J^{1, 0} = 0, \overline{\partial}_{T_{X, J}}
  (v_g^{\ast})_J^{0, 1} = 0 \right\},
\end{equation}
where $(v_g^{\ast})_J^{1, 0}$ and $(v_g^{\ast})_J^{0, 1}$ denote respectively
the $J$-linear and $J$-anti-linear parts of the endomorphism $v^{\ast}_g$. We
remind here some lines of this basic fact. We define
\begin{eqnarray*}
  2 \nabla_{g, i, J}^{1, 0} & \assign & \nabla_g - i \nabla_{g, J \bullet} \\
  &  & \\
  2 \nabla_{g, i, J}^{0, 1} & \assign & \nabla_g + i \nabla_{g, J \bullet_{}} 
\end{eqnarray*}
Let $v_J'$ and $v_J''$ be respectively the $J$-invariant and
$J$-anti-invariant parts of $v$ and set for notation simplicity
\begin{eqnarray*}
  A' & \assign & (v^{\ast}_g)_J^{1, 0} = (v_J')_g^{\ast},\\
  &  & \\
  A'' & \assign & (v_g^{\ast})_J^{0, 1} = (v_J'')_g^{\ast} .
\end{eqnarray*}
The identity $A'' = - J \dot{J}_0$ (see \cite{Pal3}) implies directly
$\dot{\omega}_0 = v_J' J = \omega A'$. We infer
\begin{eqnarray*}
  0 & = & d \dot{\omega}_0 = d (v_J' J) .
\end{eqnarray*}
The fact that the $(1, 1)$-form $v_J' J$ is real implies that the identity $d
(v_J' J) = 0$ is equivalent to the identity $\partial_J (v_J' J) = 0$. In its
turn this is equivalent to the identity $\partial^g_{T_{X, J}} A' = 0$. We
observe indeed that for all $\xi, \eta, \mu \in C^{\infty} (X, T_X
\otimes_{\mathbbm{R}} \mathbbm{C})$ hold the equalities
\begin{eqnarray*}
  \partial_J  (\omega A') (\xi', \eta', \mu'') & = & \nabla_{g, i, J}^{1, 0} 
  (\omega A') (\xi', \eta', \mu'') - \nabla_{g, i, J}^{1, 0}  (\omega A')
  (\eta', \xi', \mu'') \\
  &  & \\
  & + & \nabla_{g, i, J}^{1, 0}  (\omega A') (\mu'', \xi', \eta')\\
  &  & \\
  & = & \omega \left( \left[ \nabla_{g, i, J}^{1, 0} A' (\xi', \eta') -
  \nabla_{g, i, J}^{1, 0} A' (\eta', \xi') \right], \mu'' \right)\\
  &  & \\
  & - & \omega \left( \nabla_{g, i, J}^{1, 0} A' (\mu'', \xi'), \eta'
  \right)\\
  &  & \\
  & = & \omega \left( \left[ \nabla_{g, J}^{1, 0} A' (\xi', \eta') -
  \nabla_{g, J}^{1, 0} A' (\eta', \xi') \right], \mu'' \right)\\
  &  & \\
  & = & \omega \left( \partial^g_{T_{X, J}} A' (\xi', \eta'), \mu'' \right) .
\end{eqnarray*}
In order to continue the study of the space $\mathbbm{D}^J_g$ we need to show
a few general and fundamental facts. We start with the following weighted
complex Weitzenb\"ock type formula.

\begin{lemma}
  Let $(X, J, g)$ be a K\"ahler manifold, let $\Omega > 0$ be a smooth volume
  form and let $A \in C^{\infty} \left( X, T^{\ast}_{X, - J} \otimes T_{X, J}
  \right)$. Then hold the identity
  \begin{equation}
    \label{Fund-exp-AntHol} \Delta^{\Omega, - J}_{T_{X, g}} A =
    \nabla_g^{\ast_{\Omega}} \nabla_{g, J}^{0, 1} A -\mathcal{R}_g \ast A + A
    \tmop{Ric}_J^{\ast} (\Omega)_{\omega}
  \end{equation}
\end{lemma}

\begin{proof}
  We observe that for bi-degree reasons hold the identities
  \begin{eqnarray*}
    \Delta^{\Omega, - J}_{T_{X, g}} A & : = & \overline{\partial}_{T_{X, J}}
    \overline{\partial}^{\ast_{g, \Omega}}_{T_{X, J}} A + \frac{1}{2} 
    \overline{\partial}^{\ast_{g, \Omega}}_{T_{X, J}}
    \overline{\partial}_{T_{X, J}} A\\
    &  & \\
    & = & \overline{\partial}_{T_{X, J}} \nabla^{\ast_{\Omega}}_{T_X, g} A +
    \frac{1}{2} \nabla^{\ast_{\Omega}}_{T_X, g} \overline{\partial}_{T_{X, J}}
    A\\
    &  & \\
    & = & \overline{\partial}_{T_{X, J}} \nabla^{\ast_{\Omega}}_g A +
    \nabla^{\ast_{\Omega}}_g \overline{\partial}_{T_{X, J}} A .
  \end{eqnarray*}
  Let
  \begin{eqnarray*}
    \widehat{\nabla^{0, 1}_{g, J} A} (\xi, \eta) & : = & \nabla^{0, 1}_{g, J}
    A (\eta, \xi) .
  \end{eqnarray*}
  Then
  \begin{eqnarray*}
    2 \Delta^{\Omega, - J}_{T_{X, g}} A & = & \nabla_g
    \nabla^{\ast_{\Omega}}_g A + J \nabla_{g, J \bullet}
    \nabla^{\ast_{\Omega}}_g A\\
    &  & \\
    & + & 2 \nabla^{\ast_{\Omega}}_g \nabla^{0, 1}_{g, J} A - 2
    \nabla^{\ast_{\Omega}}_g \widehat{\nabla^{0, 1}_{g, J} A} .
  \end{eqnarray*}
  We fix an arbitrary point $p$ and we choose an arbitrary vector field $\xi$
  such that $\nabla_g \xi (p) = 0$. Let $(e_k)_k$ be a $g$-orthonormal local
  frame such that $\nabla_g e_k  (p) = 0$. We observe the local expression
  \begin{eqnarray*}
    \nabla^{\ast_{\Omega}}_g \widehat{\nabla^{0, 1}_{g, J} A} \cdot \xi & = &
    - \nabla_{g, e_k} \widehat{\nabla^{0, 1}_{g, J} A}  (e_k, \xi) +
    \nabla^{0, 1}_{g, J} A (\xi, \nabla_g f) .
  \end{eqnarray*}
  At the point $p$ hold the identities
  \begin{eqnarray*}
    2 \nabla_{g, e_k} \widehat{\nabla^{0, 1}_{g, J} A}  (e_k, \xi) & = & 2
    \nabla_{g, e_k} \left[ \nabla^{0, 1}_{g, J} A (\xi, e_k) \right]\\
    &  & \\
    & = & \nabla_{g, e_k} \nabla_{g, \xi} A \cdot e_k + J \nabla_{g, e_k}
    \nabla_{g, J \xi} A \cdot e_k,
  \end{eqnarray*}
  and thus
  \begin{eqnarray*}
    2 \nabla^{\ast_{\Omega}}_g \widehat{\nabla^{0, 1}_{g, J} A} \cdot \xi & =
    & - \nabla_{g, e_k} \nabla_{g, \xi} A \cdot e_k - J \nabla_{g, e_k}
    \nabla_{g, J \xi} A \cdot e_k\\
    &  & \\
    & + & \nabla_{g, \xi} A \cdot \nabla_g f + J \nabla_{g, J \xi} A \cdot
    \nabla_g f .
  \end{eqnarray*}
  We obtain the identity at the point $p$,
  \begin{eqnarray*}
    2 \Delta^{\Omega, - J}_{T_{X, g}} A \cdot \xi & = & - \nabla_{g, \xi}
    \nabla_{g, e_k} A \cdot e_k + \nabla_{g, \xi}  \left( A \cdot \nabla_g f
    \right)\\
    &  & \\
    & - & J \nabla_{g, J \xi} \nabla_{g, e_k} A \cdot e_k + J \nabla_{g, J
    \xi}  \left( A \cdot \nabla_g f \right)\\
    &  & \\
    & + & 2 \nabla^{\ast_{\Omega}}_g \nabla^{0, 1}_{g, J} A \cdot \xi\\
    &  & \\
    & + & \nabla_{g, e_k} \nabla_{g, \xi} A \cdot e_k + J \nabla_{g, e_k}
    \nabla_{g, J \xi} A \cdot e_k\\
    &  & \\
    & - & \nabla_{g, \xi} A \cdot \nabla_g f - J \nabla_{g, J \xi} A \cdot
    \nabla_g f .
  \end{eqnarray*}
  We remind that for any $A \in C^{\infty} \left( X, \tmop{End} (T_X) \right)$
  and $\xi, \eta \in C^{\infty} (X, T_X)$ hold the general formula
  \begin{equation}
    \label{com-covEnd} \nabla_{g, \xi} \nabla_{g, \eta} A - \nabla_{g, \eta}
    \nabla_{g, \xi} A = \left[ \mathcal{R}_g (\xi, \eta), A \right] +
    \nabla_{g, \left[ \xi, \eta \right]} A .
  \end{equation}
  Using (\ref{com-covEnd}) and the fact that in our case $\left[ e_k, \xi
  \right] (p) = \left[ e_k, J \xi \right] (p) = 0$ we obtain
  \begin{eqnarray*}
    2 \Delta^{\Omega, - J}_{T_{X, g}} A \cdot \xi & = & \mathcal{R}_g (e_k,
    \xi) A \cdot e_k - A\mathcal{R}_g (e_k, \xi) \cdot e_k\\
    &  & \\
    & + & J \left[ \mathcal{R}_g (e_k, J \xi) A \cdot e_k - A\mathcal{R}_g
    (e_k, J \xi) \cdot e_k \right]\\
    &  & \\
    & + & 2 \nabla^{\ast_{\Omega}}_g \nabla^{0, 1}_{g, J} A \cdot \xi + A
    \nabla^2_{g, \xi} f + J A \nabla_{g, J \xi} \nabla_g f\\
    &  & \\
    & = & - (\mathcal{R}_g \ast A) \cdot \xi + A \tmop{Ric}^{\ast} (g) \cdot
    \xi\\
    &  & \\
    & - & J (\mathcal{R}_g \ast A) \cdot J \xi + J A \tmop{Ric}^{\ast} (g)
    \cdot J \xi\\
    &  & \\
    & + & 2 \nabla^{\ast_{\Omega}}_g \nabla^{0, 1}_{g, J} A \cdot \xi + A
    \partial^g_{T_{X, J}} \nabla_g f \cdot \xi\\
    &  & \\
    & = & 2 \left[ \nabla_g^{\ast_{\Omega}} \nabla_{g, J}^{0, 1} A
    -\mathcal{R}_g \ast A + A \tmop{Ric}_J^{\ast} (\Omega)_{\omega} \right]
    \cdot \xi,
  \end{eqnarray*}
  thanks to (\ref{Ant-J-curv}).
\end{proof}

Multiplying both sides of (\ref{Fund-exp-AntHol}) by $A$ and integrating by
parts we infer
\begin{eqnarray*}
  \int_X \left\langle \Delta^{\Omega, - J}_{T_{X, g}} A, A \right\rangle_g
  \Omega & = & \int_X \left[ \left\langle \nabla_{g, J}^{0, 1} A, \nabla_g A
  \right\rangle_g + \left\langle A \tmop{Ric}_J^{\ast} (\Omega)_{\omega}
  -\mathcal{R}_g \ast A, A \right\rangle_g \right] \Omega .
\end{eqnarray*}
Using the fact that $\left\langle \nabla^{1, 0}_{g, J} A_J'', \nabla^{0,
1}_{g, J} A_J'' \right\rangle_g = 0$ we obtain the integral identity
\begin{equation}
  \label{Int-BK-AntHol}  \int_X \left\langle \Delta^{\Omega, - J}_{T_{X, g}}
  A, A \right\rangle_g \Omega = \int_X \left[ | \nabla_{g, J}^{0, 1} A|^2_g +
  \left\langle A \tmop{Ric}_J^{\ast} (\Omega)_{\omega} -\mathcal{R}_g \ast A,
  A \right\rangle_g \right] \Omega .
\end{equation}
We observe also the following corollary.

\begin{corollary}
  Let $(X, J, g)$ be a K\"ahler manifold, let $\Omega > 0$ be a smooth volume
  form and let $A \in C^{\infty} \left( X, T^{\ast}_{X, - J} \otimes T_{X, J}
  \right)$. Then hold the identities
  \begin{equation}
    \label{Lich-AntiHol} \mathcal{L}^{\Omega}_g A = 2 \Delta^{\Omega, -
    J}_{T_{X, g}} A + \tmop{div}_g^{\Omega} \nabla_{g, J \bullet}  (J A) - 2 A
    \tmop{Ric}_J^{\ast} (\Omega)_{\omega},
  \end{equation}
  
  \begin{equation}
    \label{div-JA} \tmop{div}_g^{\Omega} \nabla_{g, J \bullet}  (J A) =
    \tmop{Ric}^{\ast} (g) A + A_{} \tmop{Ric}^{\ast} (g) - (J \nabla_g f) \neg
    (J \nabla_g A) .
  \end{equation}
\end{corollary}

\begin{proof}
  It is obvious that the identity (\ref{Fund-exp-AntHol}) rewrites as
  (\ref{Lich-AntiHol}). In order to show (\ref{div-JA}) let $(\eta_k)^n_{k =
  1}$ be a local complex frame of $T_{X, J}$ in a neighborhood of a point $p$
  with $\nabla_g \eta_k  (p) = 0$ such that the real frame $(e_l)^{2 n}_{l =
  1}$, $e_l = \eta_l$, $l = 1, \ldots, n$ and $e_{n + k} = J \eta_k$, $k = 1,
  \ldots, n$ is $g$-orthonormal. Then at the point $p$ hold the equalities
  \begin{eqnarray*}
    \tmop{div}_g \nabla_{g, J \bullet}  (J A_J'') & = & \sum_{l = 1}^{2 n}
    \nabla_{g, e_l} \nabla_{g, J e_l} (J A_J'')\\
    &  & \\
    & = & \sum_{k = 1}^n \left[ \nabla_{g, \eta_k} \nabla_{g, J \eta_k} (J
    A_J'') - \nabla_{g, J \eta_k} \nabla_{g, \eta_k} (J A_J'') \right]\\
    &  & \\
    & = & \sum_{k = 1}^n \left[ \mathcal{R}_g (\eta_k, J \eta_k), J A_J''
    \right],
  \end{eqnarray*}
  thanks to the general formula (\ref{com-covEnd}) and thanks to the fact that
  $\left[ \eta_k, J \eta_k \right] (p) = 0$. Using the $J$-linear and
  $J$-anti-linear properties of the tensors involved in the previous equality
  we obtain
  \begin{eqnarray*}
    \tmop{div}_g \nabla_{g, J \bullet}  (J A_J'') & = & \sum_{k = 1}^n \left[
    J\mathcal{R}_g (\eta_k, J \eta_k) A_J'' + A_J'' J\mathcal{R}_g (\eta_k, J
    \eta_k) \right] \\
    &  & \\
    & = & \tmop{Ric}^{\ast} (g) A_J'' + A_J'' \tmop{Ric}^{\ast} (g) .
  \end{eqnarray*}
  Notice indeed the identities
  \begin{eqnarray*}
    2 \tmop{Ric}^{\ast} (g) & = & \sum_{l = 1}^{2 n} J\mathcal{R}_g (e_l, J
    e_l) \\
    &  & \\
    & = & \sum_{k = 1}^n \left[ J\mathcal{R}_g (\eta_k, J \eta_k) -
    J\mathcal{R}_g (J \eta_k, \eta_k) \right]\\
    &  & \\
    & = & 2 \sum_{k = 1}^n J\mathcal{R}_g (\eta_k, J \eta_k) .
  \end{eqnarray*}
  We conclude the required formula (\ref{div-JA}).
\end{proof}

We define now the vector spaces
\begin{eqnarray*}
  \mathcal{H}_{g, \Omega}^{0, 1} \left( T_{X, J} \right) & \assign &
  \tmop{Ker} \Delta^{\Omega, - J}_{T_{X, g}} \cap C^{\infty} \left( X,
  T^{\ast}_{X, - J} \otimes T_{X, J} \right),\\
  &  & \\
  \mathcal{H}_{g, \Omega}^{0, 1} \left( T_{X, J} \right)_{\tmop{sm}} & \assign
  & \left\{ A \in \mathcal{H}_{g, \Omega}^{0, 1} \left( T_{X, J} \right) \mid
  \hspace{0.25em} A = A_g^T \right\} .
\end{eqnarray*}
\begin{lemma}
  \label{Sym-Harm}Let $(X, J)$ be a Fano manifold, let $g$ be a $J$-invariant
  K\"ahler metric with symplectic form $\omega \assign g J \in 2 \pi c_1 (X,
  \left[ J \right])$ and let $\Omega > 0$ be the unique smooth volume form with $\int_X
  \Omega = 1$ such that $\omega = \tmop{Ric}_J (\Omega)$. Then hold the
  identity
  \begin{eqnarray*}
    \mathcal{H}_{g, \Omega}^{0, 1} \left( T_{X, J} \right) & = &
    \mathcal{H}_{g, \Omega}^{0, 1} \left( T_{X, J} \right)_{\tmop{sm}} .
  \end{eqnarray*}
\end{lemma}

\begin{proof}
  We consider the decomposition $A = A_{\tmop{sm}} + A_{\tmop{as}}$, where
  $A_{\tmop{as}}$ and $A_{\tmop{as}}$ are respectively the $g$-symmetric and
  $g$-anti-symmetric parts of $A$. We observe the symmetries
  \begin{eqnarray*}
    \mathcal{R}_g \ast A_{\tmop{sm}} & = & (\mathcal{R}_g \ast
    A_{\tmop{sm}})_g^T,\\
    &  & \\
    \mathcal{R}_g \ast A_{\tmop{as}} & = & - (\mathcal{R}_g \ast
    A_{\tmop{as}})_g^T .
  \end{eqnarray*}
  The fact that $A \in C^{\infty} \left( X, T^{\ast}_{X, - J} \otimes T_{X, J}
  \right)$ implies $A_{\tmop{sm}}, A_{\tmop{as}} \in C^{\infty} \left( X,
  T^{\ast}_{X, - J} \otimes T_{X, J} \right)$ and thus
  \begin{eqnarray*}
    \nabla_g^{\ast_{\Omega}} \nabla_{g, J}^{0, 1} A_{\tmop{sm}} & = & \left(
    \nabla_g^{\ast_{\Omega}} \nabla_{g, J}^{0, 1} A_{\tmop{sm}} \right)_g^T,\\
    &  & \\
    \nabla_g^{\ast_{\Omega}} \nabla_{g, J}^{0, 1} A_{\tmop{as}} & = & - \left(
    \nabla_g^{\ast_{\Omega}} \nabla_{g, J}^{0, 1} A_{\tmop{as}} \right)_g^T,
  \end{eqnarray*}
  Then the identity (\ref{Fund-exp-AntHol}) implies the equalities
  \begin{equation}
    \label{Trans-AntLapSM}  \left( \Delta^{\Omega, - J}_{T_{X, g}}
    A_{\tmop{sm}} \right)_g^T - \Delta^{\Omega, - J}_{T_{X, g}} A_{\tmop{sm}}
    = \left[ \tmop{Ric}_J^{\ast} (\Omega)_{\omega}, A_{\tmop{sm}} \right],
  \end{equation}
  
  \begin{equation}
    \label{Trans-AntLapAS}  \left( \Delta^{\Omega, - J}_{T_{X, g}}
    A_{\tmop{as}} \right)_g^T + \Delta^{\Omega, - J}_{T_{X, g}} A_{\tmop{as}}
    = \left[ A_{\tmop{as}}, \tmop{Ric}_J^{\ast} (\Omega)_{\omega} \right] .
  \end{equation}
  We deduce that in the case $\tmop{Ric}_J (\Omega) = \lambda \omega$, with
  $\lambda = \pm 1, 0$, the condition $A \in \mathcal{H}_{g, \Omega}^{0, 1}
  \left( T_{X, J} \right)$ is equivalent to the conditions $A_{\tmop{sm}},
  A_{\tmop{as}} \in \mathcal{H}_{g, \Omega}^{0, 1} \left( T_{X, J} \right)$.
  We focus now on the Fano case $\lambda = 1$. We remind the identity
  $\mathcal{R}_g \ast A_{\tmop{as}} = 0$. (See (\ref{vanish-curv-anti}) in the
  appendix.) Thus if $A \in \mathcal{H}_{g, \Omega}^{0, 1} \left( T_{X, J}
  \right)$ and $\tmop{Ric}_J (\Omega) = \omega$ then the integral formula
  (\ref{Int-BK-AntHol}) reduces to
  \begin{eqnarray*}
    0 & = & \int_X \left[ | \nabla_{g, J}^{0, 1} A_{\tmop{as}} |^2_g +
    |A_{\tmop{as}} |^2_g \right] \Omega,
  \end{eqnarray*}
  which shows $A_{\tmop{as}} = 0$ and thus the required conclusion of the
  lemma.
\end{proof}

We obtain also the following statement (the case $c_1 < 0$ has been proved in \cite{D-W-W2}).

\begin{lemma}
  \label{KE-Hsym}Let $(X, J, g)$ be a compact non Ricci flat K\"ahler-Einstein
  manifold. Then hold the identity
  \begin{eqnarray*}
    \mathcal{H}_g^{0, 1} \left( T_{X, J} \right) & = & \mathcal{H}_g^{0, 1}
    \left( T_{X, J} \right)_{\tmop{sm}} .
  \end{eqnarray*}
\end{lemma}
\begin{proof}
  Using the identities (\ref{Trans-AntLapSM}) and (\ref{Trans-AntLapAS}) with
  $\Omega = C d V_g$ we deduce that in the K\"ahler-Einstein case $\tmop{Ric}
  (g) = \lambda g$, with $\lambda = \pm 1, 0$, the condition $A \in
  \mathcal{H}_g^{0, 1} \left( T_{X, J} \right)$ is equivalent to the
  conditions $A_{\tmop{sm}}, A_{\tmop{as}} \in \mathcal{H}_g^{0, 1} \left(
  T_{X, J} \right)$. On the other hand the identities (\ref{Lich-AntiHol}) and
  (\ref{div-JA}) imply in the case $\Omega = C d V_g$ the formula
  \begin{equation}
    \label{BKN-Alin-dV} \mathcal{L}_g A = 2 \Delta^{- J}_{T_{X, g}} A +
    [\tmop{Ric}^{\ast} (g), A],
  \end{equation}
  for any $A \in C^{\infty} \left( X, T^{\ast}_{X, - J} \otimes T_{X, J}
  \right)$. The fact that $\mathcal{R}_g \ast A_{\tmop{as}} = 0$ implies the
  formula
  \begin{eqnarray*}
    \Delta_g A_{\tmop{as}} & = & 2 \Delta^{- J}_{T_{X, g}} A_{\tmop{as}} +
    [\tmop{Ric}^{\ast} (g), A_{\tmop{as}}] .
  \end{eqnarray*}
  We conclude that in the K\"ahler-Einstein case $\tmop{Ric} (g) = \lambda g$,
  with $\lambda = \pm 1, 0$, any $A \in \mathcal{H}_g^{0, 1} \left( T_{X, J}
  \right)$ satisfies $\nabla_g A_{\tmop{as}} = 0$. Then the formula
  (\ref{div-JA}) with $\Omega = C d V_g$ implies
  \begin{eqnarray*}
    0 = \tmop{div}_g \nabla_{g, J \bullet}  (J A_{\tmop{as}}) & = &
    \tmop{Ric}^{\ast} (g) A_{\tmop{as}} + A_{\tmop{as}} \tmop{Ric}^{\ast} (g)
    = 2 \lambda A_{\tmop{as}} .
  \end{eqnarray*}
  We deduce $A_{\tmop{as}} = 0$ in the case $\lambda = \pm 1$. This shows the
  required conclusion.
\end{proof}

We denote by
\begin{eqnarray*}
  \Lambda^{\Omega}_{g, J} \assign \tmop{Ker} (\Delta^{\Omega}_{g, J} -
  2\mathbbm{I}) & \subset & C_{\Omega}^{\infty} (X, \mathbbm{C})_0,
\end{eqnarray*}
and by
\begin{eqnarray*}
  \Lambda^{\Omega, \bot}_{g, J} & \assign & \left[ \tmop{Ker}
  (\Delta^{\Omega}_{g, J} - 2\mathbbm{I}) \right]^{\bot_{\Omega}} \subseteq
  C_{\Omega}^{\infty} (X, \mathbbm{C})_0,
\end{eqnarray*}
its $L_{\Omega}^2$-orthogonal inside $C_{\Omega}^{\infty} (X, \mathbbm{C})_0$.
We obtain as corollary of lemma (\ref{Sym-Harm}) the following fundamental
fact.

\begin{corollary}
  \label{dec-VarJ}{\tmstrong{$($Decomposition of the variation of the
  complex structure$)$}}
  
  Let $(X, J)$ be a Fano manifold, let $g$ be a $J$-invariant K\"ahler metric
  with symplectic form $\omega \assign g J \in 2 \pi c_1 (X, \left[ J
  \right])$ and let $\Omega > 0$ be the unique smooth volume form with $\int_X \Omega =
  1$ such that $\omega = \tmop{Ric}_J (\Omega)$. Then for all $v \in
  \mathbbm{D}^J_g$ there exists a unique $\psi \in \Lambda^{\Omega, \bot}_{g,
  J}$ and a unique $A \in \mathcal{H}_{g, \Omega}^{0, 1} \left( T_{X, J}
  \right)$ such that
  \begin{eqnarray*}
    (v_g^{\ast})_J^{0, 1} & = & \overline{\partial}_{T_{X, J}} \nabla_{g, J} 
    \overline{\psi} + A .
  \end{eqnarray*}
\end{corollary}

\begin{proof}
  We observe that the identity
  \begin{eqnarray*}
    \overline{\partial}_{T_{X, J}} (v_g^{\ast})_J^{0, 1} & = & 0,
  \end{eqnarray*}
  combined with the $\Omega$-Hodge isomorphism
  \begin{eqnarray*}
    \mathcal{H}_{g, \Omega}^{0, 1} \left( T_{X, J} \right) & \simeq & H^{0, 1}
    (X, T_{X, J}) 
    \\
    \\
    &\assign & \frac{\left\{ B \in C^{\infty} (X, \Lambda_J^{0, 1}
    T^{\ast}_X \otimes_{\mathbbm{C}} T_{X, J}) \mid \overline{\partial}_{T_{X,
    J}} B = 0 \right\}}{\left\{ \overline{\partial}_{T_{X, J}} \xi \mid \xi
    \in C^{\infty} (X, T_X) \right\}},
  \end{eqnarray*}
  implies the decomposition
  \begin{eqnarray*}
    (v_g^{\ast})_J^{0, 1} & = & \overline{\partial}_{T_{X, J}} \xi + A,
  \end{eqnarray*}
  with $\xi \in C^{\infty} (X, T_X)$ and unique $A \in \mathcal{H}_{g,
  \Omega}^{0, 1} \left( T_{X, J} \right)$. Then the fact that the endomorphism
  $(v_g^{\ast})_J^{0, 1}$ is $g$-symmetric combined with lemma \ref{Sym-Harm}
  implies that $\overline{\partial}_{T_{X, J}} \xi$ is also $g$-symmetric.
  Then formula (\ref{clos-hol-vct}) implies that for all $\eta, \mu \in
  C^{\infty} (X, T^{0, 1}_X)$ holds the identity
  \begin{eqnarray*}
    \overline{\partial}_J  (\xi_J^{1, 0} \neg \omega)  (\eta, \mu) & = &
    \omega \left( \overline{\partial}_{T_{X, J}} \xi \cdot \eta, \mu \right) +
    \omega \left( \eta, \overline{\partial}_{T_{X, J}} \xi \cdot \mu \right)\\
    &  & \\
    & = & g \left( J \overline{\partial}_{T_{X, J}} \xi \cdot \eta, \mu
    \right) + g \left( J \eta, \overline{\partial}_{T_{X, J}} \xi \cdot \mu
    \right)\\
    &  & \\
    & = & g \left( \left[ \left( \overline{\partial}_{T_{X, J}} \xi
    \right)_g^T - \overline{\partial}_{T_{X, J}} \xi \right] \cdot J \eta, \mu
    \right)\\
    &  & \\
    & = & 0 .
  \end{eqnarray*}
  Then the argument showing the surjectivity of the map in lemma
  \ref{iso-CxGrad} in the section \ref{CxBochSec} implies the existence of a
  function $\Psi \in C_{\Omega}^{\infty} (X, \mathbbm{C})_0$ such that \\
  $\xi =\nabla_{g, J}  \overline{\Psi}$. This combined with the identity (\ref{kern-HessOI}) implies the existence
  and uniqueness of $\psi \in \Lambda^{\Omega, \bot}_{g, J}$ such that
  \begin{eqnarray*}
    \overline{\partial}_{T_{X, J}} \xi & = & \overline{\partial}_{T_{X, J}}
    \nabla_{g, J}  \overline{\psi} .
  \end{eqnarray*}
  We infer the required conclusion.
\end{proof}

We show now the inclusion (\ref{Kahl-VarC1}). Time deriving the condition
$\omega_t \assign g_t J_t \in 2 \pi c_1$ we infer $\left\{ \dot{\omega}_0
\right\}_d = 0$. Then (\ref{Kahl-VarC1}) follows from the complex
decomposition identity
\begin{eqnarray*}
  v^{\ast}_g & = & g^{- 1} \dot{g}_0 = \omega^{- 1} \dot{\omega}_0 - J
  \dot{J}_0 = (v_J')_g^{\ast} + (v''_J)_g^{\ast} .
\end{eqnarray*}

\section{The decomposition of the space $\mathbbm{F}_{g, \Omega}$ in the
soliton case}

\begin{lemma}
  \label{decomp-F}Let $(X, g, \Omega)$ be a compact shrinking Ricci soliton.
  Then the linear map
  \begin{eqnarray*}
    \text{$T_{g, \Omega}$} : C_{\Omega}^{\infty} (X, \mathbbm{R})_0 \oplus
    \left[ \tmop{Ker} \nabla^{\ast_{\Omega}}_g \cap C^{\infty} (X, S^2
    T^{\ast}_X) \right] & \longrightarrow & \mathbbm{F}_{g, \Omega}\\
    &  & \\
    (\varphi, \theta) & \longmapsto & \left( \nabla_g d \varphi + \theta,
    (\varphi - \Delta^{\Omega}_g \varphi) \Omega \right),
  \end{eqnarray*}
  is an isomorphism of vector spaces.
\end{lemma}

\begin{proof}
  {\tmstrong{STEP I}}. We observe first that in the compact shrinking Ricci
  soliton case the first eigenvalue $\lambda_1 (\Delta^{\Omega}_g)$ of
  $\Delta^{\Omega}_g$ satisfies the inequality $\lambda_1 (\Delta^{\Omega}_g)
  > 1$. Indeed multiplying both sides of the identity (\ref{com-grad-Lap})
  with $\nabla_g u$ and integrating we infer
  \begin{eqnarray*}
    \int_X \left\langle \nabla_g \Delta^{\Omega}_g u, \nabla_g u
    \right\rangle_g \Omega & = & \int_X \left[ \left\langle \Delta^{\Omega}_g
    \nabla_g u, \nabla_g u \right\rangle_g + \tmop{Ric}_g (\Omega) (\nabla_g
    u, \nabla_g u) \right] \Omega\\
    &  & \\
    & = & \int_X \left[ \left| \nabla^2_g u \right|^2_g + \tmop{Ric}_g
    (\Omega) (\nabla_g u, \nabla_g u) \right] \Omega .
  \end{eqnarray*}
  Let now $u \in C_{\Omega}^{\infty} (X, \mathbbm{R})_0$ be an eigen-function
  corresponding to $\lambda_1 (\Delta^{\Omega}_g) > 0$. By definition $u
  \nequiv 0$. Thus by the previous integral identity we deduce
  \begin{eqnarray*}
    \lambda_1 (\Delta^{\Omega}_g) \int_X \left| \nabla_g u \right|^2_g \Omega
    & = & \int_X \left\langle \nabla_g \Delta^{\Omega}_g u, \nabla_g u
    \right\rangle_g \Omega\\
    &  & \\
    & = & \int_X \left[ \left| \nabla^2_g u \right|^2_g + \left| \nabla_g u
    \right|^2_g \right] \Omega\\
    &  & \\
    & > & \int_X \left| \nabla_g u \right|^2_g \Omega > 0,
  \end{eqnarray*}
  which implies the required estimate.
  
  {\tmstrong{STEP II}}. Multiplying both sides of the the identity
  (\ref{com-grad-Lap}) with $g$ we obtain
  \begin{equation}
    \label{com-d-Lap} d \Delta^{\Omega}_g u = \Delta^{\Omega}_g d u + d u
    \cdot \tmop{Ric}^{\ast}_g (\Omega) .
  \end{equation}
  Let now $(v, V) \assign T_{g, \Omega}  (\varphi, \theta)$ and observe the
  equalities
  \begin{eqnarray*}
    \nabla^{\ast_{\Omega}}_g v & = & \nabla^{\ast_{\Omega}}_g \nabla_g d
    \varphi = \Delta^{\Omega}_g d \varphi = d (\Delta^{\Omega}_g \varphi -
    \varphi) .
  \end{eqnarray*}
  The last one follows from (\ref{com-d-Lap}). We infer that the linear map
  $T_{g, \Omega}$ is well defined. The fact that in the soliton case $h_{g,
  \Omega} = 0$ the differential operator $\Delta^{\Omega}_g -\mathbbm{I}$ is
  invertible over $C_{\Omega}^{\infty} (X, \mathbbm{R})_0$ implies the
  injectivity of the map $T_{g, \Omega}$.
  
  In order to show the surjectivity of the map $T_{g, \Omega}$ let $(v, V)
  \in \mathbbm{F}_{g, \Omega}$ and define the function
  \begin{eqnarray*}
    \varphi & \assign & (\mathbbm{I}- \Delta^{\Omega}_g)^{- 1}
    V^{\ast}_{\Omega} \in \text{$C_{\Omega}^{\infty} (X, \mathbbm{R})_0$} .
  \end{eqnarray*}
  Then the identity
  \begin{eqnarray*}
    \nabla^{\ast_{\Omega}}_g \nabla_g d \varphi & = & d (\Delta^{\Omega}_g
    \varphi - \varphi),
  \end{eqnarray*}
  implies that the tensor $\theta \assign v - \nabla_g d \varphi$ satisfies
  $\nabla^{\ast_{\Omega}}_g \theta = 0$. We deduce the orthogonal
  decomposition with respect to the scalar product (\ref{Glb-Rm-m})
  \begin{equation}
    \label{dec-formul-var} v = \nabla_g d \varphi + \theta,
  \end{equation}
  with $\nabla^{\ast_{\Omega}}_g \theta = 0$. We deduce the required
  surjectivity statement.
\end{proof}

We need to introduce a few notations. From now on we assume $H_d^1 (X,
\mathbbm{R}) = 0$ (this is the case of any Fano manifold) and we observe that
the first projection map
\[ p_1 : \mathbbm{F}_{g, \Omega} \longrightarrow \mathbbm{S}_{g, \Omega} : =
   \left\{ v \in C^{\infty} (X, S^2 T_X) \mid d \nabla^{\ast_{\Omega}}_g v = 0
   \right\}, \]
is an isomorphism. Over a compact K\"ahler manifold we define the real vector
spaces
\begin{eqnarray*}
  \mathbbm{S}^J_{g, \Omega} & \assign & \mathbbm{S}_{g, \Omega} \cap
  \mathbbm{D}^J_g,\\
  &  & \\
  \mathbbm{S}^J_{g, \Omega} (0) & \assign & \mathbbm{S}_{g, \Omega} \cap
  \mathbbm{D}^J_{g, 0},\\
  &  & \\
  \mathbbm{S}^J_{g, \Omega} \left[ 0 \right] & \assign & \mathbbm{S}_{g,
  \Omega} \cap \mathbbm{D}^J_{g, \left[ 0 \right]},
\end{eqnarray*}
and
\begin{eqnarray*}
  \mathbbm{E}^{\Omega}_{g, J} & \assign & \left\{ \psi \in \Lambda^{\Omega,
  \bot}_{g, J} \mid \Delta^{\Omega}_{g, J} (\Delta^{\Omega}_{g, J} -
  2\mathbbm{I}) \psi = \overline{\Delta^{\Omega}_{g, J} (\Delta^{\Omega}_{g,
  J} - 2\mathbbm{I}) \psi}  \right\} .
\end{eqnarray*}
With the notations introduced so far we can state the following decomposition
result.

\begin{lemma}
  \label{Decomp-SKah}Let $(J, g)$ be a K\"ahler-Ricci-Soliton and let $\Omega > 0$ be
  the unique smooth volume form such that $g J = \tmop{Ric}_J (\Omega)$ and
  $\int_X \Omega = 1$. Then the linear map
  \begin{eqnarray*}
    C_{\Omega}^{\infty} (X, \mathbbm{R})_0 \oplus \mathbbm{E}^{\Omega}_{g, J}
    \oplus \mathcal{H}_{g, \Omega}^{0, 1} \left( T_{X, J} \right) &
    \longrightarrow & \mathbbm{S}^J_{g, \Omega} (0)\\
    &  & \\
    (\varphi, \psi, A) & \longmapsto & v,\\
    &  & \\
    (v^{\ast}_g)_J^{1, 0} & \assign & \partial^g_{T_{X, J}} \nabla_g  (\varphi
    + \tau),\\
    &  & \\
    (v^{\ast}_g)_J^{0, 1} & \assign & \overline{\partial}_{T_{X, J}}
    \nabla_{g, J}  (\varphi + \overline{\psi}) + A,
  \end{eqnarray*}
  with $\tau \in C_{\Omega}^{\infty} (X, \mathbbm{R})_0$ the unique solution
  of the equation
  \begin{equation}
    \label{eq-div-free} - \overline{\Delta^{\Omega}_{g, J} \tau} =
    (\Delta^{\Omega}_{g, J} - 2\mathbbm{I}) \psi,
  \end{equation}
  is an isomorphism of real vector spaces. In particular the linear map
  \begin{eqnarray*}
    \mathbbm{E}^{\Omega}_{g, J} \oplus \mathcal{H}_{g, \Omega}^{0, 1} \left(
    T_{X, J} \right) & \longrightarrow & \mathbbm{S}^J_{g, \Omega} \left[ 0
    \right]\\
    &  & \\
    (\psi, A) & \longmapsto & v,\\
    &  & \\
    (v^{\ast}_g)_J^{0, 1} & \assign & \overline{\partial}_{T_{X, J}}
    \nabla_{g, J}  (\varphi + \overline{\psi}) + A,
  \end{eqnarray*}
  with $\varphi \in C_{\Omega}^{\infty} (X, \mathbbm{R})_0$ the unique
  solution of the equation
  \begin{equation}
    \label{eq-div-freeFi}  \overline{\Delta^{\Omega}_{g, J} \varphi} =
    (\Delta^{\Omega}_{g, J} - 2\mathbbm{I}) \psi,
  \end{equation}
  is also an isomorphism of real vector spaces.
\end{lemma}

\begin{proof}
  Let first $v \in \mathbbm{S}_{g, \Omega}$ and observe that the decomposition
  formula (\ref{dec-formul-var}) rewrites as
  \begin{eqnarray*}
    v^{\ast}_g & = & \partial^g_{T_{X, J}} \nabla_g \varphi +
    \overline{\partial}_{T_{X, J}} \nabla_g \varphi + \theta_g^{\ast} .
  \end{eqnarray*}
  This implies that $v \in \mathbbm{D}^J_g$ if and only if $\theta \in
  \mathbbm{D}^J_g$, and also $v \in \mathbbm{D}^J_{g, 0}$ if and only if
  $\theta \in \mathbbm{D}^J_{g, 0}$.
  
  Let now $v \in \mathbbm{S}^J_{g, \Omega} (0)$. Then the decomposition of
  the variation of the complex structure in corollary \ref{dec-VarJ} implies
  the existence of unique $\tau \in C_{\Omega}^{\infty} (X, \mathbbm{R})_0$,
  $\psi \in \Lambda^{\Omega, \bot}_{g, J}$ and $A \in \mathcal{H}_{g,
  \Omega}^{0, 1} \left( T_{X, J} \right)$ such that
  \begin{equation}
    \label{Cx-decTETA} \theta^{\ast}_g = \partial^g_{T_{X, J}} \nabla_g \tau +
    \overline{\partial}_{T_{X, J}} \nabla_{g, J}  \overline{\psi} + A .
  \end{equation}
  For bi-degree reasons the condition $\nabla^{\ast_{\Omega}}_g \theta = 0$ is
  equivalent to the identity
  \begin{eqnarray*}
    0 & = & 2 \partial^{\ast_{g, \Omega}}_{T_{X, J}} \partial^g_{T_{X, J}}
    \nabla_g \tau + 2 \overline{\partial}^{\ast_{g, \Omega}}_{T_{X, J}}
    \overline{\partial}_{T_{X, J}} \nabla_{g, J}  \overline{\psi} .
  \end{eqnarray*}
  The latter is equivalent to the equation
  \begin{eqnarray*}
    0 & = & \nabla_{g, J}  \left[ \Delta^{\Omega}_{g, J} \tau +
    \overline{(\Delta^{\Omega}_g - 2\mathbbm{I}) \psi}  \right],
  \end{eqnarray*}
  thanks to the complex Bochner identities (\ref{del-OmBoch-fnctSOL}) and
  (\ref{dbar-OmBoch-fnctSOL}). We remind that if $u \in C_{\Omega}^{\infty}
  (X, \mathbbm{C})_0$ satisfies $\nabla_{g, J} u = 0$ then $u = 0$. (See the
  proof of the injectivity statement in lemma \ref{iso-CxGrad} in the section
  \ref{CxBochSec}.) We conclude that the condition $\nabla^{\ast_{\Omega}}_g
  \theta = 0$ is equivalent to the equation (\ref{eq-div-free}) via the
  decomposition (\ref{Cx-decTETA}) of $\theta$.
  
  Then the required decomposition statement concerning the space
  $\mathbbm{S}^J_{g, \Omega} (0)$ follows from the fact that the condition
  $\tau$ real valued is equivalent to the equation defining $\psi \in
  \mathbbm{E}^{\Omega}_{g, J}$. \ In order to see this we show first the
  commutation identity
  \begin{equation}
    \label{Com-Lap-B}  \left[ \Delta^{\Omega}_g, B^{\Omega}_{g, J} \right] = 0
    .
  \end{equation}
  Indeed using an arbitrary $g$-orthonormal local frame $(e_k)_k$ we obtain
  \begin{eqnarray*}
    \Delta^{\Omega}_g B^{\Omega}_{g, J} u & = & \Delta^{\Omega}_g  \left[ g
    (\nabla_g u, J \nabla_g f) \right]\\
    &  & \\
    & = & g (\Delta^{\Omega}_g \nabla_g u, J \nabla_g f) - 2 g (\nabla^2_g u
    \cdot e_k, J \nabla^2_g f \cdot e_k) + g (\nabla_g u, J \Delta^{\Omega}_g
    \nabla_g f)\\
    &  & \\
    & = & g (\Delta^{\Omega}_g \nabla_g u + \nabla_g u, J \nabla_g f) - 2
    \tmop{Tr}_{\mathbbm{R}} \left( \nabla^2_g u J \nabla^2_g f \right)
  \end{eqnarray*}
  thanks to formula (\ref{com-grad-Lap}) applied to $f$ and thanks to the fact
  that $(\Delta^{\Omega}_g - 2\mathbbm{I}) f = 0$. Moreover the endomorphism
  $J \nabla^2_g f$ is $g$-anti-symmetric since in the soliton case $\left[ J,
  \nabla^2_g f \right] = 0$. We deduce
  \begin{eqnarray*}
    \Delta^{\Omega}_g B^{\Omega}_{g, J} u & = & g (\nabla_g \Delta^{\Omega}_g
    u, J \nabla_g f) = B^{\Omega}_{g, J} \Delta^{\Omega}_g u,
  \end{eqnarray*}
  thanks to formula (\ref{com-grad-Lap}) applied to $u$. We infer the identity
  (\ref{Com-Lap-B}) which implies
  \begin{equation}
    \label{com-CxLap-Conj}  \left[ \Delta^{\Omega}_{g, J},
    \overline{\Delta^{\Omega}_{g, J} }  \right] = 2 i \left[ B^{\Omega}_{g,
    J}, \Delta^{\Omega}_g \right] = 0 .
  \end{equation}
  Multiplying both sides of (\ref{eq-div-free}) with $\Delta^{\Omega}_{g, J}$
  we obtain
  \begin{equation}
    \label{conjug-sym} - \left( \Delta^{\Omega}_{g, J} 
    \overline{\Delta^{\Omega}_{g, J} }  \right) \tau = \Delta^{\Omega}_{g, J} 
    (\Delta^{\Omega}_{g, J} - 2\mathbbm{I}) \psi .
  \end{equation}
  The invertible operator
  \begin{eqnarray*}
    \Delta^{\Omega}_{g, J}  \overline{\Delta^{\Omega}_{g, J} } :
    C_{\Omega}^{\infty} (X, \mathbbm{C})_0 & \longrightarrow &
    C_{\Omega}^{\infty} (X, \mathbbm{C})_0,
  \end{eqnarray*}
  is real thanks to (\ref{com-CxLap-Conj}). We deduce that the condition
  $\tau$ real valued is equivalent to the left hand side of (\ref{conjug-sym})
  being real valued, thus equivalent to the equation defining $\psi \in
  \mathbbm{E}^{\Omega}_{g, J}$.
  
  We observe finally that a variation $v \in \mathbbm{S}^J_{g, \Omega} \left[
  0 \right] \subset \mathbbm{S}^J_{g, \Omega} (0)$ corresponds to $\varphi = -
  \tau$, i.e. to $(\varphi, \psi)$ solution of the equation
  (\ref{eq-div-freeFi}).
\end{proof}

{\tmstrong{Remark 1}}. If we write $\psi = \psi_1 + i \psi_2$, with $\psi_1,
\psi_2 \in C_{\Omega}^{\infty} (X, \mathbbm{R})_0$, then (\ref{eq-div-free})
is equivalent to the system
\begin{equation}
 \left\{ \begin{array}{l}
     - \Delta^{\Omega}_g \tau = (\Delta^{\Omega}_g - 2\mathbbm{I}) \psi_1 +
     B^{\Omega}_{g, J} \psi_2,\\
     \\
     - B^{\Omega}_{g, J} \tau = (\Delta^{\Omega}_g - 2\mathbbm{I}) \psi_2 -
     B^{\Omega}_{g, J} \psi_1 .\label{systemE}
   \end{array} \right.
\end{equation}
Moreover separating real and imaginary parts in the equation defining $\psi
\in \mathbbm{E}^{\Omega}_{g, J}$ and using the commutation identity
(\ref{Com-Lap-B}) we obtain
\begin{equation}
   \mathbbm{E}^{\Omega}_{g, J} = \left\{ \psi \in
  \Lambda^{\Omega, \bot}_{g, J} \mid \left[ \Delta^{\Omega}_g
  (\Delta^{\Omega}_g - 2\mathbbm{I}) - (B^{\Omega}_{g, J})^2 \right] \psi_2 =
  2 (\Delta^{\Omega}_g -\mathbbm{I}) B^{\Omega}_{g, J} \psi_1 \right\}.\label{Rdiv-Eig}
\end{equation}
Using (\ref{Rdiv-Eig}) and the complex Bochner formula
(\ref{dbar-OmBoch-fnctSOL}) we obtain also the identity
\[ \mathbbm{E}^{\Omega}_{g, J} = \left\{ \psi \in \Lambda^{\Omega, \bot}_{g,
   J} \mid - \tmop{div}^{\Omega} \overline{\partial}^{\ast_{g, \Omega}}_{T_{X,
   J}} \overline{\partial}_{T_{X, J}} \nabla_g \psi_2 = (\Delta^{\Omega}_g
   -\mathbbm{I}) B^{\Omega}_{g, J} \psi_1 \right\} . \]
{\tmstrong{Remark 2}}. We observe that the linear map
\begin{equation}
  \label{invol-Lap-B} \Delta^{\Omega}_{g, J} : \Lambda^{\Omega, \bot}_{g, J}
  \longrightarrow \Lambda^{\Omega, \bot}_{g, J},
\end{equation}
is well defined and it represents an isomorphism of complex vector spaces. In
fact this follows from the identity
\begin{eqnarray*}
  2 \int_X u\, \overline{v} \,\Omega & = & \int_X \Delta^{\Omega}_{g, J} u\,
  \overline{v} \,\Omega,
\end{eqnarray*}
for all $v \in \Lambda^{\Omega}_{g, J}$. Thus the linear map
\begin{equation}
  \label{invol-Lap-2I-B} \Delta^{\Omega}_{g, J} - 2\mathbbm{I}:
  \Lambda^{\Omega, \bot}_{g, J} \longrightarrow \Lambda^{\Omega, \bot}_{g, J},
\end{equation}
is also well defined and represents an isomorphisms of complex vector spaces.
The surjectivity of the latter follows from the finiteness theorem for elliptic
operators. By definition of $\mathbbm{E}^{\Omega}_{g, J}$ we deduce the
existence of the isomorphism of real vector spaces
\[ \Delta^{\Omega}_{g, J}  (\Delta^{\Omega}_{g, J} - 2\mathbbm{I}) :
   \mathbbm{E}^{\Omega}_{g, J} \longrightarrow \Lambda^{\Omega, \bot}_{g, J}
   \cap C_{\Omega}^{\infty} (X, \mathbbm{R})_0 . \]
We notice also the inclusion
\begin{eqnarray*}
  \Lambda^{\Omega, \bot}_{g, J} \cap C_{\Omega}^{\infty} (X, \mathbbm{R})_0 &
  \supseteq & (\Delta^{\Omega}_{g, J} - 2\mathbbm{I}) \overline{
  (\Delta^{\Omega}_{g, J} - 2\mathbbm{I})} C_{\Omega}^{\infty} (X,
  \mathbbm{R})_0 .
\end{eqnarray*}

\section{The geometric meaning of the space $\mathbbm{F}^J_{g, \Omega} \left[
0 \right]$}

We define the subspaces
\begin{eqnarray*}
  \mathbbm{F}^J_{g, \Omega} (0) & \assign & \left\{ (v, V) \in \mathbbm{F}_{g,
  \Omega} \mid v \in \mathbbm{S}^J_{g, \Omega} (0) \right\},\\
  &  & \\
  \mathbbm{F}^J_{g, \Omega} \left[ 0 \right] & \assign & \left\{ (v, V) \in
  \mathbbm{F}_{g, \Omega} \mid v \in \mathbbm{S}^J_{g, \Omega} \left[ 0
  \right] \right\} .
\end{eqnarray*}
In the previous section we gave a parametrization of the space
$\mathbbm{S}^J_{g, \Omega} \left( 0 \right)$, and thus of $\mathbbm{F}^J_{g,
\Omega} \left( 0 \right)$, which is fundamental for the computation of a
general second variation formula for the $\mathcal{W}$ functional at a
K\"ahler-Ricci soliton point. In this section we give a simpler
parametrization of the sub-space $\mathbbm{F}^J_{g, \Omega} \left[ 0 \right]$
and a useful geometric interpretation of it. We show first a quite general variation formula for the Chern-Ricci form.

\begin{lemma}
  \label{Lm-var-O-Rc-fm}Let $(g_t, J_t)_t \subset \mathcal{K}\mathcal{S}$,
  $(\Omega_t)_t \subset \mathcal{V}_1$ be two smooth families such that
  $\dot{J}_t = ( \dot{J}_t)_{g_t}^T$. Then hold the first variation formula
  \begin{equation}
    \label{vr-O-Rc-fm} 2 \frac{d}{d t} \tmop{Ric}_{J_t} (\Omega_t) = - d
    \left( g_t \nabla_{g_t}^{\ast_{\Omega_t}} \dot{J}_t + 2 d_{J_t}^c
    \dot{\Omega}_t^{\ast} \right) .
  \end{equation}
\end{lemma}

\begin{proof}
  In the case of a fixed volume form $\Omega > 0$ we have the variation
  formula (see \cite{Pal6})
  \begin{eqnarray*}
    2 \frac{d}{d t} \tmop{Ric}_{J_t} (\Omega) & = & - d \left( g_t
    \nabla_{g_t}^{\ast_{\Omega}} \dot{J}_t \right) .
  \end{eqnarray*}
  For an arbitrary family $(\Omega_t)_t \subset \mathcal{V}_1$ we fix an
  arbitrary time $\tau$ and we time derive at $t = \tau$ the decomposition
  \begin{eqnarray*}
    \tmop{Ric}_{J_t} (\Omega_t) & = & \tmop{Ric}_{J_t} (\Omega_{\tau}) - d
    d_{J_t}^c \log \frac{\Omega_t}{\Omega_{\tau}} .
  \end{eqnarray*}
  We obtain the required variation formula.
\end{proof}

We show now that for any point $(g, \Omega) \in \mathcal{S}_{\omega}$ hold the
inclusion (\ref{TConeS}). Indeed for any smooth curve $(g_t, \Omega_t)_t
\subset \mathcal{S}_{\omega}$, with $\left( g_0, \Omega_0 \right) = \left( g,
\Omega \right)$ we have $\dot{g}^{\ast}_t = - J_t \dot{J_t}$ and thus
\begin{eqnarray*}
  0 & = & 2 \frac{d}{d t} \tmop{Ric}_{J_t} (\Omega_t) = - d
  \left[ \left( \nabla_{g_t}^{\ast_{\Omega_t}} 
  \dot{g}_t^{\ast} + \nabla_{g_t}  \dot{\Omega}_t^{\ast} \right) \neg \,\omega
  \right],
\end{eqnarray*}
thanks to the variation formula (\ref{vr-O-Rc-fm}). Then the inclusion
(\ref{TConeS}) follows from (\ref{TConeM}) and Cartan's formula for the Lie derivative of differential forms. 

We can provide at this point the geometric interpretation of the sub-space $\mathbbm{F}^J_{g, \Omega} \left[ 0 \right]$.
\begin{lemma}
  \label{Orto-symplec}For any point $(g, \Omega) \in \mathcal{S}_{\omega}$
  hold the identities $($\ref{geom-F}$)$ and $($\ref{TConeS1}$)$.
\end{lemma}
\begin{proof}
  We remind that by the orthogonal decomposition in corollary \ref{dec-VarJ}
  any element $v \in \mathbbm{D}^J_{g, \left[ 0 \right]}$ decomposes as
  \begin{eqnarray*}
    v^{\ast}_g & = & \overline{\partial}_{T_{X, J}} \nabla_{g, J} 
    \overline{\psi}_v + A_v,
  \end{eqnarray*}
  with unique $\psi_v \in \Lambda^{\Omega, \bot}_{g, J}$ and $A_v \in
  \mathcal{H}_{g, \Omega}^{0, 1} \left( T_{X, J} \right)$. Moreover the
  weighted complex Bochner identity (\ref{dbar-OmBoch-fnctSOL}) implies the
  equality
  \begin{equation}
    \label{F0-express}  \overline{\partial}^{\ast_{g, \Omega}}_{T_{X, J}}
    v^{\ast}_g + \nabla_g V^{\ast}_{\Omega} = \frac{1}{2} \nabla_{g, J} \left[
    \overline{(\Delta^{\Omega}_{g, J} - 2\mathbbm{I}) \psi_v} + 2
    V^{\ast}_{\Omega} \right],
  \end{equation}
  for any $(v, V) \in \mathbbm{D}^J_{g, \left[ 0 \right]} \times
  T_{\mathcal{V}_1}$. Thus
  \begin{equation}
    \label{Cool-F0} \mathbbm{F}^J_{g, \Omega} \left[ 0 \right] = \left\{ (v,
    V) \in \mathbbm{D}^J_{g, \left[ 0 \right]} \times T_{\mathcal{V}_1} \mid
    (\Delta^{\Omega}_{g, J} - 2\mathbbm{I}) \psi_v = - 2 V^{\ast}_{\Omega}
    \right\} .
  \end{equation}
  Let
  \begin{eqnarray}
    R_{\psi} & \assign & \tmop{Re} \left[ (\Delta^{\Omega}_{g, J} -
    2\mathbbm{I}) \psi \right],\label{defRPSI}
    \\\nonumber
    &  & \\
    I_{\psi} & \assign & \tmop{Im} \left[ (\Delta^{\Omega}_{g, J} -
    2\mathbbm{I}) \psi \right],\label{defIPSI}
  \end{eqnarray}
  (for any $z \in \mathbbm{C}$ we write $z = \tmop{Re} z + i \tmop{Im} z$) and
  observe that (\ref{F0-express}) implies the identity
  \begin{eqnarray*}
    \left( \nabla_g^{\ast_{\Omega}} v_g^{\ast} + \nabla_g V^{\ast}_{\Omega}
    \right) \neg \,\omega & = & \frac{1}{2} d I_{\psi_v} + d_J^c \left(
    R_{\psi_v} + 2 V_{\Omega}^{\ast} \right),
  \end{eqnarray*}
  for any $(v, V) \in \mathbbm{D}^J_{g, \left[ 0 \right]} \times
  T_{\mathcal{V}_1}$. Thus
  \begin{equation}
    \label{inclusTS} \mathbbm{T}^J_{g, \Omega} = \left\{ (v, V) \in
    \mathbbm{D}^J_{g, \left[ 0 \right]} \times T_{\mathcal{V}_1}
    \mid_{_{_{_{_{}}}}} R_{\psi_v} = - 2 V^{\ast}_{\Omega} \right\} .
  \end{equation}
  We notice now the equalities
  \begin{eqnarray*}
    T_{\left[ g, \Omega \right]_{\omega}, (g, \Omega)} & = & \left\{ (L_{\xi}
    g, L_{\xi} \Omega) \mid \xi \in C^{\infty} (X, T_X) : L_{\xi} \omega = 0
    \right\}\\
    &  & \\
    & = & \left\{ (L_{J \nabla_g u} g, L_{J \nabla_g u} \Omega) \mid u \in
    C_{\Omega}^{\infty} (X, \mathbbm{R})_0 \right\}\\
    &  & \\
    & = & \left\{ (2 g J \overline{\partial}_{T_{X, J}} \nabla_g u,
    \tmop{div}^{\Omega} (J \nabla_g u) \Omega) \mid u \in C_{\Omega}^{\infty}
    (X, \mathbbm{R})_0 \right\},
  \end{eqnarray*}
  indeed
  \begin{eqnarray*}
    (L_{J \nabla_g u} g)_g^{\ast} & = & J \nabla^2_g u - \nabla^2_g u J = 2 J
    \overline{\partial}_{T_{X, J}} \nabla_g u .
  \end{eqnarray*}
  We deduce that $(v, V) \in T^{\bot_G}_{\left[ g, \Omega \right]_{\omega},
  (g, \Omega)}$ if and only if for all $u \in C_{\Omega}^{\infty} (X,
  \mathbbm{R})_0$ hold the equalities
  \begin{eqnarray*}
    0 & = & 2 \int_X \left[ \left\langle J \overline{\partial}_{T_{X, J}}
    \nabla_g u, v_g^{\ast} \right\rangle_g - \tmop{div}^{\Omega} (J \nabla_g
    u) \cdot V^{\ast}_{\Omega} \right] \Omega\\
    &  & \\
    & = & - 2 \int_X \left\langle \nabla_g u, J \left(
    \overline{\partial}^{\ast_{g, \Omega}}_{T_{X, J}} v^{\ast}_g + \nabla_g
    V^{\ast}_{\Omega} \right) \right\rangle_g \Omega\\
    &  & \\
    & = & 2 \int_X u \cdot \tmop{div}^{\Omega} \left[ J \left(
    \overline{\partial}^{\ast_{g, \Omega}}_{T_{X, J}} v^{\ast}_g + \nabla_g
    V^{\ast}_{\Omega} \right) \right] \Omega .
  \end{eqnarray*}
  If we assume $(v, V) \in \mathbbm{T}^J_{g, \Omega}$ then
  \begin{equation}
    \overline{\partial}^{\ast_{g, \Omega}}_{T_{X, J}} v^{\ast}_g + \nabla_g
    V^{\ast}_{\Omega} = - \frac{1}{2} J \nabla_g I_{\psi_v}  \label{FexpT},
  \end{equation}
  thanks to (\ref{F0-express}) and (\ref{inclusTS}). Thus if $(v, V) \in
  T^{\bot_G}_{\left[ g, \Omega \right]_{\omega}, (g, \Omega)} \cap
  \mathbbm{T}^J_{g, \Omega}$ then
  \begin{eqnarray*}
    0 & = & - \int_X u \cdot \Delta^{\Omega}_g I_{\psi_v} \Omega,
  \end{eqnarray*}
  for all $u \in C_{\Omega}^{\infty} (X, \mathbbm{R})_0$, i.e.
  $\Delta^{\Omega}_g I_{\psi_v} = 0$, which is equivalent to the condition
  $I_{\psi_v} = 0$. We infer
  \begin{eqnarray*}
    T^{\bot_G}_{\left[ g, \Omega \right]_{\omega}, (g, \Omega)} \cap
    \mathbbm{T}^J_{g, \Omega} & \subseteq & \mathbbm{F}^J_{g, \Omega} \left[ 0
    \right] .
  \end{eqnarray*}
  The reverse inclusion is obvious. We deduce the identity (\ref{geom-F}).
  Then the identity (\ref{TConeS1}) follows from the inclusion (\ref{TConeS}).
\end{proof}

\section{The sign of the second variation of the $\mathcal{W}$ functional at a
K\"ahler-Ricci soliton point}\label{Sign-Sect}

\begin{proposition}
  \label{Sec-VarWKRS}Let $(X,J, g)$ be a compact K\"ahler-Ricci-Soliton and let $\Omega > 0$ be
  the unique smooth volume form such that $g J = \tmop{Ric}_J (\Omega)$ and
  $\int_X \Omega = 1$. Let also $(g_t, \Omega_t)_{t \in \mathbbm{R}} \subset
  \mathcal{M} \times \mathcal{V}_1$ be a smooth curve with $(g_0, \Omega_0) =
  (g, \Omega)$ and with $( \dot{g}_0, \dot{\Omega}_0) = (v, V) \in
  \mathbbm{F}^J_{g, \Omega} (0)$. Then with the notations of lemma
  \ref{Decomp-SKah} hold the second variation formula
  \begin{eqnarray*}
    \frac{d^2}{d t^2} _{\mid_{t = 0}} \mathcal{W} (g_t, \Omega_t) & = &
    \nabla_G D\mathcal{W} (g, \Omega) (v, V ; v, V)\\
    &  & \\
    & = & \frac{1}{2} \int_X (\Delta^{\Omega}_g -\mathbbm{I})
    (\Delta^{\Omega}_g - 2\mathbbm{I}) \varphi \cdot (\Delta^{\Omega}_g -
    2\mathbbm{I}) \varphi \Omega\\
    &  & \\
    & - & \frac{1}{2} \int_X \left[ (\Delta^{\Omega}_g
    -\mathbbm{I})_{_{_{_{_{}}}}} \Delta^{\Omega}_g \tau \cdot
    \Delta^{\Omega}_g \tau + P^{\Omega}_{g, J} \tmop{Im} \psi \cdot \tmop{Im}
    \psi + \left| A \right|^2_g F \right] \Omega,
  \end{eqnarray*}
  where
  \begin{eqnarray*}
    P^{\Omega}_{g, J} & \assign & (\Delta^{\Omega}_{g, J} - 2\mathbbm{I})
    \overline{ (\Delta^{\Omega}_{g, J} - 2\mathbbm{I})},
  \end{eqnarray*}
  is a non-negative self-adjoint real elliptic operator with respect to the
  $L_{\Omega}^2$-hermitian product. In particular if $(v, V) \in
  \mathbbm{F}^J_{g, \Omega} \left[ 0 \right]$ then
  \begin{eqnarray*}
    \frac{d^2}{d t^2} _{\mid_{t = 0}} \mathcal{W} (g_t, \Omega_t) & = &
    \nabla_G D\mathcal{W} (g, \Omega) (v, V ; v, V)\\
    &  & \\
    & = & - \frac{1}{2} \int_X \left[ 4 \left| (\Delta^{\Omega}_g
    -\mathbbm{I}) \varphi \right|^2 + P^{\Omega}_{g, J} \tmop{Im} \psi \cdot
    \tmop{Im} \psi + \left| A \right|^2_g F_{_{_{_{_{_{}}}}}} \right] \Omega .
  \end{eqnarray*}
\end{proposition}

\begin{proof}
  {\tmstrong{STEP I}}. Let $(X,g, \Omega)$ be a compact shrinking Ricci soliton point
  and let $(g_t, \Omega_t)_{t \in \mathbbm{R}} \subset \mathcal{M} \times
  \mathcal{V}_1$ be a smooth curve with $(g_0, \Omega_0) = (g, \Omega)$ and
  with arbitrary speed $( \dot{g}_0, \dot{\Omega}_0) = (v, V) \in
  \mathbbm{F}_{g, \Omega}$. We know from lemma \ref{Sec-Var-W}
  \begin{eqnarray*}
    \frac{d^2}{d t^2} _{\mid_{t = 0}} \mathcal{W} (g_t, \Omega_t) & = &
    \nabla_G D\mathcal{W} (g, \Omega) (v, V ; v, V)\\
    &  & \\
    & = & - \frac{1}{2} \int_X \left[ \left\langle \mathcal{L}^{\Omega}_g v,
    v \right\rangle_g - 2 (\Delta^{\Omega}_g - 2\mathbbm{I}) V_{\Omega}^{\ast}
    \cdot V^{\ast}_{\Omega} \right] \Omega .
  \end{eqnarray*}
  By the considerations in the beginning of section \ref{inv-F} we deduce that
  in the soliton case $h_{g, \Omega} = 0$ holds the identity
  \begin{equation}
    \label{inv-F-sol} \nabla^{\ast_{\Omega}}_g \mathcal{L}^{\Omega}_g v + d
    (\Delta^{\Omega}_g V_{\Omega}^{\ast} - 2 V_{\Omega}^{\ast}) = 0,
  \end{equation}
  for all $(v, V) \in \mathbbm{F}_{g, \Omega}$. Applying the operator
  $\nabla^{\ast_{\Omega}}_g$ to both sides of this identity we infer
  \begin{equation}
    \label{fund-ID}  (\nabla^{\ast_{\Omega}}_g)^2 \mathcal{L}^{\Omega}_g v +
    \Delta^{\Omega}_g (\Delta^{\Omega}_g - 2\mathbbm{I}) V_{\Omega}^{\ast} = 0
    .
  \end{equation}
  For any function $\varphi \in C_{\Omega}^{\infty} (X, \mathbbm{R})_0$ let
  $(v, V) \assign T_{g, \Omega}  (\varphi, 0)$. Integrating by parts and using
  the identity (\ref{fund-ID}) we infer the equalities
  \begin{eqnarray*}
    &  & \nabla_G D\mathcal{W} (g, \Omega) (v, V ; v, V)\\
    &  & \\
    & = & - \frac{1}{2} \int_X \left[ (\nabla^{\ast_{\Omega}}_g)^2
    \mathcal{L}^{\Omega}_g v \cdot \varphi + 2 (\Delta^{\Omega}_g -
    2\mathbbm{I}) (\Delta^{\Omega}_g -\mathbbm{I}) \varphi \cdot (\mathbbm{I}-
    \Delta^{\Omega}_g) \varphi \right] \Omega\\
    &  & \\
    & = & - \frac{1}{2} \int_X \Delta^{\Omega}_g (\Delta^{\Omega}_g -
    2\mathbbm{I}) (\Delta^{\Omega}_g -\mathbbm{I}) \varphi \cdot \varphi
    \Omega\\
    &  & \\
    & - & \frac{1}{2} \int_X 2 (\Delta^{\Omega}_g - 2\mathbbm{I})
    (\Delta^{\Omega}_g -\mathbbm{I}) \varphi \cdot (\mathbbm{I}-
    \Delta^{\Omega}_g) \varphi \Omega\\
    &  & \\
    & = & \frac{1}{2} \int_X (\Delta^{\Omega}_g - 2\mathbbm{I})
    (\Delta^{\Omega}_g -\mathbbm{I}) \varphi \cdot (\Delta^{\Omega}_g -
    2\mathbbm{I}) \varphi \Omega\\
    &  & \\
    & = & \frac{1}{2} \int_X (\Delta^{\Omega}_g -\mathbbm{I})
    (\Delta^{\Omega}_g - 2\mathbbm{I}) \varphi \cdot (\Delta^{\Omega}_g -
    2\mathbbm{I}) \varphi \Omega .
  \end{eqnarray*}
  {\tmstrong{Remark 1}}. We can also compute the integral
  \[ \int_X \left\langle \mathcal{L}^{\Omega}_g \nabla^2_g \varphi, \nabla^2_g
     \varphi \right\rangle_g \Omega, \]
  in the previous expansion via the formula (\ref{com-Lich-Hess}). Indeed
  \begin{eqnarray*}
    \int_X \left\langle \mathcal{L}^{\Omega}_g \nabla^2_g \varphi, \nabla^2_g
    \varphi \right\rangle_g \Omega & = & \int_X \left\langle \nabla^2_g
    (\Delta^{\Omega}_g - 2\mathbbm{I}) \varphi, \nabla^2_g \varphi
    \right\rangle_g \Omega\\
    &  & \\
    & = & \int_X \left\langle \Delta^{\Omega}_g \nabla_g (\Delta^{\Omega}_g -
    2\mathbbm{I}) \varphi, \nabla_g \varphi \right\rangle_g \Omega\\
    &  & \\
    & = & \int_X \left\langle \nabla_g (\Delta^{\Omega}_g -\mathbbm{I})
    (\Delta^{\Omega}_g - 2\mathbbm{I}) \varphi, \nabla_g \varphi
    \right\rangle_g \Omega,
  \end{eqnarray*}
  thanks to the identity (\ref{com-grad-Lap}). We conclude integrating by
  parts
  \begin{eqnarray*}
    \int_X \left\langle \mathcal{L}^{\Omega}_g \nabla^2_g \varphi, \nabla^2_g
    \varphi \right\rangle_g \Omega & = & \int_X \Delta^{\Omega}_g
    (\Delta^{\Omega}_g - 2\mathbbm{I}) (\Delta^{\Omega}_g -\mathbbm{I})
    \varphi \cdot \varphi \Omega .
  \end{eqnarray*}
  {\tmstrong{Remark 2}}. We set $\Phi \assign (\Delta^{\Omega}_g -
  2\mathbbm{I}) \varphi \in C_{\Omega}^{\infty} (X, \mathbbm{R})_0$. Then the
  previous variation formula rewrites also as
  \begin{eqnarray*}
    \frac{d^2}{d t^2} _{\mid_{t = 0}} \mathcal{W} (g_t, \Omega_t) & = &
    \nabla_G D\mathcal{W} (g, \Omega) (v, V ; v, V)\\
    &  & \\
    & = & \frac{1}{2} \int_X \left[ \left| \nabla_g \Phi \right|^2_g - \Phi^2
    \right] \Omega \geqslant 0 .
  \end{eqnarray*}
  the last inequality follows from the variational characterization of the
  first eigenvalue of $\Delta^{\Omega}_g$,
  \begin{eqnarray*}
    \lambda_1 (\Delta^{\Omega}_g) & = & \inf \left\{  \frac{\int_X \left|
    \nabla_g u \right|^2_g \Omega}{\int_X u^2 \Omega} \mid u \in
    C_{\Omega}^{\infty} (X, \mathbbm{R})_0 \smallsetminus \{0\} \right\},
  \end{eqnarray*}
  which satisfies the inequality $\lambda_1 (\Delta^{\Omega}_g) > 1$.
  
  {\tmstrong{STEP II}}. Let $(v, V) \in \mathbbm{F}_{g, \Omega}$. Using the
  $L^2$-orthogonal decomposition (\ref{dec-formul-var}) in the proof of lemma
  \ref{decomp-F}, we expand the integral term
  \begin{eqnarray*}
    \int_X \left\langle \mathcal{L}^{\Omega}_g v, v \right\rangle_g \Omega & =
    & \int_X \left[ \left\langle \mathcal{L}^{\Omega}_g \nabla_g d \varphi
    +\mathcal{L}^{\Omega}_g \theta, \nabla_g d \varphi \right\rangle_g +
    \left\langle \mathcal{L}^{\Omega}_g \nabla_g d \varphi
    +\mathcal{L}^{\Omega}_g \theta, \theta \right\rangle_g \right] \Omega .
  \end{eqnarray*}
  We observe that
  \begin{eqnarray*}
    \int_X \left\langle \mathcal{L}^{\Omega}_g \theta, \nabla_g d \varphi
    \right\rangle_g \Omega & = & \int_X \left\langle \nabla^{\ast_{\Omega}}_g
    \mathcal{L}^{\Omega}_g \theta, d \varphi \right\rangle_g \Omega = 0,
  \end{eqnarray*}
  since $\nabla^{\ast_{\Omega}}_g \mathcal{L}^{\Omega}_g \theta = 0$ thanks to
  the identity (\ref{inv-F-sol}) applied to $(\theta, 0) \in \mathbbm{F}_{g,
  \Omega}$. On the other hand formula (\ref{com-Lich-Hess}) implies
  \begin{eqnarray*}
    \int_X \left\langle \mathcal{L}^{\Omega}_g \nabla_g d \varphi, \theta
    \right\rangle_g \Omega & = & \int_X \left\langle d (\Delta^{\Omega}_g -
    2\mathbbm{I}) \varphi, \nabla^{\ast_{\Omega}}_g \theta \right\rangle_g
    \Omega = 0 .
  \end{eqnarray*}
  We conclude the decomposition identity
  \begin{eqnarray*}
    \int_X \left\langle \mathcal{L}^{\Omega}_g v, v \right\rangle_g \Omega & =
    & \int_X \left[ \left\langle \mathcal{L}^{\Omega}_g \nabla_g d \varphi,
    \nabla_g d \varphi \right\rangle_g + \left\langle \mathcal{L}^{\Omega}_g
    \theta, \theta \right\rangle_g \right] \Omega .
  \end{eqnarray*}
  Then step I implies
  \begin{eqnarray*}
    \frac{d^2}{d t^2} _{\mid_{t = 0}} \mathcal{W} (g_t, \Omega_t) & = &
    \nabla_G D\mathcal{W} (g, \Omega) (v, V ; v, V)\\
    &  & \\
    & = & \frac{1}{2} \int_X (\Delta^{\Omega}_g -\mathbbm{I})
    (\Delta^{\Omega}_g - 2\mathbbm{I}) \varphi \cdot (\Delta^{\Omega}_g -
    2\mathbbm{I}) \varphi \Omega\\
    &  & \\
    & - & \frac{1}{2} \int_X \left\langle \mathcal{L}^{\Omega}_g \theta,
    \theta \right\rangle_g \Omega .
  \end{eqnarray*}
  On the other hand using the decomposition (\ref{Cx-decTETA}) of $\theta$ and
  the decomposition formula (\ref{lin-ant-lin-Integ-dec}) we infer
  \begin{eqnarray*}
    \int_X \left\langle \mathcal{L}^{\Omega}_g \theta, \theta \right\rangle_g
    \Omega & = & \int_X \left\langle \mathcal{L}^{\Omega}_g \partial^g_{T_{X,
    J}} \nabla_g \tau, \partial^g_{T_{X, J}} \nabla_g \tau \right\rangle_g
    \Omega\\
    &  & \\
    & + & \int_X \left\langle \mathcal{L}^{\Omega}_g
    \overline{\partial}_{T_{X, J}} \nabla_{g, J}  \overline{\psi},
    \overline{\partial}_{T_{X, J}} \nabla_{g, J}  \overline{\psi} + A
    \right\rangle_g \Omega\\
    &  & \\
    & + & \int_X \left\langle \mathcal{L}^{\Omega}_g A,
    \overline{\partial}_{T_{X, J}} \nabla_{g, J}  \overline{\psi} + A
    \right\rangle_g \Omega .
  \end{eqnarray*}
  Using the identities (\ref{com-Lich-HessIO}), (\ref{com-Lich-HessOI}) and
  the property (\ref{Jcurv}) we deduce
  \begin{eqnarray*}
    \int_X \left\langle \mathcal{L}^{\Omega}_g \theta, \theta \right\rangle_g
    \Omega & = & \int_X \left\langle \partial^g_{T_{X, J}} \nabla_g
    (\Delta^{\Omega}_g - 2\mathbbm{I}) \tau, \partial^g_{T_{X, J}} \nabla_g
    \tau \right\rangle_g \Omega\\
    &  & \\
    & + & \int_X \left\langle \overline{\partial}_{T_{X, J}} \nabla_{g, J}
    (\Delta^{\Omega}_g - 2\mathbbm{I}) \overline{\psi},
    \overline{\partial}_{T_{X, J}} \nabla_{g, J}  \overline{\psi} + A
    \right\rangle_g \Omega\\
    &  & \\
    & + & \int_X \left[ \left\langle \overline{\partial}^{\ast_{g,
    \Omega}}_{T_{X, J}} \mathcal{L}^{\Omega}_g A, \nabla_{g, J} 
    \overline{\psi}  \right\rangle_g + \left\langle \mathcal{L}^{\Omega}_g A,
    A \right\rangle_g \right] \Omega .
  \end{eqnarray*}
  By bi-degree reasons $\nabla^{\ast_{\Omega}}_g A = 0$, which means $(g A, 0)
  \in \mathbbm{F}_{g, \Omega}$. We infer $\nabla^{\ast_{\Omega}}_g
  \mathcal{L}^{\Omega}_g A = 0$ thanks to the identity (\ref{inv-F-sol}). Then
  the property (\ref{ant-lin-Lich}) implies
  \begin{eqnarray}
    \overline{\partial}^{\ast_{g, \Omega}}_{T_{X, J}} \mathcal{L}^{\Omega}_g A
    & = & 0, \label{divLA} 
  \end{eqnarray}
  by bi-degree reasons. Integrating by parts further and using the weighted
  complex Bochner identities (\ref{del-OmBoch-fnctSOL}),
  (\ref{dbar-OmBoch-fnctSOL}) we obtain
  \begin{eqnarray*}
    \int_X \left\langle \mathcal{L}^{\Omega}_g \theta, \theta \right\rangle_g
    \Omega & = & \int_X \left\langle \nabla_g (\Delta^{\Omega}_g -
    2\mathbbm{I}) \tau, \partial^{\ast_{g, \Omega}}_{T_{X, J}}
    \partial^g_{T_{X, J}} \nabla_g \tau \right\rangle_g \Omega\\
    &  & \\
    & + & \int_X \left\langle \nabla_{g, J} (\Delta^{\Omega}_g -
    2\mathbbm{I}) \overline{\psi}, \overline{\partial}^{\ast_{g,
    \Omega}}_{T_{X, J}} \overline{\partial}_{T_{X, J}} \nabla_{g, J} 
    \overline{\psi}  \right\rangle_g \Omega\\
    &  & \\
    & + & \int_X \left\langle \mathcal{L}^{\Omega}_g A, A \right\rangle_g
    \Omega\\
    &  & \\
    & = & \frac{1}{2} \int_X \left\langle \nabla_g (\Delta^{\Omega}_g -
    2\mathbbm{I}) \tau, \nabla_{g, J} \Delta^{\Omega}_{g, J} \tau
    \right\rangle_g \Omega\\
    &  & \\
    & + & \frac{1}{2} \int_X \left\langle \nabla_{g, J} (\Delta^{\Omega}_g -
    2\mathbbm{I}) \overline{\psi}, \nabla_{g, J}
    \overline{(\Delta^{\Omega}_{g, J} - 2\mathbbm{I}) \psi}  \right\rangle_g
    \Omega\\
    &  & \\
    & + & \int_X \left\langle \mathcal{L}^{\Omega}_g A, A \right\rangle_g
    \Omega .
  \end{eqnarray*}
  Using the integration by parts formulas (\ref{rlcpx-int-part}) and
  (\ref{cpx-int-part}) in the subsection \ref{Integ-parts} of the appendix A, we
  deduce
  \begin{eqnarray*}
    \int_X \left\langle \mathcal{L}^{\Omega}_g \theta, \theta \right\rangle_g
    \Omega & = & \frac{1}{4} \int_X \Delta^{\Omega}_{g, J}  (\Delta^{\Omega}_g
    - 2\mathbbm{I}) \tau \cdot \Delta^{\Omega}_{g, J} \tau \Omega\\
    &  & \\
    & + & \frac{1}{4} \int_X \overline{\Delta^{\Omega}_{g, J} 
    (\Delta^{\Omega}_g - 2\mathbbm{I}) \tau} \cdot
    \overline{\Delta^{\Omega}_{g, J} \tau} \Omega\\
    &  & \\
    & + & \frac{1}{4} \int_X \Delta^{\Omega}_{g, J}  (\Delta^{\Omega}_g -
    2\mathbbm{I}) \psi \cdot \overline{ (\Delta^{\Omega}_{g, J} -
    2\mathbbm{I}) \psi} \Omega\\
    &  & \\
    & + & \frac{1}{4} \int_X \overline{\Delta^{\Omega}_{g, J} 
    (\Delta^{\Omega}_g - 2\mathbbm{I}) \psi} \cdot (\Delta^{\Omega}_{g, J} -
    2\mathbbm{I}) \psi \Omega\\
    &  & \\
    & + & \int_X \left\langle \mathcal{L}^{\Omega}_g A, A \right\rangle_g
    \Omega .
  \end{eqnarray*}
  We observe now that the commutation identity (\ref{Com-Lap-B}) implies
  \begin{eqnarray*}
    \int_X \left\langle \mathcal{L}^{\Omega}_g \theta, \theta \right\rangle_g
    \Omega & = & \frac{1}{4} \int_X (\Delta^{\Omega}_g - 2\mathbbm{I})
    \Delta^{\Omega}_{g, J} \tau \cdot \Delta^{\Omega}_{g, J} \tau \Omega\\
    &  & \\
    & + & \frac{1}{4} \int_X (\Delta^{\Omega}_g - 2\mathbbm{I})
    \overline{\Delta^{\Omega}_{g, J} \tau} \cdot \overline{\Delta^{\Omega}_{g,
    J} \tau} \Omega\\
    &  & \\
    & + & \frac{1}{4} \int_X (\Delta^{\Omega}_g - 2\mathbbm{I})
    \Delta^{\Omega}_{g, J} \psi \cdot \overline{ (\Delta^{\Omega}_{g, J} -
    2\mathbbm{I}) \psi} \Omega\\
    &  & \\
    & + & \frac{1}{4} \int_X (\Delta^{\Omega}_g - 2\mathbbm{I})
    \overline{\Delta^{\Omega}_{g, J} \psi} \cdot (\Delta^{\Omega}_{g, J} -
    2\mathbbm{I}) \psi \Omega\\
    &  & \\
    & + & \int_X \left\langle \mathcal{L}^{\Omega}_g A, A \right\rangle_g
    \Omega .
  \end{eqnarray*}
  Completing the square we obtain
  \begin{eqnarray*}
    \int_X \left\langle \mathcal{L}^{\Omega}_g \theta, \theta \right\rangle_g
    \Omega & = & \frac{1}{4} \int_X (\Delta^{\Omega}_g - 2\mathbbm{I})
    \Delta^{\Omega}_{g, J} \tau \cdot \Delta^{\Omega}_{g, J} \tau \Omega\\
    &  & \\
    & + & \frac{1}{4} \int_X (\Delta^{\Omega}_g - 2\mathbbm{I})
    \overline{\Delta^{\Omega}_{g, J} \tau} \cdot \overline{\Delta^{\Omega}_{g,
    J} \tau} \Omega\\
    &  & \\
    & + & \frac{1}{4} \int_X (\Delta^{\Omega}_g - 2\mathbbm{I}) 
    (\Delta^{\Omega}_{g, J} - 2\mathbbm{I}) \psi \cdot \overline{
    (\Delta^{\Omega}_{g, J} - 2\mathbbm{I}) \psi} \Omega\\
    &  & \\
    & + & \frac{1}{4} \int_X (\Delta^{\Omega}_g - 2\mathbbm{I}) \overline{
    (\Delta^{\Omega}_{g, J} - 2\mathbbm{I}) \psi} \cdot (\Delta^{\Omega}_{g,
    J} - 2\mathbbm{I}) \psi \Omega\\
    &  & \\
    & + & \frac{1}{2} \int_X (\Delta^{\Omega}_g - 2\mathbbm{I}) \psi \cdot
    \overline{ (\Delta^{\Omega}_{g, J} - 2\mathbbm{I}) \psi} \Omega\\
    &  & \\
    & + & \frac{1}{2} \int_X (\Delta^{\Omega}_g - 2\mathbbm{I})
    \overline{\psi} \cdot (\Delta^{\Omega}_{g, J} - 2\mathbbm{I}) \psi
    \Omega\\
    &  & \\
    & + & \int_X \left\langle \mathcal{L}^{\Omega}_g A, A \right\rangle_g
    \Omega .
  \end{eqnarray*}
  Using the equation (\ref{eq-div-free}) we infer
  \begin{eqnarray*}
    \int_X \left\langle \mathcal{L}^{\Omega}_g \theta, \theta \right\rangle_g
    \Omega & = & \frac{1}{4} \int_X (\Delta^{\Omega}_g - 2\mathbbm{I})  \left(
    \Delta^{\Omega}_{g, J} \tau + \overline{\Delta^{\Omega}_{g, J} \tau} 
    \right) \cdot \left( \Delta^{\Omega}_{g, J} \tau +
    \overline{\Delta^{\Omega}_{g, J} \tau}  \right) \Omega\\
    &  & \\
    & + & \int_X | (\Delta^{\Omega}_{g, J} - 2\mathbbm{I}) \psi |^2 \Omega\\
    &  & \\
    & + & \frac{i}{2} \int_X \left[ B^{\Omega}_{g, J} \psi \cdot
    \overline{(\Delta^{\Omega}_{g, J} - 2\mathbbm{I}) \psi} - B^{\Omega}_{g,
    J}  \overline{\psi} \cdot (\Delta^{\Omega}_{g, J} - 2\mathbbm{I}) \psi
    \right] \Omega\\
    &  & \\
    & + & \int_X \left\langle \mathcal{L}^{\Omega}_g A, A \right\rangle_g
    \Omega\\
    &  & \\
    & = & \int_X \left[ (\Delta^{\Omega}_g - 2\mathbbm{I}) \Delta^{\Omega}_g
    \tau \cdot \Delta^{\Omega}_g \tau + | \Delta^{\Omega}_{g, J} \tau |^2 +
    \left\langle \mathcal{L}^{\Omega}_g A, A \right\rangle_g \right] \Omega\\
    &  & \\
    & - & \frac{i}{2} \int_X \left[ B^{\Omega}_{g, J} \psi \cdot
    \Delta^{\Omega}_{g, J} \tau - B^{\Omega}_{g, J}  \overline{\psi} \cdot
    \overline{\Delta^{\Omega}_{g, J} \tau}  \right] \Omega .
  \end{eqnarray*}
  We observe now that the operator $B^{\Omega}_{g, J}$ is
  $L_{\Omega}^2$-anti-adjoint. This implies in particular
  \begin{eqnarray*}
    \int_X \Delta^{\Omega}_g \tau \cdot B^{\Omega}_{g, J} \tau \; \Omega & = &
    0,
  \end{eqnarray*}
  and
  \begin{eqnarray*}
    & - & \frac{i}{2} \int_X \left[ B^{\Omega}_{g, J} \psi \cdot
    \Delta^{\Omega}_{g, J} \tau - B^{\Omega}_{g, J}  \overline{\psi} \cdot
    \overline{\Delta^{\Omega}_{g, J} \tau}  \right] \Omega\\
    &  & \\
    & = & \frac{i}{2} \int_X \left[ \psi \cdot \Delta^{\Omega}_{g, J}
    B^{\Omega}_{g, J} \tau - \overline{\psi} \cdot
    \overline{\Delta^{\Omega}_{g, J} B^{\Omega}_{g, J} \tau}  \right] \Omega\\
    &  & \\
    & = & \frac{i}{2} \int_X \left[  \overline{\Delta^{\Omega}_{g, J} \psi} -
    \Delta^{\Omega}_{g, J}  \overline{\psi}  \right] B^{\Omega}_{g, J} \tau
    \Omega\\
    &  & \\
    & = & - \int_X \left( \Delta^{\Omega}_g \psi_2 + B^{\Omega}_{g, J} \psi_1
    \right) B^{\Omega}_{g, J} \tau \Omega .
  \end{eqnarray*}
  thanks to the commutation identity (\ref{Com-Lap-B}). Thus
  \begin{eqnarray*}
    \int_X \left\langle \mathcal{L}^{\Omega}_g \theta, \theta \right\rangle_g
    \Omega & = & \int_X \left[ (\Delta^{\Omega}_g -\mathbbm{I})
    \Delta^{\Omega}_g \tau \cdot \Delta^{\Omega}_g \tau + |B^{\Omega}_{g, J}
    \tau |^2 + \left\langle \mathcal{L}^{\Omega}_g A, A \right\rangle_g
    \right] \Omega\\
    &  & \\
    & - & \int_X \left( \Delta^{\Omega}_g \psi_2 + B^{\Omega}_{g, J} \psi_1
    \right) B^{\Omega}_{g, J} \tau \Omega\\
    &  & \\
    & = & \int_X \left[ (\Delta^{\Omega}_g -\mathbbm{I}) \Delta^{\Omega}_g
    \tau \cdot \Delta^{\Omega}_g \tau + \left\langle \mathcal{L}^{\Omega}_g A,
    A \right\rangle_g \right] \Omega\\
    &  & \\
    & - & 2 \int_X (\Delta^{\Omega}_g -\mathbbm{I}) \psi_2 \cdot
    B^{\Omega}_{g, J} \tau \Omega,
  \end{eqnarray*}
  thanks to the second equation in (\ref{systemE}). Using again the second equation in
  (\ref{systemE}) we expand the term
  \begin{eqnarray*}
    - 2 \int_X (\Delta^{\Omega}_g -\mathbbm{I}) \psi_2 \cdot B^{\Omega}_{g, J}
    \tau \Omega & = & 2 \int_X (\Delta^{\Omega}_g -\mathbbm{I}) \psi_2 \cdot
    (\Delta^{\Omega}_g - 2\mathbbm{I}) \psi_2 \Omega\\
    &  & \\
    & - & 2 \int_X (\Delta^{\Omega}_g -\mathbbm{I}) \psi_2 \cdot
    B^{\Omega}_{g, J} \psi_1 \Omega\\
    &  & \\
    & = & \int_X \left[ \left| (\Delta^{\Omega}_g - 2\mathbbm{I}) \psi_2
    \right|^2 + \Delta^{\Omega}_g \psi_2 \cdot (\Delta^{\Omega}_g -
    2\mathbbm{I}) \psi_2 \right] \Omega\\
    &  & \\
    & - & 2 \int_X \psi_2 \cdot (\Delta^{\Omega}_g -\mathbbm{I})
    B^{\Omega}_{g, J} \psi_1 \Omega\\
    &  & \\
    & = & \int_X \left[ (\Delta^{\Omega}_g - 2\mathbbm{I})^2 +
    (B^{\Omega}_{g, J})^2 \right] \psi_2 \cdot \psi_2 \Omega,
  \end{eqnarray*}
  thanks to the expression (\ref{Rdiv-Eig}). We observe further that the
  formula
  \begin{eqnarray}
    P^{\Omega}_{g, J} & = & (\Delta^{\Omega}_g - 2\mathbbm{I})^2 +
    (B^{\Omega}_{g, J})^2, \label{PalOP} 
  \end{eqnarray}
  hold thanks to the commutation identity (\ref{Com-Lap-B}). We conclude
  \begin{eqnarray*}
    \int_X \left\langle \mathcal{L}^{\Omega}_g \theta, \theta \right\rangle_g
    \Omega & = & \int_X \left[ (\Delta^{\Omega}_g -\mathbbm{I})
    \Delta^{\Omega}_g \tau \cdot \Delta^{\Omega}_g \tau + P^{\Omega}_{g, J}
    \psi_2 \cdot \psi_2 + \left\langle \mathcal{L}^{\Omega}_g A, A
    \right\rangle_g \right] \Omega,
  \end{eqnarray*}
  which implies the required formula for the variations $(v, V) \in
  \mathbbm{F}^J_{g, \Omega} (0)$.
  
  {\tmstrong{STEP III}}. We compute now the stability integral involving $A$.
  The trivial identity
  \begin{eqnarray*}
    \mathcal{L}^{\Omega}_g A & = & \mathcal{L}_g A + \nabla_g f \neg \nabla_g
    A,
  \end{eqnarray*}
  combined with the formula (\ref{BKN-Alin-dV}) implies
  \begin{eqnarray*}
    \mathcal{L}^{\Omega}_g A & = & 2 \overline{\partial}_{T_{X, J}}
    \overline{\partial}^{\ast_g}_{T_{X, J}} A + \left[ \tmop{Ric}^{\ast} (g),
    A \right] + \nabla_g f \neg \nabla_g A,
  \end{eqnarray*}
  since $\overline{\partial}_{T_{X, J}} A = 0$. Integrating by parts we deduce
  \begin{eqnarray*}
    \int_X \left\langle \mathcal{L}^{\Omega}_g A, A \right\rangle_g \Omega & =
    & \int_X \left[ 2 \left\langle \overline{\partial}_{T_{X, J}}
    \overline{\partial}^{\ast_g}_{T_{X, J}} A, A \right\rangle_g +
    \left\langle \nabla_g f \neg \nabla_g A, A \right\rangle_g \right]
    \Omega\\
    &  & \\
    & = & \int_X \left[ 2 \left\langle \overline{\partial}^{\ast_g}_{T_{X,
    J}} A, \overline{\partial}^{\ast_{g, \Omega}}_{T_{X, J}} A \right\rangle_g
    + \frac{1}{2} \nabla_g f . \left| A \right|^2_g \right] \Omega\\
    &  & \\
    & = & \frac{1}{2}  \int_X \Delta^{\Omega}_g f \left| A \right|^2_g
    \Omega\\
    &  & \\
    & = & \int_X F \left| A \right|^2_g \Omega,
  \end{eqnarray*}
  since $\overline{\partial}^{\ast_{g, \Omega}}_{T_{X, J}} A = 0$ and $(J, g)$
  is a K\"ahler-Ricci-Soliton. (This last identity has been obtained by Hall-Murphy \cite{Ha-Mu2} using a
  different integration by parts method.)
  
  We show now the second variation formula corresponding to the particular
  case $(v, V) \in \mathbbm{F}^J_{g, \Omega} \left[ 0 \right]$. With this
  assumption hold the relation $\varphi = - \tau$. Thus we rearrange the
  expression
  \begin{eqnarray*}
    E & : = & \frac{1}{2} \int_X (\Delta^{\Omega}_g -\mathbbm{I})
    (\Delta^{\Omega}_g - 2\mathbbm{I}) \varphi \cdot (\Delta^{\Omega}_g -
    2\mathbbm{I}) \varphi \Omega\\
    &  & \\
    & - & \frac{1}{2} \int_X (\Delta^{\Omega}_g -\mathbbm{I})
    \Delta^{\Omega}_g \varphi \cdot \Delta^{\Omega}_g \varphi \Omega\\
    &  & \\
    & = & \frac{1}{2} \int_X \left[ - 4 (\Delta^{\Omega}_g -\mathbbm{I})
    \Delta^{\Omega}_g \varphi \cdot \varphi + 4 (\Delta^{\Omega}_g
    -\mathbbm{I}) \varphi \cdot \varphi \right] \Omega\\
    &  & \\
    & = & - 2 \int_X (\Delta^{\Omega}_g -\mathbbm{I})^2 \varphi \cdot \varphi
    \Omega\\
    &  & \\
    & = & - 2 \int_X \left| (\Delta^{\Omega}_g -\mathbbm{I}) \varphi
    \right|^2 \Omega,
  \end{eqnarray*}
  which implies the required formula in the particular case $(v, V) \in
  \mathbbm{F}^J_{g, \Omega} \left[ 0 \right]$.
  
  Let $A \assign \Delta^{\Omega}_{g, J} - 2\mathbbm{I}$ and observe that $[A,
  \bar{A}] = - 2 i [\Delta^{\Omega}_g, B^{\Omega}_{g, J}] = 0$, thanks to
  (\ref{Com-Lap-B}). Then the statement concerning the operator
  $P^{\Omega}_{g, J}$ follows from the elementary lemma below.
\end{proof}

\begin{lemma}
  Let $H \assign L_{\Omega}^2 (X, \mathbbm{C})_0$ and $A, B : D \subset H
  \longrightarrow H$ be closed densely defined linear operators such that $0
  \leqslant A = A^{\ast}$, $0 \leqslant B = B^{\ast}$, $\left[ A, B \right] =
  0$. If $A$ and B are differential operators of same order with $A$ elliptic
  then $A B \geqslant 0$. In particular if $[A, \bar{A}] = 0$ then $A \bar{A}
  \geqslant 0$.
\end{lemma}

\begin{proof}
  Let $E_{\lambda_k} (A) \subset H$ be the eigenspace of $A$ corresponding to
  an eigenvalue $\lambda_k \in \mathbbm{R}_{\geqslant 0}$. Then the identity
  $\left[ A, B \right] = 0$ implies that the restriction $B : E_{\lambda_k}
  (A) \longrightarrow E_{\lambda_k} (A)$ is well defined and represents a
  non-negative self-adjoint operator. We deduce by the spectral theorem in
  finite dimensions the existence of an orthonormal basis $(e_{k, l})_{l \in
  I_k} \subset E_{\lambda_k} (A)$ such that $B e_{k, l} = \mu_{k, l} e_{k,
  l}$, with $\mu_{k, l} \in \mathbbm{R}_{\geqslant 0}$. Moreover $A e_{k, l} =
  \lambda_k e_{k, l}$. We consider a strictly monotone increasing
  parametrization $(\lambda_k)_k$. Then any $u \in H$ writes as
  \begin{eqnarray*}
    u & = & \sum_{k \geqslant 0}  \sum_{l \in I_k} c_{k, l} e_{k, l},
  \end{eqnarray*}
  $c_{k, l} \in \mathbbm{C}$. In particular for $u \in C^{\infty} (X,
  \mathbbm{C})_0$ hold the expressions
  \begin{eqnarray*}
    A u & = & \sum_{k \geqslant 0}  \sum_{l \in I_k} \lambda_k c_{k, l} e_{k,
    l},\\
    &  & \\
    B u & = & \sum_{k \geqslant 0}  \sum_{l \in I_k} \mu_{k, l} c_{k, l} e_{k,
    l},
  \end{eqnarray*}
  and
  \begin{eqnarray*}
    (A B u, u)_{\Omega} & = & (B u, A u)_{\Omega} = \sum_{k \geqslant 0} 
    \sum_{l \in I_k} \lambda_k \mu_{k, l} \left| c_{k, l} \right|^2 \geqslant
    0 .
  \end{eqnarray*}
  The inequality in the general case $u \in D$ follows from the density of the
  smooth functions in the graph topology of $A$. In order to see that $\bar{A}
  \geqslant 0$ we observe the trivial equalities
  \begin{eqnarray*}
    0 & \leqslant & \int_X A u \cdot \overline{u} \Omega = \int_X u \cdot
    \overline{A u} \Omega = \int_X \bar{A} \overline{u} \cdot u \Omega =
    \int_X \bar{A} v \cdot \overline{v} \Omega,
  \end{eqnarray*}
  with $v : = \overline{u}$. In order to show its self-adjointness we observe
  also the trivial equalities
  \begin{eqnarray*}
    \int_X \bar{A} u \cdot \overline{v} \Omega & = & \int_X \overline{v} \cdot
    \overline{A \overline{u}} \Omega = \int_X A \overline{v} \cdot u \Omega =
    \int_X u \cdot \overline{ \overline{A} v} \Omega .
  \end{eqnarray*}
\end{proof}

We deduce the following corollary of proposition \ref{Sec-VarWKRS}.

\begin{corollary}
  \label{neg-varW}In the setting of proposition \ref{Sec-VarWKRS} assume $(v,
  V) \in \mathbbm{F}^J_{g, \Omega} \left[ 0 \right]$ with $A_v = 0$. Then
  \begin{eqnarray*}
    \frac{d^2}{d t^2} _{\mid_{t = 0}} \mathcal{W} (g_t, \Omega_t) & \leqslant
    & 0,
  \end{eqnarray*}
  with equality if and only if $(v, V) = (0, 0)$.
\end{corollary}

We notice indeed that the equality hold if and only if $\varphi = 0$.

\subsection{The K\"ahler-Einstein case}\label{KEcase}

In the K\"ahler-Einstein case the complex operator $\Delta^{\Omega}_{g, J}$
reduces to the real operator $\Delta^{\Omega}_g$. Let
\begin{eqnarray*}
  \Lambda_g & \assign & \tmop{Ker}_{\mathbbm{R}} (\Delta_g - 2\mathbbm{I})
  \subset C^{\infty} (X, \mathbbm{R})_0,
\end{eqnarray*}
and let $\Lambda^{\bot}_g \subset C^{\infty} (X, \mathbbm{R})_0$ be its
$L^2$-orthogonal with respect to the measure $d V_g$. We observe the
decomposition $\Lambda^{\Omega}_{g, J} = \Lambda_g \oplus i \Lambda_g$ , which
implies the decomposition
\begin{eqnarray*}
  \Lambda^{\Omega, \bot}_{g, J} & = & \Lambda^{\bot}_g \oplus i
  \Lambda^{\bot}_g,
\end{eqnarray*}
and thus the identity $\mathbbm{E}^{\Omega}_{g, J} = \Lambda^{\bot}_g$. With
the notations of lemma \ref{Decomp-SKah} let $\Phi \assign (\Delta_g -
2\mathbbm{I}) \varphi$, and $\Psi \assign (\Delta_g - 2\mathbbm{I}) \psi$.
Then the second variation formulas in proposition \ref{Sec-VarWKRS} reduces to
\begin{eqnarray*}
  \frac{d^2}{d t^2} _{\mid_{t = 0}} \mathcal{W} (g_t, \Omega_t) & = & \nabla_G
  D\mathcal{W} (g, \Omega) (v, V ; v, V)\\
  &  & \\
  & = & \frac{1}{2 \tmop{Vol}_g (X)} \int_X [ (\Delta_g -\mathbbm{I}) \Phi
  \cdot \Phi - (\Delta_g -\mathbbm{I}) \Psi \cdot \Psi] d V_g,
\end{eqnarray*}
in the case $(v, V) \in \mathbbm{F}^J_g (0)$ and to
\begin{eqnarray*}
  \frac{d^2}{d t^2} _{\mid_{t = 0}} \mathcal{W} (g_t, \Omega_t) & = & \nabla_G
  D\mathcal{W} (g, \Omega) (v, V ; v, V)\\
  &  & \\
  & = & - \frac{2}{\tmop{Vol}_g (X)} \int_X  \left| \Delta^{- 1}_g (\Delta_g
  -\mathbbm{I}) \Psi \right|^2 d V_g \geqslant 0,
\end{eqnarray*}
in the case $(v, V) \in \mathbbm{F}^J_g \left[ 0 \right]$, with equality if
and only if $v^{\ast}_g \in \mathcal{H}_g^{0, 1} \left( T_{X, J} \right)$.

{\tmstrong{Proof of step II in the K\"ahler-Einstein case}}.

The most difficult part in the proof of proposition \ref{Sec-VarWKRS} is the
computation of the stability integral
\[ \int_X \left\langle \mathcal{L}^{\Omega}_g \theta, \theta \right\rangle_g
   \Omega, \]
in step II of the proof. In the K\"ahler-Einstein case the argument is much
more simple. We include the details for readers convenience.

We remind first the isomorphism $g^{- 1} : S_{\mathbbm{R}}^2 T^{\ast}_X \simeq
\tmop{End}_g (T_X)$. We have the
$g$-orthogonal spiting
\begin{eqnarray*}
  \tmop{End}_g (T_X) & = & E_{g, J}' \oplus_g E_{g, J}'',\\
  &  & \\
  E_{g, J}' & \assign & \tmop{End}_g (T_X) \cap C^{\infty} \left( X,
  T^{\ast}_{X, J} \otimes T_{X, J} \right),\\
  &  & \\
  E_{g, J}'' & \assign & \tmop{End}_g (T_X) \cap C^{\infty} \left( X,
  T^{\ast}_{X, - J} \otimes T_{X, J} \right) .
\end{eqnarray*}
We observe that if $\alpha \in \Lambda_{\mathbbm{R}}^2 T^{\ast}_X$ then holds
the identity $\alpha^{\ast}_{\omega} = - J \alpha^{\ast}_g$, where
$\alpha^{\ast}_{\omega} \assign \omega^{- 1} \alpha$. We define also the
vector bundle
\begin{eqnarray*}
  \Lambda_{J, \mathbbm{R}}^{1, 1} & \assign & \Lambda_J^{1, 1} T^{\ast}_X \cap
  \Lambda_{\mathbbm{R}}^2 T^{\ast}_X,
\end{eqnarray*}
and we notice the isomorphism $\omega^{- 1} : \Lambda_{J, \mathbbm{R}}^{1, 1}
\simeq E_{g, J}'$. Moreover the identity (\ref{Trns-curv-id}) combined with
the properties (\ref{lin-Lich}) and (\ref{ant-lin-Lich}) implies that the maps
\begin{equation}
  \label{g-lin-Lich} \mathcal{L}^{\Omega}_g : C^{\infty} \left( X, E_{g, J}'
  \right) \longrightarrow C^{\infty} (X, E_{g, J}'),
\end{equation}

\begin{equation}
  \label{g-ant-lin-Lich} \mathcal{L}^{\Omega}_g : C^{\infty} \left( X, E_{g,
  J}'' \right) \longrightarrow C^{\infty} (X, E_{g, J}''),
\end{equation}
are well defined. We observe also that by (\ref{Jcurv}) and
(\ref{alt-ast-curv-Id}) we deduce the formula
\begin{equation}
  \label{Lich-alt-Lich} \omega \mathcal{L}^{\Omega}_g \alpha^{\ast}_{\omega} =
  \Delta^{\Omega}_g \alpha + 2\mathcal{R}_g \ast \alpha,
\end{equation}
for all $\alpha \in C^{\infty} (X, \Lambda_{J, \mathbbm{R}}^{1, 1})$. Notice
indeed that the endomorphism $J \alpha^{\ast}_{\omega}$ is $g$-anti-symmetric
thanks to the $J$-linearity of $\alpha^{\ast}_{\omega}$. We deduce using
(\ref{Lich-alt-Lich}) and the identity $\tmop{Tr}_{\omega} \alpha =
\tmop{Tr}_{\mathbbm{R}} \alpha^{\ast}_{\omega}$,
\begin{eqnarray*}
  \tmop{Tr}_{\omega} \left( \Delta^{\Omega}_g \alpha + 2\mathcal{R}_g \ast
  \alpha \right) & = & \tmop{Tr}_{\mathbbm{R}} \left( \mathcal{L}^{\Omega}_g
  \alpha^{\ast}_{\omega} \right) .
\end{eqnarray*}
This combined with the identity (\ref{Trace-Curv}), which in our case rewrites
as
\begin{eqnarray*}
  \tmop{Tr}_{\mathbbm{R}} \left( \mathcal{R}_g \ast \alpha^{\ast}_{\omega}
  \right) & = & \tmop{Tr}_{\mathbbm{R}} \left[ \tmop{Ric}^{\ast} (g)
  \alpha^{\ast}_{\omega} \right],
\end{eqnarray*}
implies that in the Einstein case hold the trace formula
\begin{equation}
  \label{Trac-Lich} \tmop{Tr}_{\omega} \left( \Delta_g \alpha + 2\mathcal{R}_g
  \ast \alpha \right) = (\Delta_g - 2\mathbbm{I}) \tmop{Tr}_{\omega} \alpha .
\end{equation}
We observe also the identity
\begin{equation}
  \label{scal-prod-id}  \left\langle \alpha, \beta \right\rangle_g =
  \left\langle \alpha_{\omega}^{\ast}, \beta_{\omega}^{\ast} \right\rangle_g,
\end{equation}
for all $\alpha, \beta \in \Lambda_{J, \mathbbm{R}}^{1, 1}$. Indeed we
consider the equalities
\begin{eqnarray*}
  \left\langle \alpha, \beta \right\rangle_g & = & \left\langle
  \alpha_g^{\ast}, \beta_g^{\ast} \right\rangle_g\\
  &  & \\
  & = & \tmop{Tr}_{\mathbbm{R}} \left[ \alpha_g^{\ast} (\beta_g^{\ast})_g^T
  \right]\\
  &  & \\
  & = & - \tmop{Tr}_{\mathbbm{R}} \left[ \alpha_g^{\ast} \beta_g^{\ast}
  \right]\\
  &  & \\
  & = & \tmop{Tr}_{\mathbbm{R}} \left[ J \alpha_g^{\ast} J \beta_g^{\ast}
  \right]\\
  &  & \\
  & = & \tmop{Tr}_{\mathbbm{R}} \left[ \alpha_{\omega}^{\ast}
  \beta_{\omega}^{\ast} \right] \\
  &  & \\
  & = & \left\langle \alpha_{\omega}^{\ast}, \beta_{\omega}^{\ast}
  \right\rangle_g .
\end{eqnarray*}
We deduce by the identity (\ref{WedgTrac}) in the appendix and by the Stokes
theorem that over a compact K\"ahler manifold if $\alpha, \beta \in C^{\infty}
(X, \Lambda_{J, \mathbbm{R}}^{1, 1})$, $d \alpha = d \beta = 0$ satisfy
$\left\{ \alpha \right\}_d = 0$, or $\left\{ \beta \right\}_d = 0$ then
\begin{equation}
  \label{scal-Trc-Intg} 2 \int_X \left\langle \alpha, \beta \right\rangle_g d
  V_g = \int_X \tmop{Tr}_{\omega} \alpha \tmop{Tr}_{\omega} \beta d V_g .
\end{equation}
(Notice indeed the identity $\left\langle \alpha, \beta \right\rangle_g =
\left\langle \alpha, \beta \right\rangle_{\omega}$ for all $\alpha, \beta \in
\Lambda_{J, \mathbbm{R}}^{1, 1}$.) We decompose now the endomorphism section
$\theta^{\ast}_g = A_J' + A_J''$ and we estimate the integral
\begin{eqnarray*}
  \int_X \left\langle \mathcal{L}^{\Omega}_g \theta, \theta \right\rangle_g
  \Omega & = & \int_X \left\langle \mathcal{L}^{\Omega}_g \theta^{\ast}_g,
  \theta^{\ast}_g \right\rangle_g \Omega\\
  &  & \\
  & = & \int_X \left[ \left\langle \mathcal{L}^{\Omega}_g A_J', A_J'
  \right\rangle_g + \left\langle \mathcal{L}^{\Omega}_g A_J'', A''_J
  \right\rangle_g \right] \Omega .
\end{eqnarray*}
The last equality hold thanks to the identity (\ref{lin-ant-lin-Integ-dec}).
Let $\alpha \assign \omega A_J'$ and assume $\left\{ \alpha \right\}_d = 0$.
Using the identity (\ref{Wei-alt}) we obtain
\begin{eqnarray*}
  \Delta^{\Omega}_g \alpha + 2\mathcal{R}_g \ast \alpha & = &
  (\Delta^{\Omega}_{d, g} - 2\mathbbm{I}) \alpha = d \nabla_g^{\ast_{\Omega}}
  \alpha - 2 \alpha,
\end{eqnarray*}
and thus $\left\{ \Delta^{\Omega}_g \alpha + 2\mathcal{R}_g \ast \alpha
\right\}_d = 0$. Then the identities (\ref{scal-prod-id}),
(\ref{Lich-alt-Lich}), (\ref{scal-Trc-Intg}) and (\ref{Trac-Lich}) imply
\begin{eqnarray*}
  \int_X \left\langle \mathcal{L}^{}_g A_J', A_J' \right\rangle_g d V_g  & = &
  \int_X \left\langle \Delta_g \alpha + 2\mathcal{R}_g \ast \alpha, \alpha
  \right\rangle_g d V_g\\
  &  & \\
  & = & \frac{1}{2}  \int_X  (\Delta_g - 2\mathbbm{I}) \tmop{Tr}_{\omega}
  \alpha \cdot \tmop{Tr}_{\omega} \alpha d V_g \geqslant 0,
\end{eqnarray*}
since $\lambda_1 (\Delta_g) \geqslant 2$ in the K\"ahler-Einstein case.
(Notice that the condition $\int_X \tmop{Tr}_{\omega} \alpha d V_g = 0$ hold
thanks to the assumption $\left\{ \alpha \right\}_d = 0$.) In the set up of
lemma \ref{Decomp-SKah} we have $\alpha = i \partial_J \overline{\partial}_J
\tau$ and
\begin{eqnarray*}
  A_J'' & = & \overline{\partial}_{T_{X, J}} \nabla_g \psi + A,
\end{eqnarray*}
with $\psi \in \Lambda^{\bot}_g$ and $A \in \mathcal{H}_g^{0, 1} \left( T_{X,
J} \right)$. Thus by the previous computation
\begin{eqnarray*}
  \int_X \left\langle \mathcal{L}^{}_g A_J', A_J' \right\rangle_g d V_g & = &
  \frac{1}{2}  \int_X  (\Delta_g - 2\mathbbm{I}) \Delta_g \tau \cdot \Delta_g
  \tau d V_g .
\end{eqnarray*}
On the other hand formula (\ref{BKN-Alin-dV}) implies in the K\"ahler-Einstein
case
\begin{eqnarray*}
  \mathcal{L}_g A''_J & = & 2 \Delta^{- J}_{T_{X, g}} A_J'' = 2
  \overline{\partial}_{T_{X, J}} \overline{\partial}^{\ast_g}_{T_{X, J}}
  A_J'',
\end{eqnarray*}
since $\overline{\partial}_{T_{X, J}} A_J'' = 0$. Integrating by parts we
deduce
\begin{eqnarray*}
  \int_X \left\langle \mathcal{L}^{}_g A_J'', A''_J \right\rangle_g d V_g & =
  & 2 \int_X \left| \overline{\partial}^{\ast_g}_{T_{X, J}} A''_J \right|^2_g
  d V_g = 2 \int_X \left| \overline{\partial}^{\ast_g}_{T_{X, J}}
  \overline{\partial}_{T_{X, J}} \nabla_g \psi \right|^2_g d V_g .
\end{eqnarray*}
In the K\"ahler-Einstein case the complex Bochner type formula
(\ref{dbar-OmBoch-fnctSOL}) combined with the equation (\ref{eq-div-free})
implies
\begin{eqnarray*}
  2 \overline{\partial}^{\ast_g}_{T_{X, J}} \overline{\partial}_{T_{X, J}}
  \nabla_g \psi & = & \nabla_g  (\Delta_g - 2\mathbbm{I}) \psi = - \nabla_g
  \Delta_g \tau .
\end{eqnarray*}
We obtain
\begin{eqnarray*}
  \int_X \left\langle \mathcal{L}^{}_g A_J'', A''_J \right\rangle_g d V_g & =
  & \frac{1}{2} \int_X \left| \nabla_g \Delta_g \tau \right|^2_g d V_g =
  \frac{1}{2} \int_X \Delta^2_g \tau \cdot \Delta_g \tau d V_g,
\end{eqnarray*}
and thus the required formula
\begin{eqnarray*}
  \int_X \left\langle \mathcal{L}^{}_g \theta, \theta \right\rangle_g d V_g &
  = & \int_X (\Delta_g -\mathbbm{I}) \Delta_g \tau \cdot \Delta_g \tau d V_g
  \geqslant 0 .
\end{eqnarray*}
We notice also that the latter implies the statement of theorem
\ref{KEstability}. Indeed in the equality case holds $\Delta_g \tau = 0$ since
$\lambda_1 (\Delta_g) \geqslant 2$. Then the equation (\ref{eq-div-free})
implies $\psi = 0$. The conclusion follows from the decomposition identity
(\ref{Cx-decTETA}).

\begin{remark}
  We consider the particular case of a smooth curve $\left( g_t, \Omega_t
  \right)_t \subset \mathcal{S}_{\omega}$ with $g_0$ K\"ahler-Einstein metric.
  Time deriving twice the expression
  \begin{eqnarray*}
    \mathcal{W} (g_t, \Omega_t) & = & 2 \int_X \log \left(
    \frac{\omega^n}{\Omega_t} \right) \Omega_t - 2 \log n!,
  \end{eqnarray*}
  we infer
  \begin{eqnarray*}
    \frac{d^2}{d t^2} _{\mid_{t = 0}} \mathcal{W} (g_t, \Omega_t) & = & - 2
    \int_X \left| \dot{\Omega}^{\ast}_0 \right|^2 \Omega_0 - 2 \int_X \log
    \left( \frac{\Omega_0}{\omega^n} \right)  \ddot{\Omega}_0\\
    &  & \\
    & = & - 2 \int_X \left| \dot{\Omega}^{\ast}_0 \right|^2 \Omega_0,
  \end{eqnarray*}
  thanks to the K\"ahler-Einstein condition. Then a trivial change of
  variables allows to deduce our previous second variation formula in the
  particular case $( \dot{g}_0, \dot{\Omega}_0) \in \mathbbm{F}^{J_0}_{g_0,
  \Omega_0} \left[ 0 \right]$.
\end{remark}

\subsection{The case of variations in the direction $\mathbbm{T}^J_{g,
\Omega}$}

\begin{proposition}
  \label{TSec-Var-W}Let $(X,J, g)$ be a compact K\"ahler-Ricci-Soliton and let $\Omega > 0$ be the
  unique smooth volume form such that $g J = \tmop{Ric}_J (\Omega)$ and
  $\int_X \Omega = 1$. Let also $(g_t, \Omega_t)_{t \in \mathbbm{R}} \subset
  \mathcal{M} \times \mathcal{V}_1$ be a smooth curve with $(g_0, \Omega_0) =
  (g, \Omega)$ and with $( \dot{g}_0, \dot{\Omega}_0) = (v, V) \in
  \mathbbm{T}^J_{g, \Omega}$. Then with respect to the decomposition
  \begin{eqnarray*}
    v^{\ast}_g & = & \overline{\partial}_{T_{X, J}} \nabla_{g, J} 
    \overline{\psi} + A,
  \end{eqnarray*}
  with unique $\psi \in \Lambda^{\Omega, \bot}_{g, J}$ and $A \in
  \mathcal{H}_{g, \Omega}^{0, 1} \left( T_{X, J} \right)$, hold the second
  variation formula
  \begin{eqnarray*}
    \frac{d^2}{d t^2} _{\mid_{t = 0}} \mathcal{W} (g_t, \Omega_t) & = &
    \nabla_G D\mathcal{W} (g, \Omega) (v, V ; v, V)\\
    &  & \\
    & = & - \frac{1}{2} \int_X \left[ P^{\Omega}_{g, J} \tmop{Re} \psi \cdot
    \tmop{Re} \psi_{_{_{_{_{_{}}}}}} + \left| A \right|^2_g F \right] \Omega,
  \end{eqnarray*}
  where
  \begin{eqnarray*}
    P^{\Omega}_{g, J} & \assign & (\Delta^{\Omega}_{g, J} - 2\mathbbm{I})
    \overline{ (\Delta^{\Omega}_{g, J} - 2\mathbbm{I})},
  \end{eqnarray*}
  is a non-negative self-adjoint real elliptic operator with respect to the
  $L_{\Omega}^2$-hermitian product. Moreover if $(v, V) \in \mathbbm{F}^J_{g,
  \Omega} [0]$ then the previous formula writes as
  \begin{eqnarray*}
    \frac{d^2}{d t^2} _{\mid_{t = 0}} \mathcal{W} (g_t, \Omega_t) & = &
    \nabla_G D\mathcal{W} (g, \Omega) (v, V ; v, V)\\
    &  & \\
    & = & - \frac{1}{2} \int_X \left[ 4 \left| V_{\Omega}^{\ast} \right|^2 
    \; + \; P^{\Omega}_{g, J} \tmop{Im} \psi \cdot \tmop{Im}
    \psi_{_{_{_{_{_{}}}}}} + \left| A \right|^2_g F \right] \Omega .
  \end{eqnarray*}
\end{proposition}

\begin{proof}
  {\tmstrong{Step I}}. Reconsidering a computation in the poof of step II of
  the proposition \ref{Sec-VarWKRS} we see that for all variations $v \in
  \mathbbm{D}^J_{g, \left[ 0 \right]}$ holds the identity
  \begin{eqnarray*}
    \int_X \left\langle \mathcal{L}^{\Omega}_g v, v \right\rangle_g \Omega & =
    & \int_X \left\langle \mathcal{L}^{\Omega}_g \overline{\partial}_{T_{X,
    J}} \nabla_{g, J}  \overline{\psi}, \overline{\partial}_{T_{X, J}}
    \nabla_{g, J}  \overline{\psi} + A \right\rangle_g \Omega\\
    &  & \\
    & + & \int_X \left\langle \mathcal{L}^{\Omega}_g A,
    \overline{\partial}_{T_{X, J}} \nabla_{g, J}  \overline{\psi} + A
    \right\rangle_g \Omega .
  \end{eqnarray*}
  Using the identity (\ref{com-Lich-HessOI}) and the property (\ref{Jcurv}) we
  deduce
  \begin{eqnarray*}
    \int_X \left\langle \mathcal{L}^{\Omega}_g v, v \right\rangle_g \Omega & =
    & \int_X \left\langle \overline{\partial}_{T_{X, J}} \nabla_{g, J}
    (\Delta^{\Omega}_g - 2\mathbbm{I}) \overline{\psi},
    \overline{\partial}_{T_{X, J}} \nabla_{g, J}  \overline{\psi} + A
    \right\rangle_g \Omega\\
    &  & \\
    & + & \int_X \left[ \left\langle \overline{\partial}^{\ast_{g,
    \Omega}}_{T_{X, J}} \mathcal{L}^{\Omega}_g A, \nabla_{g, J} 
    \overline{\psi}  \right\rangle_g + \left\langle \mathcal{L}^{\Omega}_g A,
    A \right\rangle_g \right] \Omega .
  \end{eqnarray*}
  Using the identity (\ref{divLA}), integrating by parts further and using the
  weighted complex Bochner identity (\ref{dbar-OmBoch-fnctSOL}) we obtain
  \begin{eqnarray*}
    \int_X \left\langle \mathcal{L}^{\Omega}_g v, v \right\rangle_g \Omega & =
    & \int_X \left\langle \nabla_{g, J} (\Delta^{\Omega}_g -
    2\mathbbm{I})_{_{_{_{_{}}}}} \overline{\psi},
    \overline{\partial}^{\ast_{g, \Omega}}_{T_{X, J}}
    \overline{\partial}_{T_{X, J}} \nabla_{g, J}  \overline{\psi} 
    \right\rangle_g \Omega\\
    &  & \\
    & + & \int_X \left\langle \mathcal{L}^{\Omega}_g A, A \right\rangle_g
    \Omega\\
    &  & \\
    & = & \frac{1}{2} \int_X \left\langle \nabla_{g, J} (\Delta^{\Omega}_g -
    2\mathbbm{I})_{_{_{_{_{}}}}} \overline{\psi}, \nabla_{g, J}
    \overline{(\Delta^{\Omega}_{g, J} - 2\mathbbm{I}) \psi}  \right\rangle_g
    \Omega\\
    &  & \\
    & + & \int_X \left\langle \mathcal{L}^{\Omega}_g A, A \right\rangle_g
    \Omega .
  \end{eqnarray*}
  Using the integration by parts formula (\ref{cpx-int-part}) in the
  subsection \ref{Integ-parts} of the appendix A we infer
  \begin{eqnarray}
    \int_X \left\langle \mathcal{L}^{\Omega}_g v, v \right\rangle_g \Omega & =
    & \frac{1}{4} \int_X \Delta^{\Omega}_{g, J}  (\Delta^{\Omega}_g -
    2\mathbbm{I}) \psi \cdot \overline{ (\Delta^{\Omega}_{g, J} -
    2\mathbbm{I}) \psi} \Omega\nonumber
    \\\nonumber
    &  & \\
    & + & \frac{1}{4} \int_X \overline{\Delta^{\Omega}_{g, J} 
    (\Delta^{\Omega}_g - 2\mathbbm{I}) \psi} \cdot (\Delta^{\Omega}_{g, J} -
    2\mathbbm{I}) \psi \Omega\nonumber
    \\\nonumber
    &  & \\
    & + & \int_X \left\langle \mathcal{L}^{\Omega}_g A, A \right\rangle_g
    \Omega\label{expression-E}
   \end{eqnarray}
  for all $v \in \mathbbm{D}^J_{g, \left[ 0 \right]}$.
  
  {\tmstrong{STEP II}}. We show first the variation formula in the case $(v,
  V) \in \mathbbm{F}^J_{g, \Omega} [0]$ since the proof is simpler. Using the expression (\ref{Cool-F0}) for the space
  $\mathbbm{F}^J_{g, \Omega} [0]$ inside the identity (\ref{expression-E}) we deduce the
  equalities
  \begin{eqnarray*}
    \int_X \left\langle \mathcal{L}^{\Omega}_g v, v \right\rangle_g \Omega & =
    & - \int_X (\Delta^{\Omega}_g - 2\mathbbm{I}) \tmop{Re}
    (\Delta^{\Omega}_{g, J} \psi) \cdot V^{\ast}_{\Omega}  \, \Omega\\
    &  & \\
    & + & \int_X \left\langle \mathcal{L}^{\Omega}_g A, A \right\rangle_g
    \Omega\\
    &  & \\
    & = & - \int_X (\Delta^{\Omega}_g - 2\mathbbm{I})  (\Delta^{\Omega}_g
    \psi_1 \noplus + B^{\Omega}_{g, J} \psi_2) \cdot V^{\ast}_{\Omega}  \,
    \Omega\\
    &  & \\
    & + & \int_X \left\langle \mathcal{L}^{\Omega}_g A, A \right\rangle_g
    \Omega .
  \end{eqnarray*}
  Let write $\psi = \psi_1 + i \psi_2$, with $\psi_j$ real valued functions.
  Then the condition in the expression (\ref{Cool-F0}) rewrites as
  \begin{eqnarray}
    (\Delta^{\Omega}_g - 2\mathbbm{I})_{_{_{_{_{}}}}} \psi_1 \noplus +
    B^{\Omega}_{g, J}  \, \psi_2 = - 2 V^{\ast}_{\Omega}, \label{Real-Cnd-F} 
  \\\nonumber
  \\
    (\Delta^{\Omega}_g - 2\mathbbm{I})_{_{_{_{_{}}}}} \psi_2 \noplus -
    B^{\Omega}_{g, J}  \, \psi_1 = 0 . \label{Im-Cnd-F}
  \end{eqnarray}
  We use now the condition (\ref{Real-Cnd-F}) in the formula
  \begin{eqnarray*}
    - 2 \frac{d^2}{d t^2} _{\mid_{t = 0}} \mathcal{W} (g_t, \Omega_t) & = & -
    2 \nabla_G D\mathcal{W} (g, \Omega) (v, V ; v, V)\\
    &  & \\
    & = & \int_X \left[ \left\langle \mathcal{L}^{\Omega}_g v, v
    \right\rangle_g - 2 (\Delta^{\Omega}_g - 2\mathbbm{I}) V_{\Omega}^{\ast}
    \cdot V^{\ast}_{\Omega} \right] \Omega\\
    &  & \\
    & = & - 2 \int_X (\Delta^{\Omega}_g - 2\mathbbm{I}) \psi_1 \cdot
    V^{\ast}_{\Omega}  \, \Omega\\
    &  & \\
    & + & \int_X \left\langle \mathcal{L}^{\Omega}_g A, A \right\rangle_g
    \Omega .
  \end{eqnarray*}
  Using again the condition (\ref{Real-Cnd-F}), we expand the integral
  \begin{eqnarray*}
    - 2 \int_X (\Delta^{\Omega}_g - 2\mathbbm{I}) \psi_1 \cdot
    V^{\ast}_{\Omega}  \, \Omega & = & \int_X (\Delta^{\Omega}_g -
    2\mathbbm{I}) \psi_1 \cdot \left[ (\Delta^{\Omega}_g -
    2\mathbbm{I})_{_{_{_{_{}}}}} \psi_1 \noplus + B^{\Omega}_{g, J}  \, \psi_2
    \right]  \, \Omega\\
    &  & \\
    & = & \int_X \psi_1 \left[ (\Delta^{\Omega}_g -
    2\mathbbm{I})^2_{_{_{_{_{}}}}} \psi_1 \noplus + B^{\Omega}_{g, J}
    (\Delta^{\Omega}_g - 2\mathbbm{I}) \, \psi_2 \right]  \, \Omega\\
    &  & \\
    & = & \int_X \psi_1 \left[ (\Delta^{\Omega}_g -
    2\mathbbm{I})^2_{_{_{_{_{}}}}} \psi_1 \noplus + \left( B^{\Omega}_{g, J}
    \right)^2 \, \psi_1 \right]  \, \Omega,
  \end{eqnarray*}
  thanks to the identities (\ref{Com-Lap-B}) and (\ref{Im-Cnd-F}). 
  Using the formula (\ref{PalOP}), we infer

  \begin{eqnarray*}
    \frac{d^2}{d t^2} _{\mid_{t = 0}} \mathcal{W} (g_t, \Omega_t) & = &
    \nabla_G D\mathcal{W} (g, \Omega) (v, V ; v, V)\\
    &  & \\
    & = & - \frac{1}{2} \int_X \left[ P^{\Omega}_{g, J} \psi_1 \cdot
    \psi_{1_{_{_{_{_{}}}}}} + \left\langle \mathcal{L}^{\Omega}_g A, A
    \right\rangle_g \right] \Omega .
  \end{eqnarray*}
  Using again the condition (\ref{Real-Cnd-F}) and the commutation identity
  (\ref{Com-Lap-B}) we expand the integral
  \begin{eqnarray*}
    & &\int_X 4 \left| V_{\Omega}^{\ast} \right|^2 \Omega 
    \\
    \\
    & = & \int_X \left[
    \left| (\Delta^{\Omega}_g - 2\mathbbm{I}) \psi_1 \right|_{_{_{_{}}}}^2 +
    \left| B^{\Omega}_{g, J} \psi_2 \right|^2 + 2 (\Delta^{\Omega}_g -
    2\mathbbm{I}) \psi_1 \cdot B^{\Omega}_{g, J} \psi_2 \right] \Omega\\
    &  & \\
    & = & \int_X \left[ (\Delta^{\Omega}_g - 2\mathbbm{I})_{_{_{_{}}}}^2
    \psi_1 \cdot \psi_1 - (B^{\Omega}_{g, J})^2 \psi_2 \cdot \psi_2 + 2 \psi_1
    \cdot B^{\Omega}_{g, J} (\Delta^{\Omega}_g - 2\mathbbm{I}) \psi_2 \right]
    \Omega,
  \end{eqnarray*}
  thanks to the fact that the operator $B^{\Omega}_{g, J}$ is
  $L_{\Omega}^2$-anti-adjoint. Using again this fact and the condition
  (\ref{Im-Cnd-F}) we deduce
  \begin{eqnarray*}
    & &\int_X 2 \psi_1 \cdot B^{\Omega}_{g, J} (\Delta^{\Omega}_g - 2\mathbbm{I})
    \psi_2 \Omega 
    \\
    \\
    & = & \int_X \psi_1 \cdot (B^{\Omega}_{g, J})^2 \psi_1
    \Omega\\
    &  & \\
    & - & \int_X B^{\Omega}_{g, J}  \, \psi_1 \cdot (\Delta^{\Omega}_g -
    2\mathbbm{I}) \psi_2 \Omega\\
    &  & \\
    & = & \int_X \left[ (B^{\Omega}_{g, J})^2 \psi_1 \cdot \psi_1 -
    (\Delta^{\Omega}_g - 2\mathbbm{I}) \psi_2 \cdot (\Delta^{\Omega}_g -
    2\mathbbm{I})_{_{_{_{_{}}}}} \psi_2 \right] \Omega,
  \end{eqnarray*}
  and thus
  \begin{eqnarray*}
    \int_X 4 \left| V_{\Omega}^{\ast} \right|^2 \Omega & = & \int_X \left[
    P^{\Omega}_{g, J} \psi_1 \cdot \psi_{1_{_{_{_{_{}}}}}} - P^{\Omega}_{g, J}
    \psi_2 \cdot \psi_{2_{_{_{_{_{}}}}}} \right] \Omega .
  \end{eqnarray*}
  We infer the second variation formula
  \begin{eqnarray*}
    \frac{d^2}{d t^2} _{\mid_{t = 0}} \mathcal{W} (g_t, \Omega_t) & = &
    \nabla_G D\mathcal{W} (g, \Omega) (v, V ; v, V)\\
    &  & \\
    & = & - \frac{1}{2} \int_X \left[ 4 \left| V_{\Omega}^{\ast} \right|^2 
    \; + \; P^{\Omega}_{g, J} \psi_2 \cdot \psi_{2_{_{_{_{_{}}}}}} +
    \left\langle \mathcal{L}^{\Omega}_g A, A \right\rangle_g \right] \Omega .
  \end{eqnarray*}
  The conclusion follows from the computation in the beginning of step III in
  the proof of the proposition \ref{Sec-VarWKRS}.
  
  {\tmstrong{STEP III}}. We show now the second variation formula in the more
  general case of variations $(v, V) \in \mathbbm{T}^J_{g, \Omega}$. We
  observe first that the general expression of $\nabla^2_G \mathcal{W} (g,
  \Omega)$ obtained at the end of the proof of lemma \ref{Sec-Var-W} implies
  that over a shrinking-Ricci-Soliton point holds the variation formula
  \begin{eqnarray*}
    &  & - 2 \frac{d^2}{d t^2} _{\mid_{t = 0}} \mathcal{W} (g_t, \Omega_t)\\
    &  & \\
    & = & - 2 \nabla_G D\mathcal{W} (g, \Omega) (v, V ; v, V)\\
    &  & \\
    & = & \int_X \left\langle \mathcal{L}^{\Omega}_g v -
    L_{\nabla_g^{\ast_{\Omega}} v_g^{\ast} + \nabla_g V^{\ast}_{\Omega}} g, v
    \right\rangle_g \Omega\\
    &  & \\
    & - & 2 \int_X \left[ (\Delta^{\Omega}_g - 2\mathbbm{I})
    V_{\Omega}^{\ast} - \tmop{div}^{\Omega} \left( \nabla_g^{\ast_{\Omega}}
    v_g^{\ast} + \nabla_g V^{\ast}_{\Omega} \right)_{_{_{_{_{}}}}} \right]
    V^{\ast}_{\Omega}  \; \Omega\\
    &  & \\
    & = & \int_X \left[ \left\langle \mathcal{L}^{\Omega}_g v, v
    \right\rangle_g - 2 (\Delta^{\Omega}_g - 2\mathbbm{I}) V_{\Omega}^{\ast}
    \cdot V^{\ast}_{\Omega} \right] \Omega\\
    &  & \\
    & - & 2 \int_X \left\langle \nabla_g \left( \nabla_g^{\ast_{\Omega}}
    v_g^{\ast} + \nabla_g V^{\ast}_{\Omega} \right) _{_{_{_{_{}}}}},
    v^{\ast}_g \right\rangle_g \Omega\\
    &  & \\
    & - & 2 \int_X \left\langle \nabla_g^{\ast_{\Omega}} v_g^{\ast} +
    \nabla_g V^{\ast}_{\Omega_{_{_{_{}}}}}, \nabla_g V^{\ast}_{\Omega}
    \right\rangle_g \Omega\\
    &  & \\
    & = & \int_X \left[ \left\langle \mathcal{L}^{\Omega}_g v, v
    \right\rangle_g - 2 (\Delta^{\Omega}_g - 2\mathbbm{I}) V_{\Omega}^{\ast}
    \cdot V^{\ast}_{\Omega} - 2 \left| \nabla_g^{\ast_{\Omega}} v_g^{\ast}
    +_{_{_{_{}}}} \nabla_g V^{\ast}_{\Omega} \right|^2_g \right] \Omega,
  \end{eqnarray*}
  for arbitrary directions $\left( v, V \right) \in T_{\mathcal{M} \times
  \mathcal{V}_1}$. Using now the fact that in the case $(v, V) \in
  \mathbbm{T}^J_{g, \Omega}$ hold the expressions $R_{\psi} = - 2
  V^{\ast}_{\Omega}$, (we use here the definitions (\ref{defRPSI}), (\ref{defIPSI})) and (\ref{FexpT}) we obtain
  \begin{eqnarray*}
    &  & - 2 \nabla_G D\mathcal{W} (g, \Omega) (v, V ; v, V)\\
    &  & \\
    & = & \int_X \left[ \left\langle \mathcal{L}^{\Omega}_g v, v
    \right\rangle_g - \frac{1}{2} (\Delta^{\Omega}_g - 2\mathbbm{I}) R_{\psi}
    \cdot R_{\psi} - \frac{1}{2}  \left| \nabla_g I_{\psi} \right|^2_g \right]
    \Omega,
  \end{eqnarray*}
  for all $(v, V) \in \mathbbm{T}^J_{g, \Omega}$. Thanks to the commutation
  identity (\ref{Com-Lap-B}) we can rewrite the identity (\ref{expression-E}) as
  \begin{eqnarray*}
    &  & \int_X \left[ \left\langle \mathcal{L}^{\Omega}_g v, v
    \right\rangle_g - \left\langle \mathcal{L}^{\Omega}_g A, A \right\rangle_g
    \right] \Omega\\
    &  & \\
    & = & \frac{1}{4} \int_X (\Delta^{\Omega}_g - 2\mathbbm{I})
    \Delta^{\Omega}_{g, J} \psi \cdot \overline{ (\Delta^{\Omega}_{g, J} -
    2\mathbbm{I}) \psi} \Omega\\
    &  & \\
    & + & \frac{1}{4} \int_X (\Delta^{\Omega}_g - 2\mathbbm{I})
    \overline{\Delta^{\Omega}_{g, J} \psi} \cdot (\Delta^{\Omega}_{g, J} -
    2\mathbbm{I}) \psi \Omega .
  \end{eqnarray*}
  Adding and subtracting $2 \psi$ to the factor $\Delta^{\Omega}_{g, J} \psi$
  and respectively $2 \overline{\psi}$ to the factor
  $\overline{\Delta^{\Omega}_{g, J} \psi}$, we infer
  \begin{eqnarray*}
    &  & \int_X \left[ \left\langle \mathcal{L}^{\Omega}_g v, v
    \right\rangle_g - \left\langle \mathcal{L}^{\Omega}_g A, A \right\rangle_g
    \right] \Omega\\
    &  & \\
    & = & \frac{1}{4} \int_X (\Delta^{\Omega}_g - 2\mathbbm{I})
    (\Delta^{\Omega}_{g, J} - 2\mathbbm{I}) \psi \cdot \overline{
    (\Delta^{\Omega}_{g, J} - 2\mathbbm{I}) \psi} \Omega\\
    &  & \\
    & + & \frac{1}{4} \int_X (\Delta^{\Omega}_g - 2\mathbbm{I}) \overline{
    (\Delta^{\Omega}_{g, J} - 2\mathbbm{I}) \psi} \cdot (\Delta^{\Omega}_{g,
    J} - 2\mathbbm{I}) \psi \Omega\\
    &  & \\
    & + & \frac{1}{2} \int_X (\Delta^{\Omega}_g - 2\mathbbm{I}) \psi \cdot
    \overline{ (\Delta^{\Omega}_{g, J} - 2\mathbbm{I}) \psi} \Omega\\
    &  & \\
    & + & \frac{1}{2} \int_X (\Delta^{\Omega}_g - 2\mathbbm{I})
    \overline{\psi} \cdot (\Delta^{\Omega}_{g, J} - 2\mathbbm{I}) \psi \Omega
    .
  \end{eqnarray*}
  We deduce the equalities
  \begin{eqnarray*}
    &  & - 2 \nabla_G D\mathcal{W} (g, \Omega) (v, V ; v, V)\\
    &  & \\
    & = & \int_X \left[ (\Delta^{\Omega}_g - 2\mathbbm{I})
    (\Delta^{\Omega}_{g, J} - 2\mathbbm{I}) \psi \cdot \overline{\psi} +
    \frac{1}{2} (\Delta^{\Omega}_g - 2\mathbbm{I}) I_{\psi} \cdot I_{\psi} -
    \frac{1}{2}  \left| \nabla_g I_{\psi} \right|^2_g \right] \Omega\\
    &  & \\
    & + & \int_X \left\langle \mathcal{L}^{\Omega}_g A, A \right\rangle_g
    \Omega\\
    &  & \\
    & = & \int_X \left\{ \left[ P^{\Omega}_{g, J} - i B^{\Omega}_{g, J}
    (\Delta^{\Omega}_{g, J} - 2\mathbbm{I})_{_{_{_{_{}}}}} \right] \psi \cdot
    \overline{\psi} - I_{\psi} \cdot I_{\psi} + \left\langle
    \mathcal{L}^{\Omega}_g A, A \right\rangle_g \right\} \Omega .
  \end{eqnarray*}
  Using the expression
  \begin{eqnarray*}
    I_{\psi} & = & (\Delta^{\Omega}_g - 2\mathbbm{I})_{_{_{_{_{}}}}} \psi_2
    \noplus - B^{\Omega}_{g, J}  \, \psi_1,
  \end{eqnarray*}
  we find the formula
  \begin{eqnarray*}
    &  & - 2 \nabla_G D\mathcal{W} (g, \Omega) (v, V ; v, V)\\
    &  & \\
    & = & \int_{X_{_{_{_{_{}}}}}} \left[ P^{\Omega}_{g, J} - i B^{\Omega}_{g,
    J} (\Delta^{\Omega}_g - 2\mathbbm{I})_{_{_{_{_{}}}}} - (B^{\Omega}_{g,
    J})^2 \right] \psi \cdot \overline{\psi}  \; \Omega\\
    &  & \\
    & - & \int_X \left[ \left| (\Delta^{\Omega}_g - 2\mathbbm{I}) \psi_2
    \right|^2 + \left| B^{\Omega}_{g, J} \psi_1 \right|^2 - 2
    (\Delta^{\Omega}_g - 2\mathbbm{I})_{_{_{_{_{_{_{}}}}}}} \psi_2 \cdot
    B^{\Omega}_{g, J}  \, \psi_1 \right] \Omega\\
    &  & \\
    & + & \int_X \left\langle \mathcal{L}^{\Omega}_g A, A \right\rangle_g
    \Omega .
  \end{eqnarray*}
  The fact that the operator $B^{\Omega}_{g, J}$ is
  $L_{\Omega}^2$-anti-adjoint combined with the commutation identity
  (\ref{Com-Lap-B}) implies that $B^{\Omega}_{g, J} (\Delta^{\Omega}_g -
  2\mathbbm{I})$ is also $L_{\Omega}^2$-anti-adjoint. We deduce in particular
  the identity
  \begin{eqnarray*}
    \int_X B^{\Omega}_{g, J} (\Delta^{\Omega}_g - 2\mathbbm{I}) \psi_j \cdot
    \psi_j  \, \Omega & = & 0,
  \end{eqnarray*}
  and thus the equality
  \begin{eqnarray*}
    &  & - 2 \nabla_G D\mathcal{W} (g, \Omega) (v, V ; v, V)\\
    &  & \\
    & = & \int_X \left[ P^{\Omega}_{g, J} \psi_1 \cdot \psi_1  \; + \;
    P^{\Omega}_{g, J} \psi_2 \cdot \psi_2 \right] \Omega\\
    &  & \\
    & + & \int_X \left[ B^{\Omega}_{g, J} (\Delta^{\Omega}_g -
    2\mathbbm{I})_{_{_{_{_{_{}}}}}} \psi_2 \cdot \psi_1 - B^{\Omega}_{g, J}
    (\Delta^{\Omega}_g - 2\mathbbm{I}) \psi_1 \cdot \psi_2 \right] \Omega\\
    &  & \\
    & - & \int_X \left[ (B^{\Omega}_{g, J})^2 \psi_1 \cdot \psi_1 +
    (B^{\Omega}_{g, J})_{_{_{_{_{}}}}}^2 \psi_2 \cdot \psi_2 \right] \Omega\\
    &  & \\
    & - & \int_X \left[ \left| (\Delta^{\Omega}_g - 2\mathbbm{I}) \psi_2
    \right|^2 + \left| B^{\Omega}_{g, J} \psi_1 \right|^2 - 2
    (\Delta^{\Omega}_g - 2\mathbbm{I})_{_{_{_{_{_{_{}}}}}}} \psi_2 \cdot
    B^{\Omega}_{g, J}  \, \psi_1 \right] \Omega\\
    &  & \\
    & + & \int_X \left\langle \mathcal{L}^{\Omega}_g A, A \right\rangle_g
    \Omega.
  \end{eqnarray*}
  Using the fact that the operator $B^{\Omega}_{g, J}$ is
  $L_{\Omega}^2$-anti-adjoint and the commutation identity (\ref{Com-Lap-B})
  we can simplify in order to obtain the required variation formula.
\end{proof}

\section{Positivity of the metric $G_{g, \Omega}$ over the space
$\mathbbm{T}^J_{g, \Omega}$}

\begin{lemma}
  \label{positiv-G}For any $(g, \Omega) \in \mathcal{S}_{\omega}$ the
  restriction of the symmetric form $G_{g, \Omega}$ to the vector space
  $\mathbbm{T}^J_{g, \Omega}$, with $J : = g^{- 1} \omega$, is positive
  definite.
\end{lemma}

\begin{proof}
  Let $\left( u, U \right) \nocomma, \left( v, V \right) \in \mathbbm{T}^J_{g,
  \Omega}$. Using the expression (\ref{inclusTS}) for the space
  $\mathbbm{T}^J_{g, \Omega}$ we have
  \begin{eqnarray*}
    u^{\ast}_g & = & \overline{\partial}_{T_{X, J}} \nabla_{g, J} 
    \overline{\varphi} + A,\\
    &  & \\
    - 2 U^{\ast}_{\Omega} & = & \tmop{Re} \left[_{_{_{_{_{}}}}}
    (\Delta^{\Omega}_{g, J} - 2\mathbbm{I}) \varphi \right],
  \end{eqnarray*}
  and
  \begin{eqnarray*}
    v^{\ast}_g & = & \overline{\partial}_{T_{X, J}} \nabla_{g, J} 
    \overline{\psi} + B,\\
    &  & \\
    - 2 V^{\ast}_{\Omega} & = & \tmop{Re} \left[_{_{_{_{_{}}}}}
    (\Delta^{\Omega}_{g, J} - 2\mathbbm{I}) \psi \right],
  \end{eqnarray*}
  with unique $\varphi, \psi \in \Lambda^{\Omega, \bot}_{g, J}$ and $A, B \in
  \mathcal{H}_{g, \Omega}^{0, 1} \left( T_{X, J} \right)$. We decompose now
  the term
  \begin{eqnarray*}
    G_{g, \Omega}  (u, U ; v, V) & = & \int_X \left[ \left\langle u, v
    \right\rangle_{g_{_{_{}}}} - 2 U^{\ast}_{\Omega} \cdot V^{\ast}_{\Omega}
    \right] \Omega\\
    &  & \\
    & = & \int_X \left[ \left\langle \overline{\partial}_{T_{X, J}}
    \nabla_{g, J_{_{_{_{}}}}}  \overline{\varphi}, \overline{\partial}_{T_{X,
    J}} \nabla_{g, J}  \overline{\psi}  \right\rangle_g + \left\langle A, B
    \right\rangle_g \right] \Omega\\
    &  & \\
    & - & \frac{1}{2} \int_X \tmop{Re} \left[_{_{_{_{_{}}}}}
    (\Delta^{\Omega}_{g, J} - 2\mathbbm{I}) \varphi \right] \tmop{Re}
    \left[_{_{_{_{_{}}}}} (\Delta^{\Omega}_{g, J} - 2\mathbbm{I}) \psi \right]
    \Omega .
  \end{eqnarray*}
  Integrating by parts and using the weighted complex Bochner formula
  (\ref{dbar-OmBoch-fnctSOL}) we transform the integral
  \begin{eqnarray*}
    I_1 & \assign & \int_X \left\langle \overline{\partial}_{T_{X, J}}
    \nabla_{g, J_{_{_{_{}}}}}  \overline{\varphi}, \overline{\partial}_{T_{X,
    J}} \nabla_{g, J}  \overline{\psi}  \right\rangle_g \Omega\\
    &  & \\
    & = & \int_X \left\langle \overline{\partial}^{\ast_{g, \Omega}}_{T_{X,
    J}}  \overline{\partial}_{T_{X, J}} \nabla_{g, J_{_{_{_{}}}}} 
    \overline{\varphi}, \nabla_{g, J}  \overline{\psi}  \right\rangle_g
    \Omega\\
    &  & \\
    & = & \frac{1}{2}  \int_X \left\langle \nabla_{g, J} 
    \overline{(\Delta^{\Omega}_{g, J} - 2\mathbbm{I}) \varphi}
    _{_{_{_{_{}}}}}, \nabla_{g, J}  \overline{\psi}  \right\rangle_g \Omega .
  \end{eqnarray*}
  Using the integration by parts formula (\ref{rlcpx-int-part}) in the
  subsection \ref{Integ-parts} of the appendix we deduce
  \begin{eqnarray*}
    I_1 & = & \frac{1}{4}  \int_X  \left[ \Delta^{\Omega}_{g, J}
    (\Delta^{\Omega}_{g, J} - 2\mathbbm{I}) \varphi_{_{_{_{_{}}}}} \cdot
    \overline{\psi} \; + \overline{\Delta^{\Omega}_{g, J} (\Delta^{\Omega}_{g,
    J} - 2\mathbbm{I}) \varphi} \cdot \, \psi \right] \Omega\\
    &  & \\
    & = & \frac{1}{4}  \int_X  \left[ (\Delta^{\Omega}_{g, J} - 2\mathbbm{I})
    \varphi_{_{_{_{_{}}}}} \cdot \overline{\Delta^{\Omega}_{g, J} \psi} \; +
    \overline{(\Delta^{\Omega}_{g, J} - 2\mathbbm{I}) \varphi} \cdot \,
    \Delta^{\Omega}_{g, J} \psi \right] \Omega .
  \end{eqnarray*}
  Adding and subtracting $2 \overline{\psi}$ to the factor
  $\overline{\Delta^{\Omega}_{g, J} \psi}$ and respectively $2 \psi$ to the
  factor $\Delta^{\Omega}_{g, J} \psi$, we infer
  \begin{eqnarray*}
    I_1 & = & \frac{1}{2}  \int_X  \left[ (\Delta^{\Omega}_{g, J} -
    2\mathbbm{I}) \varphi_{_{_{_{_{}}}}} \cdot \overline{\psi} \; +
    \overline{(\Delta^{\Omega}_{g, J} - 2\mathbbm{I}) \varphi} \cdot \, \psi
    \right] \Omega\\
    &  & \\
    & + & \frac{1}{4}  \int_X  \left[ (\Delta^{\Omega}_{g, J} - 2\mathbbm{I})
    \varphi_{_{_{_{_{}}}}} \cdot \overline{(\Delta^{\Omega}_{g, J} -
    2\mathbbm{I}) \psi}  \; + \overline{(\Delta^{\Omega}_{g, J} -
    2\mathbbm{I}) \varphi} \cdot (\Delta^{\Omega}_{g, J} - 2\mathbbm{I})
    \psi_{_{_{_{_{}}}}} \right] \Omega\\
    &  & \\
    & = & \frac{1}{2}  \int_X  \left[ (\Delta^{\Omega}_{g, J} - 2\mathbbm{I})
    \varphi_{_{_{_{_{}}}}} \cdot \overline{\psi} \; + \overline{\varphi} \cdot
    (\Delta^{\Omega}_{g, J} - 2\mathbbm{I}) \psi_{_{_{_{_{}}}}} \right]
    \Omega\\
    &  & \\
    & + & \frac{1}{4}  \int_X  \left[ (\Delta^{\Omega}_{g, J} - 2\mathbbm{I})
    \varphi_{_{_{_{_{}}}}} \cdot \overline{(\Delta^{\Omega}_{g, J} -
    2\mathbbm{I}) \psi}  \; + \overline{(\Delta^{\Omega}_{g, J} -
    2\mathbbm{I}) \varphi} \cdot (\Delta^{\Omega}_{g, J} - 2\mathbbm{I})
    \psi_{_{_{_{_{}}}}} \right] \Omega .
  \end{eqnarray*}
  We deduce the general formula
  \begin{eqnarray*}
  &&  G_{g, \Omega}  (u, U ; v, V) 
  \\
  \\
  & = &  \int_X \left\{ \frac{1}{2}  \left[
    (\Delta^{\Omega}_{g, J} - 2\mathbbm{I}) \varphi_{_{_{_{_{}}}}} \cdot
    \overline{\psi} \; + (\Delta^{\Omega}_{g, J} - 2\mathbbm{I})
    \psi_{_{_{_{_{}}}}} \cdot \overline{\varphi} \right] + \left\langle A, B
    \right\rangle_g \right\} \Omega\\
    &  & \\
    & + & \frac{1}{2} \int_X \tmop{Im} \left[_{_{_{_{_{}}}}}
    (\Delta^{\Omega}_{g, J} - 2\mathbbm{I}) \varphi \right] \tmop{Im}
    \left[_{_{_{_{_{}}}}} (\Delta^{\Omega}_{g, J} - 2\mathbbm{I}) \psi \right]
    \Omega .
  \end{eqnarray*}
  In particular
  \begin{eqnarray*}
    G_{g, \Omega}  (u, U ; u, U) & = & \int_X  \left[ (\Delta^{\Omega}_{g, J}
    - 2\mathbbm{I}) \varphi_{_{_{_{_{}}}}} \cdot \overline{\varphi} \; +
    \left| A \right|^2_g \right]\Omega\\
    &  & \\
    & + & \frac{1}{2} \int_X \left\{ \tmop{Im} \left[_{_{_{_{_{}}}}}
    (\Delta^{\Omega}_{g, J} - 2\mathbbm{I}) \varphi \right] \right\}^2 \Omega
    \; \geqslant \; 0,
  \end{eqnarray*}
  with equality if and only if $\varphi = 0$ and $A = 0$, i.e. $\left( u, U
  \right) = \left( 0, 0 \right)$, thanks to the variational characterization of
  the first eigenvalue $\lambda_1 \left( \Delta^{\Omega}_{g, J} \right)
  \geqslant 2$ of the elliptic operator $\Delta^{\Omega}_{g, J}$.
\end{proof}

\begin{corollary}
  \label{zero-ortog}For any $(g, \Omega) \in \mathcal{S}_{\omega}$ hold the
  identity
  \begin{equation}
    \label{E1CxE1} \tmop{Ker}_{\mathbbm{R}} (\Delta^{\Omega}_g - 2\mathbbm{I})
    = \tmop{Ker}_{\mathbbm{R}} (\Delta^{\Omega}_{g, J} - 2\mathbbm{I}),
  \end{equation}
  with $J : = g^{- 1} \omega$.
\end{corollary}

\begin{proof}
  Let $u \in C_{\Omega}^{\infty} (X, \mathbbm{R})_0$ and $(\varphi_t)_{t \in
  \mathbbm{R}} \subset \tmop{Symp}^0 (X, \omega)$ the 1-parameter sub-group
  generated by the symplectic vector field $\xi \assign (d u)_{\omega}^{\ast}
  = - J \nabla_g u$. We set $J_t \assign \varphi^{\ast}_t J$, $g_t \assign
  \varphi^{\ast}_t g$, $\Omega_t \assign \varphi^{\ast}_t \Omega$ and we
  compute $\dot{g}_0 = L_{\xi} g$ and $\dot{\Omega}_0 = L_{\xi} \Omega$. The
  expression of the tangent space to the symplectic orbit $\left[ g, \Omega
  \right]_{\omega}$ in the proof of lemma \ref{Orto-symplec} implies
  \begin{eqnarray*}
    \dot{g}^{\ast}_0 & = & - 2 J \overline{\partial}_{T_{X, J}} \nabla_g u,\\
    &  & \\
    \dot{\Omega}^{\ast}_0 & = & - B^{\Omega}_{g, J} u .
  \end{eqnarray*}
  Then the weighted complex Bochner formula (\ref{dbar-OmBoch-fnctSOL})
  implies
  \begin{eqnarray*}
    \overline{\partial}^{\ast_{g, \Omega}}_{T_{X, J}}  \dot{g}^{\ast}_0 +
    \nabla_g  \dot{\Omega}^{\ast}_0 & = & - 2 J \overline{\partial}^{\ast_{g,
    \Omega}}_{T_{X, J}} \overline{\partial}_{T_{X, J}} \nabla_g u + \nabla_g 
    \dot{\Omega}^{\ast}_0\\
    &  & \\
    & = & - J \nabla_g (\Delta^{\Omega}_g - 2\mathbbm{I}) u + \nabla_g
    B^{\Omega}_{g, J} u + \nabla_g  \dot{\Omega}^{\ast}_0\\
    &  & \\
    & = & - J \nabla_g (\Delta^{\Omega}_g - 2\mathbbm{I}) u .
  \end{eqnarray*}
  We deduce that $( \dot{g}_0, \dot{\Omega}_0) \in \mathbbm{F}^J_{g, \Omega}
  \left[ 0 \right]$ if and only if $(\Delta^{\Omega}_g - 2\mathbbm{I}) u = 0$.
  On the other hand the (strict) positivity of the metric $G_{g, \Omega}$ over
  $\mathbbm{T}^J_{g, \Omega} \supset T_{\left[ g, \Omega \right]_{\omega}, (g,
  \Omega)}$ implies
  \begin{eqnarray*}
    T_{\left[ g, \Omega \right]_{\omega}, (g, \Omega)} \cap T^{\bot_G}_{\left[
    g, \Omega \right]_{\omega}, (g, \Omega)} & = & 0 .
  \end{eqnarray*}
  Then lemma \ref{Orto-symplec} implies
  \begin{eqnarray*}
    T_{\left[ g, \Omega \right]_{\omega}, (g, \Omega)} \cap T^{\bot_G}_{\left[
    g, \Omega \right]_{\omega}, (g, \Omega)} & = & T_{\left[ g, \Omega
    \right]_{\omega}, (g, \Omega)} \cap \mathbbm{F}^J_{g, \Omega} \left[ 0
    \right] = \left\{ 0 \right\},
  \end{eqnarray*}
  So if $( \dot{g}_0, \dot{\Omega}_0) \in \mathbbm{F}^J_{g, \Omega} \left[ 0
  \right]$ then $( \dot{g}_0, \dot{\Omega}_0) = (0, 0)$. We infer the
  inclusion
  \begin{eqnarray*}
    \tmop{Ker}_{\mathbbm{R}} (\Delta^{\Omega}_g - 2\mathbbm{I}) & \subseteq &
    \tmop{Ker}_{\mathbbm{R}} B^{\Omega}_{g, J},
  \end{eqnarray*}
  and thus the required identity (\ref{E1CxE1}).
\end{proof}

\subsection{Double splitting of the space $\mathbbm{T}^J_{g, \Omega}$}

Let $H^k$, with $H^0 = L^2$, be a Sobolev space of sections over $X$. For any
subset $S$ of smooth sections over $X$ we denote with $H^k S$ its closure with
respect to the $H^k$-topology. The pseudo-Riemannian metric $G_{g, \Omega}$ is
obviously continuous with respect to the $L^2$-topology. At the moment we are
unable to say if the topology induced by $G_{g, \Omega}$ over $L^2
\mathbbm{T}^J_{g, \Omega}$ is equivalent with the $L^2$-topology. Nevertheless
we can show the following basic decomposition result

\begin{corollary}
  \label{Ortg-dec-TS} For any $(g, \Omega) \in \mathcal{S}_{\omega}$ holds the
  decomposition identity
  \[ L^2 \mathbbm{T}^J_{g, \Omega} = L^2 T_{\left[ g, \Omega \right]_{\omega},
     (g, \Omega)} \oplus_G L^2 \mathbbm{F}_{g, \Omega}^J [0], \]
  with $J : = g^{- 1} \omega$.
\end{corollary}

\begin{proof}
  We set
  \begin{eqnarray*}
    \Lambda^{\Omega}_{g, \mathbbm{R}} & \assign & \tmop{Ker}_{\mathbbm{R}}
    (\Delta^{\Omega}_g - 2\mathbbm{I}),
  \end{eqnarray*}
  and let $\Lambda^{\Omega, \bot}_{g, \mathbbm{R}} \subset L_{\Omega}^2 (X,
  \mathbbm{R})_0$ be its $L^2$-orthogonal with respect to the measure
  $\Omega$. Then corollary \ref{zero-ortog} and its proof shows that the map
  \begin{eqnarray*}
    \chi : \Lambda^{\Omega, \bot}_{g, \mathbbm{R}} \cap C_{\Omega}^{\infty}
    (X, \mathbbm{R})_0 & \longrightarrow & T_{\left[ g, \Omega
    \right]_{\omega}, (g, \Omega)},\\
    &  & \\
    \varphi & \longmapsto & \left( 2 \omega_{_{_{_{_{_{}}}}}}
    \overline{\partial}_{T_{X, J}} \nabla_g \varphi, \left( B^{\Omega}_{g, J}
    \varphi \right) \, \Omega \right),
  \end{eqnarray*}
  is an isomorphism. We notice also that the expression of the metric $G_{g,
  \Omega}$ obtained at the end of the proof of lemma \ref{positiv-G} hold true
  for arbitrary functions $\Phi$ and $\Psi$. So we put $\left( u, U \right)
  \assign \chi \left( \varphi \right)$ and $\Phi = \Psi = - 2 i \varphi$ in
  this formula. Using the fact that the operator $B^{\Omega}_{g, J}$ is
  $L_{\Omega}^2$-anti-adjoint and the expression
  \begin{eqnarray*}
    \tmop{Im} \left[_{_{_{_{_{}}}}} (\Delta^{\Omega}_{g, J} - 2\mathbbm{I})
    \Phi \right] & = & - 2 (\Delta^{\Omega}_g - 2\mathbbm{I}) \varphi,
  \end{eqnarray*}
  we infer
  \begin{eqnarray*}
    G_{g, \Omega}  (u, U ; u, U) & = & 2 \int_X \left[ \left|
    (\Delta^{\Omega}_g - 2\mathbbm{I}) \varphi \right|^2 + 2
    (\Delta^{\Omega}_g - 2\mathbbm{I}) \varphi_{_{_{_{_{}}}}} \cdot \varphi
    \right] \Omega\\
    &  & \\
    & = : & \gamma_{g, \Omega} \left( \varphi, \varphi \right) \; \geqslant
    \; 0,
  \end{eqnarray*}
  (with equality if and only if $\varphi = 0$). We remind now that the proof
  of the weighted Bochner formula (\ref{dbar-OmBoch-fnctSOL}) shows the
  identity
  \begin{eqnarray*}
    - 2 \tmop{div}^{\Omega} \overline{\partial}^{\ast_{g, \Omega}}_{T_{X, J}}
    \overline{\partial}_{T_{X, J}} \nabla_g & = & \Delta^{\Omega}_g
    (\Delta^{\Omega}_g - 2\mathbbm{I}) - (B^{\Omega}_{g, J})^2 .
  \end{eqnarray*}
  Thus the operator
  \begin{eqnarray*}
    \left( \overline{\partial}_{T_{X, J}} \nabla_{g_{_{_{}}}}
    \right)^{\ast_{g, \Omega}} \overline{\partial}_{T_{X, J}} \nabla_g & = & -
    \tmop{div}^{\Omega} \overline{\partial}^{\ast_{g, \Omega}}_{T_{X, J}}
    \overline{\partial}_{T_{X, J}} \nabla_g,
  \end{eqnarray*}
  is elliptic. This implies (see for example \cite{Ebi}) that the immage
$$
\overline{\partial}_{T_{X, J}} \nabla_g \left[ H^{k + 2} (X,
\mathbbm{R})_{_{_{_{}}}} \right] \subset H^k\,,
$$ is closed in the $H^k$-topology, for all integers $k\geqslant 0$.
We infer that the map
\begin{eqnarray}
  \overline{\partial}_{T_{X, J}} \nabla_g : \Lambda^{\Omega, \bot}_{g,
  \mathbbm{R}} \cap H^{k + 2} (X, \mathbbm{R}) & \longrightarrow &
  \overline{\partial}_{T_{X, J}} \nabla_g \left[ H^{k + 2} (X,
  \mathbbm{R})_{_{_{_{}}}} \right] \subset H^k,\qquad\qquad\label{topo-iso-Hess}
\end{eqnarray}
  is a topological isomorphism. We deduce that the extension in the sense of
  distributions
  \begin{eqnarray*}
    \chi : \Lambda^{\Omega, \bot}_{g, \mathbbm{R}} \cap H^2 (X, \mathbbm{R}) &
    \longrightarrow & L^2 T_{\left[ g, \Omega \right]_{\omega}, (g, \Omega)},
  \end{eqnarray*}
  of the map $\chi$ is also a topological isomorphism and a $(\gamma_{g,
  \Omega}, G_{g, \Omega})$-isometry. The fact that the map
  \begin{eqnarray*}
    \Delta^{\Omega}_g - 2\mathbbm{I}: \Lambda^{\Omega, \bot}_{g, \mathbbm{R}}
    \cap H^2 (X, \mathbbm{R}) & \longrightarrow & \Lambda^{\Omega, \bot}_{g,
    \mathbbm{R}},
  \end{eqnarray*}
  is a topological isomorphism provides the estimate
  \begin{eqnarray*}
    \gamma_{g, \Omega} \left( \varphi, \varphi \right) & \geqslant & 2 \int_X
    \left| (\Delta^{\Omega}_g - 2\mathbbm{I}) \varphi \right|^2 \Omega\\
    &  & \\
    & \geqslant & 2 \|(\Delta^{\Omega}_g - 2\mathbbm{I})^{- 1} \|^{- 2} \cdot
    \| \varphi \|^2_{H^2} .
  \end{eqnarray*}
  Then the Lax-Milgram theorem implies that the map
  \begin{eqnarray*}
    \gamma_{g, \Omega} : \Lambda^{\Omega, \bot}_{g, \mathbbm{R}} \cap H^2 (X,
    \mathbbm{R}) & \longrightarrow & \left[ \Lambda^{\Omega, \bot}_{g,
    \mathbbm{R}} \cap H^2 (X, \mathbbm{R})_{_{_{_{_{_{}}}}}} \right]^{\ast},
  \end{eqnarray*}
  is a topological isomorphism. (The sign $\ast$ here denotes the topological
  dual). We infer that the restricted map
  \begin{eqnarray*}
    G_{g, \Omega} : L^2 T_{\left[ g, \Omega \right]_{\omega}, (g, \Omega)} &
    \longrightarrow & \left[ L^2 T_{\left[ g, \Omega \right]_{\omega}, (g,
    \Omega)_{_{}}} \right]^{\ast},
  \end{eqnarray*}
  is also a topological isomorphism thanks to the fact that the extended map
  $\chi$ is a $(\gamma_{g, \Omega}, G_{g, \Omega})$-isometry. Applying the
  elementary lemma \ref{abst-ortog} below to \ the spaces $E \assign L^2
  \left( X, S^2 T^{\ast}_X \right) \oplus L_{\Omega}^2 \left( X, \mathbbm{R}
  \right)_0$ and $V \assign L^2 T_{\left[ g, \Omega \right]_{\omega}, (g,
  \Omega)}$ we deduce the $G$-orthogonal decomposition
  \begin{eqnarray*}
    L^2 \left( X, S^2 T^{\ast}_X \right) \oplus L_{\Omega}^2 \left( X,
    \mathbbm{R} \right)_0 & = & L^2 T_{\left[ g, \Omega \right]_{\omega}, (g,
    \Omega)} \oplus L^2 T^{\bot_G}_{\left[ g, \Omega \right]_{\omega}, (g,
    \Omega)},
  \end{eqnarray*}
  and thus
  \begin{eqnarray*}
    L^2 \mathbbm{T}^J_{g, \Omega} & = & L^2 T_{\left[ g, \Omega
    \right]_{\omega}, (g, \Omega)} \oplus \left[ L^2 T^{\bot_G}_{\left[ g,
    \Omega \right]_{\omega}, (g, \Omega)} \cap_{_{_{_{_{}}}}} L^2
    \mathbbm{T}^J_{g, \Omega} \right] \\
    &  & \\
    & = & L^2 T_{\left[ g, \Omega \right]_{\omega}, (g, \Omega)} \oplus L^2 
    \left[ T^{\bot_G}_{\left[ g, \Omega \right]_{\omega}, (g, \Omega)}
    \cap_{_{_{_{_{}}}}} \mathbbm{T}^J_{g, \Omega} \right] .
  \end{eqnarray*}
  Then the conclusion follows from the identity (\ref{geom-F}).
\end{proof}

\begin{lemma}
  \label{abst-ortog}Let $E$ be a real Banach space, $E^{\ast}$ its topological
  dual and $G : E \times E \longrightarrow \mathbbm{R}$ be a topologically non
  degenerate bilinear form, i.e. $G : E \longrightarrow E^{\ast}$ is an
  isomorphism. If there exists a closed subspace $V \subset E$ such that the
  restriction $G : V \times V \longrightarrow \mathbbm{R}$ is also
  topologically non degenerate then $E = V \oplus V^{\bot_G}$.
\end{lemma}

\begin{proof}
  Let $j : V \longhookrightarrow E$ be the canonical inclusion and notice the
  trivial identity
  \begin{eqnarray*}
    V^{\bot} & \assign & \{\alpha \in E^{\ast} \mid \alpha \cdot v = 0
    \nocomma, \forall v \in V\} = \tmop{Ker} j^{\ast} .
  \end{eqnarray*}
  By assumption for any element $e \in E$ there exists a unique $v \in V$ such
  that $j^{\ast} \left( e \neg G \right) = j^{\ast} \left( v \neg G \right)$.
  Thus $(e - v) \neg G \in V^{\bot}$. By definition the restriction $G :
  V^{\bot_G} \longrightarrow V^{\bot}$ provides an isomorphism. We conclude $e
  - v \in V^{\bot_G}$.
\end{proof}

We notice that the condition $V \cap V^{\bot_G} = \{0\}$ is equivalent to
$\tmop{Ker} (G : V \longrightarrow V^{\ast}) = \{0\}$ but in general not
sufficient to insure the surjectivity of $G : V \longrightarrow V^{\ast}$.

\subsection{Triple splitting of the space $\mathbbm{T}^J_{g, \Omega}$}\label{Trip-Split}

By abuse of notations we will denote by $G_{g, \Omega}$ the scalar product
over $\Lambda^{\Omega, \bot}_{g, J} \subset C^{\infty}$ induced by the
isomorphism
\begin{eqnarray*}
  \eta \; : \Lambda^{\Omega, \bot}_{g, J} \oplus \mathcal{H}_{g, \Omega}^{0,
  1} \left( T_{X, J} \right) & \longrightarrow & \mathbbm{T}^J_{g, \Omega} \\
  &  & \\
  \left( \psi, A \right) & \longmapsto & \left( g \left(
  \overline{\partial}_{T_{X, J}} \nabla_{g, J}  \overline{\psi} +
  A_{_{_{_{}}}} \right)_{}, - \frac{1}{2} \tmop{Re} \left[
  (\Delta^{\Omega}_{g, J} - 2\mathbbm{I}) \psi \right] \Omega \right) .
\end{eqnarray*}
Explicitly
\begin{eqnarray*}
  G_{g, \Omega} \left( \varphi, \psi \right) & = & \frac{1}{2} \int_X \left[
  (\Delta^{\Omega}_{g, J} - 2\mathbbm{I}) \varphi_{_{_{_{_{}}}}} \cdot
  \overline{\psi} \; + (\Delta^{\Omega}_{g, J} - 2\mathbbm{I})
  \psi_{_{_{_{_{}}}}} \cdot \overline{\varphi} \right] \Omega\\
  &  & \\
  & + & \frac{1}{2} \int_X \tmop{Im} \left[_{_{_{_{_{}}}}}
  (\Delta^{\Omega}_{g, J} - 2\mathbbm{I}) \varphi \right] \tmop{Im}
  \left[_{_{_{_{_{}}}}} (\Delta^{\Omega}_{g, J} - 2\mathbbm{I}) \psi \right]
  \Omega .
\end{eqnarray*}
We introduce the vector space
\begin{eqnarray*}
  \mathbbm{E}^J_{g, \Omega} & \assign & \left\{ u \in \Lambda^{\Omega,
  \bot}_{g, J} \mid (\Delta^{\Omega}_{g, J} -
  2\mathbbm{I}) u \in \Lambda^{\Omega,
  \bot}_{g, J} \cap C_{\Omega}^{\infty} \left( X, \mathbbm{R} \right)_0
  \right\},
\end{eqnarray*}
and we observe that the expression (\ref{Cool-F0}) for the space
$\mathbbm{F}^J_{g, \Omega} \left[ 0 \right]$ shows that the map $\eta$
restricts to the isomorphism
\begin{eqnarray*}
  \eta \; : \mathbbm{E}^J_{g, \Omega} \oplus \mathcal{H}_{g, \Omega}^{0, 1}
  \left( T_{X, J} \right) & \longrightarrow & \mathbbm{F}^J_{g, \Omega} \left[
  0 \right] .
\end{eqnarray*}
The subspaces $\mathbbm{E}^J_{g, \Omega} [0] \assign \eta \mathbbm{E}^J_{g,
\Omega}$ and $\mathcal{H}_{g, \Omega}^{0, 1} \left( T_{X, J} \right) \subset
\mathbbm{F}^J_{g, \Omega} \left[ 0 \right]$ (embedded via the previous
isomorphism) are $G$-orthogonal thanks to the expression of the restriction of
$G$ over $\mathbbm{F}^J_{g, \Omega} \left[ 0 \right]$ computed in the proof of
lemma \ref{positiv-G}. We deduce the $G$-orthogonal decomposition
\begin{eqnarray*}
  \mathbbm{F}^J_{g, \Omega} \left[ 0 \right] & = & \mathbbm{E}^J_{g, \Omega}
  [0] \oplus_G \mathcal{H}_{g, \Omega}^{0, 1} \left( T_{X, J} \right) .
\end{eqnarray*}
Let now
\begin{eqnarray*}
  \mathbbm{O}^J_{g, \Omega} & \assign & \left( \mathbbm{E}^J_{g, \Omega}
  \right)^{\bot_G} \cap \Lambda^{\Omega, \bot}_{g, J},
\end{eqnarray*}
and observe that the decomposition in corollary \ref{Ortg-dec-TS} implies
\begin{eqnarray*}
  L^2 T_{\left[ g, \Omega \right]_{\omega}, (g, \Omega)} & = & \left[ L^2
  \mathbbm{F}_{g, \Omega}^J [0]_{_{_{_{_{}}}}} \right]^{\bot_G} \cap L^2
  \mathbbm{T}^J_{g, \Omega}\\
  &  & \\
  & = & \mathbbm{F}_{g, \Omega}^J [0]_{_{_{_{_{}}}}}^{\bot_G} \cap L^2
  \mathbbm{T}^J_{g, \Omega},
\end{eqnarray*}
and thus
\begin{eqnarray*}
  T_{\left[ g, \Omega \right]_{\omega}, (g, \Omega)} & = & \mathbbm{F}_{g,
  \Omega}^J [0]_{_{_{_{_{}}}}}^{\bot_G} \cap \mathbbm{T}^J_{g, \Omega} .
\end{eqnarray*}
We deduce that the map $\eta$ restricts to a $G$-isometry
\begin{eqnarray*}
  \eta : \mathbbm{O}^J_{g, \Omega} & \longrightarrow & T_{\left[ g, \Omega
  \right]_{\omega}, (g, \Omega)} .
\end{eqnarray*}
We show now the smooth $G$-orthogonal decomposition (\ref{doub-spl-T}).

\begin{proof}.
The decomposition in the statement of corollary \ref{Ortg-dec-TS} implies that
for any $\left( u, U \right) \in \mathbbm{T}^J_{g, \Omega}$ there exists
$\left( \theta, \Theta \right) \in L^2 T_{\left[ g, \Omega \right]_{\omega},
(g, \Omega)}$ and $\left( v, V \right) \in L^2 \mathbbm{F}_{g, \Omega}^J [0]$
such that $\left( u, U \right) = \left( \theta, \Theta \right) + \left( v, V
\right)$. On the other hand, the fact that the map (\ref{topo-iso-Hess})
is a topological isomorphism implies the existence of $\rho \in
\Lambda^{\Omega, \bot}_{g, \mathbbm{R}} \cap H^2 (X, \mathbbm{R})$ such that
\begin{eqnarray*}
  \theta^{\ast}_g & = & - 2 J \overline{\partial}_{T_{X, J}} \nabla_g \rho .
\end{eqnarray*}
Thus $\Theta^{\ast}_{\Omega} = - B^{\Omega}_{g, J} \rho \in H^1 \left( X,
\mathbbm{R} \right)$ and $V^{\ast}_{\Omega} \in H^1 \left( X, \mathbbm{R}
\right)$. Then the identity $(\Delta^{\Omega}_{g, J} - 2\mathbbm{I}) \psi_v =
- 2 V^{\ast}_{\Omega}$ implies $\psi_v \in \Lambda^{\Omega, \bot}_{g, J} \cap
H^3 \left( X, \mathbbm{C} \right)$ by elliptic regularity. We deduce $\left(
v, V \right) \in H^1 \mathbbm{F}_{g, \Omega}^J [0]$ and thus $\left( \theta,
\Theta \right) \in H^1 T_{\left[ g, \Omega \right]_{\omega}, (g, \Omega)}$.
The conclusion follws by induction.

We provide also a second argument. Combining formula (\ref{FexpT}) with a
computation in the proof of corollary \ref{zero-ortog} we deduce the identity
$I_{\psi_{\theta}} = 2 (\Delta^{\Omega}_g - 2\mathbbm{I}) \rho$. On the other
hand we observe that $\psi_{\theta} - 2 i \rho \in \Lambda^{\Omega}_{g, J}$.
In more explicit terms
\begin{eqnarray*}
  (\Delta^{\Omega}_{g, J} - 2\mathbbm{I}) \psi_{\theta} & = & 2 B^{\Omega}_{g,
  J} \rho + 2 i (\Delta^{\Omega}_g - 2\mathbbm{I}) \rho \in \Lambda^{\Omega,
  \bot}_{g, J} \cap L^2 \left( X, \mathbbm{C} \right) .
\end{eqnarray*}
We infer
\begin{eqnarray*}
  (\Delta^{\Omega}_{g, J} - 2\mathbbm{I}) \psi_{\theta} & = &  \left[
  B^{\Omega}_{g, J} (\Delta^{\Omega}_g - 2\mathbbm{I})^{- 1} - i\mathbbm{I}
  \right] I_{\psi_{\theta}} .
\end{eqnarray*}
(Notice that $I_{\psi} \in \Lambda^{\Omega, \bot}_{g, \mathbbm{R}}$ for any $\psi \in H^2 \left( X, \mathbbm{C} \right)$ thanks
to corollary \ref{zero-ortog}). The fact that $I_{\psi_{\theta}} = I_{\psi_u}$
is smooth implies that
\begin{eqnarray*}
  \psi_{\theta} & = & (\Delta^{\Omega}_{g, J} - 2\mathbbm{I})^{- 1} \left[
  B^{\Omega}_{g, J} (\Delta^{\Omega}_g - 2\mathbbm{I})^{- 1} - i\mathbbm{I}
  \right] I_{\psi_u},
\end{eqnarray*}
is also smooth. We infer the required smooth decomposition.
\end{proof}

We deduce the $G$-orthogonal triple splitting
\begin{eqnarray}
   \mathbbm{T}^J_{g, \Omega}  =   T_{\left[ g, \Omega
  \right]_{\omega}, (g, \Omega)} \oplus_G  \mathbbm{E}_{g, \Omega}^J [0] \oplus_G \mathcal{H}_{g, \Omega}^{0, 1} \left( T_{X, J} \right).\label{smt-trp-spl-T}
\end{eqnarray}
We infer in particular the $G$-orthogonal decomposition
\begin{equation}
   \Lambda^{\Omega, \bot}_{g, J} =  \mathbbm{O}^J_{g, \Omega} \oplus_G
   \mathbbm{E}^J_{g, \Omega}, \label{decOrtLJ}
\end{equation}
thanks to the fact that the map $\eta$ is a topological isomorphism.
We observe now the following elementary lemma.

\begin{lemma}
  Let $T : D \subset L^2 \left( X, \mathbbm{C} \right) \longrightarrow L^2
  \left( X, \mathbbm{C} \right)$ be a closed densely defined
  $L_{\Omega}^2$-self-adjoint operator such that $[T, \overline{T}] = 0$.
  Then
  \begin{eqnarray*}
    \tmop{Ker} (T \overline{T}) \cap L^2 \left( X, \mathbbm{R} \right) & = &
    \left\{ \tmop{Re} u \mid u \in \tmop{Ker} T \right\} .
  \end{eqnarray*}
\end{lemma}

\begin{proof}
  The assumption $[T, \overline{T}] = 0$ implies that the restriction 
  $$
  \overline{T} : \tmop{Ker} T \longrightarrow \tmop{Ker} T,
  $$ is well defined.
  This combined with the fact that $\overline{T} $ is also
  $L_{\Omega}^2$-self-adjoint implies that the restriction
  \begin{eqnarray*}
    \overline{T} : D \cap (\tmop{Ker} T)^{\bot} & \longrightarrow &
    (\tmop{Ker} T)^{\bot},
  \end{eqnarray*}
  is also well defined. The inclusion $\tmop{Ker} (T \overline{T}) \supseteq
  \tmop{Ker} T + \tmop{Ker} \overline{T}$ is obvious. In order to show the
  reverse inclusion let $u \in \tmop{Ker} (T \overline{T})$, i.e.
  $\overline{T} u \in \tmop{Ker} T$, and consider the decomposition $u = u_1 +
  u_2$ with $u_1 \in \tmop{Ker} T$ and $u_2 \in (\tmop{Ker} T)^{\bot}$. Then
  $\overline{T} u \in \tmop{Ker} T$ if and only if $\overline{T} u_2 \in
  \tmop{Ker} T$ since $\overline{T} u_1 \in \tmop{Ker} T$. But $\overline{T}
  u_2 \in \tmop{Ker} T$ if and only if $\overline{T} u_2 = 0$ since
  $\overline{T} u_2 \in (\tmop{Ker} T)^{\bot}$. We infer the reverse
  inclusion. Thus
  \begin{eqnarray*}
    \tmop{Ker} (T \overline{T}) & = & \left\{ u + \overline{v} \mid u, v \in
    \tmop{Ker} T \right\},
  \end{eqnarray*}
  which implies the required conclusion.
\end{proof}

We remind that if $\left( g, \Omega \right) \in \mathcal{S}_{\omega}$ is a K\"ahler-Ricci-Soliton
with $J : = g^{- 1} \omega$ then
\begin{eqnarray*}
  \left[ \Delta^{\Omega}_{g, J} - 2\mathbbm{I}, \overline{\Delta^{\Omega}_{g,
  J} - 2\mathbbm{I}} _{_{_{_{_{}}}}} \right] & = & 0,
\end{eqnarray*}
which allows to apply the previous lemma to the $L_{\Omega}^2$-self-adjoint
operator $P^{\Omega}_{g, J}$. Thus
\begin{eqnarray*}
  \tmop{Ker} P^{\Omega}_{g, J} \cap C_{\Omega}^{\infty} \left( X, \mathbbm{R}
  \right)_0 & = & \left\{ \tmop{Re} u \mid u \in \Lambda^{\Omega}_{g, J}
  \right\} = : \tmop{Re} \Lambda^{\Omega}_{g, J} .
\end{eqnarray*}
The finiteness theorem for elliptic operators implies
\begin{eqnarray*}
  P^{\Omega}_{g, J} C_{\Omega}^{\infty} \left( X, \mathbbm{R} \right)_0 & = &
  \left( \tmop{Re} \Lambda^{\Omega}_{g, J_{_{}}} \right)^{\bot} \cap
  C_{\Omega}^{\infty} \left( X, \mathbbm{R} \right)_0 \supseteq
  \Lambda^{\Omega, \bot}_{g, J} \cap C_{\Omega}^{\infty} \left( X, \mathbbm{R}
  \right)_0 .
\end{eqnarray*}
The last inclusion is obvious. The inclusion $P^{\Omega}_{g, J}
C_{\Omega}^{\infty} \left( X, \mathbbm{R} \right)_0 \subseteq \Lambda^{\Omega,
\bot}_{g, J} \cap C_{\Omega}^{\infty} \left( X, \mathbbm{R} \right)_0$ is also
obvious. We conclude
\begin{equation}
  P^{\Omega}_{g, J} C_{\Omega}^{\infty} \left( X, \mathbbm{R} \right)_0 =
  \Lambda^{\Omega, \bot}_{g, J} \cap C_{\Omega}^{\infty} \left( X, \mathbbm{R}
  \right)_0 = \left( \tmop{Re} \Lambda^{\Omega}_{g, J_{_{}}} \right)^{\bot}
  \cap C_{\Omega}^{\infty} \left( X, \mathbbm{R} \right)_0 . \label{ImgP}
\end{equation}
\begin{lemma}
  \label{exprO}If $\left( g, \Omega \right) \in \mathcal{S}_{\omega}$ is a K\"ahler-Ricci-Soliton
  then holds the identity
  \begin{eqnarray*}
    \mathbbm{O}^J_{g, \Omega} & = & \left\{ \psi \in \Lambda^{\Omega,
    \bot}_{g, J} \mid P^{\Omega}_{g, J} \tmop{Re} \psi =  0
    \right\},
  \end{eqnarray*}
  with $J : = g^{- 1} \omega$.
\end{lemma}

\begin{proof}
  We notice that for any $\varphi \in \mathbbm{E}_{g, \Omega}^J$ and $\psi \in
  \mathbbm{O}^J_{g, \Omega}$ holds the identity
  \begin{eqnarray*}
    0 & = & G_{g, \Omega} \left( \varphi, \psi \right) = \int_X
    (\Delta^{\Omega}_{g, J} - 2\mathbbm{I}) \varphi \cdot \tmop{Re} \psi
    \Omega .
  \end{eqnarray*}
  We infer
  \begin{eqnarray*}
    \mathbbm{O}^J_{g, \Omega} & = & \left\{ \psi \in \Lambda^{\Omega,
    \bot}_{g, J} \mid \tmop{Re} \psi \in  \left[
    \Lambda^{\Omega, \bot}_{g, J} \cap C_{\Omega}^{\infty} \left( X,
    \mathbbm{R} \right)_0 \right]^{\bot} \right\} \\
    &  & \\
    & = & \left\{ \psi \in \Lambda^{\Omega, \bot}_{g, J} \mid \tmop{Re} \psi
    \in  \tmop{Re} \Lambda^{\Omega}_{g, J} \right\},
  \end{eqnarray*}
  thanks to (\ref{ImgP}).
\end{proof}
\section{Infinitesimal properties of the function $\underline{H}$}

By abuse of notations also we will consider from now on $\Lambda^{\Omega,
\bot}_{g, \mathbbm{R}} \subset C^{\infty}$. 
We observe that lemma \ref{Carac-KRS} implies; $\left( g, \Omega \right) \in
\mathcal{S}_{\omega}$ is a K\"ahler-Ricci-Soliton if and only if $\underline{H}^{}_{g, \Omega} =
0$. Furthermore the identity (\ref{H-KRS}) rewrites as
\begin{eqnarray*}
  2 \underline{H}^{}_{g, \Omega}  =  - (\Delta^{\Omega}_{g, J} -
  2\mathbbm{I}) F \in \Lambda^{\Omega, \bot}_{g, J} \cap
  C_{\Omega}^{\infty} \left( X, \mathbbm{R} \right)_0,
\end{eqnarray*}
for all $\left( g, \Omega \right) \in \mathcal{S}_{\omega}$. We show now the
following fact.

\begin{lemma}
  If $\left( g, \Omega \right) \in \mathcal{S}_{\omega}$ is a K\"ahler-Ricci-Soliton then the
  linear map
  \begin{equation}
    D_{g, \Omega}  \underline{H}^{}_{} : \mathbbm{E}^J_{g, \Omega} [0] \;
    \longrightarrow \; \Lambda^{\Omega, \bot}_{g, J} \cap C_{\Omega}^{\infty}
    \left( X, \mathbbm{R} \right)_0, \label{iso-DH}
  \end{equation}
  with $J : = g^{- 1} \omega$, is well defined and represents an isomorphism
  of real vector spaces.
\end{lemma}

\begin{proof}
  The identity $2 \underline{H}^{}_{g, \Omega} = 2 H_{g, \Omega} -\mathcal{W}
  \left( g, \Omega \right)$ combined with the basic variation formula
  (\ref{var-H}) implies
  \begin{eqnarray*}
    2 D_{g, \Omega}  \underline{H}^{}  \left( v, V \right) & = &
    (\Delta^{\Omega}_g - 2\mathbbm{I}) V_{\Omega}^{\ast},
  \end{eqnarray*}
  for all $\left( v, V \right) \in \mathbbm{F}_{g, \Omega}$ over a shrinking
  Ricci soliton point $\left( g, \Omega \right)$. In our K\"ahler-Ricci soliton set up the latter
  rewrites as
  \begin{equation}
    2 D_{g, \Omega}  \underline{H}^{}  \left( v, V \right) = - \frac{1}{2} 
    (\Delta^{\Omega}_g - 2\mathbbm{I}) (\Delta^{\Omega}_{g, J} - 2\mathbbm{I})
    \psi_v, \label{dif-H-E}
  \end{equation}
  for all $\left( v, V \right) \in \mathbbm{F}^J_{g, \Omega} \left[ 0
  \right]$. The commutation identity
  \begin{eqnarray*}
    \left[ \Delta^{\Omega}_g - 2\mathbbm{I}_{_{_{_{}}}}, \Delta^{\Omega}_{g,
    J} - 2\mathbbm{I} \right] = 0, &  & 
  \end{eqnarray*}
  implies the inclusion
  \begin{equation}
    (\Delta^{\Omega}_g - 2\mathbbm{I}) \Lambda^{\Omega}_{g, J} \subseteq
    \Lambda^{\Omega}_{g, J}, \label{invar-LamJ}
  \end{equation}
  and thus
  \begin{equation}
    (\Delta^{\Omega}_g - 2\mathbbm{I}) \Lambda^{\Omega, \bot}_{g, J} \subseteq
    \Lambda^{\Omega, \bot}_{g, J} . \label{invar-Ort-LamJ}
  \end{equation}
  Then the identity (\ref{dif-H-E}) shows that the map (\ref{iso-DH}) is well
  defined. We will deduce that it is an isomorphism if we show that the map
  \begin{equation}
    \Delta^{\Omega}_g - 2\mathbbm{I}: \Lambda^{\Omega, \bot}_{g, J} \cap
    C_{\Omega}^{\infty} \left( X, \mathbbm{R} \right)_0 \longrightarrow
    \Lambda^{\Omega, \bot}_{g, J} \cap C_{\Omega}^{\infty} \left( X,
    \mathbbm{R} \right)_0, \label{iso-Lap-OrtJ}
  \end{equation}
  is an isomorphism. Indeed this is the case. The injectivity of
  (\ref{iso-Lap-OrtJ}) follows from the inclusion
  \begin{eqnarray*}
    \mathbbm{C} \Lambda^{\Omega}_{g, \mathbbm{R}} & \subseteq &
    \Lambda^{\Omega}_{g, J},
  \end{eqnarray*}
  which holds thanks to the identity (\ref{E1CxE1}). This inclusion implies
  also
  \begin{eqnarray}
    \mathbbm{C} \Lambda^{\Omega, \bot}_{g, \mathbbm{R}}  \; = \; (\mathbbm{C}
    \Lambda^{\Omega}_{g, \mathbbm{R}})^{\bot}  \supseteq  \Lambda^{\Omega,
    \bot}_{g, J}, \label{ort-Cx-incl} 
  \end{eqnarray}
  and thus
  \begin{eqnarray*}
    \Lambda^{\Omega, \bot}_{g, \mathbbm{R}}  \supseteq  \Lambda^{\Omega,
    \bot}_{g, J} \cap C_{\Omega}^{\infty} \left( X, \mathbbm{R} \right)_0 .
  \end{eqnarray*}
  We use now the obvious fact that
  \begin{eqnarray*}
    \Delta^{\Omega}_g - 2\mathbbm{I} \, : \Lambda^{\Omega, \bot}_{g,
    \mathbbm{R}} & \longrightarrow & \Lambda^{\Omega, \bot}_{g, \mathbbm{R}},
  \end{eqnarray*}
  is an isomorphism. Thus for any $f \in \Lambda^{\Omega, \bot}_{g, J} \cap
  C_{\Omega}^{\infty} \left( X, \mathbbm{R} \right)_0$ there exists a unique
  $u \in \Lambda^{\Omega, \bot}_{g, \mathbbm{R}}$ such that
  \begin{eqnarray*}
    (\Delta^{\Omega}_g - 2\mathbbm{I}) u  =  f .
  \end{eqnarray*}
  We decompose $u = u_1 + u_2$, with $u_1 \in \Lambda^{\Omega}_{g, J}$ and
  $u_2 \in \Lambda^{\Omega, \bot}_{g, J}$. Then the inclusions
  (\ref{invar-LamJ}) and (\ref{invar-Ort-LamJ}) imply the
  $L_{\Omega}^2$-orthogonal decomposition
  \begin{eqnarray*}
    (\Delta^{\Omega}_g - 2\mathbbm{I}) u_1 + (\Delta^{\Omega}_g -
    2\mathbbm{I}) u_2  =  f .
  \end{eqnarray*}
  We deduce $u_1 \in \mathbbm{C} \Lambda^{\Omega}_{g, \mathbbm{R}}$. But $u_2
  \in \mathbbm{C} \Lambda^{\Omega, \bot}_{g, \mathbbm{R}}$ thanks to the
  inclusion (\ref{ort-Cx-incl}). We infer $u_1 = 0$ since $u \in
  \Lambda^{\Omega, \bot}_{g, \mathbbm{R}}$. Thus
  \begin{eqnarray*}
    u  =  u_2  \in  \Lambda^{\Omega, \bot}_{g, J} \cap
    C_{\Omega}^{\infty} \left( X, \mathbbm{R} \right)_0 .
  \end{eqnarray*}
  We obtain the surjectivity of the map (\ref{iso-Lap-OrtJ}) and thus the
  required conclusion.
\end{proof}

\begin{lemma}
  If $\left( g, \Omega \right) \in \mathcal{S}_{\omega}$ is a K\"ahler-Ricci-Soliton then hold
  the identity
  \begin{equation}
    \tmop{Ker} D_{g, \Omega}  \underline{H}^{}_{} \cap \mathbbm{T}^J_{g,
    \Omega} \; = T_{\left[ g, \Omega \right]_{\omega}, (g, \Omega)} \oplus_G
    \mathcal{H}_{g, \Omega}^{0, 1} \left( T_{X, J} \right), \label{KerDH}
  \end{equation}
  with $J : = g^{- 1} \omega$.
\end{lemma}

\begin{proof}
  With the notations in the proof of lemma \ref{Orto-symplec}, the basic
  variation formula (\ref{var-H}) combined with the identities
  (\ref{inclusTS}) and (\ref{FexpT}) implies that for all $\left( v, V \right)
  \in \mathbbm{T}^J_{g, \Omega}$ over a K\"ahler-Ricci-Soliton point $\left(J,
  g, \Omega \right)$ hold the equalities
  \begin{eqnarray*}
    2 D_{g, \Omega}  \underline{H}^{}  \left( v, V \right) & = & - \frac{1}{2}
    (\Delta^{\Omega}_g - 2\mathbbm{I}) R_{\psi} + \frac{1}{2}  \left( L_{J
    \nabla_g I_{\psi}} \Omega_{_{_{_{}}}} \right)^{\ast}_{\Omega} \\
    &  & \\
    & = & - \frac{1}{2} (\Delta^{\Omega}_g - 2\mathbbm{I}) R_{\psi} +
    \frac{1}{2} B^{\Omega}_{g, J} I_{\psi}\\
    &  & \\
    & = & - \frac{1}{2} \tmop{Re} \left[ (\Delta^{\Omega}_g - 2\mathbbm{I})
    \overline{ \; (\Delta^{\Omega}_{g, J} - 2\mathbbm{I}) \psi}
    _{_{_{_{_{}}}}} \right]\\
    &  & \\
    & = & - \frac{1}{2} \tmop{Re} \left[ P^{\Omega}_{g, J} \overline{\psi}
    _{_{_{_{_{}}}}} \right]\\
    &  & \\
    & = & - \frac{1}{2} P^{\Omega}_{g, J} \tmop{Re} \psi,
  \end{eqnarray*}
  since $P^{\Omega}_{g, J}$ is a real operator in our K\"ahler-Ricci-Soliton case. Then lemma
  \ref{exprO} implies
  \begin{eqnarray*}
    \tmop{Ker} D_{g, \Omega}  \underline{H}^{}_{} \cap \mathbbm{T}^J_{g,
    \Omega} & \simeq & \mathbbm{O}^J_{g, \Omega} \oplus_G \mathcal{H}_{g,
    \Omega}^{0, 1} \left( T_{X, J} \right),
  \end{eqnarray*}
  i.e. the required conclusion.
\end{proof}

{\tmstrong{Proof of the main theorem \ref{Main-teorem}}}

\begin{proof}
  The inequality in the statement follows immediately from proposition
  \ref{TSec-Var-W}. If equality holds then obviously $A \in \mathcal{H}_{g,
  \Omega}^{0, 1} \left( T_{X, J} \right)_0$ and
  \begin{eqnarray*}
    \int_X P^{\Omega}_{g, J} \tmop{Re} \psi \cdot \tmop{Re}
    \psi_{_{_{_{_{_{}}}}}} \Omega & = & 0 .
  \end{eqnarray*}
  Then the spectral theorem applied to the non-negative
  $L_{\Omega}^2$-self-adjoint real elliptic operator $P^{\Omega}_{g, J}$
  implies $P^{\Omega}_{g, J} \tmop{Re} \psi = 0$. Thus the conclusion
  \[ \left( v, V \right) \in T_{\left[ g, \Omega \right]_{\omega}, (g,
     \Omega)} \oplus_G \mathcal{H}_{g, \Omega}^{0, 1} \left( T_{X, J}
     \right)_0 = \tmop{Ker} D_{g, \Omega}  \underline{H}^{}_{} \cap
     \mathbbm{T}^{J, 0}_{g, \Omega}, \]
  follows from lemma \ref{exprO} and the identity (\ref{KerDH}). In order to
  show the inclusion (\ref{incl-MainThm}) we observe that if $\left( g_t, \Omega_t \right)_{t
  \in \mathbbm{R}} \subset \tmop{KRS}_{\omega}$ is a smooth curve with $(g_0,
  \Omega_0) = (g, \Omega)$ and with $( \dot{g}_0, \dot{\Omega}_0) = (v, V)$
  then holds the identity $\underline{H}_{g_t, \Omega_t} \equiv 0$ and thus
  \[ \left( v, V \right) \in \tmop{Ker} D_{g, \Omega}  \underline{H}^{}_{}
     \cap \mathbbm{T}^J_{g, \Omega} = T_{\left[ g, \Omega \right]_{\omega},
     (g, \Omega)} \oplus_G \mathcal{H}_{g, \Omega}^{0, 1} \left( T_{X, J}
     \right), \]
  thanks to the identity (\ref{KerDH}). On the other side if we set
  $\mathcal{W}_t : =\mathcal{W} \left( g_t, \Omega_t \right)$ then
  $\dot{\mathcal{W}}_t \equiv 0$ and thus
  \[ 0 = \ddot{\mathcal{W}}_0 = \int_X \left| A \right|^2_g F \,\Omega, \]
  thanks to proposition \ref{TSec-Var-W} and lemma \ref{exprO}. We conclude
  the required inclusion.
\end{proof}

\section{Appendix A}

\subsection{The first variation of Perelman's $\mathcal{W}$ functional}

We give a short proof of Perelman's first variation formula \cite{Per} for the
$\mathcal{W}$ functional based on the identity (\ref{var-OM-Ric}). Let $(g_t,
\Omega_t)_t \subset \mathcal{M} \times \mathcal{V}_1$ be a smooth family and
set $f_t \assign \log \frac{dV_{g_t}}{\Omega_t}$. Then
\begin{eqnarray*}
  \frac{d}{dt} \hspace{0.25em} \mathcal{W} (g_t, \Omega_t) & = & \frac{d}{dt}
  \hspace{0.25em} \int_X \left[ \tmop{Tr}_{\mathbbm{R}} (g_t^{- 1} h_t)
  \hspace{0.75em} + \hspace{0.75em} 2 f_t \right] \Omega_t\\
  &  & \\
  & = & \int_X \left[ \tmop{Tr}_{\mathbbm{R}} \left( - \hspace{0.25em}
  \dot{g}_t^{\ast} h^{\ast}_t \hspace{0.75em} + \hspace{0.75em}
  \dot{h}^{\ast}_t \right) \hspace{0.75em} + \hspace{0.75em} \tmop{Tr}_{g_t}
  \dot{g} - 2 \dot{\Omega}^{\ast}_t \right] \Omega_t\\
  &  & \\
  & + & \int_X \left[ \tmop{Tr}_{g_t} h_t \hspace{0.75em} + \hspace{0.75em} 2
  f_t \right]  \dot{\Omega}_t\\
  &  & \\
  & = & \int_X \left[ - \left\langle \dot{g}_t, \hspace{0.25em} h_t
  \right\rangle_{g_t} + \left\langle g_t, \frac{d}{dt} \hspace{0.25em}
  \tmop{Ric}_{g_t} (\Omega_t) \right\rangle_{g_t} - 2 \dot{\Omega}^{\ast}_t
  \right] \Omega_t \\
  &  & \\
  & + & \int_X \left[ \tmop{Tr}_{g_t} h_t \hspace{0.75em} + \hspace{0.75em} 2
  f_t \right]  \dot{\Omega}_t .
\end{eqnarray*}
Using the variation formula (\ref{var-OM-Ric}) and integrating by parts we
infer
\begin{eqnarray*}
  \hspace{0.25em} \int_X \left\langle g_t, \frac{d}{dt} \hspace{0.25em}
  \tmop{Ric}_{g_t} (\Omega_t) \right\rangle_{g_t} \Omega_t & = & \int_X \left[
  - \hspace{0.75em} \frac{1}{2}  \left\langle g_t,
  \nabla_{g_t}^{\ast_{\Omega}} \mathcal{D}_{g_t}  \dot{g}_t
  \right\rangle_{g_t} + \Delta_{g_t}  \dot{\Omega}^{\ast}_t  \right]
  \Omega_t\\
  &  & \\
  & = & - \int_X \left[ \frac{1}{2}  \left\langle \nabla_{g_t} g_t,
  \mathcal{D}_{g_t}  \dot{g}_t \right\rangle_{g_t} + \left\langle \nabla_{g_t}
  \dot{\Omega}^{\ast}_t, \nabla_{g_t} f_t \right\rangle_{g_t} \right]
  \Omega_t\\
  &  & \\
  & = & - \int_X \dot{\Omega}^{\ast}_t \Delta^{\Omega_t}_{g_t} f_t \Omega_t,
\end{eqnarray*}
which implies Perelman's first variation formula
\begin{eqnarray*}
  \frac{d}{dt} \hspace{0.25em} \mathcal{W} (g_t, \Omega_t) & = & - \int_X
  \left[ \left\langle \dot{g}_t, h_t \right\rangle_{g_t} - 2
  \dot{\Omega}^{\ast}_t (H_t - 1) \right] \Omega_t \\
  &  & \\
  & = & - \int_X \left[ \left\langle \dot{g}_t, h_t \right\rangle_{g_t} - 2
  \dot{\Omega}^{\ast}_t  \underline{H}_t \right] \Omega_t,
\end{eqnarray*}
since $\int_X \dot{\Omega}_t = 0$.

\subsection{Basic complex identities}\label{Integ-parts}

We provide a useful expression of the hermitian product $\left\langle \cdot,
\cdot \right\rangle_{\omega}$ on $T^{\ast}_X \otimes_{\mathbbm{R}}
\mathbbm{C}$, which is the sesquilinear extension of the dual of $g$. We
observe first that for any $\xi \in T_X \otimes_{\mathbbm{R}} \mathbbm{C}$ and
any $\alpha \in T^{\ast}_X \otimes_{\mathbbm{R}} \mathbbm{C}$ hold the
elementary equalities
\begin{eqnarray*}
  0 & = & \xi \neg (\alpha \wedge \omega^n) = (\alpha \cdot \xi) \omega^n -
  \alpha \wedge (\xi \neg \omega^n) .
\end{eqnarray*}
We obtain the formula
\begin{equation}
  \label{contract-wg-frml}  (\alpha \cdot \xi) \omega^n = n \alpha \wedge (\xi
  \neg \omega) \wedge \omega^{n - 1},
\end{equation}
and thus
\begin{equation}
  \label{Trac-forml} 2 \alpha \cdot \xi = \tmop{Tr}_{\omega} \left[ \alpha
  \wedge (\xi \neg \omega) \right] .
\end{equation}
Using (\ref{Trac-forml}) we deduce that for all $\alpha, \beta \in T^{\ast}_X
\otimes_{\mathbbm{R}} \mathbbm{C}$ hold the equalities
\begin{eqnarray*}
  2 \left\langle \alpha, \beta \right\rangle_{\omega} & = & 2 \alpha \cdot
  \bar{\beta}_g^{\ast} = \tmop{Tr}_{\omega} \left[ \alpha \wedge (
  \bar{\beta}_g^{\ast} \neg \omega) \right],
\end{eqnarray*}
where $\bar{\beta}_g^{\ast} \assign g^{- 1} \bar{\beta}$. Then the identity
$\bar{\beta} J = - \bar{\beta}_g^{\ast} \neg \omega$ implies the formula
\begin{eqnarray*}
  2 \left\langle \alpha, \beta \right\rangle_{\omega} & = & -
  \tmop{Tr}_{\omega} \left[ \alpha \wedge ( \bar{\beta} J) \right] .
\end{eqnarray*}
Thus in the case $\alpha, \beta \in \Lambda_J^{1, 0} T^{\ast}_X$ we deduce the
identities
\begin{eqnarray*}
  2 \left\langle \alpha, \beta \right\rangle_{\omega} & = & \tmop{Tr}_{\omega}
  \left( i \alpha \wedge \bar{\beta} \right),\\
  &  & \\
  \overline{ \left\langle \alpha, \beta \right\rangle_{\omega}} & = &
  \left\langle \overline{\alpha}, \bar{\beta} \right\rangle_{\omega} .
\end{eqnarray*}
We show now the following integration by parts formulas

\begin{lemma}
  For any $u, v \in C^{\infty} (X, \mathbbm{C})$ holds the integration by parts
  identity
  \begin{equation}
    \label{cpx-int-part} \int_X \left[ \Delta^{\Omega}_{g, J} u \cdot
    \overline{v} + \overline{\Delta^{\Omega}_{g, J} u} \cdot v \right] \Omega
    = 2 \int_X g (\nabla_{g, J}  \overline{u}, \nabla_{g, J}  \overline{v})
    \Omega .
  \end{equation}
  If $u \in C^{\infty} (X, \mathbbm{R})$ then holds also the integration by
  parts identity
  \begin{equation}
    \label{rlcpx-int-part} \int_X \left[ \Delta^{\Omega}_{g, J} u \cdot v +
    \overline{\Delta^{\Omega}_{g, J} u} \cdot \overline{v} \right] \Omega = 2
    \int_X g (\nabla_g u, \nabla_{g, J} v) \Omega .
  \end{equation}
\end{lemma}

\begin{proof}
  Using the complex decomposition (\ref{dec-Cx-grad}) and the fact that
  hermitian product $\left\langle \cdot, \cdot \right\rangle_{\omega}$ on
  $T^{\ast}_X \otimes_{\mathbbm{R}} \mathbbm{C}$ is the sesquilinear extension
  of the dual of $g$, we deduce
  \begin{eqnarray*}
    g (\nabla_{g, J} u, \nabla_{g, J} v) & = & \left\langle \partial_J
    \overline{u} + \overline{\partial}_J u, \partial_J \overline{v} +
    \overline{\partial}_J v \right\rangle_g \\
    &  & \\
    & = & \left\langle \partial_J \overline{u} + \overline{\partial}_J u,
    \partial_J \overline{v} + \overline{\partial}_J v \right\rangle_{\omega}\\
    &  & \\
    & = & \left\langle \partial_J \overline{u}, \partial_J \overline{v}
    \right\rangle_{\omega} + \left\langle \overline{\partial}_J u,
    \overline{\partial}_J v \right\rangle_{\omega}\\
    &  & \\
    & = & \left\langle \partial_J \overline{u}, \partial_J \overline{v}
    \right\rangle_{\omega} + \overline{ \left\langle \partial_J \overline{u},
    \partial_J \overline{v} \right\rangle_{\omega}} .
  \end{eqnarray*}
  Integrating by parts and taking the conjugate we infer the identity
  \begin{equation}
    \label{genCx-int-part} 2 \int_X g (\nabla_{g, J} u, \nabla_{g, J} v)
    \Omega = \int_X \left[ \Delta^{\Omega}_{g, J}  \overline{u} \cdot v +
    \overline{\Delta^{\Omega}_{g, J}  \overline{u} } \cdot \overline{v}
    \right] \Omega .
  \end{equation}
  Replacing $u$ with $\overline{u}$, $v$ with $\overline{v}$ in
  (\ref{genCx-int-part}) we obtain (\ref{cpx-int-part}). In the case $u \in
  C^{\infty} (X, \mathbbm{R})$ formula (\ref{genCx-int-part}) implies directly
  (\ref{rlcpx-int-part}).
\end{proof}

We show now that for all $\alpha, \beta \in \Lambda_J^{1, 1} T^{\ast}_X \cap
\Lambda_{\mathbbm{R}}^2 T^{\ast}_X$ holds the identity
\begin{equation}
  \label{WedgTrac} 4 n (n - 1)  \frac{\alpha \wedge \beta \wedge \omega^{n -
  2}}{\omega^n} = \tmop{Tr}_{\omega} \alpha \tmop{Tr}_{\omega} \beta - 2
  \left\langle \alpha, \beta \right\rangle_{\omega} .
\end{equation}
Indeed we consider the local expressions
\[ \omega = \frac{i}{2} \sum_k \zeta^{\ast}_k \wedge \bar{\zeta}^{\ast}_k,
   \hspace{0.75em} \hspace{0.75em} \alpha = i \sum_{k, l} \alpha_{k \bar{l}}
   \hspace{0.25em} \zeta^{\ast}_k \wedge \bar{\zeta}^{\ast}_l, \hspace{0.75em}
   \hspace{0.75em} \beta = i \sum_{k, l} \beta_{k \bar{l}} \hspace{0.25em}
   \zeta^{\ast}_k \wedge \bar{\zeta}^{\ast}_l, \]
and we set
\begin{eqnarray*}
  \Psi : = \alpha \wedge \beta & = & \sum_{k_1, k_2, l_1, l_2} \alpha_{k_1
  \bar{l}_1} \beta_{k_2 \bar{l}_2} \hspace{0.25em} \zeta^{\ast}_{k_1} \wedge
  \zeta^{\ast}_{k_2} \wedge \bar{\zeta}^{\ast}_{l_1} \wedge
  \bar{\zeta}^{\ast}_{l_2}\\
  &  & \\
  & = & \sum_{|K| = |L| = 2} \Psi_{K, L} \hspace{0.25em} \zeta^{\ast}_K
  \wedge \bar{\zeta}^{\ast}_L,
\end{eqnarray*}
where $K = (k_1, k_2), \hspace{0.25em} 1 \leq k_1 < k_2 \leq n$ and the same
holds for $L$. Explicitly the coefficients $\Psi_{K, L}$ are given by the
expression
\[ \Psi_{K, L} = \alpha_{k_1 \bar{l}_1} \beta_{k_2 \bar{l}_2} + \alpha_{k_2
   \bar{l}_2} \beta_{k_1 \bar{l}_1} - \alpha_{k_1 \bar{l}_2} \beta_{k_2
   \bar{l}_1} - \alpha_{k_2 \bar{l}_1} \beta_{k_1 \bar{l}_2} . \]
We conclude the identity
\begin{eqnarray*}
  4 n (n - 1)  \frac{\Psi \wedge \omega^{n - 2}}{\omega^n} & = & 16 \sum_{|L|
  = 2} \Psi_{L, L} = 16 \sum_{k, l} \alpha_{k \bar{k}} \beta_{l \bar{l}} - 16
  \sum_{k, l} \alpha_{k \bar{l}} \beta_{l \bar{k}}\\
  &  & \\
  & = & \tmop{Tr}_{\omega} \alpha \tmop{Tr}_{\omega} \beta - 2 \left\langle
  \alpha, \beta \right\rangle_{\omega} .
\end{eqnarray*}

\subsection{Action of the curvature on alternating 2-forms}

We observe that as in the symmetric case we can define an action of the
curvature operator over alternating 2-forms as follows
\begin{eqnarray*}
  ( \mathcal{R}_g \ast \alpha) (\xi, \eta) & \assign & - \tmop{Tr}_g \left[
  \alpha \left( \mathcal{R}_g \left( \xi, \cdot \left) \eta, \cdot \right)
  \right], \right. \right.
\end{eqnarray*}
for any $\alpha \in \Lambda^2 T^{\ast}_X$. The tensor $\mathcal{R}_g \ast
\alpha$ is anti-symmetric. In fact let $(e_k)_k$ be a $g (x)$-orthonormal base
of $T_{X, x}$ and consider the local expression $\alpha^{\ast}_g = A_{l, k}
e^{\ast}_k \otimes e_l$, with $A_{l, k} = - A_{k, l}$. Then
\begin{eqnarray*}
  ( \mathcal{R}_g \ast \alpha) (\xi, \eta) & = & - g \left( \alpha^{\ast}_g 
  \mathcal{R}_g (\xi, e_k) \eta, e_k \right) \\
  &  & \\
  & = & g \left( \mathcal{R}_g (\xi, e_k) \eta, \alpha^{\ast}_g e_k \right)\\
  &  & \\
  & = & R_g (\xi, e_k, \alpha_g^{\ast} e_k, \eta) \\
  &  & \\
  & = & R_g (\xi, e_k, A_{l, k} e_l, \eta)\\
  &  & \\
  & = & - R_g (\xi, A_{k, l} e_k, e_l, \eta)\\
  &  & \\
  & = & - R_g (\xi, \alpha_g^{\ast} e_l, e_l, \eta)\\
  &  & \\
  & = & - R_g (\eta, e_l, \alpha_g^{\ast} e_l, \xi)\\
  &  & \\
  & = & - ( \mathcal{R}_g \ast \alpha) (\eta, \xi),
\end{eqnarray*}
thanks to the symmetry properties of the curvature form. We observe also that
the previous computation shows the identity
\begin{eqnarray*}
  ( \mathcal{R}_g \ast \alpha) (\xi, \eta) & = & - R_g (\xi, e_k, \eta,
  \alpha_g^{\ast} e_k)\\
  &  & \\
  & = & - g (\mathcal{R}_g (\xi, e_k) \alpha_g^{\ast} e_k, \eta)\\
  &  & \\
  & = & - g \left( \left( \mathcal{R}_g \ast \alpha^{\ast}_g \right) \xi,
  \eta \right),
\end{eqnarray*}
i.e.
\begin{equation}
  \label{alt-ast-curv-Id}  (\mathcal{R}_g \ast \alpha)_g^{\ast} =
  -\mathcal{R}_g \ast \alpha_g^{\ast} .
\end{equation}
On the other hand using the algebraic Bianchi identity we obtain the
equalities

\begin{eqnarray*}
  ( \mathcal{R}_g \ast \alpha) (\xi, \eta) & = & - g \left( \alpha^{\ast}_g 
  \mathcal{R}_g (\xi, e_k) \eta, e_k \right)\\
  &  & \\
  & = & g \left( \alpha^{\ast}_g  \mathcal{R}_g (e_k, \eta) \xi, e_k \right)
  + g \left( \alpha^{\ast}_g  \mathcal{R}_g (\eta, \xi) e_k, e_k \right)\\
  &  & \\
  & = & - g \left( \alpha^{\ast}_g  \mathcal{R}_g (\eta, e_k) \xi, e_k
  \right) - g \left( \mathcal{R}_g (\eta, \xi) e_k, \alpha^{\ast}_g e_k
  \right)\\
  &  & \\
  & = & ( \mathcal{R}_g \ast \alpha) (\eta, \xi) + g \left( \mathcal{R}_g
  (\eta, \xi) \alpha^{\ast}_g e_k, e_k \right)\\
  &  & \\
  & = & - ( \mathcal{R}_g \ast \alpha) (\xi, \eta) - \tmop{Tr}_{\mathbbm{R}}
  \left[ \mathcal{R}_g (\xi, \eta) \alpha^{\ast}_g \right],
\end{eqnarray*}
and thus the formula
\begin{equation}
  \label{curv-act-form}  ( \mathcal{R}_g \ast \alpha) (\xi, \eta) = -
  \frac{1}{2} \tmop{Tr}_{\mathbbm{R}} \left[ \mathcal{R}_g (\xi, \eta)
  \alpha^{\ast}_g \right] .
\end{equation}
We assume further that $(X, J, g)$ is K\"ahler and $\alpha$ is
$J$-anti-invariant. In this case $\alpha^{\ast}_g = A$ is $J$-anti-linear and
so is the endomorphism $\mathcal{R}_g (\xi, \eta) \alpha^{\ast}_g$. We deduce
\begin{equation}
  \label{vanish-curv-anti}  \mathcal{R}_g \ast \alpha = 0 \text{, \ \ i.e. \ \
  } \mathcal{R}_g \ast A = 0 .
\end{equation}
thanks to the identity (\ref{alt-ast-curv-Id}).

\subsection{Weighted Weitzenb\"ock formula for alternating 2-forms}

We show the weighted Weitzenb\"ock type formula
\begin{equation}
  \label{Wei-alt} \Delta^{\Omega}_{d, g} \alpha = \Delta^{\Omega}_g \alpha +
  2\mathcal{R}_g \ast \alpha + \alpha \tmop{Ric}^{\ast}_g (\Omega) +
  \tmop{Ric}_g (\Omega) \alpha_g^{\ast},
\end{equation}
for any alternating 2-form $\alpha$ over a Riemannian manifold. For this
purpose we fix an arbitrary point $x_0$ and we choose the vector fields $\xi$
and $\eta$ such that $0 = \nabla_g \xi (x_0) = \nabla_g \eta (x_0)$. Let
$(e_k)_k$ be a $g$-orthonormal local frame such that $\nabla_g e_k  (x_0) =
0$. Then at the point $x_0$ holds the identities
\begin{eqnarray*}
  d \nabla_g^{\ast} \alpha (\xi, \eta) & = & \nabla_{g, \xi} \nabla_g^{\ast}
  \alpha \cdot \eta - \nabla_{g, \eta} \nabla_g^{\ast} \alpha \cdot \xi\\
  &  & \\
  & = & \nabla_{g, \xi} \left[ \nabla_g^{\ast} \alpha \cdot \eta \right] -
  \nabla_{g, \eta} \left[ \nabla_g^{\ast} \alpha \cdot \xi \right]\\
  &  & \\
  & = & - \nabla_{g, \xi} \left[ \nabla_{g, e_k} \alpha (e_k, \eta) \right] +
  \nabla_{g, \eta} \left[ \nabla_{g, e_k} \alpha (e_k, \xi) \right]\\
  &  & \\
  & = & - \nabla_{g, \xi} \nabla_{g, e_k} \alpha (e_k, \eta) + \nabla_{g,
  \eta} \nabla_{g, e_k} \alpha (e_k, \xi),
\end{eqnarray*}
and
\begin{eqnarray*}
  \nabla_g^{\ast} d \alpha (\xi, \eta) & = & - \nabla_{g, e_k} d \alpha (e_k,
  \xi, \eta)\\
  &  & \\
  & = & - \nabla_{g, e_k}  \left[ d \alpha (e_k, \xi, \eta) \right]\\
  &  & \\
  & = & - \nabla_{g, e_k}  \left[ \nabla_{g, e_k} \alpha (\xi, \eta) -
  \nabla_{g, \xi} \alpha (e_k, \eta) + \nabla_{g, \eta} \alpha (e_k, \xi)
  \right]\\
  &  & \\
  & = & - \nabla_{g, e_k} \nabla_{g, e_k} \alpha (\xi, \eta) + \nabla_{g,
  e_k} \nabla_{g, \xi} \alpha (e_k, \eta) - \nabla_{g, e_k} \nabla_{g, \eta}
  \alpha (e_k, \xi) .
\end{eqnarray*}
We remind now that for any vector fields $\mu, \zeta$ such that $[\mu, \zeta]
(x_0) = 0$ holds the identity at the point $x_0$
\begin{eqnarray*}
  \nabla_{g, \mu} \nabla_{g, \zeta} \alpha - \nabla_{g, \zeta} \nabla_{g, \mu}
  \alpha & = & \mathcal{R}_g (\zeta, \mu) \neg \alpha,
\end{eqnarray*}
where the contraction operation $T \neg : \Lambda^2 T^{\ast}_X \longrightarrow
\Lambda^2 T^{\ast}_X$ associated to an endomorphism $T \in \tmop{End} (T_X)$
is defined by the formula
\begin{eqnarray*}
  (T \neg \alpha)  (\xi, \eta) & \assign & \alpha (T \xi, \eta) + \alpha (\xi,
  T \eta) .
\end{eqnarray*}
We deduce
\begin{eqnarray*}
  \left( \nabla_{g, e_k} \nabla_{g, \xi} \alpha - \nabla_{g, \xi} \nabla_{g,
  e_k} \right) (e_k, \eta) & = & \alpha \left( \mathcal{R}_g (\xi, e_k) e_k,
  \eta \right) + \alpha \left( e_k, \mathcal{R}_g (\xi, e_k) \eta \right)\\
  &  & \\
  & = & \left[ \alpha \tmop{Ric}^{\ast} (g) + (\mathcal{R}_g \ast \alpha)
  \right] (\xi, \eta),
\end{eqnarray*}
and also
\begin{eqnarray*}
  \left( \nabla_{g, \eta} \nabla_{g, e_k} \alpha - \nabla_{g, e_k} \nabla_{g,
  \eta} \alpha \right) (e_k, \xi) & = & - \left[ \alpha \tmop{Ric}^{\ast} (g)
  + (\mathcal{R}_g \ast \alpha) \right] (\eta, \xi) \\
  &  & \\
  & = & - g \left( \alpha^{\ast}_g \tmop{Ric}^{\ast} (g) \eta, \xi \right) +
  (\mathcal{R}_g \ast \alpha) (\xi, \eta)\\
  &  & \\
  & = & g (\eta, \tmop{Ric}^{\ast} (g) \alpha^{\ast}_g \xi) + (\mathcal{R}_g
  \ast \alpha) (\xi, \eta)\\
  &  & \\
  & = & \left[ \tmop{Ric} (g) \alpha^{\ast}_g + (\mathcal{R}_g \ast \alpha)
  \right] (\xi, \eta) .
\end{eqnarray*}
Summing up the terms $d \nabla_g^{\ast} \alpha (\xi, \eta)$ and
$\nabla_g^{\ast} d \alpha (\xi, \eta)$ and using these last identities we infer
the formula (\ref{Wei-alt}) in the case $\Omega = C d V_g$. In order to obtain
the general case we observe the decompositions
\begin{eqnarray*}
  d \nabla^{\ast_{\Omega}}_g \alpha & = & d \nabla_g^{\ast} \alpha + d
  (\nabla_g f \neg \alpha),\\
  &  & \\
  \nabla^{\ast_{\Omega}}_g d \alpha & = & \nabla_g^{\ast} d \alpha + \nabla_g
  f \neg d \alpha,
\end{eqnarray*}
and the identities at the point $x_0$,
\begin{eqnarray*}
  d (\nabla_g f \neg \alpha)  (\xi, \eta) & = & \nabla_{g, \xi}  (\nabla_g f
  \neg \alpha) \cdot \eta - \nabla_{g, \eta}  (\nabla_g f \neg \alpha) \cdot
  \xi\\
  &  & \\
  & = & \nabla_{g, \xi} \left[ \alpha (\nabla_g f, \eta) \right] - \nabla_{g,
  \eta} \left[ \alpha (\nabla_g f, \xi) \right]\\
  &  & \\
  & = & \nabla_{g, \xi} \alpha (\nabla_g f, \eta) + \alpha (\nabla^2_{g, \xi}
  f, \eta) \\
  &  & \\
  & - & \nabla_{g, \eta} \alpha (\nabla_g f, \xi) - \alpha (\nabla^2_{g,
  \eta} f, \xi)\\
  &  & \\
  & = & \nabla_{g, \xi} \alpha (\nabla_g f, \eta) + \alpha (\nabla^2_{g, \xi}
  f, \eta) \\
  &  & \\
  & - & \nabla_{g, \eta} \alpha (\nabla_g f, \xi) - \nabla_g d f
  \alpha_g^{\ast}  (\xi, \eta),\\
  &  & \\
  (\nabla_g f \neg d \alpha)  (\xi, \eta) & = & (\nabla_g f \neg \nabla_g
  \alpha)  (\xi, \eta) - \nabla_{g, \xi} \alpha (\nabla_g f, \eta) +
  \nabla_{g, \eta} \alpha (\nabla_g f, \xi) .
\end{eqnarray*}
Summing up we infer the required formula (\ref{Wei-alt}).

\section{Appendix B}

\subsection{Reformulation of the weighted complex Bochner identity
(\ref{dbar-OmBoch-fnctSOL})}

We define the complex operator
\begin{eqnarray*}
  \Delta^{\Omega}_{g, - J} & \assign & \overline{\Delta^{\Omega}_{g, J}} .
\end{eqnarray*}
With this notation the weighted complex Bochner type identity
(\ref{dbar-OmBoch-fnctSOL}) rewrites also as
\begin{eqnarray*}
  2 \Delta^{\Omega, - J}_{T_{X, g}} \nabla_{g, J} u & = & \nabla_{g, J} 
  (\Delta^{\Omega}_{g, - J} - 2\mathbbm{I}) u,
\end{eqnarray*}
for all $u \in C^{\infty} (X, \mathbbm{C})$. We show now that the fundamental
identity (\ref{dbar-OmBoch-fnctSOL}) implies an other important formula. We
need a few preliminaries.

\begin{lemma}
  For any $u, v \in C^{\infty} (X, \mathbbm{C})$ holds the integration by parts
  identity
  \[ \int_X \Delta^{\Omega}_{g, - J} u \cdot \overline{v} \Omega = 2 \int_X
     \left\langle \nabla_{g, J} u, \nabla_{g, J} v \right\rangle_{\omega}
     \Omega . \]
\end{lemma}

\begin{proof}
  We define the complex components of the $g$-gradient as
  \begin{eqnarray*}
    \nabla^{1, 0}_{g, J} u & \assign & (\nabla_g u)_J^{1, 0} \in C^{\infty}
    (X, T^{1, 0}_{X, J}),\\
    &  & \\
    \nabla^{0, 1}_{g, J} u & \assign & (\nabla_g u)_J^{0, 1} \in C^{\infty}
    (X, T^{0, 1}_{X, J}) .
  \end{eqnarray*}
  With these notations holds the decomposition formula
  \begin{equation}
    \label{Dec-Cx-grad} \nabla_{g, J} u = \nabla^{1, 0}_{g, J} u + \nabla^{0,
    1}_{g, J}  \overline{u} .
  \end{equation}
  We observe that for all $\xi, \eta \in T_X$ holds the identity
  \begin{eqnarray*}
    \left\langle \xi, \eta \right\rangle_{\omega} \equiv h (\xi, \eta) & = & 2
    i \omega (\eta_J^{0, 1}, \xi_J^{1, 0}) .
  \end{eqnarray*}
  This combined with (\ref{Dec-Cx-grad}) implies
  \begin{eqnarray*}
    \left\langle \nabla_{g, J} u, \nabla_{g, J} v \right\rangle_{\omega} & = &
    2 i \omega (\nabla^{0, 1}_{g, J}  \overline{v}, \nabla^{1, 0}_{g, J} u) .
  \end{eqnarray*}
  We observe now that the complex spiting of the $g$-gradient
  \begin{eqnarray*}
    \nabla_g u & = & \nabla^{1, 0}_{g, J} u + \nabla^{0, 1}_{g, J} u
  \end{eqnarray*}
  implies the identities
  \begin{eqnarray*}
    \nabla^{1, 0}_{g, J} u \neg \,\omega & = & i \overline{\partial}_J u,\\
    &  & \\
    \nabla^{0, 1}_{g, J} u \neg \,\omega & = & - i \partial_J u .
  \end{eqnarray*}
  Using this and the identity (\ref{Trac-forml}) we deduce
  \begin{eqnarray*}
    \left\langle \nabla_{g, J} u, \nabla_{g, J} v \right\rangle_{\omega} & = &
    2 \partial_J \overline{v} \cdot \nabla^{1, 0}_{g, J} u\\
    &  & \\
    & = & \tmop{Tr}_{\omega} \left[ \partial_J \overline{v} \wedge
    (\nabla^{1, 0}_{g, J} u \neg \omega) \right]\\
    &  & \\
    & = & \tmop{Tr}_{\omega} \left[ i \partial_J \overline{v} \wedge
    \overline{\partial}_J u \right]\\
    &  & \\
    & = & 2 \left\langle \partial_J \overline{v}, \partial_J \overline{u}
    \right\rangle_{\omega} .
  \end{eqnarray*}
  We infer the equalities
  \begin{eqnarray*}
    \int_X \left\langle \nabla_{g, J} u, \nabla_{g, J} v
    \right\rangle_{\omega} \Omega & = & 2 \overline{\int_X \left\langle
    \partial_J \overline{u}, \partial_J \overline{v} \right\rangle_{\omega}
    \Omega}\\
    &  & \\
    & = & \overline{\int_X \Delta^{\Omega}_{g, J}  \overline{u} \cdot v
    \Omega}\\
    &  & \\
    & = & \int_X \Delta^{\Omega}_{g, - J} u \cdot \overline{v} \Omega .
  \end{eqnarray*}
\end{proof}

We equip $C_{\Omega}^{\infty} (X, \mathbbm{C})_0$ with the
$L_{\Omega}^2$-product (\ref{L2Om-prod}) and the space \\
$C^{\infty} (X,
\Lambda_J^{0, 1} T^{\ast}_X \otimes_{\mathbbm{C}} T_{X, J})$ with the
$L_{\omega, \Omega}^2$-hermitian product (\ref{L2omOm-prod}). Then the formal
adjoint of $\mathcal{H}^{0, 1}_{g, J}$ with respect to such products
\begin{eqnarray*}
  (\mathcal{H}^{0, 1}_{g, J})^{\ast_{\omega, \Omega}} : C^{\infty} (X,
  \Lambda_J^{0, 1} T^{\ast}_X \otimes_{\mathbbm{C}} T_{X, J}) &
  \longrightarrow & C_{\Omega}^{\infty} (X, \mathbbm{C})_0,\\
  &  & \\
  \int_X \left\langle \mathcal{H}^{0, 1}_{g, J} u, A \right\rangle_{\omega}
  \Omega & = & \int_X u \cdot \overline{ (\mathcal{H}^{0, 1}_{g,
  J})^{\ast_{\omega, \Omega}} A} \Omega,
\end{eqnarray*}
satisfies the identity
\begin{eqnarray*}
  (\mathcal{H}^{0, 1}_{g, J})^{\ast_{\omega, \Omega}} & = & \nabla_{g,
  J}^{\ast_{\omega, \Omega}}  \overline{\partial}^{\ast_{g, \Omega}}_{T_{X,
  J}} .
\end{eqnarray*}
Moreover lemma implies the identity
\begin{eqnarray*}
  \Delta^{\Omega}_{g, - J} & = & \nabla_{g, J}^{\ast_{\omega, \Omega}}
  \nabla_{g, J} .
\end{eqnarray*}
Then the complex Bochner type identity (\ref{dbar-OmBoch-fnctSOL}) implies
\begin{eqnarray*}
  2 (\mathcal{H}^{0, 1}_{g, J})^{\ast_{\omega, \Omega}} \mathcal{H}^{0, 1}_{g,
  J}  \overline{v} & = & 2 \nabla_{g, J}^{\ast_{\omega, \Omega}} 
  \overline{\partial}^{\ast_{g, \Omega}}_{T_{X, J}} \overline{\partial}_{T_{X,
  J}} \nabla_{g, J}  \overline{v}\\
  &  & \\
  & = & \overline{\Delta^{\Omega}_{g, J}  (\Delta^{\Omega}_{g, J} -
  2\mathbbm{I}) v},
\end{eqnarray*}
or in other terms
\begin{eqnarray*}
  2 (\mathcal{H}^{0, 1}_{g, J})^{\ast_{\omega, \Omega}} \mathcal{H}^{0, 1}_{g,
  J} u & = & \Delta^{\Omega}_{g, - J}  (\Delta^{\Omega}_{g, - J} -
  2\mathbbm{I}) u,
\end{eqnarray*}
for all $u \in C^{\infty} (X, \mathbbm{C})$.

\subsection{A Poisson structure on the first eigenspace of
$\Delta^{\Omega}_{g, J}$}\label{Poiss-Strc}

For any complex valued function $u$ and symplectic form $\omega$ we define the
complex vector field
\begin{eqnarray*}
  (d u)_{\omega}^{\ast} & \assign & \omega^{- 1} d u = - J \nabla_g u,
\end{eqnarray*}
and the Poisson bracket
\begin{eqnarray*}
  \left\{ u, v \right\}_{\omega} & \assign & d v \cdot (d u)_{\omega}^{\ast} =
  - \omega ((d u)_{\omega}^{\ast}, (d v)_{\omega}^{\ast}) .
\end{eqnarray*}
We define the Poisson bracket over the space $C_{\Omega}^{\infty} (X,
\mathbbm{C})$ as
\begin{eqnarray*}
  \left\{ u, v \right\}_{\omega, \Omega} & \assign & \left\{ u, v
  \right\}_{\omega} - \int_X \left\{ u, v \right\}_{\omega} \Omega .
\end{eqnarray*}
With these notations holds the following lemma (see also \cite{Fut}, \cite{Gau}).

\begin{lemma}
  \label{Pisson-Holo}Let $(X, J)$ be a Fano manifold and let $g$ be a
  $J$-invariant K\"ahler metric such that $\omega \assign g J \in 2 \pi c_1
  (X, \left[ J \right])$. Let also $\Omega > 0$ be the unique smooth volume
  form with$\int_X \Omega = 1$ such that $\tmop{Ric}_J (\Omega) = \omega$.
  
  {\tmstrong{$A)$}} Then the map
  \begin{eqnarray*}
    \chi : \left( \overline{\tmop{Ker} (\Delta^{\Omega}_{g, J} -
    2\mathbbm{I})}, i \left\{ \cdot, \cdot \right\}_{\omega, \Omega} \right) &
    \longrightarrow & \left( H^0 (X, T_{X, J}), \left[ \cdot, \cdot \right]
    \right)\\
    &  & \\
    u & \longmapsto & \nabla_{g, J} u,
  \end{eqnarray*}
  is well defined and it represents an isomorphism of complex lie algebras.
  
  {\tmstrong{$B)$}} The first eigenvalue $\lambda_1 (\Delta^{\Omega}_{g, J})$
  of the operator $\Delta^{\Omega}_{g, J}$ satisfies the estimate $\lambda_1
  (\Delta^{\Omega}_{g, J}) \geqslant 2$, with equality in the case $H^0 (X,
  T_{X, J})\neq 0$.
  
  {\tmstrong{$C)$}} If we set $\tmop{Kill}_g \assign \tmop{Lie}
  (\tmop{Isom}^0_g)$ then the map
  \begin{equation}
    \label{SpRealEig-Killmap} J \nabla_g : \tmop{Ker}_{\mathbbm{R}}
    (\Delta^{\Omega}_g - 2\mathbbm{I}) \longrightarrow \tmop{Kill}_g,
  \end{equation}
  is well defined and it represents an isomorphism of real vector spaces.
  
  {\tmstrong{$D)$}} The hermitian form
  \begin{eqnarray*}
    (u, v) & \longmapsto & \int_X i \left\{ u, \bar{v} \right\}_{\omega}
    \Omega,
  \end{eqnarray*}
  over $\overline{\tmop{Ker} (\Delta^{\Omega}_{g, J} - 2\mathbbm{I})}$ is
  non-negative and let $(\mu_j)^N_{j = 0} \subset \mathbbm{R}_{\geqslant 0}$,
  $\mu_0 = 0$, be its spectrum with respect to the $L_{\Omega}^2$-product. If
  $g$ is a $J$-invariant K\"ahler-Ricci soliton then holds the decomposition
  \begin{eqnarray*}
    H^0 (X, T_{X, J}) & = & \bigoplus_{j = 0}^N V_{\mu_j},\\
    &  & \\
    V_{\mu_j} & \assign & \left\{ \xi \in H^0 (X, T_{X, J}) \mid \left[
    \nabla_g f, \xi \right] = \mu_j \xi \right\},\\
    &  & \\
    V_0 & = & \tmop{Kill}_g \oplus J \tmop{Kill}_g .
  \end{eqnarray*}
\end{lemma}

\begin{proof}
  {\tmstrong{Step A}}. In this step we show the statement {\tmstrong{A}}. The
  fact that $\chi$ is an isomorphism follows from corollary
  \ref{Eigenf-HolVct}. We show now that $\chi$ is also a morphism of complex
  Lie algebras. Let
  \begin{eqnarray*}
    \mathbbm{K}_{\pm} & \assign & \tmop{Ker} (\Delta^{\Omega}_{g, \pm J} -
    2\mathbbm{I}) .
  \end{eqnarray*}
  For any $\xi \in H^0 (X, T_{X, J})$ we denote $u^{\xi} \assign \chi^{- 1}
  (\xi) \in \mathbbm{K}_-$ and we decompose $u^{\xi} = u_1^{\xi} + i
  u_2^{\xi}$, with $u_j^{\xi} \in C_{\Omega}^{\infty} (X, \mathbbm{R})_0$. For
  any $u, v \in \mathbbm{K}_- = \overline{\mathbbm{K}_+}$ we set $\xi \assign
  \nabla_{g, J} u$, $\eta \assign \nabla_{g, J} v$ and as in \cite{Gau} we
  observe the identities
  \begin{eqnarray*}
    L_{\left[ \xi, \eta \right]} \omega & = & L_{\xi} L_{\eta} \omega -
    L_{\eta} L_{\xi} \omega\\
    &  & \\
    & = & 2 L_{\xi} i \partial_J \overline{\partial}_J u^{\eta}_1 - 2
    L_{\eta} i \partial_J \overline{\partial}_J u^{\xi}_1\\
    &  & \\
    & = & 2 i \partial_J \overline{\partial}_J  \left( \xi . u^{\eta}_1 -
    \eta . u^{\xi}_1 \right),
  \end{eqnarray*}
  since $\xi, \eta$ are holomorphic. We infer that for some constant $C_1 \in
  \mathbbm{R}$ holds the identities
  \begin{eqnarray*}
    u^{\left[ \xi, \eta \right]}_1 + C_1 & = & \xi . u^{\eta}_1 - \eta .
    u^{\xi}_1 \\
    &  & \\
    & = & \xi . v_1 - \eta . u_1 \\
    &  & \\
    & = & g (\nabla_g v_1, \nabla_g u_1 + J \nabla_g u_2)\\
    &  & \\
    & - & g (\nabla_g u_1, \nabla_g v_1 + J \nabla_g v_2)\\
    &  & \\
    & = & g (\nabla_g v_1, J \nabla_g u_2) - g (\nabla_g u_1, J \nabla_g
    v_2)\\
    &  & \\
    & = & \omega \left( \nabla_g u_2, \nabla_g v_1 \right) - \omega \left(
    \nabla_g v_2, \nabla_g u_1 \right)\\
    &  & \\
    & = & \omega \left( J \nabla_g u_2, J \nabla_g v_1 \right) + \omega (J
    \nabla_g u_1, J \nabla_g v_2)\\
    &  & \\
    & = & \omega \left( (d u_2)_{\omega}^{\ast}, (d v_1)_{\omega}^{\ast}
    \right) + \omega ((d u_1)_{\omega}^{\ast}, (d v_2)_{\omega}^{\ast})\\
    &  & \\
    & = & - \left\{ u_1, v_2 \right\}_{\omega} - \left\{ u_2, v_1
    \right\}_{\omega} .
  \end{eqnarray*}
  On the other hand
  \begin{eqnarray*}
    u^{\left[ \xi, \eta \right]}_2 & = & - u^{J \left[ \xi, \eta \right]}_1 =
    - u^{\left[ \xi, J \eta \right]}_1,
  \end{eqnarray*}
  since $\xi$ is holomorphic. We infer that for some constant $C_2 \in
  \mathbbm{R}$ holds the identities
  \begin{eqnarray*}
    u^{\left[ \xi, \eta \right]}_2 + C_2 & = & - \xi . u_1^{J \eta} + J \eta .
    u_1^{\xi} \\
    &  & \\
    & = & \xi . u^{\eta}_2 + J \eta . u_1^{\xi} \\
    &  & \\
    & = & \xi . v_2 + J \eta . u_1 \\
    &  & \\
    & = & g (\nabla_g v_2, \nabla_g u_1 + J \nabla_g u_2)\\
    &  & \\
    & + & g (\nabla_g u_1, J \nabla_g v_1 - \nabla_g v_2)\\
    &  & \\
    & = & g (\nabla_g v_2, J \nabla_g u_2) + g \left( \nabla_g u_1, J
    \nabla_g v_1 \right)\\
    &  & \\
    & = & \omega \left( \nabla_g u_2, \nabla_g v_2 \right) + \omega \left(
    \nabla_g v_1, \nabla_g u_1 \right)\\
    &  & \\
    & = & \omega \left( J \nabla_g u_2, J \nabla_g v_2 \right) - \omega
    \left( J \nabla_g u_1, J \nabla_g v_1 \right)\\
    &  & \\
    & = & \omega \left( (d u_2)_{\omega}^{\ast}, (d v_2)_{\omega}^{\ast}
    \right) - \omega \left( (d u_1)_{\omega}^{\ast}, (d v_1)_{\omega}^{\ast}
    \right)\\
    &  & \\
    & = & \left\{ u_1, v_1 \right\}_{\omega} - \left\{ u_2, v_2
    \right\}_{\omega} .
  \end{eqnarray*}
  We conclude that for all $u, v \in \mathbbm{K}_-$ holds the identity
  \begin{eqnarray*}
    \nabla_{g, J} i \left\{ u, v \right\}_{\omega, \Omega} & = & \left[
    \nabla_{g, J} u, \nabla_{g, J} v \right],
  \end{eqnarray*}
  which shows that $i \left\{ \cdot, \cdot \right\}_{\omega, \Omega}$ is a
  complex Lie algebra product over $\mathbbm{K}_-$ and that the map $\chi$ is
  a morphism of complex Lie algebras.
  
  {\tmstrong{Step B,C}}. The statements {\tmstrong{B}} and {\tmstrong{C}}
  follow from corollary \ref{Eigenf-HolVct} and the remarkable identity
  (\ref{E1CxE1}).
  
  {\tmstrong{Step D}}. We show now the statement {\tmstrong{D}}. We observe
  first that for all $u, v \in C^{\infty} (X, \mathbbm{C})$ holds the identity
  \begin{eqnarray*}
    \int_X i \left\{ u, v \right\}_{\omega} \Omega & = & - \int_X i
    B^{\Omega}_{g, J} u \cdot v \Omega .
  \end{eqnarray*}
  Indeed thanks to the computations in step {\tmstrong{A}} we deduce
  \begin{eqnarray*}
    \int_X i \left\{ u, v \right\}_{\omega} \Omega & = & - \int_X \left[
    \left\{ u_1, v_2 \right\}_{\omega} + \left\{ u_2, v_1 \right\}_{\omega}
    \right] \Omega\\
    &  & \\
    & + & i \int_X \left[ \left\{ u_1, v_1 \right\}_{\omega} - \left\{ u_2,
    v_2 \right\}_{\omega} \right] \Omega\\
    &  & \\
    & = & \int_X \left[ \left\langle \nabla_g v_1, J \nabla_g u_2
    \right\rangle_g + \left\langle \nabla_g u_1, J \nabla_g v_2
    \right\rangle_g \right] \Omega\\
    &  & \\
    & + & i \int_X \left[ \left\langle \nabla_g v_2, J \nabla_g u_2
    \right\rangle_g + \left\langle \nabla_g u_1, J \nabla_g v_1
    \right\rangle_g \right] \Omega .
  \end{eqnarray*}
  Integrating by parts we infer
  \begin{eqnarray*}
    \int_X i \left\{ u, v \right\}_{\omega} \Omega & = & \int_X \left[ v_1
    \cdot B^{\Omega}_{g, J} u_2 - u_1 \cdot B^{\Omega}_{g, J} v_2 \right]
    \Omega\\
    &  & \\
    & + & i \int_X \left[ v_2 \cdot B^{\Omega}_{g, J} u_2 + u_1 \cdot
    B^{\Omega}_{g, J} v_1 \right] \Omega\\
    &  & \\
    & = & \int_X \left[ B^{\Omega}_{g, J} u_1 \cdot v_2 + B^{\Omega}_{g, J}
    u_2 \cdot v_1 \right] \Omega\\
    &  & \\
    & - & i \int_X \left[ B^{\Omega}_{g, J} u_1 \cdot v_1 - B^{\Omega}_{g, J}
    u_2 \cdot v_2 \right] \Omega\\
    &  & \\
    & = & - \int_X i B^{\Omega}_{g, J} u \cdot v \Omega,
  \end{eqnarray*}
  thanks to the fact that $B^{\Omega}_{g, J}$ is $L_{\Omega}^2$-anti-adjoint.
  Thus if $u \in \mathbbm{K}_-$
  \begin{eqnarray*}
    \int_X i \left\{ u, \bar{u} \right\}_{\omega} \Omega & = & \int_X
    (\Delta^{\Omega}_g - 2\mathbbm{I}) u \cdot \bar{u} \Omega \geqslant 0 .
  \end{eqnarray*}
  The K\"ahler-Ricci-Soliton assumption implies the commutation identity
  \begin{eqnarray*}
    \left[ \Delta^{\Omega}_{g, J}, \Delta^{\Omega}_{g, - J} \right] & = & 0 .
  \end{eqnarray*}
  We infer that
  \begin{equation}
    \label{KerLapK} \Delta^{\Omega}_{g, J} - 2\mathbbm{I}: \mathbbm{K}_-
    \longrightarrow \mathbbm{K}_-,
  \end{equation}
  is a well defined non-negative $L_{\Omega}^2$-self-adjoint operator and let
  $(\lambda_j)^N_{j = 0} \subset \mathbbm{R}_{\geqslant 0}$, $\lambda_0 = 0$
  be it's spectrum. Notice also that by definition of $\mathbbm{K}_-$ this
  operator coincides with the operator
  \begin{eqnarray*}
    - 2 i B^{\Omega}_{g, J} : \mathbbm{K}_- & \longrightarrow & \mathbbm{K}_-
    .
  \end{eqnarray*}
  Thus $u \in \mathbbm{K}_-$ is an eigen-vector corresponding to the
  eigenvalue $\lambda_j$ if and only if $u \in \mathbbm{K}_-$ satisfies
  \begin{eqnarray*}
    (J \nabla_g f) . u & = & \frac{\lambda_j}{2} i u .
  \end{eqnarray*}
  This rewrites as
  \begin{eqnarray*}
    i \left\{ f, u \right\}_{\omega, \Omega} & = & \frac{\lambda_j}{2} u,
  \end{eqnarray*}
  and is equivalent to the equation
  \begin{eqnarray*}
    \left[ \nabla_g f, \nabla_{g, J} u \right] & = & \frac{\lambda_j}{2}
    \nabla_{g, J} u .
  \end{eqnarray*}
  Notice also that the kernel of (\ref{KerLapK}) is given by the identity
  \begin{eqnarray*}
    \mathbbm{K}_+ \cap \mathbbm{K}_- & = & \mathbbm{K}_{\mathbbm{R}} \oplus
    J\mathbbm{K}_{\mathbbm{R}},\\
    &  & \\
    \mathbbm{K}_{\mathbbm{R}} & \assign & \tmop{Ker}_{\mathbbm{R}}
    (\Delta^{\Omega}_{g, \pm J} - 2\mathbbm{I}) .
  \end{eqnarray*}
  We deduce the required conclusion with $\mu_j = \lambda_j / 2$.
\end{proof}

\subsection{Consequences of the Bochner-Kodaira-Nakano formula}

The holomorphic and antiholomorphic Hodge Laplacian operators are related by
the Bochner-Kodaira-Nakano identity. At the level of $T_X$-valued 1-forms it
reduces to the identity
\begin{equation}
  \label{Wg-BKN-Id} \Delta^{- J}_{T_{X, g}} A = \Delta^J_{T_{X, g}} A +
  \frac{1}{6}  \left( J\mathcal{R}_g \wedge A \right) (\omega^{\ast} \wedge
  \bullet),
\end{equation}
where $\omega^{\ast} \equiv \omega^{- 1} \in C^{\infty} (X, \Lambda^{1, 1}_J
T_X \cap \Lambda_{\mathbbm{R}}^2 T_X)$ is the dual element associated to
$\omega$. If in holomorphic coordinates $\omega$ writes as
\begin{eqnarray*}
  \omega & = & \frac{i}{2} \omega_{k, \bar{l}} d z_k \wedge d \bar{z}_l,
\end{eqnarray*}
then
\begin{eqnarray*}
  \omega^{\ast} & = & 2 i \omega^{l, \bar{k}}  \frac{\partial}{\partial z_k}
  \wedge \frac{\partial}{\partial \bar{z}_l} .
\end{eqnarray*}
The factor $1 / 6$ in front of the last term on the right hand side of
(\ref{Wg-BKN-Id}) is due to the convention
\begin{eqnarray*}
  v_1 \wedge \ldots \wedge v_p & \assign & \sum_{\sigma \in S_p}
  \varepsilon_{\sigma} v_{\sigma_1} \otimes \cdots \otimes v_{\sigma_p} .
\end{eqnarray*}
We explicit the latter term. For this purpose we observe first that for any
$\alpha \in \Lambda^{1, 1}_J T^{\ast}_X \otimes_{\mathbbm{C}} E$ holds the
identity
\begin{eqnarray*}
  \tmop{Tr}_{\omega} \alpha & = & - \tmop{Tr}_g  \left[ \alpha (J \cdot, \cdot
  \left) \right] . \right.
\end{eqnarray*}
We infer the expressions
\begin{eqnarray*}
  \frac{1}{\mathbbm{6}6}  (J\mathcal{R}_g \wedge A) (\omega^{\ast} \wedge \xi)
  & = & (J\mathcal{R}_g \wedge A) \left( 2 i \omega^{l, \bar{k}} 
  \frac{\partial}{\partial z_k}, \frac{\partial}{\partial \bar{z}_l}, \xi
  \right)\\
  &  & \\
  & = & - \frac{1}{2} \tmop{Tr}_{\omega} [ (J\mathcal{R}_g \wedge A)
  (\cdummy, \cdummy, \xi)]\\
  &  & \\
  & = & \frac{1}{2}  (J\mathcal{R}_g \wedge A) (J e_k, e_k, \xi),
\end{eqnarray*}
for an arbitrary $g$-orthonormal real frame $(e_k)_k$. We explicit the
exterior product using the $J$-invariant properties of the curvature operator.
We obtain
\begin{eqnarray*}
 && (J\mathcal{R}_g \wedge A) (J e_k, e_k, \xi)
 \\
 \\
 & = & J\mathcal{R}_g  (J e_k,
  e_k) A \xi - J\mathcal{R}_g  (J e_k, \xi) A e_k + J\mathcal{R}_g  (e_k, \xi)
  A J e_k \\
  &  & \\
  & = & - \tmop{Tr}_{\omega} (J\mathcal{R}_g) A \xi -\mathcal{R}_g  (J e_k,
  \xi) J A e_k +\mathcal{R}_g  (e_k, \xi) J A J e_k\\
  &  & \\
  & = & - 2 \tmop{Ric}^{\ast} (g) A \xi +\mathcal{R}_g  (\xi, J e_k) J A e_k
  - \left[ \mathcal{R}_g \ast (J A J) \right] \xi\\
  &  & \\
  & = & - 2 \tmop{Ric}^{\ast} (g) A \xi -\mathcal{R}_g  (\xi, \eta_k) J A J
  \eta_k - \left[ \mathcal{R}_g \ast (J A J) \right] \xi,
\end{eqnarray*}
where $\eta_k \assign J e_k$. But $(\eta_k)_k$ is also a $g$-orthonormal real
frame. We infer
\begin{eqnarray*}
  (J\mathcal{R}_g \wedge A) (J e_k, e_k, \xi) & = & - 2 \tmop{Ric}^{\ast} (g)
  A \xi - 2 \left[ \mathcal{R}_g \ast (J A J) \right] \xi .
\end{eqnarray*}
We deduce that the Bochner-Kodaira-Nakano identity rewrites at the level of
$T_X$-valued 1-forms as
\begin{equation}
  \label{BKN-id} \Delta^{- J}_{T_{X, g}} A = \Delta^J_{T_{X, g}} A -
  \tmop{Ric}^{\ast} (g) A -\mathcal{R}_g \ast \left( A_J'' - A_J' \right),
\end{equation}
where $A_J'$ and $A_J''$ are respectively the $J$-linear and $J$-anti-linear
parts of $A$. Using the Weitzenb\"ock type formula in lemma
\ref{OmTX-Lap-RmLap} with $\Omega = C d V_g$ we infer
\begin{eqnarray*}
  \mathcal{L}^{\Omega}_g A & = & \Delta_g A + \nabla_g f \neg \nabla_g A -
  2\mathcal{R}_g \ast A\\
  &  & \\
  & = & \Delta_{T_{X, g}} A -\mathcal{R}_g \ast A - A \tmop{Ric}^{\ast} (g) +
  \nabla_g f \neg \nabla_g A\\
  &  & \\
  & = & \left( \Delta^J_{T_{X, g}} + \Delta^{- J}_{T_{X, g}} \right) A
  -\mathcal{R}_g \ast A - A \tmop{Ric}^{\ast} (g) + \nabla_g f \neg \nabla_g A
  .
\end{eqnarray*}
Using the Bochner-Kodaira-Nakano identity (\ref{BKN-id}) we deduce the
formulas
\begin{eqnarray*}
  \mathcal{L}^{\Omega}_g A & = & 2 \Delta^J_{T_{X, g}} A - \tmop{Ric}^{\ast}
  (g) A - A \tmop{Ric}^{\ast} (g) - 2\mathcal{R}_g \ast A_J'' + \nabla_g f
  \neg \nabla_g A,\\
  &  & \\
  \mathcal{L}^{\Omega}_g A & = & 2 \Delta^{- J}_{T_{X, g}} A +
  \tmop{Ric}^{\ast} (g) A - A \tmop{Ric}^{\ast} (g) - 2\mathcal{R}_g \ast A_J'
  + \nabla_g f \neg \nabla_g A,
\end{eqnarray*}
and thus the identities
\begin{equation}
  \label{dec-Lich1} \mathcal{L}^{\Omega}_g A_J' = 2 \Delta^J_{T_{X, g}} A_J'
  - \tmop{Ric}^{\ast} (g) A_J' - A_J' \tmop{Ric}^{\ast} (g) + \nabla_g f \neg
  \nabla_g A_J',
\end{equation}

\begin{equation}
  \label{dec-Lich2} \mathcal{L}^{\Omega}_g A_J'' = 2 \Delta^{- J}_{T_{X, g}}
  A_J'' + \tmop{Ric}^{\ast} (g) A_J'' - A_J'' \tmop{Ric}^{\ast} (g) + \nabla_g
  f \neg \nabla_g A_J'' .
\end{equation}
We point out that one can obtain directly these formulas by using the methods
in the proof of identities (\ref{Fund-exp-AntHol}), (\ref{Lich-AntiHol}) and
(\ref{div-JA}). We remind now that the properties (\ref{lin-Lich}) and
(\ref{ant-lin-Lich}) imply that $A \in \tmop{Ker} \mathcal{L}^{}_g$ if and
only if $A_J' \in \tmop{Ker} \mathcal{L}^{}_g$ and $A_J'' \in \tmop{Ker}
\mathcal{L}^{}_g$. Thus if $A \in \tmop{Ker} \mathcal{L}^{}_g$ we infer thanks
to the identity (\ref{dec-Lich1}) with $\Omega = C d V_g$,
\begin{eqnarray*}
  0 & = & \int_X \left\langle \mathcal{L}_g A'_J, A'_J \right\rangle_g d V_g\\
  &  & \\
  & = & 2 \int_X \left[ \left\langle \Delta^J_{T_{X, g}} A_J', A'_J
  \right\rangle_g  \; - \left\langle \tmop{Ric}^{\ast} (g) A_J', A'_J
  \right\rangle_g \right] d V_g .
\end{eqnarray*}
Using the identity between Riemannian and hermitian norms of $T_X$-valued
forms we obtain
\begin{eqnarray*}
  2 \int_X \left\langle \Delta^J_{T_{X, g}} A_J', A'_J \right\rangle_g d V_g &
  = & 2 \int_X \left\langle \Delta^J_{T_{X, g}} A_J', A'_J
  \right\rangle_{\omega} d V_g\\
  &  & \\
  & = & \int_X \left[ 2 \left| \partial^{\ast_g}_{T_{X, J}} A'_J
  \right|^2_{\omega}  \; + \;  \left| \partial^g_{T_{X, J}} A'_J
  \right|^2_{\omega} \right] d V_g\\
  &  & \\
  & = & \int_X \left[ 2 \left| \partial^{\ast_g}_{T_{X, J}} A'_J \right|^2_g 
  \; + \;  \left| \partial^g_{T_{X, J}} A'_J \right|^2_g \right] d V_g .
\end{eqnarray*}
We deduce
\begin{equation}
  \label{Lich-LinZero} \int_X \left[ 2 \left| \partial^{\ast_g}_{T_{X, J}}
  A'_J \right|^2_g  \; + \;  \left| \partial^g_{T_{X, J}} A'_J \right|^2_g  \;
  - 2 \left\langle \tmop{Ric}^{\ast} (g) A_J', A'_J \right\rangle_g \right] d
  V_g = 0 .
\end{equation}
Assume from now on the K\"ahler-Einstein condition $\tmop{Ric} (g) = \lambda
g$, $\lambda = \pm 1, 0$. The identity (\ref{dec-Lich2}) with $\Omega = C d
V_g$ implies in this case
\begin{eqnarray*}
  \mathcal{L}_g A''_J & = & 2 \Delta^{- J}_{T_{X, g}} A_J'',
\end{eqnarray*}
and thus
\begin{eqnarray*}
  \tmop{Ker} \mathcal{L}_g \cap C^{\infty} \left( X, T^{\ast}_{X, - J} \otimes
  T_{X, J} \right) & = & \mathcal{H}_g^{0, 1} \left( T_{X, J} \right) .
\end{eqnarray*}
Let now $A \in \tmop{Ker} \nabla^{\ast}_g$ and observe that for be-degree
reasons holds the decomposition
\begin{eqnarray*}
  0 = \nabla^{\ast}_g A & = & \nabla^{\ast}_{T_{X, g}} A_J' +
  \nabla^{\ast}_{T_{X, g}} A_J''\\
  &  & \\
  & = & \partial^{\ast_g}_{T_{X, J}} A_J' +
  \overline{\partial}^{\ast_g}_{T_{X, J}} A_J'' .
\end{eqnarray*}
Thus if $A \in \tmop{Ker} \nabla^{\ast}_g \cap \tmop{Ker} \mathcal{L}_g$ then
$\overline{\partial}^{\ast_g}_{T_{X, J}} A_J'' = 0$ and thus
$\partial^{\ast_g}_{T_{X, J}} A_J' = 0$ which implies
\begin{eqnarray*}
  \int_X \left[ \left| \partial^g_{T_{X, J}} A'_J \right|^2_g  \; - \; 2
  \left\langle \tmop{Ric}^{\ast} (g) A_J', A'_J \right\rangle_g \right] d V_g
  & = & 0,
\end{eqnarray*}
thanks to (\ref{Lich-LinZero}). We will still denote by $\mathcal{L}_g$ the
analogue operator over $C^{\infty} \left( X, S_{\mathbbm{R}}^2 T^{\ast}_X
\right)$. We infer that if $\lambda \neq 0$ then holds the identity
\begin{equation*}
  \tmop{Ker} \nabla^{\ast}_g \cap \tmop{Ker} \mathcal{L}_g \cap
  \mathbbm{D}^J_g  =  \left\{ v \in C^{\infty} \left( X, S_{\mathbbm{R}}^2
  T^{\ast}_X \right) \mid \hspace{0.25em} v = v_J'', \;v^{\ast}_g \in
  \mathcal{H}_g^{0, 1} \left( T_{X, J} \right) \right\},
\end{equation*}
i.e. there exists an isomorphism
\begin{eqnarray*}
  \tmop{Ker} \nabla^{\ast}_g \cap \tmop{Ker} \mathcal{L}_g \cap
  \mathbbm{D}^J_g & \longrightarrow & \mathcal{H}_g^{0, 1} \left( T_{X, J}
  \right)_{\tmop{sm}}\\
  v & \longmapsto & v^{\ast}_g .
\end{eqnarray*}
But $\mathcal{H}_g^{0, 1} \left( T_{X, J} \right)_{\tmop{sm}}   = 
\mathcal{H}_g^{0, 1} \left( T_{X, J} \right)$, thanks to lemma \ref{KE-Hsym}.
We conclude the following fact.

\begin{lemma}
  \label{harmonicKER}Over any compact non Ricci flat K\"ahler-Einstein
  manifold $\left( X, J, g \right)$ there exists the canonical isomorphism
  \begin{eqnarray*}
    \tmop{Ker} \nabla^{\ast}_g \cap \tmop{Ker} \mathcal{L}_g \cap
    \mathbbm{D}^J_g & \longrightarrow & H^{0, 1} (X, T_{X, J}) \simeq H^1 (X,
    \mathcal{O}(T_{X, J})) \\
    v & \longmapsto & \left\{ v^{\ast}_g \right\} .
  \end{eqnarray*}
\end{lemma}

This result was proved by \cite{D-W-W2} in the negative K\"ahler-Einstein case
$\tmop{Ric} (g) = - g$.

\subsection{Polarized deformations of Fano manifolds}

In this subsection we review a few basic facts on deformation theory which
clarifies the Fano set up in the paper. In particular we wish the readers would avoid the frequent inaccuracies 
we found in the application of this theory to K\"ahler geometry.

\subsubsection{The Maurer-Cartan equation}

Let $(V, J_0)$ be a complex vector space of dimension $n$. We remind that the
data of a complex structure $J$ over $V$ is equivalent with a $n$-dimensional
complex subspace data $\Gamma \subset \mathbbm{C}V \assign V
\otimes_{\mathbbm{R}} \mathbbm{C}$ such that $\Gamma \cap \overline{\Gamma} =
\left\{ 0 \right\}$. A complex structure $J$ over $V$ is called
$J_0$-compatible if the projection map
\begin{eqnarray*}
  \pi^{0, 1}_{J_0} : V^{0, 1}_J & \longrightarrow & V^{0, 1}_{J_0},
\end{eqnarray*}
is surjective, i.e. a $\mathbbm{C}$-isomorphism. This is equivalent to the
condition $V_J^{0, 1} \cap V_{J_0}^{1, 0} = \left\{ 0 \right\}$, which in its
turn is equivalent to the existence of a $\mathbbm{C}$-linear map $\theta :
V^{0, 1}_{J_0} \longrightarrow V^{1, 0}_{J_0}$ such that
\begin{eqnarray*}
  V^{0, 1}_J & = & (\mathbbm{I}+ \theta) V^{0, 1}_{J_0} .
\end{eqnarray*}
If we set $\mu \assign (\theta + \bar{\theta})_{\mid V} \in \tmop{End}_{- J_0}
(V)$ then the condition $V^{0, 1}_J \cap \overline{V^{0, 1}_J} = \left\{ 0
\right\}$ is equivalent to say $\mathbbm{I}+ \mu \in \tmop{GL}_{\mathbbm{R}}
(V)$. Notice that we can obtain $\theta$ by the formula $\theta =
\mu_{\mathbbm{C}} \cdot \pi^{0, 1}_{J_0}$, with $\mu_{\mathbbm{C}} \in
\tmop{End}_{\mathbbm{C}} (\mathbbm{C}V)$ the natural complexification of
$\mu$. If we denote by $\mathcal{J} (V, J_0)$ the set of $J_0$-compatible
complex structures over $V$ and if we set
\begin{eqnarray*}
  C (V, J_0) & \assign & \left\{ \mu \in \tmop{End}_{- J_0} (V) \mid
  \mathbbm{I}+ \mu \in \tmop{GL}_{\mathbbm{R}} (V) \right\},
\end{eqnarray*}
we infer the existence of a bijection, called the Caley transform (see \cite{Gau})
\begin{eqnarray*}
  \chi : C (V, J_0) & \longrightarrow & \mathcal{J} (V, J_0) \\
  &  & \\
  \mu & \longleftrightarrow & J \assign (\mathbbm{I}+ \mu) J_0  (\mathbbm{I}+
  \mu)^{- 1}\\
  &  & \\
  \mu : = (J_0 + J)^{- 1} (J_0 - J) & \longleftrightarrow & J .
\end{eqnarray*}
Notice indeed that $\mathcal{J} (V, J_0)$ is the sub-set of the complex
structures such that $J_0 + J \in \tmop{GL}_{\mathbbm{R}} (V)$. We observe
that for any {\tmem{$\mu \in C (V, J_0)$}} as above $\mathbbm{I}- \mu \in
\tmop{GL}_{\mathbbm{R}} (V)$. Indeed
\begin{eqnarray*}
  - J_0 J  =  (\mathbbm{I}- \mu)  (\mathbbm{I}+ \mu)^{- 1} .
\end{eqnarray*}
Thus $\mu \in \tmop{End}_{- J_0} (V)$ satisfies $\mathbbm{I}+ \mu \in
\tmop{GL}_{\mathbbm{R}} (V)$ if and only if 
$$
(\mathbbm{I}- \mu^2) =
(\mathbbm{I}- \mu) (\mathbbm{I}+ \mu) \in \tmop{GL}_{J_0} (V)
$$ This last
condition is equivalent with
\begin{eqnarray*}
  \pi_{J_0}^{1, 0} \cdot (\mathbbm{I}- \mu^2)  =  (\mathbbm{I}_{V^{1,
  0}_{J_0}} - \theta \overline{\theta}_{}) \in \tmop{GL}_{\mathbbm{C}} (V^{1,
  0}_{J_0}) .
\end{eqnarray*}
We assume from now on that $(X, J_0)$ is a compact complex manifold and let
$\mathcal{J} (X, J_0)$ be the set of $J_0$-compatible smooth almost complex
structures. For any $J \in \mathcal{J} (X, J_0)$ let
\begin{eqnarray*}
  \theta  \equiv  \theta_J \in C^{\infty} (X, \Lambda^{0, 1}_{J_0}
  T^{\ast}_X \otimes_{\mathbbm{C}} T_{X, J_0}^{1, 0}), \quad(\mathbbm{I}_{T^{1,
  0}_{X, J_0}} - \theta \overline{\theta}_{}) \in \tmop{GL}_{\mathbbm{C}}
  (T_{X, J_0}^{1, 0}),
\end{eqnarray*}
be the corresponding inverse Caley transform. We show that the subset
$\mathcal{J}_{\tmop{int}} (X, J_0)$ of integrable almost complex structures is
given by the Maurer-Cartan equation
\begin{equation}
  \label{MCart-Eq}  \overline{\partial}_{T^{1, 0}_{X, J_0}} \theta +
  \frac{1}{2}  \left[ \theta, \theta \right] = 0,
\end{equation}
where for any $\alpha, \beta \in C^{\infty} (X, \Lambda^{0, \bullet}_{J_0}
T^{\ast}_X \otimes_{\mathbbm{C}} T_{X, J_0}^{1, 0})$ we define the exterior
differential Lie product
\begin{eqnarray*}
  \left[ \alpha, \beta \right] & \in & C^{\infty} (X, \Lambda^{0,
  \bullet}_{J_0} T^{\ast}_X \otimes_{\mathbbm{C}} T_{X, J_0}^{1, 0}),
\end{eqnarray*}
of degree $d = \deg \alpha + \deg \beta$ by the formula
\begin{eqnarray*}
  \left[ \alpha, \beta \right]  (\xi) & : = & \sum_{|I| = \deg \alpha}
  \varepsilon_I  \left[ \alpha (\xi_I), \beta (\xi_{\complement I}) \right],
\end{eqnarray*}
for all $\xi \in \bar{\mathcal{O}} (T_{X, J_0}^{0, 1})^{\times d}$. Notice
that this formula defines a priori only an element
\begin{eqnarray*}
  \left[ \alpha, \beta \right] & \in & \tmop{Alt}^d_{\bar{\mathcal{O}}} (
  \bar{\mathcal{O}} (T_{X, J_0}^{0, 1}) ; C^{\infty} (T_{X, J_0}^{1, 0})) .
\end{eqnarray*}
However we can define pointwise the section $\left[ \alpha, \beta \right]$ as
follows. For any $v \in (T_{X, J_0, x}^{0, 1})^{\times d}$
\begin{eqnarray*}
  \left[ \alpha, \beta \right] (v) & \assign & \left[ \alpha, \beta \right]
  (\xi)_{\mid x},
\end{eqnarray*}
with $\xi \in \bar{\mathcal{O}} (T_{X, J_0}^{0, 1})^{\times d}$ such that
$\xi_x = v$. This is well defined by the $\bar{\mathcal{O}}$-linearity of
$\left[ \alpha, \beta \right]$. Indeed the coefficients of $\xi$ with respect
to the local frame $( \bar{\zeta}_k)^n_{k = 1} \subset \bar{\mathcal{O}} (U,
T_{X, J_0}^{0, 1})$, with $\zeta_j \assign \frac{\partial}{\partial z_j}$, and
$J_0$-holomorphic coordinates $(z_1, \ldots, z_n)$, are $J_0$-anti-holomorphic
functions which value at the point $x$ is uniquely determined by $v$. The
section $\left[ \alpha, \beta \right]$ is smooth since its coefficients with
respect to the frame $( \bar{\zeta}_k)^n_{k = 1}$ are smooth functions.

Notice now that $(\mathbbm{I}+ \theta)  ( \bar{\zeta}_k)$, $k = 1, \ldots,
n$, is a local frame of the bundle $T_{X, J}^{0, 1}$ over an open set $U$.
Then the integrability of $J$ is equivalent to the condition
\begin{equation}
  \label{Integ-CX}  \left[ (\mathbbm{I}+ \theta) ( \bar{\zeta}_k),
  (\mathbbm{I}+ \theta) ( \bar{\zeta}_l) \right] \in C^{\infty} (U, T_{X,
  J}^{0, 1}),
\end{equation}
since the torsion form $\tau_J \in C^{\infty} (X, \Lambda^{0, 2}_{J_0}
T^{\ast}_X \otimes_{\mathbbm{C}} T_{X, J_0}^{1, 0})$ of $J$ satisfies
\begin{eqnarray*}
  \tau_J \left( (\mathbbm{I}+ \theta) ( \bar{\zeta}_k), (\mathbbm{I}+ \theta)
  ( \bar{\zeta}_l) \right) & = & \left[ (\mathbbm{I}+ \theta) (
  \bar{\zeta}_k), (\mathbbm{I}+ \theta) ( \bar{\zeta}_l) \right]^{1, 0}_J .
\end{eqnarray*}
We observe also the identities
\begin{eqnarray*}
  \left[ (\mathbbm{I}+ \theta) ( \bar{\zeta}_k), (\mathbbm{I}+ \theta) (
  \bar{\zeta}_l) \right] & = & \left[ \bar{\zeta}_k, \theta ( \bar{\zeta}_l)
  \right] + \left[ \theta ( \bar{\zeta}_k), \bar{\zeta}_l \right] + \left[
  \theta ( \bar{\zeta}_k), \theta ( \bar{\zeta}_l) \right]\\
  &  & \\
  & = & \left[ \bar{\zeta}_k, \theta ( \bar{\zeta}_l) \right]_{J_0}^{1, 0} -
  \left[ \bar{\zeta}_l, \theta ( \bar{\zeta}_k) \right]_{J_0}^{1, 0} +
  \frac{1}{2}  \left[ \theta, \theta \right]  ( \bar{\zeta}_k, \bar{\zeta}_l)
  \\
  &  & \\
  & = & \left( \overline{\partial}_{T^{1, 0}_{X, J_0}} \theta + \frac{1}{2} 
  \left[ \theta, \theta \right] \right) ( \bar{\zeta}_k, \bar{\zeta}_l) \in
  C^{\infty} (U, T_{X, J_0}^{1, 0}) .
\end{eqnarray*}
We have $T_{X, J}^{0, 1} \cap T_{X, J_0}^{1, 0} = 0_X$ by the
$J_0$-compatibility of $J$. We infer that if (\ref{Integ-CX}) holds then also
(\ref{MCart-Eq}) holds true and
\begin{eqnarray*}
  \left[ (\mathbbm{I}+ \theta) ( \bar{\zeta}_k), (\mathbbm{I}+ \theta) (
  \bar{\zeta}_l) \right] & = & 0 .
\end{eqnarray*}
On the other hand if (\ref{MCart-Eq}) is satisfied then the previous identity
is satisfied and thus (\ref{Integ-CX}) holds true.

\begin{remark}
  For any $\alpha, \beta \in C^{\infty} (X, \Lambda^{0, \bullet}_{J_0}
  T^{\ast}_X \otimes_{\mathbbm{C}} T_{X, J_0})$ we define the exterior
  differential Lie product
  \begin{eqnarray*}
    \left[ \alpha, \beta \right] & \in & C^{\infty} (X, \Lambda^{0,
    \bullet}_{J_0} T^{\ast}_X \otimes_{\mathbbm{C}} T_{X, J_0}),
  \end{eqnarray*}
  of degree $d = \deg \alpha + \deg \beta$ by the formula
  \begin{eqnarray*}
    \left[ \alpha, \beta \right] & : = & \left[ \pi_{_J}^{1, 0} \cdot
    \alpha_{_{\mathbbm{C}}}, \pi_{_J}^{1, 0} \cdot \beta_{_{\mathbbm{C}}}
    \right] + \overline{\left[ \pi_{_J}^{1, 0} \cdot \alpha_{_{\mathbbm{C}}},
    \pi_{_J}^{1, 0} \cdot \beta_{_{\mathbbm{C}}} \right]} .
  \end{eqnarray*}
  Then the Maurer-Cartan equation (\ref{MCart-Eq}) can be rewritten in the
  equivalent form
  \begin{eqnarray*}
    \overline{\partial}_{T_{X, J_0}}  \, \mu + \frac{1}{2}  \left[ \mu, \mu
    \right] & = & 0,
  \end{eqnarray*}
  since
  \begin{eqnarray*}
    \overline{\partial}_{T_{X, J_0}}  \, \mu & = & \overline{\partial}_{T^{1,
    0}_{X, J_0}} \theta \noplus + \overline{ \overline{\partial}_{T^{1, 0}_{X,
    J_0}} \theta},\\
    &  & \\
    \left[ \mu, \mu \right] & = & \left[ \theta, \theta \right] +
    \overline{\left[ \theta, \theta \right]} .
  \end{eqnarray*}
\end{remark}

Let now $B \subset \mathbbm{C}^p$ be the unitary open ball and observe that,
by a refinement of Ehresmann theorem for any proper holomorphic submersion
$\pi : \mathfrak{X} \longrightarrow B$ of a complex manifold $\mathfrak{X}$
onto $B$ with central fiber $(X, J_0) = \pi^{- 1} (0)$ there exists a smooth
map $\varphi : \mathfrak{X} \longrightarrow X$ such that the map
\begin{eqnarray*}
  (\varphi, \pi) : \mathfrak{X} & \longrightarrow & X \times B,
\end{eqnarray*}
is a diffeomorphism with $\varphi_{\mid_X} =\mathbbm{I}_X$ and with
$\varphi^{- 1} (x) \subset \mathfrak{X}$ complex sub-variety for all $x \in X$.

Let now $\theta \assign (\theta_t)_{t \in B} \subset C^{\infty} (X,
\Lambda^{0, 1}_{J_0} T^{\ast}_X \otimes_{\mathbbm{C}} T_{X, J_0}^{1, 0})$ with
$\theta_0 = 0$ and
\begin{eqnarray*}
  \det (\mathbbm{I}_{T^{1, 0}_{X, J_0}} - \theta_t \overline{\theta}_t) &
  \neq & 0,
\end{eqnarray*}
be a smooth family of $J_0$-compatible complex structures. We observe that the
almost complex manifold
\begin{eqnarray*}
  \mathfrak{X} \hspace{1.2em} \equiv \hspace{1.2em} \left( \mathfrak{X},
  \theta \right) & \assign & \bigsqcup_{t \in B} (X, \theta_t),
\end{eqnarray*}
is a complex one if and only if $\theta_t$ satisfies the Maurer-Cartan
equation (\ref{MCart-Eq}) for all $t \in B$ and the map
\begin{eqnarray*}
  t \in B & \longmapsto & \theta_t (x) \in \Lambda^{0, 1}_{J_0} T^{\ast}_{X,
  x} \otimes_{\mathbbm{C}} T_{X, J_0, x}^{1, 0},
\end{eqnarray*}
is holomorphic for all $x \in X$. Indeed the distribution $T_{\mathfrak{X},
\theta}^{0, 1}$ is integrable if and only if its local generators
$\bar{\tau}_r \assign \frac{\partial}{\partial \overline{t}_r}$ , $r = 1,
\ldots, p$, $(\mathbbm{I}+ \theta_t)  ( \bar{\zeta}_k)$, $k = 1, \ldots, n$,
satisfy the conditions
\begin{equation}
  \label{Part-InexCx}  \left[ \bar{\tau}_r, (\mathbbm{I}+ \theta_t) (
  \bar{\zeta}_k) \right] \in C^{\infty} (U, T_{X, \theta_t}^{0, 1}),
\end{equation}
and
\begin{equation}
  \label{Integ-CXt}  \left[ (\mathbbm{I}+ \theta_t) ( \bar{\zeta}_k),
  (\mathbbm{I}+ \theta_t) ( \bar{\zeta}_l) \right] \in C^{\infty} (U, T_{X,
  \theta_t}^{0, 1}) .
\end{equation}
The latter is equivalent with the Maurer-Cartan equation (\ref{MCart-Eq}). Let
$\theta_t = \theta^{k, l}_t  \bar{\zeta}^{\ast}_k \otimes \zeta_l$ be the
local expression of $\theta_t$. Then the identity
\begin{eqnarray*}
  \left[ \bar{\tau}_r, (\mathbbm{I}+ \theta_t) ( \bar{\zeta}_k) \right] & = &
  \bar{\tau}_r . \theta^{k, l}_t \zeta_l \in C^{\infty} (U, T_{X, J_0}^{0,
  1}),
\end{eqnarray*}
combined with the property $T_{X, \theta_t}^{0, 1} \cap T_{X, J_0}^{1, 0} =
0_X$, shows that (\ref{Part-InexCx}) holds true if and only if the map $t
\longmapsto \theta_t$ is holomorphic.

For any $p \in X$ a coordinate chart of $\mathfrak{X}$ in a open neighborhood
$U_p \times B$ of $\left( p, 0 \right)$ is given by a smooth function $f : U_p
\times B \longrightarrow \mathbbm{C}^n \times \mathbbm{C}^p$ such that
\[  \left\{ \begin{array}{l}
     \overline{\partial}_{_{J_0}} f \hspace{1.2em} + \hspace{1.2em}
     \partial_{_{J_0}} f \cdot \theta_t  \; = 0,\\
     \\
     \overline{\partial}_{_B} f \hspace{1.2em} = \hspace{1.2em} 0\\
     \\
     \det \left( d f \right)  \neq 0 .
   \end{array} \right. \]
In order to produce such family $\theta$ we need to remind a few basic facts
about Hodge theory.

\subsubsection{Basic facts about Hodge theory and
$\overline{\partial}$-equations}

Let $\omega$ be a hermitian metric over $X$ and let $(E,
\overline{\partial}_E, h)$ be a hermitian holomorphic vector bundle over it.
We define the anti-holomorphic Hodge Laplacian
\begin{eqnarray*}
  \Delta''_E & \assign & \overline{\partial}_E \overline{\partial}^{\ast}_E +
  \overline{\partial}^{\ast}_E \overline{\partial}_E,
\end{eqnarray*}
acting on the sections of $\Lambda_J^{p, q} T^{\ast}_X \otimes_{\mathbbm{C}}
E$. Let $\mathcal{E}^{p, q} (E) \assign C^{\infty} (X, \Lambda_J^{p, q}
T^{\ast}_X \otimes_{\mathbbm{C}} E)$ and set
\begin{eqnarray*}
  \mathcal{H}^{p, q} (E) & \assign & \tmop{Ker} \Delta''_E \cap
  \mathcal{E}^{p, q} \left( E \right) .
\end{eqnarray*}
We remind the $L^2$-Hodge decomposition
\begin{eqnarray*}
  \mathcal{E}^{p, q} (E) & = & \mathcal{H}^{p, q} (E) \oplus
  \overline{\partial}_E \mathcal{E}^{p, q - 1} (E) \oplus
  \overline{\partial}^{\ast}_E \mathcal{E}^{p, q + 1} (E) .
\end{eqnarray*}
We observe that if there exists two subspaces $L, V \subset C^{\infty} (X, E)$
such that the $L^2$-decomposition
\begin{eqnarray*}
  C^{\infty} (X, E) & = & L \oplus V,
\end{eqnarray*}
holds then $L$ and $V$ are closed subspaces of $C^{\infty} (X, E)$. Indeed $L =
V^{\bot}$ and $V = L^{\bot}$ by the $L^2$-decomposition. The same
consideration holds for the Sobolev spaces $W^k (X, E)$. Thus the $L^2$-Hodge
decomposition implies that the spaces $\overline{\partial}_E \mathcal{E}^{p, q
- 1} (E)$ and $\overline{\partial}^{\ast}_E \mathcal{E}^{p, q + 1} (E)$ are
closed in the smooth topology. We infer the $L^2$-decomposition
\begin{eqnarray*}
  \mathcal{E}^{p, q} (E) & = & \left[ \right. \tmop{Ker} \overline{\partial}_E
  \cap \mathcal{E}^{p, q} (E) \left] \right. \oplus
  \overline{\partial}^{\ast}_E \mathcal{E}^{p, q + 1} (E),
\end{eqnarray*}
and thus
\begin{eqnarray*}
  \tmop{Ker} \overline{\partial}_E \cap \mathcal{E}^{p, q} (E) & = &
  \mathcal{H}^{p, q} (E) \oplus \overline{\partial}_E \mathcal{E}^{p, q - 1}
  (E) .
\end{eqnarray*}
An other way to see this decomposition is the following. Let
\begin{eqnarray*}
  H_E : \mathcal{E}^{p, q} (E) & \longrightarrow & \mathcal{H}^{p, q} (E),
\end{eqnarray*}
be the $L^2$-projection operator over $\mathcal{H}^{p, q} (E)$. For any
$\alpha \in \mathcal{E}^{p, q} (E)$ there exists $\beta \in \mathcal{E}^{p, q
- 1} (E)$ and $\gamma \in \mathcal{E}^{p, q + 1} (E)$ such that
\begin{eqnarray*}
  \alpha & = & H_E \alpha + \overline{\partial}_E \beta +
  \overline{\partial}^{\ast}_E \gamma .
\end{eqnarray*}
Now if $\overline{\partial}_E \alpha = 0$ then $\overline{\partial}_E
\overline{\partial}^{\ast}_E \gamma = 0$, i.e $\overline{\partial}^{\ast}_E
\gamma = 0$. Let 
$$
W_k^{p, q} (E) \assign W^k (X, \Lambda_J^{p, q} T^{\ast}_X
\otimes_{\mathbbm{C}} E).
$$
We remind that the Green operator
\begin{eqnarray*}
  G_E : W_k^{p, q} (E) & \longrightarrow & \Delta''_E W_{k + 4}^{p, q} (E),
\end{eqnarray*}
is defined by the identity $\mathbbm{I}= H_E + \Delta''_E G_E$. The latter
implies $\tmop{Ker} G_E = \tmop{Ker} \Delta''_E$ and the $L^2$-orthogonal
decomposition
\begin{eqnarray*}
  \alpha & = & H_E \alpha + \overline{\partial}_E 
  \overline{\partial}^{\ast}_E G_E \alpha + \overline{\partial}^{\ast}_E 
  \overline{\partial}_E G_E \alpha .
\end{eqnarray*}
We show now the identity $\overline{\partial}_E G_E = G_E
\overline{\partial}_E$. Indeed $\overline{\partial}_E$-differentiating the
identity defining $G_E$ we infer
\begin{eqnarray*}
  \overline{\partial}_E \alpha & = & \overline{\partial}_E \Delta''_E G_E
  \alpha = \Delta''_E  \overline{\partial}_E G_E \alpha .
\end{eqnarray*}
Applying the same identity to $\overline{\partial}_E \alpha$ we obtain
$\overline{\partial}_E \alpha = \Delta''_E G_E \overline{\partial}_E \alpha$,
since $H_E \overline{\partial}_E = 0$ by orthogonality. Thus
\begin{eqnarray*}
  \Delta''_E  \left( \overline{\partial}_E G_E \alpha - G_E
  \overline{\partial}_E \alpha \right) & = & 0 .
\end{eqnarray*}
The fact that by definition $G_E W_k^{p, q} (E) = \Delta''_E W_{k + 4}^{p, q}
(E)$ implies the existence of $\beta \in W_{k + 4}^{p, q} (E)$ and $\gamma \in
W_{K + 3}^{p, q + 1} (E)$ such that
\begin{eqnarray*}
  G_E \alpha & = & \Delta''_E \beta,\\
  &  & \\
  G_E \overline{\partial}_E \alpha & = & \Delta''_E \gamma .
\end{eqnarray*}
We deduce
\begin{eqnarray*}
  \overline{\partial}_E G_E \alpha - G_E \overline{\partial}_E \alpha  = 
  \Delta''_E  \left( \overline{\partial}_E \beta - \gamma \right) = 0,
\end{eqnarray*}
thanks to the orthogonality of the Kernel and image of $\Delta''_E$.

We observe finally that the equation $\overline{\partial}_E \alpha = \beta$
admits a solution if and only if $\overline{\partial}_E \beta = 0$ and $H_E
\beta = 0$. In this case the unique solution of minimal $L^2$-norm is given by
$\alpha = \overline{\partial}^{\ast}_E G_E \beta$.

\subsubsection{The equation of holomorphic maps}

For any smooth map $f : \left( X, J_X \right) \longrightarrow \left( Y, J_Y
\right)$ we define the operators
\begin{eqnarray*}
  2 \partial_{_{J_X, J_Y}} f & \assign & d f  -  \left( J_Y \circ f
  \right) \cdot d f \cdot J_X  \in  C^{\infty}
  \left( X, \Lambda^{1, 0}_{_{J_X}} T^{\ast}_X \otimes_{_{\mathbbm{C}}}
  f^{\ast} T_{Y, J_Y} \right) ,\\
  &  & \\
  2 \overline{\partial}_{_{J_X, J_Y}} f & \assign & d f  +  \left( J_Y
  \circ f \right) \cdot d f \cdot J_X  \in 
  C^{\infty} \left( X, \Lambda^{0, 1}_{_{J_X}} T^{\ast}_X
  \otimes_{_{\mathbbm{C}}} f^{\ast} T_{Y, J_Y} \right) ,
\end{eqnarray*}
and we notice the elementary identities
\begin{eqnarray*}
  \partial_{_{J_X, J_Y}} f & = & \pi_{_Y}^{1, 0} \cdot d f \cdot \pi^{1,
  0}_{_X} +  \pi_{_Y}^{0, 1} \cdot d f \cdot
  \pi^{0, 1}_{_X} ,\\
  &  & \\
  \overline{\partial}_{_{J_X, J_Y}} f & = & \pi_{_Y}^{1, 0} \cdot d f \cdot
  \pi^{0, 1}_{_X}  +  \pi_{_Y}^{0, 1} \cdot d f
  \cdot \pi^{1, 0}_{_X}  .
\end{eqnarray*}
The map $f$ is called holomorphic if $\left( J_Y \circ f \right) \cdot d f = d
f \cdot J_X$. We deduce that the map $f$ is holomorphic if and only if
$\overline{\partial}_{_{J_X, J_Y}} f \; = \; 0$, thus if and only if
\begin{eqnarray*}
  \pi_{_Y}^{1, 0} \cdot d f \cdot \pi^{0, 1}_{_X} & = & 0  .
\end{eqnarray*}
We infer that a map $f : \left( X, J_{\varphi} \right) \longrightarrow \left(
Y, J_Y \right)$ is holomorphic if and only if
\begin{eqnarray*}
  \pi_{_Y}^{1, 0} \cdot d f_{\mid T_{X, J_{\varphi}}^{0, 1}} & = & 0
   .
\end{eqnarray*}
The identity
\begin{eqnarray*}
  T_{X, J_{\varphi}}^{0, 1} & = & \left( \pi_{_{J_0}}^{0, 1}  +
   \varphi \right) \mathbbm{C}T_X ,
\end{eqnarray*}
implies that $f : \left( X, J_{\varphi} \right) \longrightarrow \left( Y, J_Y
\right)$ is holomorphic if and only if
\begin{eqnarray*}
  \pi_{_Y}^{1, 0} \cdot d f \cdot \left( \pi_{_{J_0}}^{0, 1}  +
   \varphi \right) & = & 0  .
\end{eqnarray*}
This last condition rewrites as
\begin{eqnarray*}
  0 & = & \overline{\partial}_{_{J_0, J_Y}} f \cdot \pi_{_{J_0}}^{0, 1}
   +  \pi_{_Y}^{1, 0} \cdot d f \cdot
  \pi_{_{J_0}}^{1, 0} \cdot \varphi\\
  &  & \\
  & = & \overline{\partial}_{_{J_0, J_Y}} f \cdot \pi^{0, 1}_{_{J_0}}
  + \partial_{_{J_0, J_Y}} f \cdot \varphi.
\end{eqnarray*}
We explicit the latter condition in the case of a smooth map\\
$f : \left( X,
J_{\varphi} \right) \longrightarrow \left( X, J_{\theta} \right)$. Indeed
\begin{eqnarray*}
  0 & = & 2 \overline{\partial}_{_{J_0, J_{\theta}}} f \cdot \pi^{0,
  1}_{_{J_0}} +  2 \; \partial_{_{J_0,
  J_{\theta}}} f \cdot \varphi\\
  &  & \\
  & = & [\mathbbm{I} -  i \, (J_{\theta} \circ
  f)] \cdot (d f \cdot \pi^{0, 1}_{_{J_0}}  +  d f
  \cdot \varphi)\\
  &  & \\
  & = & [\mathbbm{I}  -  i \, (J_{\theta} \circ
  f)] \cdot \pi_{_{J_0}}^{1, 0} \cdot (d f \cdot \pi^{0, 1}_{_{J_0}}
   +  d f \cdot \varphi)\\
  &  & \\
  & + & [\mathbbm{I}  - i \, (J_{\theta} \circ
  f)] \cdot \pi_{_{J_0}}^{0, 1} \cdot (d f \cdot \pi^{0, 1}_{_{J_0}}
   +  d f \cdot \varphi)\\
  &  & \\
  & = & [\mathbbm{I}  - i \, (J_{\theta} \circ
  f)] \cdot \pi_{_{J_0}}^{1, 0} \cdot ( \overline{\partial}_{_{J_0}} f
   +  \partial_{_{J_0}} f \cdot \varphi)\\
  &  & \\
  & + & [\mathbbm{I}  -  i \, (J_{\theta} \circ
  f)] \cdot \pi_{_{J_0}}^{0, 1} \cdot (\partial_{_{J_0}} f  + 
  \overline{\partial}_{_{J_0}} f \cdot \varphi) \hspace{1.2em} .
\end{eqnarray*}
We explicit at this point the expression of $J_{\theta}$. For this purpose
let $\mu \; : = \; \theta \; + \overline{\theta}$ and decompose the identity
\begin{eqnarray*}
  J_{\theta} & : = & J_0  \left( \mathbbm{I}  - 
  \mu \right)  \left( \mathbbm{I} +  \mu
  \right)^{- 1}\\
  &  & \\
  & = & J_0  \left( \mathbbm{I}  - \mu \right)^2
  \left( \mathbbm{I}  -  \mu^2 \right)^{- 1}\\
  &  & \\
  & = & J_0  \left( \mathbbm{I}  -  2 \mu
   + \mu^2 \right)  \left( \mathbbm{I}
   -  \mu^2 \right)^{- 1}\\
  &  & \\
  & = & J_0 \left( \mathbbm{I}  -  2 \theta
   - 2 \overline{\theta} +
   \theta \overline{\theta}  + 
  \overline{\theta} \theta \right)  \left( \mathbbm{I} -
  \theta \overline{\theta}  - 
  \overline{\theta} \theta \right)^{- 1} .
\end{eqnarray*}
Decomposing in types we infer
\begin{eqnarray*}
  J_{\theta} & = & i \left( \mathbbm{I}_{1, 0} + 
  \theta \overline{\theta} \right) (\mathbbm{I}_{1, 0}  -
   \theta \overline{\theta})^{- 1} +
   2 i \overline{\theta} (\mathbbm{I}_{1, 0}  -
   \theta \overline{\theta})^{- 1}\\
  &  & \\
  & - & 2 i \theta (\mathbbm{I}_{0, 1}  -  \overline{\theta}
  \theta)^{- 1}  -  i (\mathbbm{I}_{0, 1}
   +  \overline{\theta} \theta) (\mathbbm{I}_{0, 1}
   -  \overline{\theta} \theta)^{- 1} .
\end{eqnarray*}
Let $A \assign \theta \overline{\theta}$. Using the trivial identity 
$$
\left(
\mathbbm{I} \; + \; A \right) \left( \mathbbm{I} \; - \; A \right)^{- 1} = \;
\mathbbm{I} \; + \; 2 A \left( \mathbbm{I} \; - \; A \right)^{- 1},
$$ 
we
conclude the expression
\begin{eqnarray*}
  J_{\theta} & = & \underbrace{i\mathbbm{I}_{1, 0}  +
   2 i \theta \overline{\theta} (\mathbbm{I}_{1, 0}
   -  \theta \overline{\theta})^{- 1}}_{_{\in
  \mathcal{E}^{1, 0} \left( T_{X, J_0}^{1, 0} \right)}}   +
  \underbrace{2 i \overline{\theta} (\mathbbm{I}_{1, 0}
   -  \theta \overline{\theta})^{- 1}}_{_{\in
  \mathcal{E}^{1, 0} \left( T_{X, J_0}^{0, 1} \right)}} \\
  &  & \\
  & - & \underbrace{2 i \theta (\mathbbm{I}_{0, 1} -
  \overline{\theta} \theta)^{- 1}}_{_{\in \mathcal{E}^{0, 1} \left( T_{X,
  J_0}^{1, 0} \right)}}   - 
  \underbrace{i\mathbbm{I}_{0, 1}  -  2 i
  \overline{\theta} \theta (\mathbbm{I}_{0, 1}  - 
  \overline{\theta} \theta)^{- 1}}_{_{\in \mathcal{E}^{0, 1} \left( T_{X,
  J_0}^{0, 1} \right)}}   .
\end{eqnarray*}
For notation simplicity we identify $J_{\theta} \equiv J_{\theta} \circ f$ and
thus $\theta \equiv \theta \circ f$. Using the previous expression we infer
the equalities
\begin{eqnarray*}
  \frac{1}{2} \; [\mathbbm{I} -  i \, (J_{\theta}
  \circ f)] \cdot \pi^{1, 0}_{_{J_0}} & = & \mathbbm{I}_{1, 0} 
  +  \theta \overline{\theta} (\mathbbm{I}_{1, 0} 
  - \theta \overline{\theta})^{- 1}  +
   \overline{\theta} (\mathbbm{I}_{1, 0}  -
   \theta \overline{\theta})^{- 1} \\
  &  & \\
  & = & (\mathbbm{I}_{1, 0}  -  \theta
  \overline{\theta})^{- 1}  +  \overline{\theta}
  (\mathbbm{I}_{1, 0}  - \theta
  \overline{\theta})^{- 1} ,\\
  &  & \\
  \frac{1}{2} \; [\mathbbm{I} -  i \, (J_{\theta}
  \circ f)] \cdot \pi^{0, 1}_{_{J_0}} & = & -  \theta
  (\mathbbm{I}_{0, 1} - \overline{\theta} \theta)^{- 1}
   -  \overline{\theta} \theta (\mathbbm{I}_{0, 1}
   - \overline{\theta} \theta)^{- 1} .
\end{eqnarray*}
The second equality follows from the trivial identity 
$$
\mathbbm{I} \; + \; A
\left( \mathbbm{I} \; - \; A \right)^{- 1} \; = \; \left( \mathbbm{I} \; - \;
A \right)^{- 1}.
$$
We deduce that the holomorphy condition for $f$ writes in
the form
\begin{eqnarray*}
  0 & = & \overline{\partial}_{_{J_0, J_{\theta}}} f \cdot \pi^{0, 1}_{_{J_0}}
  +  \partial_{_{J_0, J_{\theta}}} f \cdot
  \varphi\\
  &  & \\
  & = & (\mathbbm{I}_{1, 0} -  \theta
  \overline{\theta})^{- 1} \cdot ( \overline{\partial}_{_{J_0}} f
   +  \partial_{_{J_0}} f \cdot \varphi)
   - \theta (\mathbbm{I}_{0, 1}  -
  \overline{\theta} \theta)^{- 1} \cdot (\partial_{_{J_0}} f
   +  \overline{\partial}_{_{J_0}} f \cdot \varphi)\\
  &  & \\
  & + & \overline{\theta} (\mathbbm{I}_{1, 0}  - 
  \theta \overline{\theta})^{- 1} \cdot ( \overline{\partial}_{_{J_0}} f
  +  \partial_{_{J_0}} f \cdot \varphi)
   - \; \overline{\theta} \theta (\mathbbm{I}_{0, 1}
   - \overline{\theta} \theta)^{- 1} \cdot
  (\partial_{_{J_0}} f  + \overline{\partial}_{_{J_0}} f
  \cdot \varphi) .
\end{eqnarray*}
The fact that the second line is composed by elements in $\mathcal{E}^{0, 1}
\left( T_{X, J_0}^{1, 0} \right)$ and the third by elements in
$\mathcal{E}^{0, 1} \left( T_{X, J_0}^{0, 1} \right)$ implies that the
holomorphy condition for $f$ is equivalent to the equations
\begin{eqnarray*}
  (\mathbbm{I}_{1, 0}  -  \theta
  \overline{\theta})^{- 1} \cdot ( \overline{\partial}_{_{J_0}} f
   +  \partial_{_{J_0}} f \cdot \varphi) & = &
  \theta (\mathbbm{I}_{0, 1} \hspace{1.2em} -  \overline{\theta}
  \theta)^{- 1} \cdot (\partial_{_{J_0}} f  + 
  \overline{\partial}_{_{J_0}} f \cdot \varphi),\\
  &  & \\
  \overline{\theta} (\mathbbm{I}_{1, 0} -  \theta
  \overline{\theta})^{- 1} \cdot ( \overline{\partial}_{_{J_0}} f
  + \partial_{_{J_0}} f \cdot \varphi) & = &
  \overline{\theta} \theta (\mathbbm{I}_{0, 1} - 
  \overline{\theta} \theta)^{- 1} \cdot (\partial_{_{J_0}} f  +
   \overline{\partial}_{_{J_0}} f \cdot \varphi) .
\end{eqnarray*}
But the last one is obtained multiplying both sides of the first with
$\overline{\theta}$. We infer that the holomorphy condition for $f$ writes as
\begin{eqnarray*}
  \pi_{_{J_0}}^{1, 0} \cdot \overline{\partial}_{_{J_0}} f +
   \partial_{_{J_0}} f \cdot \varphi & = & (\mathbbm{I}_{1, 0}
  -  \theta \overline{\theta}) \theta
  (\mathbbm{I}_{0, 1}  - \overline{\theta}
  \theta)^{- 1} \cdot (\partial_{_{J_0}} f  + 
  \overline{\partial}_{_{J_0}} f \cdot \varphi).
\end{eqnarray*}
We notice now the identity
\begin{eqnarray*}
  \theta & = & (\mathbbm{I}_{1, 0}  -  \theta
  \overline{\theta}) \theta (\mathbbm{I}_{0, 1}  -
  \overline{\theta} \theta)^{- 1} .
\end{eqnarray*}
The latter follows decomposing the trivial identity
\begin{eqnarray*}
  \mu & = & \left( \mathbbm{I}  - \mu^2 \right)
  \mu \left( \mathbbm{I}  - \mu^2 \right)^{- 1}
  .
\end{eqnarray*}
We conclude finally that the map $f : \left( X, J_{\varphi} \right)
\longrightarrow \left( X, J_{\theta} \right)$ is holomorphic if and only if
\begin{eqnarray*}
  \pi_{_{J_0}}^{1, 0} \cdot \overline{\partial}_{_{J_0}} f  +
  \partial_{_{J_0}} f \cdot \varphi & = &  \left( \theta \circ
  f \right) \cdot (\partial_{_{J_0}} f + 
  \overline{\partial}_{_{J_0}} f \cdot \varphi).
\end{eqnarray*}
For any $f \in \tmop{Diff} \left( X \right)$ sufficiently close to the
identity in $C^1$-norm, the almost complex structure $f^{\ast} J_{\theta}$ is
$J_0$-compatible, i.e. $\det \left( J_0 + f^{\ast} J_{\theta} \right) 
\neq  0$. Thus there exists a unique form $\theta_f$ such that $f^{\ast}
J_{\theta}  \; = \; J_{\theta_f}$. 

By definition the map $f : \left( X,
J_{\theta_f} \right) \longrightarrow \left( X, J_{\theta} \right)$ is
holomorphic. We conclude that $\theta_f$ is given by the formula
\begin{equation}
  \pi_{_{J_0}}^{1, 0} \cdot \overline{\partial}_{_{J_0}} f  -
   \left( \theta \circ f \right) \cdot \partial_{_{J_0}} f = -
  \left[ \partial_{_{J_0}} f  - \left( \theta
  \circ f \right) \cdot \overline{\partial}_{_{J_0}} f \right] \cdot \theta_f
  , \label{teta-F}
\end{equation}
and thus
\begin{eqnarray*}
  \theta_f & = & -  \left[ \partial_{_{J_0}} f  -
   \left( \theta \circ f \right) \cdot \overline{\partial}_{_{J_0}} f
  \right]_{\mid T^{1, 0}_{X, J_0}}^{- 1} \left[ \pi_{_{J_0}}^{1, 0} \cdot
  \overline{\partial}_{_{J_0}} f  -  \left( \theta
  \circ f \right) \cdot \partial_{_{J_0}} f \right] ,
\end{eqnarray*}
as long as
\begin{eqnarray*}
  \left[ \partial_{_{J_0}} f  -  \left( \theta \circ f \right)
  \cdot \overline{\partial}_{_{J_0}} f \right]_{\mid T^{1, 0}_{X, J_0}} & \in
  & \tmop{GL}_{_{\mathbbm{C}}} \left( T^{1, 0}_{X, J_0} \right) .
\end{eqnarray*}
Adding the complex conjugate we infer
\begin{eqnarray}
  \overline{\partial}_{_{J_0}} f  -  \left( \mu
  \circ f \right) \cdot \partial_{_{J_0}} f & = & - \left[
  \partial_{_{J_0}} f  -  \left( \mu \circ f \right) \cdot
  \overline{\partial}_{_{J_0}} f \right] \cdot \mu_f , 
  \label{mu-F}
\end{eqnarray}
and thus
\begin{eqnarray*}
  \mu_f & = & -  \left[ \partial_{_{J_0}} f -
  \left( \mu \circ f \right) \cdot \overline{\partial}_{_{J_0}} f \right]^{-
  1} \left[ \overline{\partial}_{_{J_0}} f - 
  \left( \mu \circ f \right) \cdot \partial_{_{J_0}} f \right],
\end{eqnarray*}
as long as
\begin{eqnarray*}
  \partial_{_{J_0}} f  -  \left( \mu \circ f \right) \cdot
  \overline{\partial}_{_{J_0}} f & \in & \tmop{GL} \left( T_X \right) .
\end{eqnarray*}

\subsubsection{The Kuranishi space of a compact complex manifold}\label{Kuran-Sect}

Let $\left( X, J \right)$ be a complex manifold and consider
\begin{eqnarray*}
  E_J'' & \assign & T^{\ast}_{X, - J} \otimes T_{X, J},\\
  &  & \\
  E_{g, J}'' & \assign & \tmop{End}_g (T_X) \cap E_J'',\\
  &  & \\
  \mathcal{C}_J & : = & \left\{ \mu \in \mathcal{E}(E_J'') \mid \left( 1 + \mu
  \right) \in \tmop{GL}_{_{\mathbbm{R}}} \left( T_X \right),
  \overline{\partial}_{T_{X, J}}  \, \mu + \frac{1}{2}  \left[ \mu, \mu
  \right] = 0 \right\} .
\end{eqnarray*}
Then the Caley transform (see \cite{Gau}) provides a bijection
\begin{eqnarray*}
  \tmop{Cal}_J : \mathcal{C}_J & \longrightarrow & \mathcal{J}_{\tmop{int}} \\
  &  & \\
  \mu & \longleftrightarrow & J \assign (\mathbbm{I}+ \mu) J_0  (\mathbbm{I}+
  \mu)^{- 1}\\
  &  & \\
  \mu : = (J_0 + J)^{- 1} (J_0 - J) & \longleftrightarrow & J .
\end{eqnarray*}
For notations convenience we will restrict our considerations to the Fano case
even if the result that will follow and its argument holds for a general
compact complex manifold. For any polarized Fano manifold $(X, J, \omega)$ we
define also the sub-set of $\Omega$-divergence free tensors in $\mathcal{C}_J$
\begin{eqnarray*}
  \mathcal{C}^{\tmop{div}}_{J, g} & \assign & \left\{ \mu \in \mathcal{C}_J
  \mid \overline{\partial}^{\ast_{g, \Omega}}_{T_{X, J}} \mu =
  0_{_{_{_{_{}}}}} \right\} \nosymbol .
\end{eqnarray*}
We denote by $H^0 \left( T_{X, J} \right)^{\bot} \cap W_k \left( T_{X, J}
\right)$ the $L_{g, \Omega}^2$-orthogonal space to the space of holomorphic
vector fields inside $W_k \left( T_{X, J}
\right)$. For any $\xi \in \mathcal{E} \left( T_{X, J} \right)$ of
sufficiently small norm the map $e \left( \xi \right) : X \longrightarrow X$
defined by
\begin{eqnarray*}
  e \left( \xi \right)_x & : = & \exp_{g, x} \left( \xi_x \right),
\end{eqnarray*}
is a smooth diffeomorphism. For readers convenience we provide a proof (in the
Fano case) of the following fundamental result due to Kuranishi \cite{Kur}.

\begin{theorem}
  \label{Kuranishi}{\tmstrong {$($The Kuranishi space $\mathcal{K}_{J, g}$.$)$}}
   For any polarized Fano manifold $(X, J,
  \omega)$ and any integer $k > n + 1$ with $n \assign \dim_{_{\mathbbm{C}}}
  X$ there exists;
  
  {\bf(A)}
   $\varepsilon, \delta \in \mathbbm{R}_{> 0}$,
  a complex analytic subset $\mathcal{K}_{J, g} \subseteq \mathcal{H}_{g,
  \Omega}^{0, 1} \left( T_{X, J} \right) \cap B^g_{\delta} \left( 0 \right)$,
  $0 \in \mathcal{K}_{J, g}$ and a holomorphic embedding
  \begin{eqnarray*}
    \mu : \mathcal{H}_{g, \Omega}^{0, 1} \left( T_{X, J} \right) \cap
    B^g_{\delta} \left( 0 \right) & \longrightarrow & B^{W^{0, 1}_k \left(
    T_{X, J} \right)}_{\varepsilon} \left( 0 \right),
  \end{eqnarray*}
  with $\mu_0 = 0$, which restricts to a bijection
  \begin{eqnarray*}
    \mu : \mathcal{K}_{J, g} & \longrightarrow & \mathcal{C}^{\tmop{div}}_{J,
    g} \cap B^{W^{0, 1}_k \left( T_{X, J}  \right)}_{\varepsilon} \left( 0
    \right),
  \end{eqnarray*}
  with the property $d_0 \mu \left( v \right) = v$, for all $v \in
  \tmop{TC}_{\mathcal{K}_{J, g}, 0} \assign$the tangent cone of
  $\mathcal{K}_{J, g}$ at the origin.
  
  {\bf (B)} $\varepsilon_0 \nocomma \nocomma \in
  \mathbbm{R}_{> 0}$, $\varepsilon_0 < \varepsilon$, and a smooth map
  \begin{eqnarray*}
    B^{W^{0, 1}_k \left( T_{X, J} \right)}_{\varepsilon_0} \left( 0 \right) &
    \longrightarrow & H^0 \left( T_{X, J} \right)^{\bot} \cap W_k \left( T_{X,
    J} \right)\\
    &  & \\
    \varphi & \longmapsto & \xi_{\varphi},
  \end{eqnarray*}
  with $\xi_0 = 0$, such that $\overline{\partial}^{\ast_{g, \Omega}}_{T_{X,
  J}} \varphi_{e \left( \xi_{\varphi} \right)} = 0$ which restricts to an
  application
  \begin{eqnarray*}
    B^{W^{0, 1}_k \left( T_{X, J} \right)}_{\varepsilon_0} \left( 0 \right)
    \cap \mathcal{E}^{0, 1} \left( T_{X, J} \right) & \longrightarrow & H^0
    \left( T_{X, J} \right)^{\bot} \cap \mathcal{E} \left( T_{X, J} \right),
  \end{eqnarray*}
  and such that the map
  \begin{eqnarray*}
    \mathcal{C}_J \cap B^{W^{0, 1}_k \left( T_{X, J} \right)}_{\varepsilon_0}
    \left( 0 \right) & \longrightarrow & \mathcal{C}^{\tmop{div}}_{J, g} \cap
    B^{W^{0, 1}_k \left( T_{X, J} \right)}_{\varepsilon} \left( 0 \right)\\
    &  & \\
    \varphi & \longmapsto & \mu \left( \varphi \right) \assign \varphi_{e
    \left( \xi_{\varphi} \right)},
  \end{eqnarray*}
  is well defined.
\end{theorem}

\begin{proof}
  We divide Kuranishi's proof in a few steps.
  
  {\tmstrong{STEP A1}}. We show first that the system
  \[ (S_1)  \left\{ \begin{array}{l}
       \overline{\partial}_{T_{X, J}}  \, \mu + \frac{1}{2}  \left[ \mu, \mu
       \right] = 0,\\
       \\
       \overline{\partial}^{\ast_{g, \Omega}}_{T_{X, J}} \mu = 0 .
     \end{array} \right. \]
  is equivalent to the system
  \[ (S_2)  \left\{ \begin{array}{l}
       \mu + \frac{1}{2}  \overline{\partial}^{\ast_{g, \Omega}}_{T_{X, J}}
       G_{T_{X, J}}  \left[ \mu, \mu \right] = H_{T_{X, J}}  \, \mu,\\
       \\
       H_{T_{X, J}} \, \left[ \mu, \mu \right] = 0,
     \end{array} \right. \]
  provided that $\mu$ is sufficiently close to $0$. Indeed let $\mu$ be a
  solution of $(S_1)$. Then the considerations about the resolution of the
  $\overline{\partial}$-equation imply the second equation in $(S_2)$.
  Moreover if we set
  \begin{eqnarray*}
    \varphi & \assign & - \frac{1}{2}  \overline{\partial}^{\ast_{g,
    \Omega}}_{T_{X, J}} G_{T_{X, J}}  \left[ \mu, \mu \right],
  \end{eqnarray*}
  then $\alpha \assign \mu - \varphi$ satisfies $\overline{\partial}_{T_{X,
  J}}  \, \alpha = 0$ and $\overline{\partial}^{\ast_{g, \Omega}}_{T_{X, J}}
  \alpha = 0$. Thus $\alpha \in \mathcal{H}_{g, \Omega}^{0, 1} (T_{X, J})$ and
  $H_{T_{X, J}} \mu = \alpha$ since
  \begin{eqnarray*}
    H_{T_{X, J}} \overline{\partial}^{\ast_{g, \Omega}}_{T_{X, J}} & = & 0,
  \end{eqnarray*}
  by orthogonality. This shows that also the first equation in $(S_2)$ holds.
  Assume now that $\mu$ is a solution of $(S_2)$. It is clear that the second
  equation in $(S_1)$ holds true. We set
  \begin{eqnarray*}
    \psi & \assign & \overline{\partial}_{T_{X, J}}  \, \mu + \frac{1}{2} 
    \left[ \mu, \mu \right],
  \end{eqnarray*}
  and we observe the equalities
  \begin{eqnarray*}
    \psi & = & - \frac{1}{2}  \overline{\partial}_{T_{X, J}}
    \overline{\partial}^{\ast_{g, \Omega}}_{T_{X, J}} G_{T_{X, J}}  \left[
    \mu, \mu \right] + \frac{1}{2}  \left[ \mu, \mu \right]\\
    &  & \\
    & = & \frac{1}{2}  \overline{\partial}^{\ast_{g, \Omega}}_{T_{X, J}}
    \overline{\partial}_{T_{X, J}}  \, G_{T_{X, J}}  \left[ \mu, \mu \right]\\
    &  & \\
    & = & \frac{1}{2}  \overline{\partial}^{\ast_{g, \Omega}}_{T_{X, J}}
    G_{T_{X, J}} \overline{\partial}_{T_{X, J}}  \left[ \mu, \mu \right]\\
    &  & \\
    & = & \overline{\partial}^{\ast_{g, \Omega}}_{T_{X, J}} G_{T_{X, J}} 
    \left[ \overline{\partial}_{\mathcal{T}} \mu, \mu \right] .
  \end{eqnarray*}
  We deduce the identity
  \begin{eqnarray*}
    \psi & = & \overline{\partial}^{\ast_{g, \Omega}}_{T_{X, J}} G_{T_{X, J}} 
    \left[ \psi, \mu \right] .
  \end{eqnarray*}
  The assumption $k > n + 1$ implies that the Sobolev embedding $W^k \subset
  C^{k - n}$ holds true. Using the standard estimates on the Sobolev norms of
  $W^{\bullet, \bullet}_k$
  \begin{eqnarray*}
    \|G_{T_{X, J}}  \, \varphi \|_{k + 2} & \leqslant & C_0  \| \varphi
    \|_k,\\
    &  & \\
    \| \left[ \varphi, \psi \right] \|_{k - 1} & \leqslant & C_2  \| \varphi
    \|_k \| \psi \|_k,
  \end{eqnarray*}
  we obtain
  \begin{eqnarray*}
    \| \psi \|_k & \leqslant & C_1  \| \left[ \psi, \mu \right] \|_{k - 1}
    \leqslant C_1 C_2  \| \psi \|_k  \| \mu \|_k .
  \end{eqnarray*}
  Thus if $\| \mu \|_k \leqslant \varepsilon / (C_1 C_2)$ for some
  $\varepsilon \in (0, 1)$ then $(1 - \varepsilon) \| \psi \|_k \leqslant 0$,
  which holds true if and only if $\psi = 0$.
  
  {\tmstrong{STEP A2}}. We remind that the previous discussion shows that the
  first equation in $(S_2)$ is equivalent to the condition
  \begin{eqnarray*}
    F (\mu) \assign \mu + \frac{1}{2}  \overline{\partial}^{\ast_{g,
    \Omega}}_{T_{X, J}} G_{T_{X, J}}  \left[ \mu, \mu \right] & \in &
    \mathcal{H}_{g, \Omega}^{0, 1} (T_{X, J}) .
  \end{eqnarray*}
  Let $\Xi_k \subset W_k^{0, 1} (T_{X, J})$ be the subset of the elements
  satisfying this condition. We notice that the map
  \begin{eqnarray*}
    F : W^{0, 1}_k (T_{X, J}) & \longrightarrow & W^{0, 1}_k (T_{X, J}),
  \end{eqnarray*}
  is well defined and continuous thanks to the estimate
  \begin{eqnarray*}
    \| \overline{\partial}^{\ast_{g, \Omega}}_{T_{X, J}} G_{T_{X, J}}  \left[
    \mu, \mu \right] \|_k & \leqslant & C_1  \| \left[ \mu, \mu \right] \|_{k
    - 1} \leqslant C_1 C_2  \| \mu \|^2_k .
  \end{eqnarray*}
  We infer that $F$ is also holomorphic since $F -\mathbbm{I}$ is a continuous
  quadratic form. The fact that the differential of $F$ at the origin is the
  identity implies the existence of an inverse holomorphic map $F^{- 1}$ in a
  neighborhood $B^{W_k}_{\varepsilon} (0)$ of the origin. Restricting this to
  $\mathcal{H}_{g, \Omega}^{0, 1} (T_{X, J}) \cap B^{W_k}_{\varepsilon} (0)$
  we deduce the existence of a holomorphic map
  \begin{eqnarray*}
    \alpha \in \mathcal{H}_{g, \Omega}^{0, 1} (T_{X, J}) \cap
    B^{W_k}_{\varepsilon} (0) & \longmapsto & \mu_{\alpha} \in W^{0, 1}_k
    (T_{X, J}),
  \end{eqnarray*}
  such that
  \begin{eqnarray*}
    \mu_{\alpha} + \frac{1}{2}  \overline{\partial}^{\ast_{g,
    \Omega}}_{T_{X, J}}
    G_{T_{X, J}}  \left[ \mu_{\alpha}, \mu_{\alpha} \right] & = & \alpha .
  \end{eqnarray*}
  By construction $\tmop{Im} \left( \alpha \longmapsto \mu_{\alpha} \right)$
  represents a neighborhood of the origin inside $\Xi_k$. It is clear that
  $\mu_{\alpha}$ is of class $C^{k - n}$ by the Sobolev embedding. We show
  further that $\mu_{\alpha}$ is smooth for a sufficiently small choice of
  $\varepsilon$. Indeed applying the Hodge Laplacian $\Delta^{\Omega, -
  J}_{T_{X, g}}$ to both sides of the previous identity and using the
  equalities
  \begin{eqnarray*}
    \Delta^{\Omega, - J}_{T_{X, g}} \, \overline{\partial}^{\ast_{g,
    \Omega}}_{T_{X, J}} G_{T_{X, J}} =  \overline{\partial}^{\ast_{g,
    \Omega}}_{T_{X, J}} \Delta^{\Omega, - J}_{T_{X, g}} G_{T_{X, J}}
     =  \overline{\partial}^{\ast_{g,
    \Omega}}_{T_{X, J}} ,
  \end{eqnarray*}
  (notice that $\overline{\partial}^{\ast_{g, \Omega}}_{T_{X, J}} H_{T_{X, J}}
  \; = \; 0$) we obtain the equation
  \begin{eqnarray*}
    \Delta^{\Omega, - J}_{T_{X, g}} \mu_{\alpha} \noplus \noplus \; + \; 
    \frac{1}{2}  \overline{\partial}^{\ast_{g, \Omega}}_{T_{X, J}} \left[
    \mu_{\alpha}, \mu_{\alpha} \right] & = & 0 \;,
  \end{eqnarray*}
  which rewrites also as
  \begin{eqnarray*}
    \Delta^{\Omega, - J}_{T_{X, g}}  \, \mu_{\alpha} \noplus \noplus \; + \; 
    \frac{1}{2} \; \mu_{\alpha} \ast \nabla_{g_{}}^2 \mu_{\alpha} & = &
    \frac{1}{2} \nabla_{g_{}} \mu_{\alpha} \ast \nabla_{g_{}} \mu_{\alpha} +
    \frac{1}{2} \; \mu_{\alpha} \ast \nabla_{g_{}}^{} \mu_{\alpha} \ast
    \nabla_g f_{_{_{_{}}}}  \;,
  \end{eqnarray*}
  where $\ast$ denotes adequate contraction operators. The fact that the
  $C^0$-norm of $\mu_{\alpha}$ can be made arbitrary small for sufficiently
  small $\varepsilon$ implies that the operator
  \[ \Delta^{\Omega, - J}_{T_{X, g}} \noplus \noplus \; + \;  \frac{1}{2} \;
     \mu_{\alpha} \ast \nabla_{g_{}}^2  \;, \]
  is elliptic. Then the smoothness of $\mu_{\alpha}$ follows by standard
  elliptic bootstrapping. We denote by $\mathcal{K}_{J, g}$ the zero set of
  the holomorphic map
  \begin{eqnarray*}
    \chi : \mathcal{H}_{g, \Omega}^{0, 1} (T_{X, J}) \cap
    B^{W_k}_{\varepsilon} (0) & \longrightarrow & \mathcal{H}_{g, \Omega}^{0,
    2} (T_{X, J})\\
    &  & \\
    \alpha & \longmapsto & H_{T_{X, J}} \left[ \mu_{\alpha}, \mu_{\alpha}
    \right]  \; .
  \end{eqnarray*}
  Then the set $\left\{ \mu_{\alpha} \mid \alpha \in \mathcal{K}_{J, g}
  \right\}$ covers the set of the solutions of the system $\left( S_2 \right)$
  in a neighborhood of the origin.
  
  {\tmstrong{STEP B}}. We observe first that $\overline{\partial}^{\ast_{g,
  \Omega}}_{T_{X, J}} \varphi = 0$ if and only if $G_{T_{X, J}}
  \overline{\partial}^{\ast_{g, \Omega}}_{T_{X, J}} \varphi = 0$. Indeed
  \[ \tmop{Im} \overline{\partial}^{\ast_{g, \Omega}}_{T_{X, J}} \bot
     \tmop{Ker} G_{T_{X, J}}, \]
  since $\tmop{Ker} G_{T_{X, J}} = \tmop{Ker} \Delta^{\Omega, - J}_{T_{X,
  g}}$. Thus in order to construct the application $\varphi \longmapsto
  \xi_{\varphi}$ we need to find the zeros of the map
  \begin{eqnarray*}
    R : W^{0, 1}_k \left( T_{X, J} \right) \times \left[ H^0 \left( T_{X, J}
    \right)^{\bot} \cap B^{W_k \left( T_{X, J} \right)}_{\varepsilon_0} \left(
    0 \right) \right] & \longrightarrow & H^0 \left( T_{X, J} \right)^{\bot}
    \cap W_k \left( T_{X, J} \right)\\
    &  & \\
    \left( \varphi, \xi \right) & \longmapsto & G_{T_{X, J}}
    \overline{\partial}^{\ast_{g, \Omega}}_{T_{X, J}} \varphi_{e \left( \xi
    \right)},\\
    &  & \\
    \left( 0, 0 \right) & \longmapsto & 0 .
  \end{eqnarray*}
  For notations simplicity we denote $\Psi : \left( \mu, f \right) \longmapsto
  \mu_f$. With these notations the formula (\ref{mu-F}) writes as
  \begin{eqnarray*}
    \overline{\partial}_{_J} e \left( t \xi \right) & = & - \partial_{_J} e
    \left( t \xi \right) \cdot \Psi \left( 0, e \left( t \xi \right) \right) .
  \end{eqnarray*}
  Time deriving this identity at $t = 0$ and using the fact that $\frac{d}{d
  t}_{\mid t = 0} e \left( t \xi \right) = \xi$, $\Psi \left( 0, \tmop{Id}_X
  \right) = 0$ and $e \left( 0 \right) = \tmop{Id}_X$ we obtain
  \begin{eqnarray*}
    \overline{\partial}_{T_{X, J}}  \, \xi & = & - D_f \Psi \left( 0,
    \tmop{Id}_X \right) \cdot \xi,
  \end{eqnarray*}
  where $D_f \Psi$ denotes the partial Frechet derivative of $\Psi$ in the
  variable $f$. We observe now that for any $\xi \in W_k \left( T_{X, J}
  \right)$ holds the decomposition formula
  \begin{eqnarray*}
    \xi & = & H_{T_{X, J}}  \, \xi \noplus + G_{T_{X, J}}
    \overline{\partial}^{\ast_{g, \Omega}}_{T_{X, J}}
    \overline{\partial}_{T_{X, J}} \xi .
  \end{eqnarray*}
  Thus if $\xi \in H^0 \left( T_{X, J} \right)^{\bot} \cap W_k \left( T_{X, J}
  \right)$ then holds the identity
  \begin{eqnarray*}
    \xi & = & G_{T_{X, J}} \overline{\partial}^{\ast_g}_{T_{X, J}}
    \overline{\partial}_{T_{X, J}} \xi .
  \end{eqnarray*}
  We conclude the identity $D_{\xi} R \left( 0, 0 \right) =\mathbbm{I}$ and
  the existence of the map $\varphi \longmapsto \xi_{\varphi}$ by the implicit
  function theorem. In local coordinates we can consider the expansion
  \begin{eqnarray*}
    e \left( \xi \right) & = & \tmop{Id}_X + \xi + O ( \left| \xi
    \right|^2) .
  \end{eqnarray*}
  Then the formula (\ref{mu-F}) implies the local identity
  \begin{eqnarray*}
    \varphi_{e \left( \xi \right)} & = & \overline{\partial}_{T_{X, J}} \xi +
    \varphi + Q \left( \varphi, \xi \right),
  \end{eqnarray*}
  with $Q$ an analytic function (depending on the local coordinates) Then the
  condition $\overline{\partial}^{\ast_{g, \Omega}}_{T_{X, J}} \varphi_{e
  \left( \xi_{\varphi} \right)} = 0$ implies
  \[ \Delta^{\Omega, - J}_{T_{X, g}} \xi_{\varphi} +
     \overline{\partial}^{\ast_{g, \Omega}}_{T_{X, J}} \varphi +
     \overline{\partial}^{\ast_{g, \Omega}}_{T_{X, J}} Q \left( \varphi,
     \xi_{\varphi} \right) = 0 . \]
  Thus $\xi_{\varphi}$ is smooth if $\varphi$ is smooth by elliptic
  regularity.
\end{proof}

\subsubsection{Parametrization of a sub-space of the $\omega$-compatible
complex structures}\label{Pol-Kuran-Sect}

Let $\left( X, J, \omega \right)$ be a polarized Fano manifold and consider
the set
\begin{eqnarray*}
  \mathcal{C}_{\omega, J} & \assign & \left\{ \mu \in \mathcal{E}(E_{g,
  J}'') \mid g \left( 1 \pm \mu \right) > 0, \overline{\partial}_{T_{X, J}} 
  \, \mu + \frac{1}{2}  \left[ \mu, \mu \right] = 0 \right\},
\end{eqnarray*}
with $g \assign - \omega J$. Then the Caley transform restricts to a bijection
(see \cite{Gau})
\begin{eqnarray*}
  \tmop{Cal}_J : \mathcal{C}_{\omega, J} & \longrightarrow &
  \mathcal{J}_{\omega} .
\end{eqnarray*}
We define also the sub-set of $\Omega$-divergence free tensors in
$\mathcal{C}_{\omega, J}$
\begin{eqnarray*}
  \mathcal{C}^{\tmop{div}}_{\omega, J} & \assign & \left\{ \mu \in
  \mathcal{C}_{\omega, J} \mid \overline{\partial}^{\ast_{g, \Omega}}_{T_{X,
  J}} \mu = 0_{_{_{_{_{}}}}} \right\} \nosymbol .
\end{eqnarray*}
\begin{definition}
  \label{Pol-Kuranishi}{\tmstrong{$($The Kuranishi space of polarized deformations$)$}} For
  any polarized Fano manifold $(X, J, \omega)$ we define the Kuranishi space
  of $\omega$-polarized complex deformations as the complex analytic subset
  \begin{eqnarray*}
    \mathcal{K}^{\omega}_J & \assign & \left\{ \alpha \in \mathcal{K}_{J, g}
    \mid \mu_{\alpha} = \left( \mu_{\alpha} \right)_g^T \right\} .
  \end{eqnarray*}
\end{definition}

With these notations the map $\mu$ in theorem \ref{Kuranishi} restricts to a
bijection
\begin{eqnarray*}
  \mu : \mathcal{K}^{\omega}_J & \longrightarrow &
  \mathcal{C}^{\tmop{div}}_{\omega, J} \cap B^{W^{0, 1}_k \left( T_{X, J}
  \right)}_{\varepsilon} \left( 0 \right) .
\end{eqnarray*}
For any $\alpha\in \mathcal{K}^{\omega}_J$ we define $J_{\alpha}:=\tmop{Cal}_J\mu_{\alpha}$.
Let also $\mathcal{U}_{\omega} \subset C_{\Omega}^{\infty} \left( X,
\mathbbm{C} \right)_0$ be an open neighborhood of the origin such that $\omega
+ d d_{_{J_{\alpha}}}^c u_1 > 0$ for all $\alpha \in \mathcal{K}^{\omega}_J$
and $u = u_1 + i u_2 \in \mathcal{U}_{\omega}$, with $u_j$ real valued. We
define the real vector field
\begin{eqnarray*}
  \xi^{\alpha, u}_t & \assign & - \left( \omega + t d d_{_{J_{\alpha}}}^c u_1
  \right)^{- 1} \left( d_{_{J_{\alpha}}}^c u_1 \noplus + \frac{1}{2} d u_2
  \right),
\end{eqnarray*}
for all $t \in \left( - \varepsilon, 1 + \varepsilon \right)$, for some small
$\varepsilon > 0$. We define also the family of diffeomorphisms $\left(
\Phi^{\alpha, u}_t \right)_{t \in \left( - \varepsilon, 1 + \varepsilon
\right)}$ over $X$ given by $\partial_t \Phi^{\alpha, u}_t = \xi^{\alpha, u}_t
\circ \Phi_t^{\alpha, u}$, with $\Phi_0^{\alpha, u} = \tmop{Id}_X$. We set
finally
\begin{eqnarray*}
  J_{\alpha, u} & \assign & \left( \Phi^{\alpha, u}_1 \right)^{\ast}
  J_{\alpha} .
\end{eqnarray*}
With these notations holds the following lemma.

\begin{lemma}
  \label{param-pol-Cxdef}The map
  \begin{eqnarray*}
    \mathcal{K}^{\omega}_J \times \mathcal{U}_{\omega} & \longrightarrow &
    \mathcal{J}_{\omega},\\
    &  & \\
    \left( \alpha, u \right) & \longmapsto & J_{\alpha, u},
  \end{eqnarray*}
  is well defined and its differential at the origin is given by the fiberwise
  injection
  \begin{eqnarray*}
    \tmop{TC}_{\mathcal{K}^{\omega}_J, 0} \oplus \Lambda^{\Omega, \bot}_{g, J}
    & \longrightarrow & \tmop{TC}_{\mathcal{J}_{\omega}, J}\\
    &  & \\
    \left( A, v \right) & \longmapsto & - J \left[ \overline{\partial}_{T_{X,
    J}} \nabla_{g, J}  \overline{v} + 2 A_{_{_{_{}}}} \right] .
  \end{eqnarray*}
\end{lemma}

\begin{proof}
  Let denote for simplicity $\omega_t \assign \omega + t d d_{_{J_{\alpha}}}^c
  u_1$ and we observe the elementary identities
  \begin{eqnarray*}
    \dot{\omega}_t & = & d d_{_{J_{\alpha}}}^c u_1 = - d (\xi^{\alpha, u}_t
    \neg \omega_t) = - L_{\xi^{\alpha, u}_t} \omega_t ._{}
  \end{eqnarray*}
  We infer
  \begin{eqnarray*}
    \frac{d}{d t}  \left[_{_{_{_{_{}}}}} \left( \Phi^{\alpha, u}_t
    \right)^{\ast} \omega_t \right] & = & \left( \Phi^{\alpha, u}_t
    \right)^{\ast}  \left( \dot{\omega}_t \noplus + L_{\xi^{\alpha, u}_t}
    \omega_{t_{_{_{}}}} \right) = 0,
  \end{eqnarray*}
  and thus $\left( \Phi^{\alpha, u}_1 \right)^{\ast} \omega_1 = \left(
  \Phi^{\alpha, u}_0 \right)^{\ast} \omega_0 = \omega$, i.e.
  \begin{eqnarray*}
    \left( \Phi^{\alpha, u}_1 \right)^{\ast} \left( \omega + d
    d_{_{J_{\alpha}}}^c u_1 \right)  & = & \omega .
  \end{eqnarray*}
  The fact that the complex structure $J_{\alpha}$ is integrable implies that
  the form $\omega_1$ is $J_{\alpha}$-invariant. (This is no longer true in
  the non-integrable case!) We conclude $J_{\alpha, u} \in
  \mathcal{J}_{\omega}$.
  
  We compute now the differential at the origin. We consider for this purpose
  a smooth family $\left( u \left( s \right) \right)_s \subset
  \mathcal{U}_{\omega}$ such that $u \left( 0 \right) = 0$ and $\dot{u} \left(
  0 \right) = v$. We denote for simplicity $\xi_{t, s} \assign \xi_t^{0, u
  \left( s \right)}$ and $\Phi_{t, s} \assign \Phi^{0, u \left( s \right)}_t$.
  Then deriving with respect to $s$ at $s = 0$ the identity
  \begin{eqnarray*}
    \frac{\partial}{\partial t} \Phi_{t, s} & = & \xi_{t, s} \circ \Phi_{t,
    s},
  \end{eqnarray*}
  and using the fact that $\xi_{t, 0} = 0$, (which implies in particular
  $\Phi_{t, 0} = \tmop{Id}_X$) we obtain
  \begin{eqnarray*}
    \frac{\partial}{\partial s} _{\mid_{s = 0}}  \frac{\partial}{\partial t}
    \Phi_{t, s} & = & \frac{\partial}{\partial s} _{\mid_{s = 0}} \xi_{t, s} 
    \; + \; d_x \xi_{t, 0} \cdot \frac{\partial}{\partial s} _{\mid_{s = 0}}
    \Phi_{t, s}\\
    &  & \\
    & = & \frac{\partial}{\partial s} _{\mid_{s = 0}} \xi_{t, s} .
  \end{eqnarray*}
  On the other hand deriving with respect to $s$ at $s = 0$ the identity
  \begin{eqnarray*}
    \xi_{t, s} \neg \left( \omega + t d d_{_J}^c u_1 \left( s
    \right)_{_{_{_{}}}} \right) & = & - \left( d_{_J}^c u_1 \left( s \right)
    \noplus + \frac{1}{2} d u_2 \left( s \right) \right),
  \end{eqnarray*}
  we obtain
  \begin{eqnarray*}
    \left( \frac{\partial}{\partial s} _{\mid_{s = 0}} \xi_{t, s} \right) \neg
    \; \omega & = & - \left( d_{_J}^c v_1 \noplus + \frac{1}{2} d v_2 \right),
  \end{eqnarray*}
  and thus
  \begin{eqnarray*}
    \frac{\partial}{\partial s} _{\mid_{s = 0}} \xi_{t, s} & = & - \frac{1}{2}
    \nabla_{g, J}  \overline{v} .
  \end{eqnarray*}
  Commuting the derivatives in $s$ and $t$ we infer the identity
  \begin{eqnarray*}
    \frac{\partial}{\partial t}  \frac{\partial}{\partial s} _{\mid_{s = 0}}
    \Phi_{t, s} & = & - \frac{1}{2} \nabla_{g, J}  \overline{v} .
  \end{eqnarray*}
  Integrating in $t$ from $0$ to $1$ we deduce
  \begin{eqnarray*}
    \eta \assign \frac{\partial}{\partial s} _{\mid_{s = 0}} \Phi_{1, s} & = &
    - \frac{1}{2} \nabla_{g, J}  \overline{v},
  \end{eqnarray*}
  since $\Phi_{0, s} = \tmop{Id}_X$. We infer
  \begin{eqnarray*}
    \frac{d}{d s} _{\mid_{s = 0}} J_{0, u \left( s \right)} & = & L_{\eta} J =
    - J \overline{\partial}_{T_{X, J}} \nabla_{g, J}  \overline{v} .
  \end{eqnarray*}
  Assume now $\left( \alpha \left( s \right) \right)_s \subset \mathcal{K}_{J,
  \omega}$ is a smooth curve with $\alpha \left( 0 \right) = 0$ and
  $\dot{\alpha} \left( 0 \right) = A$. Then
  \begin{eqnarray*}
    \frac{d}{d s} _{\mid_{s = 0}} J_{\alpha \left( s \right), u \left( s
    \right)} & = & \frac{d}{d s} _{\mid_{s = 0}} J_{\alpha \left( s \right)} +
    \frac{d}{d s} _{\mid_{s = 0}} J_{0, u \left( s \right)},
  \end{eqnarray*}
  with
  \begin{eqnarray*}
    \frac{d}{d s} _{\mid_{s = 0}} J_{\alpha \left( s \right)} & = & - 2 J A,
  \end{eqnarray*}
  thanks to the properties of the differential of the Caley transform (see \cite{Gau}).
\end{proof}

\begin{lemma}
  For any point $J \in \mathcal{J}_{\omega}$ holds the inclusions 
  \begin{eqnarray*}
  &&  \left[ \overline{\partial}_{T_{X, J}} \nabla_{g, J} C^{\infty} \left( X,
    \mathbbm{C} \right)_{_{_{_{}}}} \right] \oplus_{\Omega}
    \tmop{TC}_{\mathcal{K}^{\omega}_J, 0} 
    \\
    \\
    & \subseteq &
    \tmop{TC}_{\mathcal{J}_{\omega}, J}\\
    &  & \\
    & \subseteq & \left[ \overline{\partial}_{T_{X, J}} \nabla_{g, J}
    C^{\infty} \left( X, \mathbbm{C} \right)_{_{_{_{}}}} \right]
    \oplus_{\Omega} \tmop{TC}_{\mathcal{K}_{J, g}, 0} .
  \end{eqnarray*}
\end{lemma}

\begin{proof}
  The first inclusion is a direct consequence of lemma \ref{param-pol-Cxdef}.
  In order to show the second one let $\left( \varphi_t \right)_t \subset
  \mathcal{C}_J \cap B^{W^{0, 1}_k \left( T_{X, J} \right)}_{\varepsilon_0}
  \left( 0 \right)$ with $\varphi_0 = 0$ and set for notation simplicity $e_t
  \assign e \left( \xi_{\varphi_t} \right)$. With these notations, the identity
  (\ref{mu-F}) writes as
  \begin{eqnarray*}
    \overline{\partial}_{_J} e_t  -  \left(
    \varphi_t \circ e_t \right) \cdot \partial_{_J} e_t  \;=\;  - 
    \left[ \partial_{_J} e_t  -  \left( \varphi_t \circ e_t
    \right) \cdot \overline{\partial}_{_J} e_t \right] \cdot \mu \left(
    \varphi_t \right) .
  \end{eqnarray*}
  Time deriving this at $t = 0$ and using the obvious equality $\dot{e}_0 = D
  \xi \left( 0 \right) \dot{\varphi}_0$ we deduce the equality
  \begin{eqnarray*}
    \overline{\partial}_{T_{X, J}} \left[ D_{_{_{_{}}}} \xi \left( 0 \right)
    \dot{\varphi}_0 \right] - \dot{\varphi}_0  \;= \; - \frac{d}{d t} _{\mid_{t
    = 0}} \mu \left( \varphi_t \right) .
  \end{eqnarray*}
  This combined with the identity $\overline{\partial}^{\ast_{g,
  \Omega}}_{T_{X, J}} \mu \left( \varphi_t \right) \equiv 0$ implies
  \begin{eqnarray*}
    \overline{\partial}^{\ast_{g, \Omega}}_{T_{X, J}} 
    \overline{\partial}_{T_{X, J}} \left[ D_{_{_{_{}}}} \xi \left( 0 \right)
    \dot{\varphi}_0 \right] - \overline{\partial}^{\ast_{g, \Omega}}_{T_{X,
    J}}  \dot{\varphi}_0 \; = \; 0 .
  \end{eqnarray*}
  Thus if $\overline{\partial}^{\ast_{g, \Omega}}_{T_{X, J}}  \dot{\varphi}_0
  = 0$ then $D_{_{_{_{}}}} \xi \left( 0 \right) \dot{\varphi}_0 = 0$ and
  \begin{eqnarray*}
    \dot{\varphi}_0 \; = \;  \frac{d}{d t} _{\mid_{t = 0}} \mu \left( \varphi_t
    \right) .
  \end{eqnarray*}
  We infer the equality
  \begin{eqnarray*}
    &  & \left\{ A \in \mathcal{H}_{g, \Omega}^{0, 1} (T_{X, J}) \mid \exists
    \left( J_t \right)_t \subset \mathcal{J}_{\tmop{int}} :J_0=J,\; \dot{J}_0 = A
    \right\}\\
    &  & \\
    & = & \left\{ A \in \mathcal{H}_{g, \Omega}^{0, 1} (T_{X, J}) \mid
    \exists \left( \varphi_t \right)_t \subset \mathcal{C}^{\tmop{div}}_{J,g} :\varphi_0=0,\;
    \dot{\varphi}_0 = A \right\}\; =\; \tmop{TC}_{\mathcal{K}_{J, g}, 0} .
  \end{eqnarray*}
  By gauge transformation we deduce
  \begin{eqnarray*}
    &  & \left\{ A \in \mathcal{H}_{g, \Omega}^{0, 1} (T_{X, J}) \mid \exists
    \left( J_t \right)_t \subset \mathcal{J}_{\omega} :J_0=J,\; H_{T_{X, J}} 
    \dot{J}_0 = A \right\}\\
    &  & \\
    & \subseteq & \left\{ A \in \mathcal{H}_{g, \Omega}^{0, 1} (T_{X, J})
    \mid \exists \left( J_t \right)_t \subset \mathcal{J}_{\tmop{int}} :J_0=J,\;
    \dot{J}_0 = A \right\},
  \end{eqnarray*}
  and thus the required inclusion.
\end{proof}

This result combined with the existence of the isomorphism $\eta$ and with the triple decomposition identity
(\ref{smt-trp-spl-T}) implies the inclusions (\ref{First-incTC}) and (\ref{Second-incTC}) in the introduction of the paper.
\\
\\
\\
{\noindent}\tmtextbf{Acknowledgments. }We warmly thank Jean-Michel Bismut for
giving us the opportunity to explain our Soliton-K\"ahler-Ricci flow at his
weekly seminar in Orsay.

We warmly thank also Duong H. Phong for sharing with us a huge number of quite
delicate points in contemporary K\"ahler geometry.

We are also very grateful to Robert Bryant for explaining us in detail the
key $\tmop{Spin}^c$ computation in \cite{D-W-W2} and to Xianzhe Dai for providing us
more technical details about their beautiful work.

We thank Paul Gauduchon for sending us a copy of his monumental book \cite{Gau}
and for shearing his great knowledge in K\"ahler geometry with us.

We thank also Stuart Hall for explaining to us in great detail the
construction of the Dancer-Wang K\"ahler-Ricci-Soliton and the vanishing of
the stability integral in \cite{Ha-Mu2} in some particular harmonic directions
over this soliton.{\hspace*{\fill}}{\medskip}

\vspace{1cm}
\noindent
Nefton Pali
\\
Universit\'{e} Paris Sud, D\'epartement de Math\'ematiques 
\\
B\^{a}timent 425 F91405 Orsay, France
\\
E-mail: \textit{nefton.pali@math.u-psud.fr}

\begin{thebibliography}{000000}
\bibitem[Ach]{Ach} \textsc{Ache, A.G.,} 
\emph{On the uniqueness of asymptotic limits of the Ricci flow}, arXiv:1211.3387v2, (2013).
\bibitem[Bes]{Bes} \textsc{Besse, A.L.,} 
\emph{Einstein Manifolds}, Springer-Verlag, 2007.
\bibitem[C-H-I]{C-H-I} \textsc{Cao, H.D, Hamilton, R.S., Ilmanen, T.,}
\emph{Gaussian densities and stability for some Ricci solitons}, 
arXiv:math.DG/0404169.
\bibitem[Ca-Zhu]{Ca-Zhu} \textsc{Cao, H.D, Zhu, M.,} 
\emph{On second variation of Perelman's Ricci shrinker entropy}, 
arXiv:1008.0842v5 (2011), Math. Ann., 353 (2012), no. 3, 747-763.
\bibitem[Ca-He]{Ca-He} \textsc{Cao, H.D, He, C.,} 
\emph{Linear stability of Perelman's $\nu$-entropy on symmetric spaces of compact type}, arXiv:1304.2697v1, (2013).
\bibitem[Ch-Wa]{Ch-Wa} \textsc{Chen, X.X, Wang, B.,} 
\emph{Space of Ricci flows (II)}, arXiv:1405.6797v1, (2014).
\bibitem[Co-Mi]{Co-Mi} \textsc{Colding, T.H., Minicozzi, W.P.,} 
\emph{On uniqueness of tangent cones for Einstein
manifolds}, arXiv:1206.4929, (2012), Invent. Math. 196 (2014), no. 3, 515–588. 
\bibitem[D-W-W1]{D-W-W1} \textsc{Dai, X., Wang, X., Wei, G.,} 
\emph{On the stability of Riemannian manifold with paralel spinors}, Invent.Math. 161, (2005), no. 1, 151-176.
\bibitem[D-W-W2]{D-W-W2} \textsc{Dai, X., Wang, X., Wei, G.,} 
\emph{On the variational stability of K\"ahler-Einstein metrics}, Commun. Anal. Geom. 15 (2007), no. 4, 669-693.
\bibitem[Don]{Don} \textsc{Donalsdon, S.K.,} 
\emph{Remarks on gauge theory, complex geometry and 4-manifold topology}. In 
\emph{Fields Medallists'lectures}, volume 5 of World Sci. Ser. 20th 
Century Math., pages 384-403. World Sci. Publ., River Edge, NJ, 1997.
\bibitem[Ebi]{Ebi} \textsc{Ebin, D.G.} \emph{The manifolds of Riemannian metrics},
in ”Proceedings of the AMS Sym-
posia on Pure Marthematics”,
XV, (1970)
\bibitem[Fut]{Fut} \textsc{Futaki, A.} \emph{An obstruction to the existence of Einstein K\"ahler metrics}, Invent. Math. 73, 1983, 437-443.
\bibitem[Fu1]{Fu1} \textsc{Futaki, A.} \emph{K\"ahler-Einstein metrics and Integral Invariants}, 
Lecture Notes in Mathematics, 1314. Springer-Verlag, Berlin, 1988, 437-443.
\bibitem[Gau]{Gau} \textsc{Gauduchon, P.} \emph{Calabi’s extremal metrics: An elementary introduction}, book in preparation.
\bibitem[Ha-Mu1]{Ha-Mu1} \textsc{Hall, S., Murphy, T.,} \emph{On the linear stability of Kähler-Ricci solitons}, arXiv:1008.1023, (2010), 
Proc. Amer. Math. Soc. 139, No. 9, (2011) 3327-3337.
\bibitem[Ha-Mu2]{Ha-Mu2} \textsc{Hall, S., Murphy, T.,} \emph{Variation of complex structures and the stability of K\"ahler-Ricci Solitons}, 
rXiv:1206.4922, (2012), Pacific J. Math. 265 (2013), no. 2, 441–454. 
\bibitem[Ham]{Ham} \textsc{Hamilton, R.S.,}
\emph{The formation of singularities in the Ricci flow}, 
Surveys in differential geometry, Vol. II (Cambridge, MA, 1993), 7–136, Int. Press, Cambridge, MA, 1995.
\bibitem[Kur]{Kur} \textsc{Kuranishi, M.,} \emph{On the locally complete families of complex analytic structures}, Ann. of Math. (2) 75 1962 536–577.
\bibitem[Kro]{Kro} \textsc{Kr\"oncke, K.,} \emph{Stability and instability of Ricci solitons}, arXiv:1403.3721v1, (2014).
\bibitem[Pal]{Pal} \textsc{Pali, N.,} \emph{Plurisubharmonic functions and positive $(1,
1)$-currents over almost complex manifolds}, Manuscripta Mathematica, volume
118, (2005) issue 3, pp. 311 - 337.
\bibitem[Pal1]{Pal1} \textsc{Pali, N.,} \emph{Characterization of Einstein-Fano manifolds via the
K\"ahler-Ricci flow}, arXiv:math/0607581v2, (2006), Indiana Univ. Math. J. 57, (2008), no. 7, 3241-3274.
\bibitem[Pal2]{Pal2} \textsc{Pali, N.,} 
\emph{A second variation formula for Perelman's $\mathcal{W}$-functional along 
the modified K\"ahler-Ricci flow}, arXiv:1201.0970v1, (2012), Mathematische Zeitschrift. 276 (2014), no. 1-2, 173-189.
\bibitem[Pal3]{Pal3} \textsc{Pali, N.,} 
\emph{The total second variation of Perelman's $\mathcal{W}$-functional}, arXiv:1201.0969v1. (2012),  Calc. Var. Partial Differential 
Equations 50 (2014), 
no. 1-2, 115–144. 
\bibitem[Pal4]{Pal4} \textsc{Pali, N.,} 
\emph{The Soliton-Ricci Flow over Compact Manifolds}, arXiv:1203.3682
\bibitem[Pal5]{Pal5} \textsc{Pali, N.,} 
\emph{The Soliton-K\"{a}hler-Ricci Flow over Fano Manifolds}, arXiv:1203.3684
\bibitem[Pal6]{Pal6} \textsc{Pali, N.,} 
\emph{Variation formulas for the complex components of the Bakry-Emery-Ricci endomorphism}, arXiv:math
\bibitem[Pal7]{Pal7} \textsc{Pali, N.,} 
\emph{The stability of the Soliton-K\"ahler-Ricci flow}, in preparation.
\bibitem[Per]{Per} \textsc{Perelman, G.,} 
\emph{The entropy formula for the Ricci flow 
and its geometric applications}, arXiv:math/0211159.
\bibitem[P-S1]{P-S1} \textsc{Phong, D.H., Sturm, J.,} 
\emph{On stability and the convergence of the K\"ahler-Ricci flow}, arXiv:0412185v1, (2004), J. Differential Geom. 72 (2006), no. 1, 149–168.
\bibitem[P-S-S-W2]{P-S-S-W2} \textsc{Phong, D.H., Song, J., Sturm, J., Weinkove, B.,} 
\emph{On the convergence of the modified K\"ahler-Ricci flow and solitons}, arXiv:0809.0941v1 (2008),  Comment. Math. Helv. 86 
(2011), no. 1, 91–112.
\bibitem[P-S-S-W3]{P-S-S-W3} \textsc{Phong, D.H., Song, J., Sturm, J., Weinkove, B.,} 
\emph{The K\"ahler-Ricci flow and the $\bar\partial$-operator on vector fields}, arXiv:0705.4048v2 (2008), 
J. Differential Geom. 81 (2009), no. 3, 631–647. 
\bibitem[Su-Wa]{Su-Wa}\textsc{Sun, S., Wang, Y.Q,} 
\emph{On the K\"ahler Ricci flow near a K\"ahler Einstein metric}, arXiv:1004.2018v3 (2013). To appear in J. Reine Angew. Math.
\bibitem[Ti-Zhu1]{Ti-Zhu1}\textsc{Tian, G., Zhu, X.H,} 
\emph{Convergence of Kähler-Ricci flow}, J. Amer. Math. Soc. 20 (2007), no. 3, 675–699.
\bibitem[Ti-Zhu2]{Ti-Zhu2}\textsc{Tian, G., Zhu, X.H,} 
\emph{Perelman's W-functional and stability of 
K\"ahler-Ricci flow}, arXiv:0801.3504v1. (2008).
\bibitem[Ti-Zhu3]{Ti-Zhu3}\textsc{Tian, G., Zhu, X.H,} 
\emph{Convergence of 
K\"ahler-Ricci flow on Fano Manifolds II}, arXiv:1102.4798v1. (2011). J. Reine Angew. Math. 678 (2013), 223–245.
\bibitem[T-Z-Z-Z]{T-Z-Z-Z}\textsc{Tian, G., Zhang, S., Zhang, Z.L., Zhu, X.H,} 
\emph{Perelman's Entropy and K\"ahler-Ricci Flow on a Fano Manifold}, 
arXiv:1107.4018. (2011), Trans. Amer. Math. Soc. 365 (2013), no. 12, 6669–6695.
\bibitem[Ti-Zha]{Ti-Zha}\textsc{Tian, G., Zhang, Z.L.,} 
\emph{Regularity of K\"ahler-Ricci flows on Fano manifolds}, arXiv:1310.5897v1, (2013).
\end{thebibliography}
\end{document}